\documentclass[aop,preprint]{imsart}

\RequirePackage{amsthm,amsmath,amsfonts,amssymb}
\RequirePackage[authoryear]{natbib}
\RequirePackage[colorlinks,citecolor=blue,urlcolor=blue]{hyperref}
\RequirePackage{graphicx}
\RequirePackage{mathtools}

\usepackage{dsfont}
\usepackage{booktabs,tabularx,array,longtable}
\usepackage{enumitem}

\startlocaldefs
\numberwithin{equation}{section}
\theoremstyle{plain}

\newtheorem{theorem}{Theorem}[section]
\newtheorem{lemma}[theorem]{Lemma}

\newtheorem{proposition}[theorem]{Proposition}

\theoremstyle{definition}

\newtheorem{example}{Example}

\newcommand{\I}{\uppercase\expandafter{\romannumeral1}}
\newcommand{\II}{\uppercase\expandafter{\romannumeral2}}
\newcommand{\E}{\mathbb E}
\newcommand{\Pp}{\mathbb P}
\newcommand{\R}{\mathbb R}

\newcommand{\Var}{\operatorname{Var}}
\newcommand{\tr}{\operatorname{tr}}
\newcommand{\rank}{\operatorname{rank}}

\newcommand{\dk}{d_K}
\newcommand{\1}{\mathbf 1}
\newcommand{\dd}{\,d}
\newcommand{\Normal}{\mathcal N(0,1)}

\newcommand{\Cov}{\operatorname{Cov}}
\newcommand{\Span}{\operatorname{span}}
\newcommand{\cF}{\mathcal F}
\newcommand{\cV}{\mathcal V}

\endlocaldefs

\begin{document}

\begin{frontmatter}
\title{Universality transitions in the logarithmic law of random determinants beyond the fourth moment}
\runtitle{Universality transitions beyond the fourth moment}

\begin{aug}
\author[A]{\fnms{Song-hao}~\snm{Liu}\ead[label=e1]{liusonghao@dlut.edu.cn}},
\author[B]{\fnms{Qi-man}~\snm{Shao}\ead[label=e2]{shaoqm@sustech.edu.cn}},
\author[C]{\fnms{Hao}~\snm{Shi}\ead[label=e3]{Hao.Shi@link.cuhk.edu.hk}}
\and
\author[C]{\fnms{Jing-yu}~\snm{Xu}\ead[label=e4]{12131253@mail.sustech.edu.cn}}


\address[A]{School of Mathematical Sciences, Dalian University of Technology, Dalian, Liaoning, China \printead[presep={,\ }]{e1}}

\address[B]{Department of Statistics and Data Science, Shenzhen International Center for Mathematics, Southern University of Science and Technology, Shenzhen, Guangdong, 518055, China \printead[presep={,\ }]{e2}}

\address[C]{Department of Statistics and Data Science, Southern University of Science and Technology, Shenzhen, Guangdong, 518055, China \printead[presep={,\ }]{e3,e4}}

\end{aug}

\begin{abstract}
For real random matrices $A_n=(a_{ij})_{1\le i,j\le n}$
with independent centered entries of unit variance,
\citet{BaoPanZhou2015} established the classical logarithmic
law under a uniform fourth-moment condition:
\[
 \frac{\log|\det A_n|-\frac12\log(n-1)!}
 {\sqrt{\frac12\log n}}
 \Rightarrow \mathcal N(0,1).
\]
We find that this Gaussian universality can be extended beyond
the fourth-moment condition, with two distinct transitions in the
center and fluctuation scale as the moment condition gradually weaken.

The first critical condition is the uniform integrability of the family
$\{|a_{ij}|^4/\sqrt{\log(e+|a_{ij}|)}\}_{n,i,j}$.
The classical law remains valid under this weaker condition,
with no change in center or scale.
As the moment condition is weakened further, the classical
center can fail.
Replacing the classical center $\frac12\log(n-1)!$ by the
actual mean $\mathbb E\log|\det A_n|$ preserves the standard
Gaussian limit at the original scale and allows the moment
condition to be relaxed to a second critical condition: uniform
integrability of $\{|a_{ij}|^4/\log(e+|a_{ij}|)\}_{n,i,j}$.
Under still weaker moment conditions, the classical scale
$\sqrt{\frac12\log n}$ can also fail.
Replacing this scale in turn by the actual standard
deviation $\sqrt{\operatorname{Var}(\log|\det A_n|)}$
yields a standard Gaussian limit whenever the expectations
of $|a_{ij}|^4/\{\log(e+|a_{ij}|)\}^{\alpha}$
are uniformly bounded over $n,i,j$, for any fixed
$\alpha\ge0$. The results using the actual mean or standard deviation
additionally assume uniformly bounded entry densities.

We establish the sharpness of the first two critical conditions via counterexamples.
We further obtain Berry--Esseen bounds for all three
normalizations, with optimal rates under a uniform
fourth-moment bound.
\end{abstract}

\end{frontmatter}

\section{Introduction}
Let $A_n=(a_{ij})_{1\le i,j\le n}$ be an $n\times n$ random
matrix with independent real entries satisfying
$\E a_{ij}=0$ and $\E a_{ij}^2=1$.
The study of random determinants $\det A_n$ goes back to
\cite{SzekeresTuran1937}.
Subsequent work investigated their moments and distributions;
see \cite{ForsytheTukey1952,NyquistRiceRiordan1954,Prekopa1967,Dembo1989}
and the monograph \cite{GirkoBook1990}.

For matrices with independent symmetric Bernoulli entries,
\cite{TaoVu2006} proved that
\[
 \sqrt{n!}\exp\{-c\sqrt{n\log n}\}
 \le |\det A_n|
 \le \sqrt{n!}\,\omega(n)
\]
with probability tending to one, for an absolute constant $c>0$
and every function $\omega(n)\to\infty$.
These bounds identify the leading order
$\log|\det A_n|=(\frac12+o(1))n\log n$
with high probability, but do not determine the limiting law
of the fluctuations.

When the entries of $A_n$ are independent standard Gaussian variables,
the Gram--Schmidt decomposition and rotational invariance give
\begin{equation}\label{eq:Gaussial-decompose}
 \log|\det A_n|
 \overset{d}=\frac12\sum_{k=1}^{n}\log\chi_k^2,
\end{equation}
where $\chi_1^2,\ldots,\chi_n^2$ are independent chi-square variables
with respective degrees of freedom $1,\ldots,n$.
This yields the logarithmic law
\begin{equation}\label{eq:CLT-for-Gaussian-case}
 \frac{\log|\det A_n|-\frac12\log(n-1)!}
 {\sqrt{\frac12\log n}}
 \Rightarrow\Normal;
\end{equation}
see, for example, \cite[Theorem~3.5]{Rouault2007}.
See also \cite{Goodman1963} for complex Wishart determinants
and \cite{DumitriuEdelman2002} for tridiagonal matrix models
extending the classical Gaussian and Wishart ensembles.

It can be shown that 
\begin{equation}
  \begin{aligned}\label{eq:classic-Gaussian-log-law-BE-bound}
    \sup_{x\in\mathbb{R}}\biggl|\mathbb{P}\biggl(
    \frac{\log(|\det A_n|)-\frac12\log(n-1)!}{\sqrt{\frac12\log n}}
    \le x\biggr)-\mathbb{P}(\Normal \le x)\biggr|
    \asymp C\log^{-1/2} n,   
  \end{aligned}
\end{equation}
and, under exact-mean centering,
\begin{equation}
  \begin{aligned}\label{eq:Gaussian-exact-BE-bound}
    \sup_{x\in\mathbb{R}}\biggl|\mathbb{P}\biggl(
    \frac{ \log(|\det A_n|)-\E \log(|\det A_n|) }{\sqrt{\frac12\log n}}
    \le x\biggr)-\mathbb{P}(\Normal \le x)\biggr|
    \asymp C\log^{-1} n.
  \end{aligned}
\end{equation}
Furthermore, under exact-variance normalization, we have
\begin{equation}
  \begin{aligned}\label{eq:Gaussian-exact-var-BE-bound}
    \sup_{x\in\mathbb{R}}\biggl|\mathbb{P}\biggl(
    \frac{ \log(|\det A_n|)-\E \log(|\det A_n|) }{\sqrt{\Var(\log(|\det A_n|))}}
    \le x\biggr)-\mathbb{P}(\Normal \le x)\biggr|
    \asymp C\log^{-3/2} n.
  \end{aligned}
\end{equation}

The universality problem asks whether the logarithmic law
\eqref{eq:CLT-for-Gaussian-case} remains valid for non-Gaussian
entries. Early claims appear in \cite{Girko1979,Girko1997},
the latter under a bounded $(4+\delta)$th moment condition
for some $\delta>0$; see \cite{tao2012central} for a discussion
of points left unclear in these proofs.
\cite{NguyenVu2014} established this law under a uniform
subexponential tail condition and obtained a Berry--Esseen bound of $(\log n)^{-1/3+o(1)}$.

\cite{BaoPanZhou2015} subsequently proved the same law under
the weaker condition
\begin{equation}\label{eq:assumption-uniform-4-moment}
 \sup_{n\ge1}\max_{1\le i,j\le n}\E|a_{ij}|^4<\infty.
\end{equation}
The results above establish that the uniform fourth-moment
bound is sufficient for the logarithmic law, but do not show
that it is optimal. \citet{li2026logarithmic} describe this condition as the
``sharp finite fourth moment condition'', without providing
an accompanying proof of optimality.

We first ask whether the classical logarithmic law
\eqref{eq:CLT-for-Gaussian-case} remains valid under weaker
moment conditions and, if so, how far these conditions can
be relaxed while retaining the same center and scale.
We then ask whether Gaussian fluctuations can persist under
still weaker moment conditions with a suitable choice of
center and scale.

In this paper, we find, surprisingly, that the Gaussian
universality of the classical logarithmic law can be extended
beyond the fourth-moment condition.
Our results further reveal two critical transitions in the
center and fluctuation scale as the moment condition is
gradually weakened.

First, we prove that the moment condition for the classical
logarithmic law \eqref{eq:CLT-for-Gaussian-case} can be relaxed
from a uniform fourth-moment bound to the first critical
condition, namely uniform integrability of
$\{|a_{ij}|^4/\sqrt{\log(e+|a_{ij}|)}\}_{n,i,j}$
(Theorem~\ref{thm:det-clt}).
Within this range, neither the classical center nor the
classical scale needs to be changed.
As the moment condition is weakened further, however,
the classical center may be the first to fail.

With uniformly bounded entry densities, using the actual
mean $\mathbb E\log|\det A_n|$ in place of the classical
center $\frac12\log(n-1)!$ allows us to retain the original
scale $\sqrt{\frac12\log n}$ and further relax the moment
condition to the second critical condition, namely uniform
integrability of
$\{|a_{ij}|^4/\log(e+|a_{ij}|)\}_{n,i,j}$
(Theorem~\ref{thm:exact-clt}).
However, under weaker moment conditions, changing the center
alone may no longer suffice, and the classical scale can
also fail.

At this point, replacing the classical scale in turn by the
actual standard deviation
$\sqrt{\operatorname{Var}(\log|\det A_n|)}$
allows Gaussian universality to extend further.
Under the same density condition, for any fixed
$\alpha\ge0$, if the expectations of
$|a_{ij}|^4/\{\log(e+|a_{ij}|)\}^{\alpha}$
are uniformly bounded over all entries and dimensions,
then centering by the actual mean and normalizing by the
actual standard deviation still yield a standard Gaussian
limit (Theorem~\ref{thm:variance-clt}).
Our counterexamples also establish the endpoint sharpness
of the first two conditions, showing that
adjustments to the center and scale can indeed be necessary
at the respective endpoints
(Examples~\ref{ex:det-threshold} and~\ref{ex:exact-threshold}).

We also establish Berry--Esseen bounds for the three
normalizations (Theorems~\ref{thm:det-be}, \ref{thm:exact-be},
and \ref{thm:variance-be}).
Under a uniform fourth-moment bound, these bounds match the
three Gaussian rates \eqref{eq:classic-Gaussian-log-law-BE-bound}--\eqref{eq:Gaussian-exact-var-BE-bound} stated above.
Under weaker moment conditions, the rates depend on the
logarithmic moment condition.
For normalization by the actual mean and standard deviation,
Example~\ref{ex:variance-rate-sharpness} gives matching lower
bounds for every fixed $\alpha\ge0$.
Thus the rate in Theorem~\ref{thm:variance-be} is optimal
under each of these moment conditions.

The rest of this paper is organized as follows.
Section~\ref{sec:main-result} states the main results and
presents examples showing the sharpness of the moment
conditions and convergence rates.
Section~\ref{sec:proof-main-results} proves the six main
theorems using auxiliary propositions, which are proved in
Section~\ref{sec:proof-propositions}.
Appendix~\ref{sec:proof-lemmas} contains the proofs of the
auxiliary lemmas, and Appendix~\ref{sec:proof-examples} gives
the proofs of the sharpness examples.

\section{Main results}\label{sec:main-result}

\subsection{Assumptions and basic notation}
A \emph{triangular array of random matrices} is a sequence $A=(A_n)_{n\ge1}$ with $A_n=(a_{ij})_{1\le i,j\le n}$. 
For every $n$, $A_n$ is an $n\times n$ random matrix, the entries $\{a_{ij}:1\le i,j\le n\}$  are independent real random variables. It is noted that $a_{ij}$ may also depend on $n$.

Assume that
\begin{equation}\label{eq:moment-control-class}
 \E a_{ij}=0,
 \qquad
 \E\bigl(a_{ij}^2\bigr)=1.
\end{equation}
Let $\mathcal L_{n} = \log|\det A_n|$ and $\sigma_n:=\sqrt{\Var(\mathcal L_n)}$. On the event $\{\det A_n=0\}$, we set $\mathcal L_{n}=-\infty$.  The three normalized statistics are
\begin{align}
 W_n^{\mathrm d}(A_n)
 :=\frac{ \mathcal L_{n} - \frac12\log(n-1)!}
 {\sqrt{\frac12\log n}},
 \quad
 W_n^{\mathrm e}(A_n)
 :=\frac{ \mathcal L_{n} - \E\mathcal L_{n}}
 {\sqrt{\frac12\log n}},
 \quad
 W_n^{\mathrm{v}}(A_n)
 :=\frac{\mathcal L_n-\E\mathcal L_n}{\sigma_n}.
\end{align}
The superscripts ``$\mathrm d$'', ``$\mathrm e$'' and ``$\mathrm v$'' stand for deterministic centering, exact-mean centering and exact-variance normalization. 
For real valued random variables $X$ and $Y$, the Kolmogorov distance is defined as
\begin{equation}
  \begin{aligned}\label{def:Kol-distance}
    d_{\mathrm K}(X,Y)
    :=\sup_{x\in\mathbb R}|\Pp(X\le x)-\Pp(Y\le x)|. 
  \end{aligned}
\end{equation}

\subsection{Main theorem}

\begin{theorem}\label{thm:det-clt}
Let $A=(A_n)_{n\ge1}$ be a triangular array of random matrices such that, for each $n$, the entries of $A_n$ are independent, satisfying \eqref{eq:moment-control-class}.
If the family
\begin{equation}
  \begin{aligned}\label{eq:det-clt-assumption}
    \left\{
    \frac{|a_{ij}|^4}
    {\sqrt{\log(e+|a_{ij}|)}}:
    n\ge2,\ 1\le i,j\le n
    \right\} \quad \text{is uniformly integrable,}
  \end{aligned}
\end{equation}
then
\begin{equation}
  \begin{aligned}\label{eq:det-clt-conclusion}
    W_n^{\mathrm d}(A_n)\Rightarrow\Normal.    
  \end{aligned}
\end{equation}
\end{theorem}

The next example shows that the uniform-integrability assumption in \eqref{eq:det-clt-assumption} cannot be weakened to boundedness in $L^1$ of the same family.

\begin{example}\label{ex:det-threshold}
Let $Y$ be symmetric with variance one and a compactly supported
$C^\infty$ density. For $m\ge1$, set $B_m=e^m$,
$L_m=\log(e+B_m)$, $N_m=\lfloor B_m^4/\sqrt{L_m}\rfloor$,
$p_m=N_m^{-1}$, and $\tau_m=p_mB_m^2$.
Let $Y,\eta_m,\varepsilon_m$ be independent, where $\eta_m$
is Bernoulli with parameter $p_m$ and
$\Pp(\varepsilon_m=1)=\Pp(\varepsilon_m=-1)=1/2$, and define
\begin{equation}\label{eq:det-example-spike-variable}
 \xi_m=\frac{Y+\eta_m\varepsilon_mB_m}{\sqrt{1+\tau_m}}.
\end{equation}
At $n=N_m$, let all entries of $A_n$ be independent, with the
first $\lfloor n/2\rfloor$ rows consisting of copies of $\xi_m$
and the remaining rows of standard normal variables.
At all other dimensions, take independent standard normal entries.
Then all entries are centered with variance one, and
\begin{equation}\label{eq:det-example-critical-bound}
 \sup_{n\ge2}\max_{1\le i,j\le n}
 \E\frac{|a_{ij}|^4}{\sqrt{\log(e+|a_{ij}|)}}<\infty.
\end{equation}
The family $\left\{ |a_{ij}|^4 / \sqrt{\log(e+|a_{ij}|)}: n\ge2,\ 1\le i,j\le n \right\}$ is not uniformly
integrable, and
\begin{equation}\label{eq:det-example-failure}
 \liminf_{m\to\infty}
 \dk\bigl(W_{N_m}^{\mathrm d}(A_{N_m}),\Normal\bigr)>0.
\end{equation}
Consequently, the uniform-integrability assumption in
Theorem~\ref{thm:det-clt} cannot, in general, be replaced by the
$L^1$-boundedness condition \eqref{eq:det-example-critical-bound}.
The proof is given in Appendix~\ref{sec:proof-examples}.
\end{example}

Motivated by \eqref{eq:Gaussian-exact-BE-bound}, we now consider the exact-mean centering $W_n^{\mathrm e}(A_n)$.

\begin{theorem}\label{thm:exact-clt}
Let $A=(A_n)_{n\ge1}$ be a triangular array of random matrices such that, for each $n$, the entries of $A_n$ are independent, centered, real-valued random variables with unit variance and uniformly bounded density $\sup\limits_{n\ge2}\max\limits_{1\le i,j\le n}\|f_{ij}^{(n)}\|_\infty\le M_0<\infty$.
If the family
\begin{equation}
  \begin{aligned}\label{eq:exact-clt-assumption}
    \left\{\frac{|a_{ij}|^{4}}{\log(e+|a_{ij}|)}: n\geq 2, 1\leq i,j\leq n \right\}  \quad \text{is uniformly integrable,}
  \end{aligned}
\end{equation}
then
\begin{equation}
  \begin{aligned}\label{eq:exact-clt-conclusion}
    W_n^{\mathrm e}(A_n)\Rightarrow\Normal.
  \end{aligned}
\end{equation}
\end{theorem}

Again, the following example shows that condition \eqref{eq:exact-clt-assumption} is sharp.

\begin{example}\label{ex:exact-threshold}
Use the construction in Example~\ref{ex:det-threshold}, now with $N_m=\left\lfloor\frac{B_m^4}{L_m}\right\rfloor$, 
and with $p_m$, $\tau_m$, $\eta_m$, $\xi_m$, and $A_n$
redefined accordingly.
The entries are centered with variance one and have uniformly
bounded $C^\infty$ densities. Moreover,
\begin{equation}\label{eq:exact-example-critical-bound}
 \sup_{n\ge2}\max_{1\le i,j\le n}
 \E\frac{|a_{ij}|^4}{\log(e+|a_{ij}|)}<\infty,
\end{equation}
but the family $ \left\{ |a_{ij}|^4/ \log(e+|a_{ij}|): n\ge2,\ 1\le i,j\le n \right\}$ is not uniformly integrable. There exist $m_k\uparrow\infty$ and
$\sigma^2\in[17/16,\,1+(\log2)/8]$ such that
\begin{equation}\label{eq:exact-example-limit}
 W_{N_{m_k}}^{\mathrm e}(A_{N_{m_k}})
 \Rightarrow\mathcal N(0,\sigma^2).
\end{equation}
Consequently, the uniform-integrability assumption in
Theorem~\ref{thm:exact-clt} cannot, in general, be replaced by the
requirement that the family in \eqref{eq:exact-clt-assumption} be
bounded in $L^1$.
The proof is given in Appendix~\ref{sec:proof-examples}.
\end{example}

Then, motivated by \eqref{eq:Gaussian-exact-var-BE-bound}, we consider the exact-variance normalization $W_n^{\mathrm v}(A_n)$.

\begin{theorem}\label{thm:variance-clt}
Let $A=(A_n)_{n\ge1}$ be a triangular array of random matrices,
and let $\alpha\ge0$ be fixed.
Suppose that, for each $n$, the entries of $A_n$ are independent,
centered, real-valued random variables with unit variance and
densities $f_{ij}^{(n)}$ satisfying
\[
 \sup_{n\ge2}\max_{1\le i,j\le n}
 \|f_{ij}^{(n)}\|_\infty\le M_0<\infty.
\]
Suppose that there is a constant $K<\infty$ such that
\begin{equation}\label{eq:variance-clt-assumption}
 \sup_{n\ge2}\max_{1\le i,j\le n}
 \E\frac{|a_{ij}|^4}
 {\{\log(e+|a_{ij}|)\}^{\alpha}}
 \le K.
\end{equation}
Then $0<\sigma_n^2<\infty$, and
\begin{equation}\label{eq:variance-clt-conclusion}
 W_n^{\mathrm v}(A_n)\Rightarrow\Normal.
\end{equation}
\end{theorem}

Next we give the Berry--Esseen bounds for the deterministic centerings, exact-mean centerings and exact-variance normalizations, respectively.

\begin{theorem}\label{thm:det-be}
Let $A=(A_n)_{n\ge1}$ be a triangular array of random matrices such that, for each $n$, the entries of $A_n = (a_{ij})_{1\leq i,j\leq n}$ are independent, real-valued random variables and let $0<\delta\le\frac12$. Suppose that there is a positive constant $K<\infty$ such that for all $n\geq 2$, $1\le i,j\le n$
\begin{equation}\label{eq:det-be-assumption}
 \E a_{ij}=0,\qquad
 \E a_{ij}^2=1,\qquad
 \sup_{n,i,j}\E\frac{|a_{ij}|^4}
 {\{\log(e+|a_{ij}|)\}^{1/2-\delta}}\le K<\infty.
\end{equation}
Then, for every $n\ge 2$,
\begin{equation}\label{eq:det-be-bound}
 \dk(W_n^{\mathrm d}(A_n),\Normal)
 \le C(K,\delta)(\log n)^{-\delta},
\end{equation}
where $C(K,\delta)$ is a finite constant depending on $K$ and $\delta$ only.
\end{theorem}

\begin{theorem}\label{thm:exact-be}
Let $A=(A_n)_{n\ge1}$ be a triangular array of random matrices such that, for each $n$, the entries of $A_n = (a_{ij})_{1\leq i,j\leq n}$ are independent, real-valued random variables and let $0<\gamma\le1$. Suppose that there are two positive constants $K, M_{0}<\infty$ such that for all $n\geq 2$, $1\le i,j\le n$ 
\begin{equation}\label{eq:exact-be-assumption}
 \E a_{ij}=0,\qquad
 \E a_{ij}^2=1,\qquad
 \sup_{n,i,j}\E\frac{|a_{ij}|^4}
 {\{\log(e+|a_{ij}|)\}^{1-\gamma}}\le K<\infty,
\end{equation}
and $a_{ij}$ has a density $f_{ij}^{(n)}$ satisfying $\|f_{ij}^{(n)}\|_\infty\le M_0$.
Then,
\begin{equation}\label{eq:exact-be-bound}
 \dk(W_n^{\mathrm e}(A_n),\Normal)
 \le C(K,M_0,\gamma)(\log n)^{-\gamma},
\end{equation}
where $C(K,M_0,\gamma)$ is a finite constant depending on $K$, $M_0$ and $\gamma$ only.
\end{theorem}

\begin{theorem}\label{thm:variance-be}
Let $A=(A_n)_{n\ge1}$ be a triangular array of random matrices such that,
for each $n$, the entries of $A_n=(a_{ij})_{1\le i,j\le n}$ are independent,
real-valued random variables, and let $\alpha\ge0$ be fixed.
Suppose that there are two positive constants $K,M_0<\infty$ such that,
for all $n\ge2$ and $1\le i,j\le n$,
\begin{equation}\label{eq:variance-be-assumption}
 \E a_{ij}=0,\qquad
 \E a_{ij}^2=1,\qquad
 \sup_{n\ge2}\max_{1\le i,j\le n}
 \E\frac{|a_{ij}|^4}
 {\{\log(e+|a_{ij}|)\}^{\alpha}}
 \le K<\infty,
\end{equation}
and $a_{ij}$ has a density $f_{ij}^{(n)}$ satisfying
$\|f_{ij}^{(n)}\|_\infty\le M_0$.
Then $0<\sigma_n^2<\infty$, and
\begin{equation}\label{eq:variance-be-bound}
 \dk(W_n^{\mathrm v}(A_n),\Normal)
 \le C(K,M_0,\alpha)
 \min\left\{
 (\log n)^{\alpha-3/2},
 \frac{\log(e+\log n)}{\sqrt{\log n}}
 \right\},
\end{equation}
where $C(K,M_0,\alpha)$ is a finite constant depending on $K$, $M_0$
and $\alpha$ only.
\end{theorem}

At $\delta=1/2$, Theorem~\ref{thm:det-be} gives the classical order
$(\log n)^{-1/2}$ under a uniform fourth moment.  At $\gamma=1$,
Theorem~\ref{thm:exact-be} gives the sharper exact-centered order $(\log n)^{-1}$. At $\alpha=0$, Theorem~\ref{thm:variance-be} gives the sharper order
$(\log n)^{-3/2}$. All three rates are optimal even for Gaussian entries.

\begin{example}
\label{ex:variance-rate-sharpness}
Fix $\alpha>0$. For all sufficiently large $n$, set
\begin{equation}\label{eq:variance-example-spike-parameters}
 (p_n,b_n)=
 \begin{cases}
 \displaystyle
 \left(\frac{(\log n)^\alpha}{n^2},\,\sqrt n\right),
 &0<\alpha\le1,\\[3mm]
 \displaystyle
 \left(
 \frac{\log n}{n^2(\log\log n)^2},\,
 \sqrt n\,(\log n)^{(\alpha-1)/4}\sqrt{\log\log n}
 \right),
 &\alpha>1.
 \end{cases}
\end{equation}
Let $G,\eta_n,\varepsilon_n$ be independent, with $G$ standard
normal, $\eta_n$ Bernoulli with parameter $p_n$, and
$\Pp(\varepsilon_n=1)=\Pp(\varepsilon_n=-1)=1/2$.
Let the entries of $A_n$ be independent copies of
\begin{equation}\label{eq:variance-example-spike-variable}
 \xi_n=\frac{G+\eta_n\varepsilon_n b_n}{\sqrt{1+p_nb_n^2}}.
\end{equation}
Use independent standard normal entries at the remaining
finitely many dimensions.
The entries have uniformly bounded $C^\infty$ densities and
satisfy \eqref{eq:variance-be-assumption} for some $K<\infty$.
For some $c=c(\alpha)>0$ and all sufficiently large $n$,
\begin{equation}\label{eq:variance-rate-lower-bound}
 \dk(W_n^{\mathrm v}(A_n),\Normal)
 \ge c\min\left\{
 (\log n)^{\alpha-3/2},
 \frac{\log(e+\log n)}{\sqrt{\log n}}
 \right\}.
\end{equation}
For $\alpha=0$, independent standard normal entries give the
same lower bound by \eqref{eq:Gaussian-exact-var-BE-bound}.
The proof is given in Appendix~\ref{sec:proof-examples}.
\end{example}

\section{Proof of main results}\label{sec:proof-main-results}
First, we give a lemma that reformulates the uniform-integrability
conditions in Theorems~\ref{thm:det-clt} and \ref{thm:exact-clt}.
The following lemma is a consequence of the de la Vallée--Poussin
criterion for uniform integrability; see, for example,
\cite[Theorem~6.19]{Klenke2014}.

\begin{lemma}\label{lem:ui-moment-control}
Let $A=(A_n)_{n\ge1}$ be a triangular array, where
$A_n=(a_{ij})_{1\le i,j\le n}$.
Fix $\alpha\in\{\frac12,1,\frac32\}$ and set
$q(x)=\{\log(e+x)\}^{\alpha}$ for $x\ge0$.
Suppose that the family
$\{|a_{ij}|^4/q(|a_{ij}|): n\ge2,\ 1\le i,j\le n\}$
is uniformly integrable.
Then there exist a constant $K<\infty$
and a finite-valued nondecreasing function
$\psi:[0,\infty)\to[1,\infty)$ such that
\begin{equation}\label{eq:equal-condition-to-uni-integ-condition}
 \psi(x)=o(q(x))
 \quad (x\to\infty),
 \qquad
 \text{and}
 \qquad
 \sup_{n\ge2}\max_{1\le i,j\le n}
 \E\frac{|a_{ij}|^4}{\psi(|a_{ij}|)}
 \le K.
\end{equation}
\end{lemma}

\subsection{Proof of Theorem~\ref{thm:det-clt}}\label{subsec:proof-det-clt}
Before giving the proof of Theorem~\ref{thm:det-clt}, we first give a truncation result that preserves the mean and variance exactly and changes each entry only on a rare event.
 
To state a single truncation result that applies to all three
normalizations, fix $\alpha\ge0$ and let
$\psi:[0,\infty)\to[1,\infty)$ be finite-valued and
nondecreasing. Define
\begin{align}
 q(x)
 &:=\{\log(e+x)\}^{\alpha},
 \qquad x\ge0,
 \label{eq:truncation-two-functions}\\
 \varepsilon_{\psi,q}(T)
 &:=\sup_{x\ge T}\frac{\psi(x)}{q(x)},
 \qquad T>0.
 \label{eq:epsilon-definition}
\end{align}
We assume that $\psi(x)=O(q(x))$ as $x\to\infty$,
so $\varepsilon_{\psi,q}(T)$ is finite for every $T>0$.
If, in addition, $\psi(x)=o(q(x))$ as $x\to\infty$,
then $\varepsilon_{\psi,q}(T)\to0$ as $T\to\infty$.
The required truncation is given by the following proposition.

\begin{proposition}[Moment-preserving truncation]
\label{prop:truncation}
Let $q$ be defined by \eqref{eq:truncation-two-functions}, and let
$\psi:[0,\infty)\to[1,\infty)$ be a finite-valued nondecreasing
function satisfying $\psi(x)=O(q(x))$ as $x\to\infty$.
Suppose that
\begin{equation}\label{eq:truncation-assumptions}
 \E X=0,\qquad
 \E X^2=1,\qquad
 \E\frac{|X|^4}{\psi(|X|)}\le K.
\end{equation}
For every sufficiently large $T$, one can couple $X$ to a mean-zero
and variance-one random variable $\widetilde X_T$ such that
\begin{align}
 |\widetilde X_T|
 &\le2T,
 &
 \Pp(\widetilde X_T\ne X)
 &\le
 CK\varepsilon_{\psi,q}(T)\frac{q(T)}{T^4},
 \label{eq:discrete-truncation-coupling}\\
 \E|\widetilde X_T|^4
 &\le
 CK\varepsilon_{\psi,q}(T)q(T),
 &
 \E|\widetilde X_T|^3
 &\le
 C\E|X|^3
 \le
 C_{K,\psi}.
 \label{eq:discrete-truncation-moments}
\end{align}
Here $C$ is an absolute constant, whereas $C_{K,\psi}<\infty$
depends only on $K$ and $\psi$.
The required lower bound on $T$ may depend on $K,\psi,\alpha$.

If, in addition, $X$ has a density $f_X$ satisfying
$\|f_X\|_\infty\le M_0$, then the coupling can be chosen so that
$\widetilde X_T$ is absolutely continuous.
The probability bound in \eqref{eq:discrete-truncation-coupling}
and both moment bounds in \eqref{eq:discrete-truncation-moments}
remain valid, while the support bound is replaced by
$|\widetilde X_T|\le3T$.
Moreover,
\begin{equation}\label{eq:continuous-truncation-density}
 \|f_{\widetilde X_T}\|_\infty
 \le M_0+\frac{C}{T}.
\end{equation}
\end{proposition}

The proof is given in Section~\ref{subsec:proof-prop-truncation}.  The proposition is preferable to an ordinary truncation followed by recentering because it leaves the original variable unchanged outside a rare event. Then, we give a proposition that gives a normal approximation for a matrix whose fourth moments may grow with $n$, but remain below the scale $\sqrt{\log n}$.

\begin{proposition}\label{prop:det-growing-fourth}
Let $B_n=(b_{ij})_{1\leq i,j \leq n}$ be an $n\times n$ random matrix with independent entries satisfying
\begin{equation}\label{eq:det-growing-assumptions}
 \E b_{ij}=0,\qquad \E b_{ij}^2=1,\qquad
 \max_{i,j}\E|b_{ij}|^3\le R_3,\qquad
 M:=\max\left\{3,\max_{i,j}\E|b_{ij}|^4\right\}.
\end{equation}
If $M\le K_0\sqrt{\log n}$, for some absolute constant $K_0>0$, then
\begin{equation}\label{eq:det-growing-bound}
 \dk(W_n^{\mathrm d}(B_n),\Normal)
 \le C(K_0,R_3)\left\{\frac{M}{\sqrt{\log n}}\right\}.
\end{equation}
\end{proposition}

Its proof is given in Section~\ref{subsec:proof-prop-det-growing-fourth}.

\begin{proof}[Proof of Theorem~\ref{thm:det-clt}]
For the deterministic-centering case, take
$q(x)=\sqrt{\log(e+x)}$.
By the uniform-integrability condition
\eqref{eq:det-clt-assumption} and
Lemma~\ref{lem:ui-moment-control}, there exist a constant
$K<\infty$ and a nondecreasing function
$\psi:[0,\infty)\to[1,\infty)$ such that
\eqref{eq:equal-condition-to-uni-integ-condition} holds
and $\varepsilon_{\psi,q}(T)\to0$ as $T\to\infty$.
Throughout the rest of the paper, write $\ell=\log n$.
Choose the truncation level
\begin{equation}\label{eq:det-clt-truncation-level}
 T_n=\sqrt n\,\ell^{1/8}.
\end{equation}
Apply the first part of Proposition~\ref{prop:truncation} independently to all entries and call the resulting truncated matrix $\widetilde A_n$.  Since $T_n^4=n^2\ell^{1/2}$ and $q(T_n)\asymp\sqrt\ell$, \eqref{eq:discrete-truncation-coupling} gives
\begin{equation}\label{eq:det-clt-global-coupling}
 \Pp(A_n\ne\widetilde A_n)
  \le n^{2}CK \varepsilon_{\psi,q}(T_n) \frac{q(T_{n})}{T_{n}^{4}}
 \le C\varepsilon_{\psi,q}(T_n)\longrightarrow0.
\end{equation}
The fourth-moment estimate in
\eqref{eq:discrete-truncation-moments} yields
\begin{equation}\label{eq:det-clt-effective-fourth}
 \max_{i,j}\E|\widetilde a_{ij}|^4
 \le C\varepsilon_{\psi,q}(T_n)\sqrt\ell
 =o(\sqrt\ell),
\end{equation}
and the third moments are uniformly bounded.
Since
$M=\max\{3,\max_{i,j}\E|\widetilde a_{ij}|^4\}
=o(\sqrt\ell)$,
Proposition~\ref{prop:det-growing-fourth} gives
\begin{equation}\label{eq:det-clt-truncated-limit}
 \dk(W_n^{\mathrm d}(\widetilde A_n),\Normal)
 \longrightarrow0.
\end{equation}
On the event $\{A_n=\widetilde A_n\}$, we have
$W_n^{\mathrm d}(A_n)=W_n^{\mathrm d}(\widetilde A_n)$.
Thus, by \eqref{eq:det-clt-truncated-limit} and
\eqref{eq:det-clt-global-coupling},
\[
 \dk(W_n^{\mathrm d}(A_n),\Normal)
 \le
 \dk(W_n^{\mathrm d}(\widetilde A_n),\Normal)
 +\Pp(A_n\ne\widetilde A_n)
 \longrightarrow0.
\]
Thus, we complete the proof of Theorem~\ref{thm:det-clt}.
\end{proof}

\subsection{Proof of Theorem~\ref{thm:exact-clt}}\label{subsec:proof-exact-clt}
We use the continuous part of Proposition~\ref{prop:truncation}
together with the following normal approximation result.

\begin{proposition}\label{prop:exact-growing-fourth}
Let $B_n=(b_{ij})_{1\leq i,j \leq n}$ be $n\times n$ random matrix with independent entries satisfying
\begin{equation}\label{eq:exact-growing-assumptions}
 \E b_{ij}=0,\qquad \E b_{ij}^2=1,\qquad
 \max_{i,j}\E|b_{ij}|^3\le R_3,
 \qquad \|f_{b_{ij}}\|_\infty\le M_1.
\end{equation}
Suppose that
\begin{equation}\label{eq:exact-growing-support}
 |b_{ij}|\le (K_0 n\ell^{3/4})^{1/2},
\end{equation}
for some fixed constant $K_0>0$, and put
\begin{equation}\label{eq:exact-growing-fourth}
 M:=\max\left\{3,\max_{i,j}\E|b_{ij}|^4\right\}.
\end{equation}
If $M=o(\ell)$, then we have 
\begin{equation}\label{eq:exact-growing-asymptotic}
 W_n^{\mathrm e}(B_n)\Rightarrow\Normal.
\end{equation}
Moreover, if $0<\gamma\le1$ and $M\le K_1\ell^{1-\gamma}$, then
\begin{equation}\label{eq:exact-growing-bound}
 \dk(W_n^{\mathrm e}(B_n),\Normal)
 \le C\frac{M}{\ell},
\end{equation}
where $C$ depends only on $\gamma,R_3,M_1,K_0,K_1$.
\end{proposition}

The proof is given in Section~\ref{subsec:proof-prop-exact-growing-fourth}.

\begin{proof}[Proof of Theorem~\ref{thm:exact-clt}]
For the exact-centering case, take $q(x)=\log(e+x)$. By the uniform-integrability assumption \eqref{eq:exact-clt-assumption} and
Lemma~\ref{lem:ui-moment-control}, there exist a constant $K<\infty$
and a finite-valued nondecreasing function
$\psi:[0,\infty)\to[1,\infty)$ such that \eqref{eq:equal-condition-to-uni-integ-condition} holds and $\varepsilon(T)\to0$ when $T\to \infty$, where $\varepsilon(T)$ is defined in \eqref{eq:epsilon-definition}.

We choose a proper truncation level $T_n$ that grows with $n$ as
\begin{equation}\label{eq:exact-clt-truncation-level}
 T_n=\sqrt n\,\ell^{3/8}.
\end{equation}
Apply the continuous part of Proposition~\ref{prop:truncation} independently to all entries.  Denote the truncated matrix by $\widetilde A_n$, let $a_{i}$ and $\widetilde a_{i}$ denote the $i$th rows of $A_n$ and $\widetilde A_n$, respectively.
Because $T_n^4=n^2\ell^{3/2}$ and $q(T_n)\asymp\ell$, \eqref{eq:discrete-truncation-coupling}--\eqref{eq:continuous-truncation-density} give
\begin{align}
 \max_{1\leq i\leq n}\Pp(a_{i}\ne\widetilde a_{i})
 &\le C\varepsilon_{\psi,q}(T_n)n^{-1}\ell^{-1/2},
 \label{eq:exact-clt-rows-coupling}\\
 \Pp(A_n\ne\widetilde A_n)
 &\le C\varepsilon_{\psi,q}(T_n)\ell^{-1/2}=o(1),
 \label{eq:exact-clt-global-coupling}\\
 \max_{i,j}\E|\widetilde a_{ij}|^4
 &\le C\varepsilon_{\psi,q}(T_n)\ell=o(\ell),
 \label{eq:exact-clt-effective-fourth}\\
 |\widetilde a_{ij}|^2&\le9n\ell^{3/4},
 \label{eq:exact-clt-support}
\end{align}
The third moments and density remain uniformly bounded.
By \eqref{eq:exact-clt-effective-fourth} and
\eqref{eq:exact-clt-support}, we have
$M=\max\{3,\max_{i,j}\E|\widetilde a_{ij}|^4\}=o(\ell)$,
and the support condition \eqref{eq:exact-growing-support}
holds for $\widetilde A_n$ with $K_0=9$.
Proposition~\ref{prop:exact-growing-fourth} therefore gives
\begin{equation}\label{eq:truncated-A-CLT}
 \dk(W_n^{\mathrm e}(\widetilde A_n),\Normal)
 \longrightarrow0, \quad \text{and} \quad W_n^{\mathrm e}( \widetilde A_n )\Rightarrow\Normal.
\end{equation}

Clearly, we can transfer the Kolmogorov distance from matrix $\widetilde A_n$ to the original matrix $A_n$ as
\begin{equation}
  \begin{aligned}\label{eq:exact-BE-bound-decomposition}
    \dk(W_n^{\mathrm e}(A_n),\Normal) \leq & \dk(W_n^{\mathrm e}(\widetilde A_n),\Normal) + \Pp(A_n\ne\widetilde A_n)\\
    & + \frac{|\E\log|\det A_n|-\E\log|\det\widetilde A_n||}{\sqrt{\ell}}.
  \end{aligned}
\end{equation}
By \eqref{eq:exact-clt-global-coupling} and \eqref{eq:truncated-A-CLT}, the first two terms on the right-hand side of \eqref{eq:exact-BE-bound-decomposition} tend to zero. Therefore, to finish the proof of $W_n^{\mathrm e}(A_n)\Rightarrow\Normal$, it remains to prove that
\begin{equation}
  \begin{aligned}\label{eq:expectation-diff-need-to-prove}
    |\E\log|\det A_n|-\E\log|\det\widetilde A_n|| = o(\sqrt{\ell}), 
  \end{aligned}
\end{equation}
which is a special case of the following lemma.

\begin{lemma}\label{prop:exact-transfer}
Let $A_n$ and $\widetilde A_n$ have independent rows, and couple corresponding rows independently across row indices.  Assume that within every row the coordinates are independent, centered, have variance one, and have densities bounded by a common constant $M_0$.  If
\begin{equation}\label{eq:row-coupling-probability}
 \max_{1\le i\le n}\Pp(a_i\ne\widetilde a_i)\le\rho_n\le e^{-1},
\end{equation}
then both logarithmic determinants are integrable and
\begin{equation}\label{eq:general-mean-transfer}
 |\E\log|\det A_n|-\E\log|\det\widetilde A_n||
 \le C_{M_0}n\rho_n\log\frac e{\rho_n}.
\end{equation}
\end{lemma}

Its proof is given in Appendix~\ref{sec:proof-lemmas}.  
Applying Lemma~\ref{prop:exact-transfer} and using \eqref{eq:exact-BE-bound-decomposition}, we obtain
\begin{equation}
\begin{aligned}\label{eq:general-distribution-transfer}
 \dk(W_n^{\mathrm e}(A_n),\Normal)
 \le& \dk(W_n^{\mathrm e}(\widetilde A_n),\Normal)+\Pp(A_n\ne\widetilde A_n)
 +C_{M_0}\frac{n\rho_n\log(e/\rho_n)}{\sqrt{\ell}}.
\end{aligned} 
\end{equation}

By \eqref{eq:exact-clt-rows-coupling}, we may take
$\rho_n=C\varepsilon_{\psi,q}(T_n)/(n\sqrt\ell)$,
which is at most $e^{-1}$ for all sufficiently large $n$.
Since $\ell=\log n$,
\[
 \frac{n\rho_n\log(e/\rho_n)}{\sqrt\ell}
 \le
 C\varepsilon_{\psi,q}(T_n)
 +\frac{C\varepsilon_{\psi,q}(T_n)}{\ell}
   \log\frac{e}{\varepsilon_{\psi,q}(T_n)}
 \longrightarrow0.
\]
Here we used $\varepsilon_{\psi,q}(T_n)\to0$ and
$x\log(e/x)\to0$ as $x\downarrow0$.
By \eqref{eq:general-mean-transfer}, this proves
\eqref{eq:expectation-diff-need-to-prove}.
Finally, \eqref{eq:general-distribution-transfer},
\eqref{eq:exact-clt-global-coupling}, and
\eqref{eq:truncated-A-CLT} give
$\dk(W_n^{\mathrm e}(A_n),\Normal)\to0$ and $W_n^{\mathrm e}(A_n)\Rightarrow\Normal$.
This completes the proof of Theorem~\ref{thm:exact-clt}.
\end{proof}

\subsection{Proof of Theorem~\ref{thm:variance-clt}}
\label{subsec:proof-variance-clt}

We first state the following estimate for the truncated matrix.

\begin{proposition}\label{prop:variance-growing-fourth}
Let $\alpha\ge0$ be fixed, and let $B=(B_n)_{n\ge2}$ be a triangular
array of $n\times n$ random matrices with independent real-valued
entries $b_{ij}$, each having a density $f_{b_{ij}}$.
Suppose that there are fixed constants
$R_3,M_1,K_0,K_1>0$ such that, for every $n\ge2$ and
$1\le i,j\le n$,
\begin{equation}\label{eq:variance-growing-assumptions}
 \E b_{ij}=0,\qquad
 \E b_{ij}^2=1,\qquad
 \E|b_{ij}|^3\le R_3,\qquad
 \|f_{b_{ij}}\|_\infty\le M_1,
\end{equation}
and
\begin{equation}\label{eq:variance-growing-support}
 |b_{ij}|\le
 \bigl(K_0n\ell^{(\alpha+3)/2}\bigr)^{1/2}
 \qquad\text{almost surely}.
\end{equation}
Put $M:=\max\{3,\max_{1\le i,j\le n}\E|b_{ij}|^4\}$
and assume that $M\le K_1\ell^\alpha$.
Then there exist constants $c,C>0$ and an integer $n_0\ge2$,
depending only on $R_3,M_1,K_0,K_1,\alpha$, such that,
for every $n\ge n_0$,
\begin{equation}\label{eq:variance-growing-variance}
 c\ell\le\Var\bigl(\log|\det B_n|\bigr)
 \le C(\ell+M),
\end{equation}
and
\begin{equation}\label{eq:variance-growing-bound}
 \dk\left(
 \frac{\log|\det B_n|-\E\log|\det B_n|}
 {\sqrt{\Var(\log|\det B_n|)}},
 \Normal
 \right)
 \le C\min\left\{
 \frac{M}{\ell^{3/2}},
 \frac{\log\ell}{\sqrt\ell}
 \right\}.
\end{equation}
\end{proposition}

The proof of Proposition~\ref{prop:variance-growing-fourth} is given
in Section~\ref{sec:proof-propositions}.

\begin{proof}[Proof of Theorem~\ref{thm:variance-clt}]
Take
$\psi(x)=q(x)=\{\log(e+x)\}^{\alpha}$
in Proposition~\ref{prop:truncation}.
Then $\varepsilon_{\psi,q}(T)=1$ for every $T>0$.
Set
\begin{equation}\label{eq:variance-clt-truncation-level}
 T_n=\sqrt n\,\ell^{(\alpha+3)/4}.
\end{equation}
Thus $T_n^4=n^2\ell^{\alpha+3}$ and
$q(T_n)\asymp\ell^\alpha$.
Constants in this proof may depend on $K,M_0,\alpha$.

For sufficiently large $n$, apply the continuous part of
Proposition~\ref{prop:truncation} independently to all entries
of $A_n$, and denote the resulting matrix by $\widetilde A_n$.
Write $a_i$ and $\widetilde a_i$ for the corresponding rows.
Equations \eqref{eq:discrete-truncation-coupling}--
\eqref{eq:continuous-truncation-density} give
\begin{align}
 \max_{1\le i\le n}\Pp(a_i\ne\widetilde a_i)
 &\le \frac{C_0}{n\ell^3},
 \label{eq:variance-clt-rows-coupling}\\
 M_n:=\max\left\{3,\max_{1\le i,j\le n}
 \E|\widetilde a_{ij}|^4\right\}
 &\le C\ell^\alpha,\nonumber\\
 |\widetilde a_{ij}|^2
 &\le 9n\ell^{(\alpha+3)/2}.\nonumber
\end{align}
The truncated entries remain independent, centered, and of
variance one. Their third absolute moments are uniformly bounded,
and their densities are bounded by $M_0+1$ for all sufficiently
large $n$.
Consequently, Proposition~\ref{prop:variance-growing-fourth}
applies to $\widetilde A_n$. With
$\widetilde\sigma_n^2=\Var(\log|\det\widetilde A_n|)$, it yields
\begin{equation}\label{eq:variance-clt-truncated}
 \widetilde\sigma_n^2\ge c\ell,
 \qquad
 \dk(W_n^{\mathrm v}(\widetilde A_n),\Normal)
 \le C\min\left\{
 \frac{M_n}{\ell^{3/2}},
 \frac{\log\ell}{\sqrt\ell}
 \right\}.
\end{equation}

On the event $\{A_n=\widetilde A_n\}$, we have
$\mathcal L_n=\log|\det\widetilde A_n|$.
Once $0<\sigma_n^2<\infty$ has been verified, we obtain
\begin{equation}\label{eq:variance-clt-comparison-decomposition}
 \begin{aligned}
 \dk(W_n^{\mathrm v}(A_n),\Normal)
 &\le \dk(W_n^{\mathrm v}(\widetilde A_n),\Normal)
 +\Pp(A_n\ne\widetilde A_n)\\
 &\quad+C\frac{|\E\mathcal L_n-\E\log|\det\widetilde A_n||}
 {\widetilde\sigma_n}
 +C\frac{|\sigma_n^2-\widetilde\sigma_n^2|}
 {\widetilde\sigma_n^2}.
 \end{aligned}
\end{equation}
Here we use the inequality
$\sup_x|\Phi(ax+b)-\Phi(x)|\le C(|b|+|a^2-1|)$
for $a>0$ and $b\in\R$.

The first term on the right-hand side of
\eqref{eq:variance-clt-comparison-decomposition}
is bounded by \eqref{eq:variance-clt-truncated}.
For the coupling term, \eqref{eq:variance-clt-rows-coupling} gives
\[
 \Pp(A_n\ne\widetilde A_n)
 \le\sum_{i=1}^n\Pp(a_i\ne\widetilde a_i)
 \le \frac{C}{\ell^3}.
\]
For the mean comparison, take $\rho_n=C_0/(n\ell^3)$.
By \eqref{eq:variance-clt-rows-coupling}, this bounds the
row-change probabilities, and $\rho_n\in(0,e^{-1}]$
for all sufficiently large $n$.
Since $\log(e/\rho_n)=\ell+3\log\ell+O(1)$,
applying Lemma~\ref{prop:exact-transfer} with the common
density bound $M_0+1$ and using
\eqref{eq:variance-clt-truncated}, we obtain
\begin{equation}\label{eq:variance-clt-mean-comparison}
 \frac{|\E\mathcal L_n-\E\log|\det\widetilde A_n||}
 {\widetilde\sigma_n}
 \le C\frac{n\rho_n\log(e/\rho_n)}{\sqrt\ell}
 \le \frac{C}{\ell^{5/2}}.
\end{equation}

It remains to control the variance term in
\eqref{eq:variance-clt-comparison-decomposition} and to verify that
$\sigma_n^2$ is finite and strictly positive.
The following lemma provides both assertions.

\begin{lemma}\label{lem:variance-transfer}
Under the assumptions of Lemma~\ref{prop:exact-transfer}, with
$0<\rho_n\le e^{-1}$, let
$\widetilde\sigma_n^2=\Var(\log|\det\widetilde A_n|)$.
Then $\mathcal L_n,\log|\det\widetilde A_n|\in L^2$ and
$\sigma_n^2,\widetilde\sigma_n^2\in(0,\infty)$. Moreover,
\begin{equation}\label{eq:general-variance-transfer}
 \begin{aligned}
 |\sigma_n^2-\widetilde\sigma_n^2|
 &\le C_{M_0}\biggl\{
 n\rho_n\log\frac e{\rho_n}\,\widetilde\sigma_n
 +n^2\rho_n^2\left(\log\frac e{\rho_n}\right)^2
 +n\rho_n\left(\log\frac e{\rho_n}\right)^2
 \biggr\},
 \end{aligned}
\end{equation}
where $C_{M_0}$ depends only on the common density bound $M_0$.
\end{lemma}

Its proof is given in Appendix~\ref{sec:proof-lemmas}.

By Lemma~\ref{lem:variance-transfer}, we have
$0<\sigma_n^2<\infty$ for every $n\ge2$ and
\eqref{eq:variance-clt-comparison-decomposition} is valid
for all sufficiently large $n$.
Applying \eqref{eq:general-variance-transfer} with the common
density bound $M_0+1$, and using
$\widetilde\sigma_n^2\ge c\ell$,
$\rho_n=C_0/(n\ell^3)$, and
$\log(e/\rho_n)\le C\ell$, we obtain
\begin{equation}\label{eq:variance-clt-variance-comparison}
 \frac{|\sigma_n^2-\widetilde\sigma_n^2|}
 {\widetilde\sigma_n^2}
 \le C\left(
 \ell^{-5/2}+\ell^{-5}+\ell^{-2}
 \right)
 \le \frac{C}{\ell^2}.
\end{equation}
Substituting these estimates into
\eqref{eq:variance-clt-comparison-decomposition} yields
\begin{equation}\label{eq:variance-clt-transfer-bound}
 \dk(W_n^{\mathrm v}(A_n),\Normal)
 \le \dk(W_n^{\mathrm v}(\widetilde A_n),\Normal)
 +\frac{C}{\ell^2}.
\end{equation}
Combining \eqref{eq:variance-clt-truncated} and
\eqref{eq:variance-clt-transfer-bound}, we conclude that
\[
 \dk(W_n^{\mathrm v}(A_n),\Normal)
 \le C\left(
 \frac{\log\ell}{\sqrt\ell}+\frac1{\ell^2}
 \right)
 \longrightarrow0.
\]
This proves Theorem~\ref{thm:variance-clt}.
\end{proof}

\subsection{Proof of Theorem~\ref{thm:det-be}}\label{subsec:proof-det-be}

Set $\psi(x)=\{\log(e+x)\}^{1/2-\delta}$,
$q(x)=\sqrt{\log(e+x)}$, and take
$T_n=\sqrt n\,\ell^{1/8}$ as in
\eqref{eq:det-clt-truncation-level}.
Then $q(T_n)\asymp\sqrt\ell$ and
\begin{equation}\label{eq:det-be-epsilon}
 \varepsilon_{\psi,q}(T_n)
 \asymp\{\log(e+T_n)\}^{-\delta}
 \asymp\ell^{-\delta}.
\end{equation}
Apply Proposition~\ref{prop:truncation} independently to
all entries to obtain $\widetilde A_n$.
The truncated entries have uniformly bounded third moments.
The truncation estimates, the union-bound calculation in
\eqref{eq:det-clt-global-coupling}, and
Proposition~\ref{prop:det-growing-fourth} give
\begin{align}
 \Pp(A_n\ne\widetilde A_n)
 &\le C\ell^{-\delta},
 \label{eq:det-be-coupling}\\
 M_n:=\max_{i,j}\E|\widetilde a_{ij}|^4
 &\le C\ell^{1/2-\delta},
 \label{eq:det-be-M}\\
 \dk(W_n^{\mathrm d}(\widetilde A_n),\Normal)
 &\le C\frac{\max\{3,M_n\}}{\sqrt\ell}
 \le C\ell^{-\delta}.
 \label{eq:det-be-truncated}
\end{align}
Combining \eqref{eq:det-be-coupling} and
\eqref{eq:det-be-truncated} by the coupling inequality
proves \eqref{eq:det-be-bound} for all sufficiently large $n$. Thus we complete the proof of Theorem~\ref{thm:det-be}.

\subsection{Proof of Theorem~\ref{thm:exact-be}}\label{subsec:proof-exact-be}
Set $\psi(x)=\{\log(e+x)\}^{1-\gamma}$,
$q(x)=\log(e+x)$, and take
$T_n=\sqrt n\,\ell^{3/8}$ as in
\eqref{eq:exact-clt-truncation-level}.
Apply the continuous part of Proposition~\ref{prop:truncation}
independently to all entries to obtain $\widetilde A_n$.
At this scale,
\begin{equation}\label{eq:exact-be-epsilon}
 \varepsilon_{\psi,q}(T_n)
 \asymp\{\log(e+T_n)\}^{-\gamma}
 \asymp\ell^{-\gamma},
 \qquad q(T_n)\asymp\ell.
\end{equation}
As in the proof of Theorem~\ref{thm:exact-clt}, \eqref{eq:exact-clt-global-coupling}--\eqref{eq:exact-clt-support} hold, and
\begin{align}
 \Pp(A_n\ne\widetilde A_n)&\le C\ell^{-1/2-\gamma},
 \label{eq:exact-be-coupling}\\
 M_n:=\max_{i,j}\E|\widetilde a_{ij}|^4
 &\le C\ell^{1-\gamma}.
 \label{eq:exact-be-M}
\end{align}
Applying Proposition~\ref{prop:exact-growing-fourth} to
$\widetilde A_n$ with $K_0=9$, and using
\eqref{eq:exact-clt-support} and \eqref{eq:exact-be-M}, we obtain
\begin{equation}\label{eq:exact-be-truncated}
 \dk(W_n^{\mathrm e}(\widetilde A_n),\Normal)
 \le C\frac{M_n}{\ell}\le C\ell^{-\gamma}.
\end{equation}
By \eqref{eq:exact-clt-rows-coupling} and
\eqref{eq:exact-be-epsilon}, we may take
$\rho_n=C/(n\ell^{1/2+\gamma})$.
For all sufficiently large $n$, $\rho_n\le e^{-1}$,
and the original and truncated entries have densities
bounded by $M_0+1$.
Since
$\log(e/\rho_n)=\ell+(\frac12+\gamma)\log\ell+O(1)$,
Lemma~\ref{prop:exact-transfer} gives
\[
 \frac{|\E\log|\det A_n|
       -\E\log|\det\widetilde A_n||}{\sqrt\ell}
 \le C\frac{n\rho_n\log(e/\rho_n)}{\sqrt\ell}
 \le C\ell^{-\gamma}.
\]
Combining this estimate with
\eqref{eq:exact-be-coupling} and
\eqref{eq:exact-be-truncated} in
\eqref{eq:general-distribution-transfer}
proves \eqref{eq:exact-be-bound}
for all sufficiently large $n$.
Enlarging the constant covers the remaining $n\ge2$.
Thus we complete the proof of Theorem~\ref{thm:exact-be}.

\subsection{Proof of Theorem~\ref{thm:variance-be}}
\label{subsec:proof-variance-be}

\begin{proof}[Proof of Theorem~\ref{thm:variance-be}]
By Lemma~\ref{lem:variance-transfer},
$0<\sigma_n^2<\infty$ for every $n\ge2$.
Use the truncated matrix $\widetilde A_n$ constructed in the proof
of Theorem~\ref{thm:variance-clt}, with the truncation level
\eqref{eq:variance-clt-truncation-level}.
This construction gives $M_n\le C\ell^\alpha$
for all sufficiently large $n$.
Combining \eqref{eq:variance-clt-truncated} and
\eqref{eq:variance-clt-transfer-bound}, we obtain
\[
 \dk(W_n^{\mathrm v}(A_n),\Normal)
 \le C\min\left\{
 \ell^{\alpha-3/2},
 \frac{\log\ell}{\sqrt\ell}
 \right\}
 +\frac{C}{\ell^2}.
\]
Since $\alpha\ge0$, the remainder $\ell^{-2}$ is bounded
by both terms in the minimum for all sufficiently large $n$.
Thus \eqref{eq:variance-be-bound} follows from
$\log\ell\le\log(e+\ell)$.
All constants depend only on $K,M_0,\alpha$.
Enlarging the constant covers the remaining $n\ge2$. This proves Theorem~\ref{thm:variance-be}.
\end{proof}

\section{Proof of propositions}\label{sec:proof-propositions}
We first give some useful lemmas which will be uesd repeatedly in the proofs of propositions. 

\begin{lemma}[Esseen smoothing]\label{lem:esseen}
For a real-valued random variable $X$, its \emph{characteristic function} is
\begin{equation}\label{eq:characteristic-function-definition}
 \varphi_X(t):=\E e^{itX},\qquad t\in\R.
\end{equation}
For every $T>0$,
\[
 \dk(X,\Normal)
 \le \frac CT+C\int_0^T\frac{|\varphi_X(t)-e^{-t^2/2}|}{t}\dd t.
\]
\end{lemma}

\begin{lemma}\label{lem:perturb}
Let $X$ and $R$ be real-valued random variables and let $Z\sim\Normal$.
\begin{enumerate}[label=(\roman*)]
\item For every $\eta>0$,
\begin{equation}\label{eq:perturb-arbitrary}
 \dk(X+R,Z)\le\dk(X,Z)+\Pp(|R|>\eta)+\eta/\sqrt{2\pi}.
\end{equation}
\item If $R$ is independent of $X$, $\E R=0$, and $\E R^2<\infty$, then
\begin{equation}\label{eq:perturb-centered}
 \dk(X+R,Z)\le\dk(X,Z)+C\E R^2.
\end{equation}
\item For $c\in\R$,
\begin{equation}\label{eq:perturb-shift}
 \dk(X+c,Z)\le\dk(X,Z)+C|c|.
\end{equation}
\item If $|r-1|\le1/2$, then
\begin{equation}\label{eq:perturb-scale}
 \dk(rX,Z)\le\dk(X,Z)+C|r-1|.
\end{equation}
\end{enumerate}
\end{lemma}

\subsection{Proof of Proposition~\ref{prop:truncation}}\label{subsec:proof-prop-truncation}

We preserve the mean and variance by a three-point correction
in the general case, and by replacing $X$ on a rare event
with a uniform random variable when $X$ has a bounded density.

\begin{proof}[Proof of Proposition~\ref{prop:truncation}]
Let $q$ be defined by \eqref{eq:truncation-two-functions}, and let
$\varepsilon(T)=\varepsilon_{\psi,q}(T)$ as in
\eqref{eq:epsilon-definition}.
For every fixed $\alpha\ge0$, we have $q(x)=o(x)$ as $x\to\infty$,
and $x\mapsto q(x)/x^r$ is decreasing for all sufficiently large
$x$, for $r\in\{2,4\}$.
Hence, for all sufficiently large $T$ and every $x\ge T$,
\begin{equation}\label{eq:truncation-tail-comparison}
 \frac{\psi(x)}{x^r}
 \le \varepsilon(T)\frac{q(x)}{x^r}
 \le \varepsilon(T)\frac{q(T)}{T^r},
 \qquad r\in\{2,4\}.
\end{equation}
Multiplying by $|X|^4/\psi(|X|)$ on $\{|X|>T\}$ and taking
expectations, first with $r=4$ and then with $r=2$, gives
\begin{align}
 \Pp(|X|>T)
 &\le CK\varepsilon(T)\frac{q(T)}{T^4},
 \label{eq:truncation-tail-probability}\\
 v_T:=\E[X^2\1_{\{|X|>T\}}]
 &\le CK\varepsilon(T)\frac{q(T)}{T^2}.
 \label{eq:truncation-tail-second}
\end{align}
On the retained region, monotonicity of $\psi$ and the definition
of $\varepsilon(T)$ yield
\begin{equation}\label{eq:truncation-retained-fourth}
 \E[|X|^4\1_{\{|X|\le T\}}]
 \le K\psi(T)
 \le K\varepsilon(T)q(T).
\end{equation}
Since $\psi=O(q)$ and $q(x)=o(x)$, we have $\psi(x)=o(x)$.
Thus there exists $R_\psi\ge1$, depending only on $\psi$, such that
$\psi(x)\le x$ for $x\ge R_\psi$. Consequently,
\begin{equation}\label{eq:truncation-original-third}
 \begin{aligned}
 \E|X|^3
 &\le R_\psi\E X^2
 +\E\frac{|X|^4}{\psi(|X|)}
 \le R_\psi+K.
 \end{aligned}
\end{equation}

We first construct the discrete correction. Define
\begin{equation}\label{eq:discrete-correction-data}
 X_0=X\1_{\{|X|\le T\}},\qquad
 m_T=\E[X\1_{\{|X|>T\}}],\qquad
 s_T=v_T+2m_T^2.
\end{equation}
Since $|m_T|\le v_T/T$ and $v_T\le1$, for all sufficiently large
$T$ we have $T|m_T|\le s_T\le2v_T\le T^2$.
Therefore the following three numbers are nonnegative and sum
to one:
\begin{equation}\label{eq:discrete-correction-law}
 \begin{aligned}
 \Pp(Z_T=T)
 &=\frac12\left(\frac{s_T}{T^2}+\frac{m_T}{T}\right),
 &\Pp(Z_T=-T)
 &=\frac12\left(\frac{s_T}{T^2}-\frac{m_T}{T}\right),\\
 \Pp(Z_T=0)
 &=1-\frac{s_T}{T^2}.&&
 \end{aligned}
\end{equation}
Take $Z_T$ independent of $X$ and put
$\widetilde X_T=X_0+Z_T$.
Then $|\widetilde X_T|\le2T$ and
$\E\widetilde X_T=\E X_0+\E Z_T=-m_T+m_T=0$.
Using independence and $\E Z_T^2=s_T$, we obtain
\[
 \E\widetilde X_T^2
 =(1-v_T)-2m_T^2+s_T=1.
\]
Using \eqref{eq:truncation-tail-probability},
\eqref{eq:truncation-tail-second}, and $s_T\le2v_T$, we obtain
\begin{equation}\label{eq:discrete-correction-change}
 \Pp(\widetilde X_T\ne X)
 \le \Pp(|X|>T)+\Pp(Z_T\ne0)
 \le CK\varepsilon(T)\frac{q(T)}{T^4}.
\end{equation}
Moreover,
\begin{equation}\label{eq:discrete-correction-moments}
 \begin{aligned}
 \E|Z_T|^4
 &=T^2s_T
 \le CK\varepsilon(T)q(T),\\
 \E|Z_T|^3
 &=Ts_T
 \le2Tv_T
 \le2\E[|X|^3\1_{\{|X|>T\}}].
 \end{aligned}
\end{equation}
Combining \eqref{eq:truncation-retained-fourth},
\eqref{eq:truncation-original-third},
\eqref{eq:discrete-correction-change}, and
\eqref{eq:discrete-correction-moments} with
$|u+v|^r\le2^{r-1}(|u|^r+|v|^r)$ proves
\eqref{eq:discrete-truncation-coupling}--
\eqref{eq:discrete-truncation-moments}.

Assume now that $X$ has a density bounded by $M_0$.
If $v_T=0$, then $X$ is supported in $[-T,T]$ almost surely,
and taking $\widetilde X_T=X$ proves the continuous part.
We may therefore assume $v_T>0$.
Set $A_T=\{|X|>T\}$, $p_T=\Pp(A_T)$, and
$q_T=2v_T/T^2$.
Since $T^2p_T\le v_T\le1$, we have
$2p_T\le q_T\le1$ for all sufficiently large $T$.

We enlarge the event $A_{T}$ to an event $E_{T}$ of probability $q_{T}$, and then replace $X$ on $E_{T}$ by some continuous random variable $R_T$ which is independent of all preceding variables.  The enlargement is done by means of an independent uniform randomizer, as follows.
\begin{equation}
  \begin{aligned}
    E_{T} = A_{T} \cup \left\{ A_{T}^{c}, U < \frac{q_T - p_T}{1 - p_T} \right\},  \qquad \Pp(E_T) = q_T,  \nonumber
  \end{aligned}
\end{equation}
where $U$ is uniform on $[0,1]$ and independent of all preceding variables.
If
\begin{equation}\label{eq:continuous-correction-data}
 m_E=\E[X\1_{E_T}],\qquad v_E=\E[X^2\1_{E_T}],\qquad
 \mu_T=\frac{m_E}{q_T},\qquad
 \sigma_T^2=\frac{v_E}{q_T}-\mu_T^2,
\end{equation}
then $2p_{T} \leq q_{T} \leq 1$ and $\E X=0$ implying $|m_E|\le|m_T|$.
Also, the added part of $E_T$ lies in $\{|X|\le T\}$, so
\begin{equation}\label{eq:continuous-correction-range}
 v_T\le v_E\le v_T+T^2q_T=3v_T,
 \qquad |\mu_T|\le T/2,
 \qquad T^2/4\le\sigma_T^2\le3T^2/2.
\end{equation}
Let $V$ be uniformly distributed on $[-\sqrt3,\sqrt3]$, independent of all preceding variables, set $R_T=\mu_T+\sigma_TV$, and define
\begin{equation}\label{eq:continuous-correction-variable}
 \widetilde X_T=X\1_{E_T^c}+R_T\1_{E_T}.
\end{equation}
The definitions in \eqref{eq:continuous-correction-data} give
\begin{equation}\label{eq:continuous-correction-matching}
 \E\widetilde X_T=\E[X\1_{E_T^c}]+q_T\mu_T=0,
 \qquad
 \E\widetilde X_T^2=\E[X^2\1_{E_T^c}]+q_T(\mu_T^2+\sigma_T^2)=1.
\end{equation}
By \eqref{eq:continuous-correction-range}, $|R_T|<3T$.
Using $q_T=2v_T/T^2$ and \eqref{eq:truncation-tail-second}, we obtain
\begin{equation}\label{eq:continuous-correction-change}
 \Pp(\widetilde X_T\ne X)=q_T
 \le CK\varepsilon(T)\frac{q(T)}{T^4}.
\end{equation}
Since $E_T\supset\{|X|>T\}$,
$\Pp(E_T)=q_T=2v_T/T^2$, and $|R_T|\le3T$,
we obtain
\begin{align*}
 \E|\widetilde X_T|^4
 &\le K\varepsilon_{\psi,q}(T)q(T)+CT^2v_T
 \le CK\varepsilon_{\psi,q}(T)q(T),\\
 \E|\widetilde X_T|^3
 &\le \E|X|^3+CTv_T
 \le C\E|X|^3
 \le C_{K,\psi}.
\end{align*}
The fourth-moment bound follows from
\eqref{eq:truncation-retained-fourth} and
\eqref{eq:truncation-tail-second}.
For the third moment, we used
$Tv_T\le\E[|X|^3\1_{\{|X|>T\}}]$
and \eqref{eq:truncation-original-third}.
This proves \eqref{eq:discrete-truncation-moments}.
It remains to check the density.
The retained component has density at most $M_0$, because its density is the original density multiplied by a number in $[0,1]$.  Conditional on replacement, $R_T$ has density $(2\sqrt3\sigma_T)^{-1}$ on an interval; hence its unconditional contribution is at most $Cq_T/T\le C/T$.  This proves \eqref{eq:continuous-truncation-density} and completes the proof of Proposition~\ref{prop:truncation}.
\end{proof}

\subsection{Common scales and notation for the determinant estimates}
\label{subsec:common-determinant-setup}

The proofs of Propositions~\ref{prop:det-growing-fourth},
\ref{prop:exact-growing-fourth}, and
\ref{prop:variance-growing-fourth} use the following common scales
and Gram--Schmidt projection notation.

Fix a constant $a\ge20$, independent of $n$, and set
\begin{equation}\label{eq:common-parameters}
 \ell=\log n,\qquad
 s_1=\lfloor\ell^{3a}\rfloor,\qquad
 s_2=\lfloor n\ell^{-20a}\rfloor,\qquad
 m=n-s_1,\qquad \lambda=n^{-1/6}.
\end{equation}
In the proofs of Propositions~\ref{prop:det-growing-fourth} and
\ref{prop:exact-growing-fourth}, we take $a=20$.
For Proposition~\ref{prop:variance-growing-fourth}, we take $a=20+4\lceil\alpha\rceil$,
where $\alpha$ is the fixed parameter in that proposition.
For each fixed $a$, we take $n$ large enough that
$1\le s_1<s_2<n/2$.

For any $n\times n$ random matrix $A_{n}$, let $a_1^T,\ldots,a_n^T\in\R^n$ be the rows of $A_{n}$, and let $A(i)$ be the $i\times n$ matrix consisting of the first \(i\) rows of $A_n$
define
\begin{align}
 \cV_i&=\Span(a_1,\ldots,a_i),
 &P_i&=I_n-A(i)^T(A(i)A(i)^T)^{-1}A(i), \quad P_0:=I_n, 
 \label{eq:common-projections}\\
 k_i&=n-i,
 &Q_i&=k_i^{-1}P_i.
 \label{eq:common-normalized-projections}
\end{align}
Here $P_i$ is the orthogonal projection onto $\cV_i^\perp$.
Write $P_i=(p_{rs}(i))$ and $Q_i=(q_{rs}(i))$,
with $q_{rr}(i)=p_{rr}(i)/k_i$.
In every application below, the first $i$ rows are linearly independent almost surely.  Consequently, $P_i$ has rank $k_i$, and
\begin{equation}\label{eq:common-projection-traces}
 \tr Q_i=k_i^{-1}\tr P_i=1,
 \qquad
 \tr Q_i^2=k_i^{-2}\tr P_i=k_i^{-1}.
\end{equation}
When the first $i$ rows are revealed, let $\cF_i=\sigma(a_1,\ldots,a_i)$ be the filtration, and write $\E_i=\E(\,\cdot\mid\cF_i)$ and $\Pp_i=\Pp(\,\cdot\mid\cF_i)$. Let
\begin{equation}\label{eq:common-heights}
 Z_{i+1}=a_{i+1}^TQ_i a_{i+1},\qquad
 d_i=\sum_{r=1}^nq_{rr}(i)^2.
\end{equation}
Hence $Z_{i+1}$ is the normalized squared Gram--Schmidt height of row $a_{i+1}$, while $d_i$ is the squared Euclidean energy of the normalized diagonal-projection vector.
We call the first $m=n-s_1$ rows the \emph{bulk} and the last $s_1$ rows the \emph{terminal block}.  Finally, $A_n^{\mathrm{hyb}}$ denotes the matrix obtained from $A_n$ by replacing its terminal block by independent standard Gaussian rows.

\subsection{Proof of Proposition~\ref{prop:det-growing-fourth}}\label{subsec:proof-prop-det-growing-fourth}
Our goal is to prove \eqref{eq:det-growing-bound}.  We compare $B_n$
with the hybrid matrix $B_n^{\mathrm{hyb}}$.  By the triangle
inequality,
\[
 \dk\bigl(W_n^{\mathrm d}(B_n),\Normal\bigr)
 \le
 \dk\bigl(
 W_n^{\mathrm d}(B_n),
 W_n^{\mathrm d}(B_n^{\mathrm{hyb}})
 \bigr)
 +
 \dk\bigl(
 W_n^{\mathrm d}(B_n^{\mathrm{hyb}}),
 \Normal
 \bigr).
\]
We estimate the two terms on the right separately.

For the first term, we replace the last $s_1$ rows one at a time.
The total replacement error is
\begin{align}\label{eq:total-replacement-error}
 \dk\bigl(
 W_n^{\mathrm d}(B_n),
 W_n^{\mathrm d}(B_n^{\mathrm{hyb}})
 \bigr)
 \le C\ell^{-a}.
\end{align}

For the second term, we use the Gram--Schmidt decomposition of
$B_n^{\mathrm{hyb}}$.  It splits the logarithmic determinant into a
non-Gaussian bulk part and a Gaussian terminal part.  We estimate the
two parts separately and obtain
\begin{equation}\label{eq:second-term-error}
 \dk\bigl(
 W_n^{\mathrm d}(B_n^{\mathrm{hyb}}),
 \Normal
 \bigr)
 \le CM\ell^{-1/2}.
\end{equation}

Combining the two bounds gives
\[
 \dk\bigl(W_n^{\mathrm d}(B_n),\Normal\bigr)
 \le
 C\ell^{-a}+CM\ell^{-1/2}
 \le
 CM\ell^{-1/2},
\]
where the last inequality uses $a=20$ and $M\ge3$.  This proves
\eqref{eq:det-growing-bound}.

\begin{proof}[Proof of Proposition~\ref{prop:det-growing-fourth}]

Some formulas below require full rank, for example the formula for $P_i$
contains $(A(i)A(i)^T)^{-1}$ and this inverse exists when the first
$i$ rows are linearly independent.  The next lemma shows that we may
assume that every square submatrix of $B_n$ is invertible almost surely.

\begin{lemma}[Reduction to almost-sure nonsingularity]\label{lem:det-smoothing}
It is enough to prove Proposition~\ref{prop:det-growing-fourth} under the additional assumption that all square submatrices of $B_n$ are invertible almost surely.
\end{lemma}

\begin{proof}[Proof of Lemma~\ref{lem:det-smoothing}]
Assume that Proposition~\ref{prop:det-growing-fourth} has been proved under the additional condition that all its square submatrices are invertible almost surely. We show that the same estimate then holds without this additional condition.

Fix $n$.  Let $\Theta_n=(\theta_{ij})$ have independent entries,
independent of $B_n$, with each $\theta_{ij}$ uniformly distributed on
$[-\sqrt3,\sqrt3]$, and set
\[
 B_n^{(\rho)}
 =
 \frac{B_n+\rho\Theta_n}{\sqrt{1+\rho^2}},
 \qquad \rho>0.
\]
The entries of $B_n^{(\rho)}$ remain independent.  Moreover, since
$b_{ij}$ and $\theta_{ij}$ are independent, centered, and have variance
one,
\[
 \E b_{ij}^{(\rho)}=0,
 \qquad
 \E\bigl(b_{ij}^{(\rho)}\bigr)^2=1.
\]

We first verify the required nonsingularity.  Fix any square submatrix
with row and column sets $I$ and $J$, where $|I|=|J|$.  Conditionally
on $B_n$,
\[
 \det B_n^{(\rho)}[I,J]
 =
 (1+\rho^2)^{-|I|/2}
 \det\bigl(B_n[I,J]+\rho\Theta_n[I,J]\bigr)
\]
is a nonzero polynomial in the entries of $\Theta_n[I,J]$: its
homogeneous part of highest degree is
$\rho^{|I|}\det\Theta_n[I,J]$.  Since these noise variables have a
joint density, the polynomial vanishes with conditional probability
zero.  There are only finitely many square submatrices, so all square
submatrices of $B_n^{(\rho)}$ are invertible almost surely.

For $p=3,4$, the inequality
$|x+y|^p\le2^{p-1}(|x|^p+|y|^p)$ gives, uniformly in
$0<\rho\le1$,
\[
 \E|b_{ij}^{(\rho)}|^p
 \le
 C_p\{\E|b_{ij}|^p+\E|\theta_{ij}|^p\}.
\]
Consequently,
\[
 \max_{i,j}\E|b_{ij}^{(\rho)}|^3\le C(R_3+1),
 \qquad
 M_\rho
 :=
 \max\left\{3,\max_{i,j}\E|b_{ij}^{(\rho)}|^4\right\}
 \le CM.
\]
In particular, $M\le K_0\sqrt{\log n}$ implies
$M_\rho\le CK_0\sqrt{\log n}$.  Applying the Proposition~\ref{prop:det-growing-fourth} under the additional nonsingularity assumption, with $K_0$ and $R_3$ replaced
by fixed multiples if necessary, yields
\[
 \dk\bigl(W_n^{\mathrm d}(B_n^{(\rho)}),\Normal\bigr)
 \le
 C(K_0,R_3)
 \left\{
 (\log n)^{-1/2}
 +
 \frac{M}{\sqrt{\log n}}
 \right\},
 \qquad 0<\rho\le1,
\]
where the constant is independent of $\rho$.

Now let $\rho_q\downarrow0$.  Using the same noise matrix
$\Theta_n$ for every $q$, we have
$B_n^{(\rho_q)}\to B_n$ entrywise almost surely, and hence
$\det B_n^{(\rho_q)}\to\det B_n$.  If $\det B_n\ne0$, the logarithms
converge in the usual sense; if $\det B_n=0$, then
$\log|\det B_n^{(\rho_q)}|\to-\infty$.  Therefore
\[
 W_n^{\mathrm d}(B_n^{(\rho_q)})
 \longrightarrow W_n^{\mathrm d}(B_n)
\]
almost surely in the extended real line.

Let $F_q$ and $F$ be the corresponding distribution functions and
let $\Phi$ be the standard normal distribution function.  At every
continuity point $x$ of $F$, one has $F_q(x)\to F(x)$, and hence
\[
 |F(x)-\Phi(x)|
 \le
 \liminf_{q\to\infty}
 \dk\bigl(W_n^{\mathrm d}(B_n^{(\rho_q)}),\Normal\bigr).
\]
Since the continuity points of $F$ are dense, while $F$ is
right-continuous and $\Phi$ is continuous, taking the supremum over
$x\in\mathbb R$ gives
\[
 \dk\bigl(W_n^{\mathrm d}(B_n),\Normal\bigr)
 \le
 \liminf_{q\to\infty}
 \dk\bigl(W_n^{\mathrm d}(B_n^{(\rho_q)}),\Normal\bigr).
\]
Combining the last two displays proves the desired bound for $B_n$,
and hence we complete the proof.
\end{proof}

We now return to the proof of Proposition~\ref{prop:det-growing-fourth}.
From now on, we assume that every square submatrix of $B_n$ is
invertible almost surely.

We now prove the bound for the first term \eqref{eq:total-replacement-error} in the proof
strategy.  We replace the last $s_1$ rows by independent standard
Gaussian rows, one row at a time.  By the triangle inequality, the
total replacement error is at most the sum of the $s_1$ one-row
errors. It is enough to prove a bound $C\ell^{-4a}$ for each step, since
\[
 C s_1\ell^{-4a}
 \le
 C\ell^{3a}\ell^{-4a}
 =
 C\ell^{-a}.
\]
The next lemma proves this one-row bound.

\begin{lemma}[Replacing one terminal row]\label{lem:det-one-row}
Let $C_n$ and $\overline C_n$ have identical first $n-1$ rows satisfying the assumptions of Proposition~\ref{prop:det-growing-fourth}; suppose the last row of $C_n$ has independent mean-zero, variance-one entries with third absolute moments at most $R_3$, and the last row of $\overline C_n$ is independent standard Gaussian.  If all square submatrices are invertible almost surely, then
\begin{equation}\label{det:eq:one-row}
 \sup_x|\Pp(W_n^{\mathrm d}(C_n)\le x)-\Pp(W_n^{\mathrm d}(\overline C_n)\le x)|
 \le C\ell^{-4a}.
\end{equation}
\end{lemma}

\begin{proof}[Proof of Lemma~\ref{lem:det-one-row}]
Let $C(n-1)$ be the common $(n-1)\times n$ matrix formed by the first
$n-1$ rows of $C_n$ and $\overline C_n$, and let
$\cF_{n-1}=\sigma(C(n-1))$.  Write the last rows of the two matrices as
$(c_{n1},\ldots,c_{nn})$ and
$(\overline c_{n1},\ldots,\overline c_{nn})$, respectively.

For each $1\le j\le n$, let $\alpha_j$ be the cofactor corresponding
to the $(n,j)$ entry.  Since the two matrices have the same first
$n-1$ rows, the cofactors $\alpha_1,\ldots,\alpha_n$ are the same for
both matrices and are $\cF_{n-1}$-measurable.  Expansion along the last
row gives
\[
 \det C_n=\sum_{j=1}^n\alpha_jc_{nj},
 \qquad
 \det\overline C_n=\sum_{j=1}^n\alpha_j\overline c_{nj}.
\]
The nonsingularity assumption implies that the common first $n-1$
rows have rank $n-1$ almost surely.  Hence the vector
$(\alpha_1,\ldots,\alpha_n)$ is nonzero.  Put $\Lambda=\left(\sum_{j=1}^n\alpha_j^2\right)^{1/2}$, $w_j=\frac{\alpha_j}{\Lambda}$.
Then $\Lambda>0$ and $\sum_jw_j^2=1$.  With
\[
 S=\sum_{j=1}^nw_jc_{nj},
 \qquad
 \overline S=\sum_{j=1}^nw_j\overline c_{nj},
 \qquad 
 \overline S \mid \mathcal F_{n-1} \sim \Normal,
\]
the two determinant expansions become
\begin{equation}\label{eq:det-one-row-linear-forms}
 \det C_n=\Lambda S,
 \qquad
 \det\overline C_n=\Lambda\overline S.
\end{equation}

Conditionally on $\cF_{n-1}$, the coefficients $w_j$ are deterministic,
while the variables $c_{n1},\ldots,c_{nn}$ are independent, centered,
have variance one, and have third absolute moments at most $R_3$.
Consequently, $\sum_{j=1}^n \E\left((w_jc_{nj})^2\mid\cF_{n-1}\right) = \sum_{j=1}^nw_j^2 = 1$,
and
$
\sum_{j=1}^n
\E\left(|w_jc_{nj}|^3\mid\cF_{n-1}\right)
\le
R_3\sum_{j=1}^n|w_j|^3
\le
R_3\max_j|w_j|$.
The classical Berry--Esseen inequality therefore gives
\begin{equation}\label{eq:det-one-row-BE}
 \sup_{u\in\R}
 \left|
 \Pp(S\le u\mid\cF_{n-1})-\Pp(\overline S\le u\mid\cF_{n-1})
 \right|
 \le
 CR_3\max_j|w_j|.
\end{equation}
For $t\ge0$, applying \eqref{eq:det-one-row-BE} at $t$ and $-t$ gives
\begin{equation}\label{eq:det-one-row-absolute-values}
 \sup_{t\ge0}
 \left|
 \Pp(|S|\le t\mid\cF_{n-1})
 -
 \Pp(|\overline S|\le t\mid\cF_{n-1})
 \right|
 \le
 CR_3\max_j|w_j|.
\end{equation}
For $x\in\R$, define
$
\tau_x
=
\Lambda^{-1}
\exp\left\{
\frac12\log(n-1)!
+
x\sqrt{\frac12\ell}
\right\}$.
By \eqref{eq:det-one-row-linear-forms},
\[
 \{W_n^{\mathrm d}(C_n)\le x\}
 =
 \{|S|\le\tau_x\},
 \qquad
 \{W_n^{\mathrm d}(\overline C_n)\le x\}
 =
 \{|\overline S|\le\tau_x\}.
\]
Since $\tau_x$ is $\cF_{n-1}$-measurable, the estimate
\eqref{eq:det-one-row-absolute-values} is uniform over all
$t\ge0$.  Hence
\[
 \left|
 \Pp(W_n^{\mathrm d}(C_n)\le x\mid\cF_{n-1})
 -
 \Pp(W_n^{\mathrm d}(\overline C_n)\le x\mid\cF_{n-1})
 \right|
 \le
 CR_3\max_j|w_j|.
\]
Taking expectations and then the supremum over $x$ gives
\begin{equation}\label{eq:det-one-row-distribution}
 \sup_x
 \left|
 \Pp(W_n^{\mathrm d}(C_n)\le x)
 -
 \Pp(W_n^{\mathrm d}(\overline C_n)\le x)
 \right|
 \le
 CR_3\,\E\max_j|w_j|.
\end{equation}

It remains to show that $\E\max_j|w_j|\le C\ell^{-4a}$.
We first relate the coefficients $w_j$ to the diagonal entrie $p_{jj}(n-1)$
of the projection matrix $P_{n-1}$ and claim that $p_{jj}(n-1) = w_{j}^{2}$. And later we use the estimate of $\E\max_jp_{jj}(n-1)$ to bound $\E\max_j|w_j|$.

Here, we show that $p_{jj}(n-1) = w_{j}^{2}$. Let
\[
 \alpha=(\alpha_1,\ldots,\alpha_n)^T,
\]
be the vector of cofactors corresponding to the last row of $C_n$.
Recall that $w=\alpha/\Lambda$.  
We first show that $\alpha\in\ker C(n-1)$. Here, $\ker C(n-1) =\{ x\in \mathbb{R}^{n} : C(n-1)x=0\}$ is the orthogonal complement of the row space of the common first $n-1$ rows. Since this kernel isone-dimensional, it will then follow that $\alpha$ spans it.

To show that $\alpha\in\ker C(n-1)$, we need to show that $C(n-1)\alpha=0$. Since 
\begin{equation}
  \begin{aligned}
    C(n-1)\alpha
    =
    \left(\sum_{j = 1}^{n} c_{1j}\alpha_j, \ldots, \sum_{j = 1}^{n} c_{(n-1)j}\alpha_j\right)^T, \nonumber
  \end{aligned}
\end{equation}
we need to show that for each $1\le r\le n-1$, $\sum_{j=1}^n c_{rj}\alpha_j=0$. 
To see this, we fix $1\le r\le n-1$ and append the $r$th row $(c_{r1},\ldots,c_{rn})$ of $C(n-1)$ to $C(n-1)$ as its last row.  The resulting $n\times n$ matrix has two identical
rows, and hence its determinant is zero.  Its cofactors along the
last row are precisely $\alpha_1,\ldots,\alpha_n$.  Expanding along the
last row therefore gives
\[
 \sum_{j=1}^n c_{rj}\alpha_j=0.
\]
Since this holds for every $1\le r\le n-1$, we obtain $C(n-1)\alpha=0$.
Because $\rank C(n-1)=n-1$ almost surely and $\alpha\ne0$,
\[
 \ker C(n-1)
 =
 \operatorname{span}\{\alpha\}
 =
 \operatorname{span}\{w\}.
\]
By \eqref{eq:common-projections}, $P_{n-1}$ is the orthogonal
projection onto $\ker C(n-1)$. Since $w$ is a unit
vector, the orthogonal projection onto $\operatorname{span}\{w\}$ maps $x\in\R^n$ to $(w^Tx)w$, and hence
\begin{equation}\label{eq:det-one-row-projection}
 P_{n-1}=ww^T,
 \qquad
 p_{jj}(n-1)=w_j^2.
\end{equation}
Then, we give a useful Lemma~\ref{lem:terminal-diagonal} to bound $\E\max_jp_{jj}(n-1)$.

\begin{lemma}\label{lem:terminal-diagonal}
Let $\alpha\ge0$, $a\ge20$, and $C_0>0$ be fixed, with
$a>3\alpha+12$.
Let $A_n$ have independent centered real-valued entries of
variance one and fourth moments at most $M$, where
$3\le M\le C_0\ell^\alpha$.
Assume that all square submatrices of $A_n$ are invertible
almost surely.
Use the notation in
\eqref{eq:common-parameters}--\eqref{eq:common-heights}.
Then, for all sufficiently large $n$, uniformly for
$n-s_2\le i\le n-1$,
\begin{equation}\label{ex:eq:terminal-diagonal-max}
 \E\max_{1\le r\le n}p_{rr}(i)\le C\ell^{-8a}.
\end{equation}
Here $C$ depends only on $C_0,\alpha,a$.
\end{lemma}

The proof of Lemma~\ref{lem:terminal-diagonal} is given in Appendix~\ref{sec:proof-lemmas}

For all sufficiently large $n$, the assumption
$M\le K_0\sqrt\ell$ implies $M\le\ell$.  Lemma
\ref{lem:terminal-diagonal}, applied with $i=n-1$, therefore gives
\[
 \E\max_jp_{jj}(n-1)\le C\ell^{-8a}.
\]
The estimate remains valid when some of the common first $n-1$ rows
are Gaussian.  Using \eqref{eq:det-one-row-projection} and Jensen's
inequality,
\[
 \begin{aligned}
 \E\max_j|w_j|
 &=
 \E\sqrt{\max_jw_j^2}
 =
 \E\sqrt{\max_jp_{jj}(n-1)}
 \le
 \left(\E\max_jp_{jj}(n-1)\right)^{1/2}
 \le
 C\ell^{-4a}.
 \end{aligned}
\]
Substituting this estimate into
\eqref{eq:det-one-row-distribution} proves
\eqref{det:eq:one-row}, after absorbing the fixed factor $R_3$ into
the constant.
\end{proof}

We now return to the proof of Proposition~\ref{prop:det-growing-fourth}.
We compare $B_n$ with $B_n^{\mathrm{hyb}}$ by replacing the last
$s_1$ rows one at a time.  Let $g_{m+1},\ldots,g_n$ be independent
standard Gaussian rows, independent of $B_n$.  For
$0\le r\le s_1$, let $B_n^{[r]}$ be the matrix obtained from $B_n$
by replacing rows $m+1,\ldots,m+r$ with
$g_{m+1},\ldots,g_{m+r}$.  Thus
\[
 B_n^{[0]}=B_n,
 \qquad
 B_n^{[s_1]}=B_n^{\mathrm{hyb}}.
\]
The same polynomial argument as in Lemma~\ref{lem:det-smoothing}
shows that every $B_n^{[r]}$ has all square submatrices invertible
almost surely.

Fix $1\le r\le s_1$.  The matrices $B_n^{[r-1]}$ and $B_n^{[r]}$
differ only in row $m+r$.  Let $\Pi_r$ be a permutation matrix that
moves row $m+r$ to the last position.  We apply the same permutation
to both matrices.  Then $\Pi_r B_n^{[r-1]}$ and
$\Pi_r B_n^{[r]}$ have the same first $n-1$ rows.  Their last rows
are the original $(m+r)$th row of $B_n$ and the Gaussian row
$g_{m+r}$, respectively.

Since $|\det\Pi_r|=1$, the row permutation does not change the
absolute determinant.  Hence
\[
 W_n^{\mathrm d}\bigl(\Pi_r B_n^{[j]}\bigr)
 =
 W_n^{\mathrm d}\bigl(B_n^{[j]}\bigr),
 \qquad
 j\in\{r-1,r\}.
\]
All assumptions of Lemma~\ref{lem:det-one-row} therefore hold for
the permuted pair.  The lemma gives
\[
 \dk\left(
 W_n^{\mathrm d}\bigl(B_n^{[r-1]}\bigr),
 W_n^{\mathrm d}\bigl(B_n^{[r]}\bigr)
 \right)
 \le
 C\ell^{-4a}.
\]
Applying the triangle inequality to the chain
$B_n^{[0]},B_n^{[1]},\ldots,B_n^{[s_1]}$, we obtain
\begin{equation}\label{eq:det-growing-total-replacement}
 \begin{aligned}
 \dk\left(
 W_n^{\mathrm d}(B_n),
 W_n^{\mathrm d}(B_n^{\mathrm{hyb}})
 \right)
 &=
 \dk\left(
 W_n^{\mathrm d}\bigl(B_n^{[0]}\bigr),
 W_n^{\mathrm d}\bigl(B_n^{[s_1]}\bigr)
 \right)\\
 &\le
 \sum_{r=1}^{s_1}
 \dk\left(
 W_n^{\mathrm d}\bigl(B_n^{[r-1]}\bigr),
 W_n^{\mathrm d}\bigl(B_n^{[r]}\bigr)
 \right)\\
 &\le
 Cs_1\ell^{-4a}
 \le
 C\ell^{-a}.
 \end{aligned}
\end{equation}

This proves the bound for the row-replacement term in the proof
strategy.  It remains to prove
\[
 \dk\bigl(
 W_n^{\mathrm d}(B_n^{\mathrm{hyb}}),
 \Normal
 \bigr)
 \le
 CM\ell^{-1/2}.
\]
The next lemma gives this estimate.

\begin{lemma}\label{lem:det-hybrid}
Under the assumptions of Proposition~\ref{prop:det-growing-fourth}, suppose in addition that all square submatrices of $B_n$ are invertible almost surely.  Then
\begin{equation}\label{det:eq:hybrid-bound}
 \dk(W_n^{\mathrm d}(B_n^{\mathrm{hyb}}),\Normal)
 \le CM\ell^{-1/2}.
\end{equation}
\end{lemma}

\begin{proof}[Proof of Lemma~\ref{lem:det-hybrid}]
We divide the proof into five steps.

\medskip
\noindent\textbf{Step 1: Gram--Schmidt and the Gaussian terminal block.}
Let $b_1^T,\ldots,b_n^T$ be the rows of
$B_n^{\mathrm{hyb}}$.  Since its first $m=n-s_1$ rows are same as
the first $m$ rows of $B_n$, the nonsingularity assumption implies that these rows
have rank $m$ almost surely.

Recall that $Z_{i+1}=b_{i+1}^TQ_i b_{i+1} = \frac{\operatorname{dist}(b_{i+1},\mathcal V_i)^2}{k_i}$.
Thus
\[
 \operatorname{dist}(b_{i+1},\mathcal V_i)^2
 =
 k_iZ_{i+1}.
\]
Since $\prod_{i=0}^{n-1}k_i=n!$, Gram--Schmidt gives
\begin{equation}\label{det:eq:GS-logdet}
 \log\{(\det B_n^{\mathrm{hyb}})^2\}
 -\log(n-1)!
 =
 \sum_{i=0}^{n-1}\log Z_{i+1}+\log n.
\end{equation}
Consequently,
\begin{equation}
  \begin{aligned}
    W_n^{\mathrm d}(B_n^{\mathrm{hyb}})
    =&
    \frac{
    \sum_{i=0}^{n-1}\log Z_{i+1}+\log n
    }{\sqrt{2\ell}}\\
    =& 
    \underbrace{
    \frac{
    \sum_{i=0}^{m-1}\log Z_{i+1}+\log n + \sum_{i=m}^{n-1}\mathbb{E}\log Z_{i+1}
    }{\sqrt{2\ell}}
    }_{\text{bulk part}} 
    + 
    \underbrace{
    \frac{\sum_{i=m}^{n-1}(\log Z_{i+1} - \mathbb{E}\log Z_{i+1})
    }{\sqrt{2\ell}}
    }_{\text{Gaussian replacement part}}.
    \nonumber     
  \end{aligned}
\end{equation} 
Firstly, we estimate the Gaussian replacement part. Set $\mathcal G_{s_1}^{\circ} = \sum_{i=m}^{n-1}\log Z_{i+1}$, we give a useful Lemma~\ref{lem:gaussian-terminal-block} to estimate the mean and variance of $\mathcal G_{s_1}^{\circ}$. Whose proof is postponed in Appendix~\ref{sec:proof-lemmas}.

\begin{lemma}[Gaussian terminal block]\label{lem:gaussian-terminal-block}
Let $1\le r\le n$.  For any $n\times n$ random matrix $A$, suppose that the first $n-r$ rows have rank
$n-r$ almost surely, and that the last $r$ rows are independent
standard Gaussian rows, independent of all preceding rows.
Then we have conditional on the first $n-r$ rows, the last $r$ normalized squared
Gram--Schmidt heights are independent and, in chronological order,
have laws
\[
 \frac{\chi_r^2}{r},
 \frac{\chi_{r-1}^2}{r-1},
 \ldots,
 \chi_1^2,
\]
where $\chi_1^2,\ldots,\chi_r^2$ are independent chi-square random
variables with respective degrees of freedom $1,\ldots,r$.

Let $\mathcal G_r^\circ$ be a random variable with law $\mathcal G_r^\circ \stackrel{d}{=} \sum_{k=1}^r\log(\chi_k^2/k)$ and denote 
\begin{equation}\label{eq:gaussian-terminal-block-definition}
 \mu_r=\E\mathcal G_r^\circ,
 \qquad
 \mathcal G_r^{\mathrm c}
 =
 \mathcal G_r^\circ-\mu_r.
\end{equation}
Then
\begin{align}
 \left|\mu_r+\sum_{k=1}^r\frac1k\right|
 \le C,
 \label{eq:gaussian-terminal-mean}\\
 \Var(\mathcal G_r^{\mathrm c})
 =
 2\sum_{k=1}^r\frac1k+O(1),
 \label{eq:gaussian-terminal-variance}\\
 \sum_{k=1}^r
 \E\left|
 \log(\chi_k^2/k)-\E\log(\chi_k^2/k)
 \right|^3
 \le C.
 \label{eq:gaussian-terminal-third}
\end{align}
The constants are absolute and uniform in $r$.
\end{lemma}

Lemma~\ref{lem:gaussian-terminal-block}, applied with $r=s_1$,
shows that $\mathcal G_{s_1}^{\circ}$ is independent of $\mathcal F_m$ and
has the same law as $\sum_{k=1}^{s_1}\log(\chi_k^2/k)$, where $\chi_k^2$ are independent Chi-square random variables with degree $k$. Denote $ \E \mathcal G_{s_1}^{\circ} = \mu_{s_1}$, by Lemma~\ref{lem:gaussian-terminal-block} we have
\[
 \qquad
 \Var( \mathcal G_{s_1}^{\circ} )
 \le C(1+\log s_1)
 \le C\log\ell.
\]
Since the bulk sum is $\mathcal F_m$-measurable,
$\mathcal G_{s_1}^{\circ}-\mu_{s_1}$ is an independent centered
perturbation.  Lemma~\ref{lem:perturb}(ii) therefore gives
\begin{equation}\label{det:eq:gaussian-block-perturb}
 \dk\left(
 W_n^{\mathrm d}(B_n^{\mathrm{hyb}}),\Normal
 \right)
 \le
 \dk\left(
 \frac{
 \sum_{i=0}^{m-1}\log Z_{i+1}
 +\log n+\mu_{s_1}
 }{\sqrt{2\ell}},
 \Normal
 \right)
 +
 C\frac{\log\ell}{\ell}.
\end{equation}

\medskip
\noindent\textbf{Step 2: lower truncation and harmonic correction.}
Let
\begin{equation}
  \begin{aligned}\label{eq:truncate-notation-definition}
    \theta=& \ell^{-a/2},
    \quad
    H=|\log\theta|=\frac a2\log\ell,\\
    L_{i+1}=& \log(Z_{i+1}\vee\theta),
    \quad
    \zeta_{i+1}= L_{i+1}+\frac1{k_i},
    \quad 0\le i<m.
  \end{aligned}
\end{equation}
Since $k_i$ ranges from $n$ down to $s_1+1$ as $i$ ranges from
$0$ to $m-1$, we have 
$
 \sum_{i=0}^{m-1}\frac1{k_i}
 =
 \sum_{k=s_1+1}^{n}\frac1k.
$
Hence
\begin{equation}\label{det:eq:harmonic-correction}
 \sum_{i=0}^{m-1}L_{i+1}
 +\log n+\mu_{s_1}
 =
 \sum_{i=0}^{m-1}\zeta_{i+1}+C_0,
\end{equation}
where
$
 C_0
 =
 \log n+\mu_{s_1}
 -
 \sum_{k=s_1+1}^{n}\frac1k$.
Moreover,
\[
 C_0
 =
 \left(
 \log n-\sum_{k=1}^{n}\frac1k
 \right)
 +
 \left(
 \mu_{s_1}+\sum_{k=1}^{s_1}\frac1k
 \right).
\]
The first parenthesis is uniformly bounded by the standard
harmonic-number estimate, while the second is uniformly bounded by \eqref{eq:gaussian-terminal-mean}.  Therefore $|C_0|\le C$.
Lemma~\ref{lem:perturb}(iii) now gives
\begin{equation}\label{det:eq:harmonic-shift-bound}
 \dk\left(
 \frac{
 \sum_{i=0}^{m-1}L_{i+1}
 +\log n+\mu_{s_1}
 }{\sqrt{2\ell}},
 \Normal
 \right)
 \le
 \dk\left(
 \frac{\sum_{i=0}^{m-1}\zeta_{i+1}}{\sqrt{2\ell}},
 \Normal
 \right)
 +
 C\ell^{-1/2}.
\end{equation}

\medskip
\noindent\textbf{Step 3: normal approximation for the corrected bulk.}
Put
\begin{equation}
  \begin{aligned}\label{def:xi-zeta-first}
    s_\zeta^2
    =
    \sum_{i=0}^{m-1}\E\zeta_{i+1}^2,
    \qquad
    \xi_j=\frac{\zeta_j}{s_\zeta},
    \qquad
    \sigma_{\zeta,j}^2
    =
    \frac{\E\zeta_{j}^2}{s_\zeta^2} = \E \xi_{j}^{2}.
    \qquad 1\le j\le m.    
  \end{aligned}
\end{equation}
Then $\xi_j$ is $\mathcal F_j$-measurable and
\[
 \sum_{j=1}^{m}\E\xi_j^2=1,\qquad  v_{\zeta,i}=\sum_{j=i+1}^{m-1}\sigma_{\zeta,j+1}^2,
 \qquad 0\le i<m,
\]
Combining Step1 and Step2, it remains to estimate $ \dk\left( \frac{\sum_{i=0}^{m-1}\zeta_{i+1}}{\sqrt{2\ell}}, \Normal \right)$ in \eqref{det:eq:harmonic-shift-bound} and the loss caused by the lower truncation (which is estimated in the step 4). In this step, we firstly estimate $\dk\left( \sum_{i=0}^{m-1}\xi_{i+1}, \Normal \right)$ and $ \left| \frac{s_\zeta}{\sqrt{2\ell}}-1 \right|$ seperately.  Then we combine these two estimates and using Lemma~\ref{lem:perturb}(iv) to get the bound of $ \dk\left( \frac{\sum_{i=0}^{m-1}\zeta_{i+1}}{\sqrt{2\ell}}, \Normal \right)$.

To estimate $\dk\left( \sum_{i=0}^{m-1}\xi_{i+1}, \Normal \right)$, we give the following lemma.

\begin{lemma}\label{lem:det-telescope}
Let $\xi_j$ be $\cF_j$-measurable, $1\le j\le N$, and put
$S_j=\sum_{r=1}^j\xi_r$.  Assume
\[
 \sigma_j^2=\E\xi_j^2,
 \qquad \sum_{j=1}^N\sigma_j^2=1,
 \qquad v_j=\sum_{r=j+1}^N\sigma_r^2.
\]
Write
\[
 m_j=\E(\xi_j\mid\cF_{j-1}),
 \quad a_j=\E(\xi_j^2\mid\cF_{j-1}),
 \quad c_j=\E(\xi_j^3\mid\cF_{j-1}).
\]
For $t\ge0$, set
\begin{align*}
 \mathcal M(t)&=\sum_j\E[e^{itS_{j-1}}e^{-t^2v_j/2}m_j],\\
 \mathcal B(t)&=\sum_j\E[e^{itS_{j-1}}e^{-t^2v_j/2}(a_j-\sigma_j^2)],\\
 \mathcal C(t)&=\sum_j\E[e^{itS_{j-1}}e^{-t^2v_j/2}c_j],\\
 \mathcal R_4(t)&=\sum_je^{-t^2v_j/2}(\E|\xi_j|^4+\sigma_j^4).
\end{align*}
There are absolute constants $c_0,C>0$ such that, whenever
$T^2\max_j\sigma_j^2\le c_0$,
\begin{equation}\label{det:eq:noncentered-telescope-bound}
 \dk(S_N,\Normal)
 \le\frac CT+C\int_0^T
 \{|\mathcal M(t)|+t|\mathcal B(t)|
 +t^2|\mathcal C(t)|+t^3\mathcal R_4(t)\}\dd t.
\end{equation}
\end{lemma}

\begin{proof}[Proof of Lemma~\ref{lem:det-telescope}]
Let $\varphi(t)=\E e^{itS_N}$.  Lemma~\ref{lem:esseen} gives
\[
 \dk(S_N,\Normal)
 \le\frac CT+C\int_0^T\frac{|\varphi(t)-e^{-t^2/2}|}{t}\dd t.
\]
Since $v_{j-1}=\sigma_j^2+v_j$, telescoping a Gaussian future yields
\begin{align*}
 \varphi(t)-e^{-t^2/2}
 =\sum_{j=1}^N\E\left[e^{itS_{j-1}}e^{-t^2v_j/2}
 \left\{\E(e^{it\xi_j}\mid\cF_{j-1})
 -e^{-t^2\sigma_j^2/2}\right\}\right].
\end{align*}
Taylor expansion gives
\[
 \E(e^{it\xi_j}\mid\cF_{j-1})
 =1+itm_j-\frac{t^2}{2}a_j-\frac{it^3}{6}c_j
 +O(t^4\E(|\xi_j|^4\mid\cF_{j-1})),
\]
while $T^2\max_j\sigma_j^2\le c_0$ gives
\[
 e^{-t^2\sigma_j^2/2}
 =1-\frac{t^2}{2}\sigma_j^2+O(t^4\sigma_j^4).
\]
Subtraction, division by $t$, and integration prove
\eqref{det:eq:noncentered-telescope-bound}. Thus we complete the proof of Lemma~\ref{lem:det-telescope}.
\end{proof}

We now return to the proof of Lemma~\ref{lem:det-hybrid}.
To apply Lemma~\ref{lem:det-telescope}, we need to check the smallness condition $T^2\max_j\sigma_{\zeta,j}^2\le c_0$. Thus, we give the following Lemma~\ref{lem:common-one-step} to estimate the moments of $\zeta_{i}$, and we can use this lemma to estimate $\max_j\sigma_{\zeta,j}^2$.

\begin{lemma}\label{lem:common-one-step}
Let $B_n=(b_{ij})_{1\le i,j\le n}$ have independent centered
real-valued entries of variance one, and assume that all its
square submatrices are invertible almost surely.
Use the notation in
\eqref{eq:common-parameters}--\eqref{eq:common-heights}
and \eqref{eq:truncate-notation-definition}.
Define
\[
 M=\max\left\{3,\max_{1\le i,j\le n}\E|b_{ij}|^4\right\},
\]
and assume that $M\le C_0\ell^\alpha$, where
$\alpha\ge0$, $a\ge20$, and $C_0>0$ are fixed and
$a>3\alpha+12$.
The same $M$ bounds the fourth moments after the terminal
Gaussian replacement, since standard Gaussian entries have
fourth moment $3$.

For $1\le r\le4$, define
\begin{equation}\label{det:eq:lambda-growing-origin}
 \lambda_i^{(r)}
 =
 \int_{T_0}^{H}u^{r-1}
 \left\{
 M^2e^{4u}k_i^{-2}
 +M\ell^{4a}k_i^{-3}
 +M^2k_i^{-2}
 \right\}\dd u,
\end{equation}
where $T_0>1$ is a sufficiently large absolute constant.
Also define
\begin{equation}\label{eq:common-one-step-error}
 H_i^*
 =
 Md_i
 +\sqrt M\sqrt{\frac{d_i}{k_i}}
 +M^{3/2}k_i^{-3/2}
 +M^2k_i^{-2}
 +\sum_{q=1}^4\lambda_i^{(q)},
 \qquad
 h_i=\E H_i^*.
\end{equation}
Then, for all sufficiently large $n$ and every $0\le i<m$,
almost surely,
\begin{equation}\label{eq:common-one-step-conditional}
 \left|\E_i\zeta_{i+1}\right|
 +
 \left|\E_i\zeta_{i+1}^2-\frac2{k_i}\right|
 +
 \left|\E_i\zeta_{i+1}^3\right|
 +
 \E_i|\zeta_{i+1}|^4
 \le CH_i^*.
\end{equation}
Consequently,
\begin{align}
 \E\left|\E_i\zeta_{i+1}\right|
 &\le Ch_i,
 \label{eq:common-one-step-mean}\\
 \E\left|
 \E_i\zeta_{i+1}^2-\E\zeta_{i+1}^2
 \right|
 &\le Ch_i,
 \label{eq:common-one-step-second}\\
 \E\left|\E_i\zeta_{i+1}^3\right|
 &\le Ch_i,
 \label{eq:common-one-step-third}\\
 \E|\zeta_{i+1}|^4
 &\le Ch_i,
 \label{eq:common-one-step-fourth}\\
 \left|
 \E\zeta_{i+1}^2-\frac2{k_i}
 \right|
 &\le Ch_i.
 \label{eq:common-one-step-second-mean}
\end{align}
Moreover,
\begin{equation}\label{eq:common-h-sum-point}
 \sum_{i=0}^{m-1}h_i\le CM,
 \qquad
 h_i\le\frac{CM}{k_i},
 \quad 0\le i<m,
\end{equation}
and
\begin{equation}\label{eq:common-h-terminal}
 \sum_{s_1\le k_i<s_2}h_i
 \le CM\ell^{-(2\alpha+12)}.
\end{equation}
Here $C$ depends only on
$C_0,\alpha,a$.
\end{lemma}

The proof of Lemma~\ref{lem:common-one-step} is given in Appendix~\ref{sec:proof-lemmas}. Using \eqref{eq:common-one-step-second-mean} and
\eqref{eq:common-h-sum-point} in Lemma~\ref{lem:common-one-step}, we have
\[
 s_\zeta^2
 =
 2\sum_{k=s_1+1}^{n}\frac1k
 +O\left(\sum_i h_i\right)
 =
 2\ell+O(M+\log\ell).
\]
Moreover,
\[
 \E\zeta_{i+1}^2
 \le
 \frac2{k_i}+Ch_i
 \le
 \frac{CM}{k_i},
\]
while $s_\zeta^2\asymp\ell$, because
$M\le K_0\sqrt\ell$ is the assumption of Proposition~\ref{prop:det-growing-fourth}.
Therefore
\[
 \sigma_{\zeta,i+1}^2
 =
 \frac{\E\zeta_{i+1}^2}{s_\zeta^2}
 \le
 \frac{CM}{\ell k_i}
 \le
 \frac{CM}{\ell s_1}.
\]
Taking the maximum over $0\le i<m$ gives \eqref{eq:det-zeta-scale}
\begin{equation}\label{eq:det-zeta-scale}
 s_\zeta^2
 =
 2\ell+O(M+\log\ell),
 \qquad
 \max_{0\le i<m}\sigma_{\zeta,i+1}^2
 \le
 \frac{CM}{\ell s_1}.
\end{equation}

Therefore, we can verify the smallness condition in Lemma~\ref{lem:det-telescope}.
We apply Lemma~\ref{lem:det-telescope} with $N=m$ and
$T=\sqrt\ell$. By \eqref{eq:det-zeta-scale},
\[
 T^2\max_{0\le i<m}\sigma_{\zeta,i+1}^2
 \le
 \ell\frac{CM}{\ell s_1}
 =
 \frac{CM}{s_1}
 \le
 \frac{C\sqrt{\ell}}{\ell}
 =
 o(1).
\]
Thus the smallness condition in
Lemma~\ref{lem:det-telescope} holds for all sufficiently large $n$. And using Lemma~\ref{lem:det-telescope} gives
\begin{equation}
  \begin{aligned}\label{eq:xi-telescope-middle-result}
    \dk\left( \sum_{i=0}^{m-1}\xi_{i+1}, \Normal \right) 
    \leq 
    C\ell^{-1/2}+C\int_0^{\sqrt{\ell}}
    \{|\mathcal M(t)|+t|\mathcal B(t)|
    +t^2|\mathcal C(t)|+t^3\mathcal R_4(t)\}\dd t,
  \end{aligned}
\end{equation}
where $\mathcal M,\mathcal B,\mathcal C$, and $\mathcal R_4$ be the
quantities in Lemma~\ref{lem:det-telescope} for the present sequence $\{\xi_i\}_{i \geq 1}$ defined in \eqref{def:xi-zeta-first}.  

To estimate the integrals in \eqref{eq:xi-telescope-middle-result}, we use the one-step estimates in Lemma~\ref{lem:common-one-step} to bound the conditional moments of $\xi_{i+1}$. The one-step estimates in Lemma~\ref{lem:common-one-step} give
\begin{equation}
  \begin{aligned}\label{eq:one-step-estimate-Lemm4.8}
    \E\left|\E_i\xi_{i+1}\right|
    \le\frac{Ch_i}{s_\zeta},
    \qquad
    \E\left|
    \E_i\xi_{i+1}^2-\E\xi_{i+1}^2
    \right|
    \le\frac{Ch_i}{s_\zeta^2},\\
    \E\left|\E_i\xi_{i+1}^3\right|
    \le\frac{Ch_i}{s_\zeta^3},
    \qquad
    \E|\xi_{i+1}|^4
    \le\frac{Ch_i}{s_\zeta^4}.
  \end{aligned}
\end{equation}
By the definitions of $\mathcal M$, $\mathcal B$, and $\mathcal C$,
each summand contains the factors $e^{itS_i}$ and
$e^{-t^2v_{\zeta,i}/2}$.  We have
\[
 |e^{itS_i}|=1,
 \qquad
 0<e^{-t^2v_{\zeta,i}/2}\le1.
\]
Therefore, the triangle inequality and \eqref{eq:one-step-estimate-Lemm4.8} give
\begin{equation}
  \begin{aligned}\label{eq:estimate-M-B-C-intergal}
  |\mathcal M(t)|
  &\le
  \sum_{i=0}^{m-1}
  e^{-t^2v_{\zeta,i}/2}
  \E\left|\E_i\xi_{i+1}\right|
  \le
  \frac C{s_\zeta}
  \sum_{i=0}^{m-1}
  h_i e^{-t^2v_{\zeta,i}/2},\\
  |\mathcal B(t)|
  &\le
  \sum_{i=0}^{m-1}
  e^{-t^2v_{\zeta,i}/2}
  \E\left|
  \E_i\xi_{i+1}^2-\E\xi_{i+1}^2
  \right|
  \le
  \frac C{s_\zeta^2}
  \sum_{i=0}^{m-1}
  h_i e^{-t^2v_{\zeta,i}/2},\\
  |\mathcal C(t)|
  &\le
  \sum_{i=0}^{m-1}
  e^{-t^2v_{\zeta,i}/2}
  \E\left|\E_i\xi_{i+1}^3\right|
  \le
  \frac C{s_\zeta^3}
  \sum_{i=0}^{m-1}
  h_i e^{-t^2v_{\zeta,i}/2}.
  \end{aligned}
\end{equation}
The fourth-moment estimate gives, in the same way,
\begin{equation}
  \begin{aligned}\label{eq:estimate-forth-moment-integral}
    \sum_{i=0}^{m-1}
    e^{-t^2v_{\zeta,i}/2}\E|\xi_{i+1}|^4
    \le
    \frac C{s_\zeta^4}
    \sum_{i=0}^{m-1}h_i e^{-t^2v_{\zeta,i}/2}.
  \end{aligned}
\end{equation}
Thus combining \eqref{eq:estimate-M-B-C-intergal} and \eqref{eq:estimate-forth-moment-integral}, to estimate \eqref{eq:xi-telescope-middle-result}, we need to calculate the integrals of the form $\int_0^{\sqrt\ell}t^p\sum_{i=0}^{m-1}h_i e^{-t^2v_{\zeta,i}/2}\dd t$ for $p=0,1,2,3$.  The following Lemma~\ref{lem:det-weighted} gives a uniform bound for these integrals.

\begin{lemma}\label{lem:det-weighted}
Define
\[
 v_{\zeta,i}=\sum_{j=i+1}^{m-1}\sigma_{\zeta,j+1}^2,
 \qquad 0\le i<m,
\]
where the empty sum is understood to be zero.  Then, for every fixed $c>0$ and $p=0,1,2,3$,
\begin{equation}\label{det:eq:weighted-growing}
 \int_0^{\sqrt\ell}t^p
 \sum_{i=0}^{m-1}h_i e^{-ct^2v_{\zeta,i}}\dd t
 \le C_pM,
\end{equation}
where $C_p$ may depend on $c$ and on the constants in the standing
assumptions, but not on $n$ or $M$.
\end{lemma}

\begin{proof}[Proof of Lemma~\ref{lem:det-weighted}]
As throughout this section, it is enough to consider all sufficiently
large $n$.

Suppose first that $k_i\ge s_2$.  By the definitions of
$v_{\zeta,i}$ and $\sigma_{\zeta,j+1}^2$, $s_\zeta^2v_{\zeta,i} = \sum_{j=i+1}^{m-1}\E\zeta_{j+1}^2$.
Using \eqref{eq:common-one-step-second-mean}, we obtain
\begin{align*}
 s_\zeta^2v_{\zeta,i}
 &\ge
 2\sum_{j=i+1}^{m-1}\frac1{k_j}
 -
 C\sum_{j=i+1}^{m-1}h_j
 =
 2\sum_{q=s_1+1}^{k_i-1}\frac1q
 -
 C\sum_{j=i+1}^{m-1}h_j.
\end{align*}
Here we used $k_j=n-j$ and $m=n-s_1$, so that $k_j$ ranges
from $k_i-1$ down to $s_1+1$ as $j$ ranges from $i+1$ to
$m-1$.

Since $k_i\ge s_2$,
\[
 \sum_{q=s_1+1}^{k_i-1}\frac1q
 \ge
 \int_{s_1+1}^{k_i}\frac{\dd x}{x}
 \ge
 \log\frac{s_2}{s_1+1}.
\]
For all sufficiently large $n$, $s_2\ge\frac12n\ell^{-20a}$,
$s_1+1\le2\ell^{3a}$,
and therefore
\[
 \log\frac{s_2}{s_1+1}
 \ge
 \ell-23a\log\ell-O(1).
\]
Moreover, \eqref{eq:common-h-sum-point} gives $\sum_{j=i+1}^{m-1}h_j\le CM$.
Consequently,
\[
 s_\zeta^2v_{\zeta,i}
 \ge
 2\ell-46a\log\ell-CM-O(1).
\]
Since $M\le K_0\sqrt\ell$, the right-hand side is at least
$c_1\ell$ for some constant $c_1>0$ and all sufficiently large
$n$.  On the other hand,
\eqref{eq:det-zeta-scale} gives $s_\zeta^2\le C\ell$.
It follows that
\[
 v_{\zeta,i}\ge c_0>0
 \qquad\text{whenever }k_i\ge s_2.
\]

Hence, for every $p=0,1,2,3$,
\[
 \int_0^{\sqrt\ell}t^pe^{-ct^2v_{\zeta,i}}\dd t
 \le
 \int_0^\infty t^pe^{-cc_0t^2}\dd t
 \le C_p.
\]
Therefore,
\begin{align*}
 \int_0^{\sqrt\ell}t^p
 \sum_{k_i\ge s_2}h_i e^{-ct^2v_{\zeta,i}}\dd t
 &=
 \sum_{k_i\ge s_2}h_i
 \int_0^{\sqrt\ell}t^pe^{-ct^2v_{\zeta,i}}\dd t
 \le
 C_p\sum_{k_i\ge s_2}h_i
 \le C_pM,
\end{align*}
where the last inequality follows from
\eqref{eq:common-h-sum-point}.

It remains to consider the range $s_1\le k_i<s_2$.  Here we use only
$e^{-ct^2v_{\zeta,i}}\le1$.  By
\eqref{eq:common-h-terminal},
\begin{align*}
 &\int_0^{\sqrt\ell}t^p
 \sum_{s_1\le k_i<s_2}
 h_i e^{-ct^2v_{\zeta,i}}\dd t
 \le
 \left(
 \sum_{s_1\le k_i<s_2}h_i
 \right)
 \int_0^{\sqrt\ell}t^p\dd t
 \le
 CM\ell^{-5}
 \frac{\ell^{(p+1)/2}}{p+1}.
\end{align*}
Since $p\le3$, one has
$\ell^{(p+1)/2}\le\ell^2$, and hence the last expression is bounded
by
\[
 C_pM\ell^{-3}\le C_pM.
\]
Combining the two ranges proves
\eqref{det:eq:weighted-growing}. Thus the proof of Lemma~\ref{lem:det-weighted} is complete.
\end{proof}

We now return to the proof of Lemma~\ref{lem:det-hybrid}.
Applying Lemma~\ref{lem:det-weighted} with $c=1/2$ and
$p=0,1,2,3$, respectively, gives
\begin{align}
 \int_0^{\sqrt\ell}|\mathcal M(t)|\dd t
 &\le C\frac M{s_\zeta},
 \label{det:eq:M-growing-final}\\
 \int_0^{\sqrt\ell}t|\mathcal B(t)|\dd t
 &\le C\frac M{s_\zeta^2},
 \label{det:eq:B-growing-final}\\
 \int_0^{\sqrt\ell}t^2|\mathcal C(t)|\dd t
 &\le C\frac M{s_\zeta^3},
 \label{det:eq:C-growing-final}\\
 \int_0^{\sqrt\ell}t^3
 \sum_{i=0}^{m-1}
 e^{-t^2v_{\zeta,i}/2}\E|\xi_{i+1}|^4\dd t
 &\le C\frac M{s_\zeta^4}.
 \label{det:eq:R-growing-zeta}
\end{align}

For \eqref{eq:xi-telescope-middle-result}, it remains to control the Gaussian Taylor remainder.  Since
$\sum_i\sigma_{\zeta,i+1}^2=1$,
\[
 \sum_{i=0}^{m-1}\sigma_{\zeta,i+1}^4
 \le
 \max_{0\le i<m}\sigma_{\zeta,i+1}^2\sum_{i=0}^{m-1}\sigma_{\zeta,i+1}^2
 =
 \max_{0\le i<m}\sigma_{\zeta,i+1}^2.
\]
Therefore
\begin{align}
 &\int_0^{\sqrt\ell}t^3
 \sum_{i=0}^{m-1}
 e^{-t^2v_{\zeta,i}/2}\sigma_{\zeta,i+1}^4\dd t
 \le
 \frac{\ell^2}{4}
 \max_{0\le i<m}\sigma_{\zeta,i+1}^2
 \le
 C\frac{M\ell}{s_1}
 =
 o\left(\frac M{\sqrt\ell}\right).
 \label{det:eq:R-growing-sigma}
\end{align}

By \eqref{eq:det-zeta-scale}, $s_\zeta^2\asymp\ell$.
Substituting
\eqref{det:eq:M-growing-final}--\eqref{det:eq:R-growing-sigma}
into \eqref{eq:xi-telescope-middle-result}, we obtain
\begin{equation}
  \begin{aligned}
  \dk\left(
  \frac{\sum_{i=0}^{m-1}\zeta_{i+1}}{s_\zeta},
  \Normal
  \right)
  &\le
  C\ell^{-1/2}
  +
  C\left\{
  \frac M{s_\zeta}
  +\frac M{s_\zeta^2}
  +\frac M{s_\zeta^3}
  +\frac M{s_\zeta^4}
  \right\}
  +
  o\left(\frac M{\sqrt\ell}\right)\\
  &\le
  C\left(
  \ell^{-1/2}+\frac M{\sqrt\ell}
  \right).
  \label{det:eq:bulk-standard-s}
  \end{aligned}
\end{equation}

Next,
\[
 \left|
 \frac{s_\zeta}{\sqrt{2\ell}}-1
 \right|
 =
 \frac{|s_\zeta^2-2\ell|}
 {\sqrt{2\ell}\{s_\zeta+\sqrt{2\ell}\}}
 \le
 C\frac{M+\log\ell}{\ell}.
\]
This quantity tends to zero because $M\le K_0\sqrt\ell$, and is
therefore at most $1/2$ for all sufficiently large $n$.  Hence
Lemma~\ref{lem:perturb}(iv) gives
\begin{equation}\label{det:eq:corrected-bulk-growing}
 \dk\left(
 \frac{\sum_{i=0}^{m-1}\zeta_{i+1}}{\sqrt{2\ell}},
 \Normal
 \right)
 \le
 C\left(
 \ell^{-1/2}+\frac M{\sqrt\ell}
 \right).
\end{equation}

\medskip
\noindent\textbf{Step 4: remove the lower truncation.}

In the step2 \eqref{eq:truncate-notation-definition}, we have truncated the logarithm $\log Z_{i+1}$ at $\theta=\ell^{-a/2}$ to control the lower tail.  In this step, we show that the truncation has negligible effect on the distribution of the bulk sum.  The following Lemma~\ref{lem:common-lower-tail} gives a bound for the lower tail of $Z_{i+1}$ and shows that the truncation error is small, whose proof is postponed in Appendix~\ref{sec:proof-lemmas}.

\begin{lemma}\label{lem:common-lower-tail}
Under the assumptions and notation of
Lemma~\ref{lem:common-one-step}, take the absolute constant
$T_0>1$ in \eqref{det:eq:lambda-growing-origin} sufficiently large.
Then, for all sufficiently large $n$, all $T_0\le u\le H$,
and every $0\le i<m$, almost surely,
\begin{equation}\label{det:eq:lower-tail-Z}
 \Pp_i(Z_{i+1}<e^{-u})
 \le C\left\{
 M^2e^{4u}k_i^{-2}
 +M\ell^{4a}k_i^{-3}
 +M^2k_i^{-2}
 \right\}.
\end{equation}
And we have the quantities $\lambda_i^{(r)}$ defined in
\eqref{det:eq:lambda-growing-origin} satisfy
\begin{equation}\label{det:eq:lambda-growing-total}
 \sum_{i=0}^{m-1}\sum_{r=1}^4\lambda_i^{(r)}
 \le C(\log\ell)^4
 \left(
 M^2\ell^{-a}
 +M\ell^{-2a}
 +M^2\ell^{-3a}
 \right).
\end{equation}
Moreover,
\begin{equation}\label{det:eq:lower-truncation-growing-error}
 \Pp\{L_{i+1}\ne\log Z_{i+1}
 \text{ for some }0\le i<m\}
 \le C\left(
 M^2\ell^{-a}
 +M\ell^{-2a}
 +M^2\ell^{-3a}
 \right).
\end{equation}
Here $C$ depends only on $C_0,\alpha,a$.
\end{lemma}

By  \eqref{det:eq:lower-truncation-growing-error} in Lemma~\ref{lem:common-lower-tail}, we obtain
\begin{align*}
 &\Pp\left(
 L_{i+1}\ne\log Z_{i+1}
 \text{ for some }0\le i<m
 \right)
 \le
 C\left\{
 M^2\ell^{-a}
 +M\ell^{-2a}
 +M^2\ell^{-3a}
 \right\}.
\end{align*}
Since $M\le K_0\sqrt\ell$ and $a=20$, the right-hand side is more less than $C\frac M{\sqrt\ell}$.
On the complementary event,
\[
 \sum_{i=0}^{m-1}L_{i+1}
 =
 \sum_{i=0}^{m-1}\log Z_{i+1}.
\]
The coupling inequality, followed by
\eqref{det:eq:harmonic-shift-bound} and
\eqref{det:eq:corrected-bulk-growing}, therefore gives
\begin{equation}\label{det:eq:untruncated-bulk-growing}
 \dk\left(
 \frac{
 \sum_{i=0}^{m-1}\log Z_{i+1}
 +\log n+\mu_{s_1}
 }{\sqrt{2\ell}},
 \Normal
 \right)
 \le
 C\left(
 \ell^{-1/2}+\frac M{\sqrt\ell}
 \right).
\end{equation}

\medskip
\noindent\textbf{Step 5: restore the Gaussian terminal block.}
Combining
\eqref{det:eq:gaussian-block-perturb} and
\eqref{det:eq:untruncated-bulk-growing}, and using $\frac{\log\ell}{\ell}\le C\ell^{-1/2}$,
we conclude that
\[
 \dk\left(
 W_n^{\mathrm d}(B_n^{\mathrm{hyb}}),\Normal
 \right)
 \le
 C\left(
 \ell^{-1/2}+\frac M{\sqrt\ell}
 \right)
 \le
 C\frac M{\sqrt\ell}.
\]
This proves \eqref{det:eq:hybrid-bound}. And thus we complete the proof of Lemma~\ref{lem:det-hybrid}.
\end{proof}

We now return to the proof of Proposition~\ref{prop:det-growing-fourth}.Combining \eqref{eq:det-growing-total-replacement} and \eqref{det:eq:hybrid-bound}, we obtain
\begin{equation}\label{eq:det-growing-combination}
 \dk(W_n^{\mathrm d}(B_n),\Normal)
 \le C\ell^{-a}+CM\ell^{-1/2}.
\end{equation}
Because $a=20$, the first term in \eqref{eq:det-growing-combination} is absorbed by $CM\ell^{-1/2}$.  This proves \eqref{eq:det-growing-bound}. Thus we complete the proof of Proposition~\ref{prop:det-growing-fourth}.
\end{proof}

\subsection{Proof of Proposition~\ref{prop:exact-growing-fourth}}\label{subsec:proof-prop-exact-growing-fourth}

\paragraph*{Proof strategy}

Our goal is to prove \eqref{eq:exact-growing-bound} and
\eqref{eq:exact-growing-asymptotic}.  We compare $B_n$ with the
hybrid matrix $B_n^{\mathrm{hyb}}$.  By the triangle inequality,
\[
 \dk\bigl(W_n^{\mathrm e}(B_n),\Normal\bigr)
 \le
 \dk\bigl(
 W_n^{\mathrm e}(B_n),
 W_n^{\mathrm e}(B_n^{\mathrm{hyb}})
 \bigr)
 +
 \dk\bigl(
 W_n^{\mathrm e}(B_n^{\mathrm{hyb}}),
 \Normal
 \bigr).
\]
We estimate the two terms on the right separately.

For the first term, we replace the last $s_1$ rows one at a time.
Each replacement changes both the logarithmic determinant and its
exact mean.  The bounded-density assumption allows us to compare the
two exact means.  Summing the one-row estimates gives
\[
 \dk\bigl(
 W_n^{\mathrm e}(B_n),
 W_n^{\mathrm e}(B_n^{\mathrm{hyb}})
 \bigr)
 \le
 C\ell^{-a}\log\ell
 =
 o(\ell^{-1}).
\]

For the second term, we use the Gram--Schmidt decomposition of
$B_n^{\mathrm{hyb}}$.  It splits the logarithmic determinant into a
non-Gaussian bulk part and a Gaussian terminal part.  In the bulk, we
subtract the conditional mean from each increment.  This gives a
martingale sum and a centered sum of conditional means.  We approximate
the martingale sum together with the Gaussian terminal part by a normal
law.  We show that the centered conditional means and the truncation
error are small.  Under the quantitative assumptions of the
proposition, this gives
\[
 \dk\bigl(
 W_n^{\mathrm e}(B_n^{\mathrm{hyb}}),
 \Normal
 \bigr)
 \le
 C\frac{M}{\ell}.
\]
If $M=o(\ell)$, the same hybrid distance tends to zero.

Combining the two quantitative bounds gives
\[
 \dk\bigl(W_n^{\mathrm e}(B_n),\Normal\bigr)
 \le
 C\ell^{-a}\log\ell
 +
 C\frac{M}{\ell}
 \le
 C\frac{M}{\ell},
\]
where the last inequality uses $a=20$ and $M\ge3$.  This proves
\eqref{eq:exact-growing-bound}.  If $M=o(\ell)$, then
\[
 \dk\bigl(W_n^{\mathrm e}(B_n),\Normal\bigr)
 \le
 o(\ell^{-1})+o(1)
 =
 o(1),
\]
which proves \eqref{eq:exact-growing-asymptotic}.

\begin{proof}[Proof of Proposition~\ref{prop:exact-growing-fourth}]

Unlike the deterministic-centering case, no smoothing reduction is
needed here.  Since the entries are independent and have densities,
every square submatrix of $B_n$ is invertible almost surely.  We first
control the row-replacement term in the proof strategy.  The next
lemma gives the required one-row estimate for $W_n^{\mathrm e}$.

\begin{lemma}[Replacing one terminal row under exact centering]\label{lem:exact-one-row}
Let $C_n$ and $\overline C_n$ have identical first $n-1$ rows.  Suppose the last row of $\overline C_n$ is standard Gaussian, and suppose that all non-Gaussian entries of $C_n$ and $\overline C_n$ satisfy the assumptions of Proposition~\ref{prop:exact-growing-fourth}. Then, we have
\begin{equation}\label{ex:eq:one-row}
 \sup_{x\in\R}\left|
 \Pp\bigl(W_n^{\mathrm e}(C_n)\le x\bigr)
 -\Pp\bigl(W_n^{\mathrm e}(\overline C_n)\le x\bigr)
 \right|
 \le C\ell^{-4a}\log\ell.
\end{equation}
\end{lemma}

\begin{proof}[Proof of Lemma~\ref{lem:exact-one-row}]
First, we use a very similar argument like proof of Lemma~\ref{lem:det-one-row}, and we briefly sketch the main steps.
Let $\mathcal{F}_{n-1}$ be the $\sigma$-algebra generated by the first common $n-1$ rows of $C_n$ and $\overline C_n$, and let $\alpha = (\alpha_{1},\dots\alpha_{n})$ be the vector of cofactors corresponding to the last row.
It is nonzero almost surely; write its Euclidean norm as $\Lambda$ and normalize it to $w=(w_1,\ldots,w_n)$ with $\sum_jw_j^2=1$.  Then
\begin{equation}
  \begin{aligned}\label{eq:det-Cn-set-Cn-bar-decompose}
    \det C_n=\Lambda S_0,\qquad
    \det\overline C_n=\Lambda S_G,    
  \end{aligned}
\end{equation}
where $S_0=\sum_jw_jc_{nj}$ and $S_G \mid \mathcal{F}_{n-1} \sim N(0,1)$.  Conditionally on the first $n-1$ rows, the classical Berry--Esseen inequality and the uniform third moment give
\begin{equation}\label{ex:eq:row-BE}
 \Delta(w):=\sup_t|\Pp(S_0\le t\mid\cF_{n-1})-\Phi(t)|
 \le CR_3\sum_j|w_j|^3
 \le CR_3\max_j|w_j|.
\end{equation}
For $r\ge0$, applying \eqref{ex:eq:row-BE} at $r$ and $-r$ gives
\[
 |\Pp(|S_0|\le r\mid\cF_{n-1})-\Pp(|S_G|\le r)|
 \le2\Delta(w) \leq CR_3\max_j|w_j|.
\]
Hence the two logarithmic determinants, when centered by the same deterministic number, have conditional Kolmogorov distance at most $CR_3\max_j|w_j|$. 
Thus, let 
\begin{equation}
  \begin{aligned}
    d_n= (\E\log(\det \overline C_n)^{2} - \E\log(\det C_n)^{2})/\sqrt{2\ell}, \nonumber
  \end{aligned}
\end{equation}
we have
\begin{equation}
  \begin{aligned}
    &\sup_x\left|
    \Pp\left( W_{n}^{\mathrm e}(C_{n})\le x\right) - \Pp\left( W_{n}^{\mathrm e}(\overline C_{n})\le x\right) \right|\\
    \leq &
    \sup_x\left|
    \Pp\left( W_{n}^{\mathrm e}(C_{n})\le x\right)
    -\Pp\left( W_{n}^{\mathrm e}(\overline C_{n}) + d_{n}\le x\right)
    \right|
    +  \sup_x\left|
    \Pp\left( W_{n}^{\mathrm e}(\overline C_{n}) + d_{n}\le x\right) - \Pp\left( W_{n}^{\mathrm e}(\overline C_{n})\le x\right) \right| \\
    \le& CR_3\E \max_j|w_j| + \sup_x\left| \Pp\left( W_{n}^{\mathrm e}(\overline C_{n}) + d_{n}\le x\right) - \Pp\left( W_{n}^{\mathrm e}(\overline C_{n})\le x\right) \right|, \nonumber
  \end{aligned}
\end{equation}
By \eqref{eq:det-one-row-projection}, we have $w_j^2=p_{jj}(n-1)$.
Lemma~\ref{lem:terminal-diagonal} and Jensen imply
$\E\max_j|w_j|\le C\ell^{-4a}$. Thus, we have
\begin{equation}
  \begin{aligned}
    &\sup_x\left|
    \Pp\left( W_{n}^{\mathrm e}(C_{n})\le x\right) - \Pp\left( W_{n}^{\mathrm e}(\overline C_{n})\le x\right) \right|\\
    \le& C\ell^{-4a} + \sup_x\left| \Pp\left( W_{n}^{\mathrm e}(\overline C_{n}) + d_{n}\le x\right) - \Pp\left( W_{n}^{\mathrm e}(\overline C_{n})\le x\right) \right|, \nonumber
  \end{aligned}
\end{equation}
Since $\log(\det\overline C_n)^{2}=2\log\Lambda+\log S_G^2$, conditionally on the first $n-1$ rows, the density of $\log S_G^2$ is $f(y)=\frac{e^{y/2}}{\sqrt{2\pi}}e^{-e^y/2}, y\in\R$, and is bounded. Thus, we have 
\begin{equation}
  \begin{aligned}\label{eq:need-dn-eatimate}
    &\sup_x\left|
    \Pp\left( W_{n}^{\mathrm e}(C_{n})\le x\right) - \Pp\left( W_{n}^{\mathrm e}(\overline C_{n})\le x\right) \right|
    \le C\ell^{-4a} + C\sqrt{\ell}|d_{n}|.
  \end{aligned}
\end{equation}

We now estimate $|d_n|$.  Since \eqref{eq:det-Cn-set-Cn-bar-decompose} gives $\det C_n=\Lambda S_0$, $\det\overline C_n=\Lambda S_G$, the common factor $\Lambda$ cancels from the difference of the exact
means.  Hence
\[
 \begin{aligned}
 \sqrt{2\ell}\,|d_n|
 &=
 2\left|
 \E\log|S_0|-\E\log|S_G|
 \right|
 \le
 2\E\left|
 \E(\log|S_0|\mid\cF_{n-1})
 -
 \E\log|S_G|
 \right|.
 \end{aligned}
\]
Thus it is enough to compare the logarithmic expectations of $S_0$
and $S_G$.

The main difficulty comes from values of $S_0$ close to zero.  The
positive part of $\log|S_0|$ is controlled by the second moment, since $\log^+|S_0|\le S_0^2$.
For the negative part, we have
\[
 \E\bigl((-\log|S_0|)_+\mid\cF_{n-1}\bigr)
 =
 \int_0^\infty
 \Pp\bigl(|S_0|\le e^{-t}\mid\cF_{n-1}\bigr)\,dt.
\]
Thus we need a bound for the probability that $S_0$ is close to zero.
Such a bound follows from a uniform density estimate.  
Indeed, if $\left\|f_{S_0\mid\cF_{n-1}}\right\|_\infty\le C$,
then
\[
 \Pp(|S_0|\le r\mid\cF_{n-1})
 \le 2Cr,
 \qquad r>0.
\]
Therefore, we first establish a uniform density bound for $S_0$.
The following Lemma~\ref{lem:projection-density} provides the required estimate.

\begin{lemma}\label{lem:projection-density}
Let $X=(X_1,\ldots,X_n)$ have independent coordinates whose densities are bounded by $B$.  If $P$ is an orthogonal projection of rank $k$, then $PX$, viewed in $\operatorname{Range}(P)$, has a density $g_P$ satisfying
\begin{equation}\label{ex:eq:projection-density}
 \|g_P\|_\infty\le (\sqrt2 B)^k.
\end{equation}
Consequently, there is $u_B = \max\{1,4\log(2\sqrt{\pi e}B)\} \ge1$, depending only on $B$, such that for every $u\ge u_B$,
\begin{equation}\label{ex:eq:projection-smallball}
 \Pp\left(\frac{X^TPX}{k}\le e^{-u}\right)\le e^{-ku/4}.
\end{equation}
For $k=1$, every unit linear form $\sum_jw_jX_j$ therefore has density bounded by $\sqrt2 B$.
\end{lemma}

\begin{proof}[Proof of Lemma~\ref{lem:projection-density}]
Firstly, \eqref{ex:eq:projection-density} was proved by \cite{livshyts2016sharp}. Then we consider \eqref{ex:eq:projection-smallball}. Let $B_2^k(r)=\{x\in\R^k:\|x\|_2\le r\}$ be the $k$-dimensional Euclidean ball of radius $r$.  Its volume satisfies
\[
 |B_2^k(r)|\le\left(\frac{\sqrt{2\pi e}\,r}{\sqrt k}\right)^k.
\]
Since $X^TPX=\|PX\|^2$, taking $r=\sqrt{k}e^{-u/2}$ gives
\[
 \Pp\left(\frac{X^TPX}{k}\le e^{-u}\right)
 =\Pp\left( \|PX\|\leq r \right)
 \le \|g_P\|_{L^\infty(E)}|B_2^k(r)|
 \le(2\sqrt{\pi e}\,B e^{-u/2})^k.
\]
Choosing $u_B=\max\{1,4\log(2\sqrt{\pi e}B)\}$, then for  \(u\ge u_B\),
$2\sqrt{\pi e}\,B e^{-u/2}\le e^{-u/4}$,
and thus proves \eqref{ex:eq:projection-smallball} and complete the proof of Lemma~\ref{lem:projection-density}.
\end{proof}

We now return to the proof of Lemma~\ref{lem:exact-one-row}.
We apply Lemma~\ref{lem:projection-density} to $S_0$.
Conditionally on $\cF_{n-1}$, the vector $w$ is deterministic and
satisfies $\sum_{j=1}^n w_j^2=1$.
Hence the rank-one case of Lemma~\ref{lem:projection-density} gives
\[
 \left\|f_{S_0\mid\cF_{n-1}}\right\|_\infty
 \le
 \sqrt2 M_1.
\]
The variable $S_G$ is standard Gaussian, so its density is also
uniformly bounded.  Moreover,
\[
 \E(S_0^2\mid\cF_{n-1})=1,
 \qquad
 \E S_G^2=1.
\]
We already know from \eqref{ex:eq:row-BE} that, conditionally on
$\cF_{n-1}$, $\dk(S_0,S_G) \le \Delta(w)$.
We now need to convert this Kolmogorov bound into a bound for the
logarithmic expectations.  The following lemma gives this estimate.

\begin{lemma}[Logarithmic expectation comparison]\label{lem:log-comparison}
Let random variables $X,Y$ have densities bounded by $B$ and satisfy $\E X^2+\E Y^2\le B$.  If $\dk(X,Y)\le\rho\le e^{-1}$, then
\begin{equation}\label{ex:eq:log-compare}
 |\E\log|X|-\E\log|Y||
 \le C_B\rho\log\frac e\rho.
\end{equation}
The same conclusion holds if $X,Y$ are coupled with $\Pp(X\ne Y)\le\rho$.
\end{lemma}

\begin{proof}[Proof of Lemma~\ref{lem:log-comparison}]
First we define $\log ^{+}|x|:=\max\{\log |x|,0\}$ and $(-\log|x|)_{+}:= \max\{-\log|x|,0\}$, thus $\log |x| = \log ^{+}|x| - (-\log|x|)_{+}$ and
\begin{equation}
  \begin{aligned}
    |\E \log |X| - \E\log |Y| | =& \left| \E\log^{+}|X| - \E\log^{+}|Y| - (\E(-\log|X|)_{+} - \E(-\log|Y|)_{+}  ) \right|\\
    \leq & \left| \E\log^{+}|X| - \E\log^{+}|Y| \right| + \left| (\E(-\log|X|)_{+} - \E(-\log|Y|)_{+}  ) \right|. \nonumber
  \end{aligned}
\end{equation}
Write $F_X,F_Y$ for the distribution functions.  For every $r\ge0$,
\[
 \big|\Pp(|X|\le r)-\Pp(|Y|\le r)\big|
 \le |F_X(r)-F_Y(r)|+|F_X((-r)-)-F_Y((-r)-)|\le2\rho.
\]
Since the two densities are bounded by $B$,
\[
 \Pp(|X|\le r)+\Pp(|Y|\le r)\le4Br.
\]
To integrate the logarithmic tails, we use the following elementary identity: for every nonnegative random variable $Z$ and every $p>0$,
\begin{equation}\label{eq:layer-cake-identity}
 \E Z^p=p\int_0^\infty t^{p-1}\Pp(Z>t)\dd t.
\end{equation}
For the negative part, since for $u\geq 0$, we have $\{(-\log|X|)_+>u\} = \{|X|<e^{-u}\}$. Applying \eqref{eq:layer-cake-identity} with $p=1$ and splitting at $u_0=\log(e/\rho)$ gives
\begin{align*}
 &\left|\E(-\log|X|)_+-\E(-\log|Y|)_+\right|
 \le \int_0^{u_0}2\rho\,\dd u
 +\int_{u_0}^\infty4Be^{-u}\,\dd u
 \le C_B\rho\log(e/\rho).
\end{align*}
For the positive part,
\[
 \E\log^+|X|=\int_0^\infty\Pp(|X|\ge e^u)\dd u = \int_0^{u_{1}}\Pp(|X|\ge e^u)\dd u + \int_{u_1}^\infty\Pp(|X|\ge e^u)\dd u.
\]
The difference of the tail probabilities is at most $2\rho$.  With
$u_1=\log(1/\rho)$, the integral over $[0,u_1]$ is bounded by
$2\rho u_1$.  On $[u_1,\infty)$, Markov's inequality and the assumed second moments give
\[
 \int_{u_1}^\infty
 \{\Pp(|X|\ge e^u)+\Pp(|Y|\ge e^u)\}\dd u
 \le B\int_{u_1}^\infty e^{-2u}\dd u\le C_B\rho^2.
\]
Combining the positive and negative parts proves \eqref{ex:eq:log-compare}.  A coupling with failure probability at most $\rho$ implies
$|\Pp(X\in A)-\Pp(Y\in A)|\le\rho$ for every Borel set $A$, and in particular $\dk(X,Y)\le\rho$. Thus, we complete the proof of Lemma~\ref{lem:log-comparison}.
\end{proof}

We now return to the proof of Lemma~\ref{lem:exact-one-row}.
On the event $\{\Delta(w)\le e^{-1}\}$, we apply
Lemma~\ref{lem:log-comparison} conditionally on $\cF_{n-1}$, with
$X=S_0$, $Y=S_G$, and $\rho=\Delta(w)$.  The density and
second-moment assumptions were verified above.  Therefore,
\[
 \left|
 \E(\log|S_0|\mid\cF_{n-1})
 -
 \E\log|S_G|
 \right|
 \le
 C_{M_1}\Delta(w)\log\frac e{\Delta(w)}.
\]
On the event $\{\Delta(w)>e^{-1}\}$, we use a uniform bound for the
logarithmic expectations.  For $S_0$,
\[
 \begin{aligned}
 \E\bigl(|\log|S_0||\mid\cF_{n-1}\bigr)
 &=
 \E\bigl(\log^+|S_0|\mid\cF_{n-1}\bigr)
 +
 \E\bigl((-\log|S_0|)_+\mid\cF_{n-1}\bigr)\\
 &\le
 \E(S_0^2\mid\cF_{n-1})
 +
 \int_0^\infty
 \Pp(|S_0|\le e^{-t}\mid\cF_{n-1})\,dt\\
 &\le
 1+
 \int_0^\infty
 \min\{1,2\sqrt2M_1e^{-t}\}\,dt\\
 &\le
 C_{M_1}.
 \end{aligned}
\]
The same type of bound holds for $S_G$.  Hence
\[
 \begin{aligned}
 &\left|
 \E(\log|S_0|\mid\cF_{n-1})
 -
 \E\log|S_G|
 \right|
 \le
 C_{M_1}\Delta(w)\log\frac e{\Delta(w)}
 +
 C_{M_1}\1_{\{\Delta(w)>e^{-1}\}}.
 \end{aligned}
\]

Since $x\mapsto x\log(e/x)$ is concave and non-decreasing on $[0,1]$ and
$\Pp(\Delta>e^{-1})\le e\E\Delta$, combining with $\E\Delta(w)\le C\ell^{-4a}$ we obtain
\begin{equation}
  \begin{aligned}
    \E (\Delta(w)\log\frac e{\Delta(w)}) \leq & \E (\Delta(w))\log\frac e{\E \Delta(w)} \leq C\ell^{-4a}\log \ell, \nonumber
  \end{aligned}
\end{equation}
and thus
\begin{equation}\label{ex:eq:row-mean}
 \sqrt{2\ell}|d_{n}| = |\E\log(\det C_n)^{2} - \E\log(\det(\overline C_n)^{2}|
 \le C\ell^{-4a}\log\ell.
\end{equation}

Combining this estimate with \eqref{eq:need-dn-eatimate} and \eqref{ex:eq:row-mean} proves \eqref{ex:eq:one-row}. Thus, we complete the proof of Lemma~\ref{lem:exact-one-row}.
\end{proof}

We now return to the proof of Proposition~\ref{prop:exact-growing-fourth}.
We compare $B_n$ with $B_n^{\mathrm{hyb}}$ by replacing the last
$s_1$ rows one at a time.  Let $g_{m+1},\ldots,g_n$ be independent
standard Gaussian rows, independent of $B_n$.  For
$0\le r\le s_1$, let $B_n^{[r]}$ be the matrix obtained from $B_n$
by replacing rows $m+1,\ldots,m+r$ with
$g_{m+1},\ldots,g_{m+r}$.  Thus
\[
 B_n^{[0]}=B_n,
 \qquad
 B_n^{[s_1]}=B_n^{\mathrm{hyb}}.
\]

Fix $1\le r\le s_1$.  The matrices $B_n^{[r-1]}$ and $B_n^{[r]}$
differ only in row $m+r$.  Let $\Pi_r$ be a permutation matrix that
moves row $m+r$ to the last position.  We apply the same permutation
to both matrices.  Then $\Pi_r B_n^{[r-1]}$ and
$\Pi_r B_n^{[r]}$ have the same first $n-1$ rows.  Their last rows
are the original $(m+r)$th row of $B_n$ and the Gaussian row
$g_{m+r}$, respectively.

The row permutation does not change the absolute determinant, since
$|\det\Pi_r|=1$.  It also does not change the exact mean.  Hence
\[
 W_n^{\mathrm e}\bigl(\Pi_r B_n^{[j]}\bigr)
 =
 W_n^{\mathrm e}\bigl(B_n^{[j]}\bigr),
 \qquad
 j\in\{r-1,r\}.
\]
All assumptions of Lemma~\ref{lem:exact-one-row} hold for the
permuted pair.  Therefore,
\[
 \dk\left(
 W_n^{\mathrm e}\bigl(B_n^{[r-1]}\bigr),
 W_n^{\mathrm e}\bigl(B_n^{[r]}\bigr)
 \right)
 \le
 C\ell^{-4a}\log\ell.
\]

Applying the triangle inequality to
$B_n^{[0]},B_n^{[1]},\ldots,B_n^{[s_1]}$, we obtain
\begin{equation}\label{eq:exact-growing-total-replacement}
 \begin{aligned}
 \dk\left(
 W_n^{\mathrm e}(B_n),
 W_n^{\mathrm e}(B_n^{\mathrm{hyb}})
 \right)
 &=
 \dk\left(
 W_n^{\mathrm e}\bigl(B_n^{[0]}\bigr),
 W_n^{\mathrm e}\bigl(B_n^{[s_1]}\bigr)
 \right)\\
 &\quad\le
 \sum_{r=1}^{s_1}
 \dk\left(
 W_n^{\mathrm e}\bigl(B_n^{[r-1]}\bigr),
 W_n^{\mathrm e}\bigl(B_n^{[r]}\bigr)
 \right)\\
 &\quad\le
 Cs_1\ell^{-4a}\log\ell\\
 &\quad\le
 C\ell^{-a}\log\ell
 =
 o(\ell^{-1}).
 \end{aligned}
\end{equation}

This proves the bound for the row-replacement term in the proof
strategy.  It remains to prove
\[
 \dk\bigl(
 W_n^{\mathrm e}(B_n^{\mathrm{hyb}}),
 \Normal
 \bigr)
 \le
 C\frac{M}{\ell}.
\]
The next lemma gives this estimate.

\begin{lemma}\label{lem:exact-hybrid}
Under the assumptions of Proposition~\ref{prop:exact-growing-fourth}, we have
\begin{equation}\label{ex:eq:hybrid-bound}
 \dk(W_n^{\mathrm e}(B_n^{\mathrm{hyb}}),\Normal)
 \le C\frac{M}{\ell}.
\end{equation}
More generally, if $M=o(\ell)$, then
\begin{equation}\label{ex:eq:hybrid-asymptotic}
 \dk(W_n^{\mathrm e}(B_n^{\mathrm{hyb}}),\Normal)
 \longrightarrow0.
\end{equation}
\end{lemma}

\begin{proof}[Proof of Lemma~\ref{lem:exact-hybrid}]
We divide the proof into five steps.  The main term is the martingale bulk plus the
centered terminal block generated by the Gaussian rows.  The centered
drift and the lower-truncation remainder will be added at the end.

\medskip
\noindent\textbf{Step 1: decompose the exact-centered logarithmic determinant.}

In the previous deterministic centering case, we first do the truncation and estimate the the Kolmogorov distance of secquence $\sum_{i = 0}^{m-1}\zeta_{i+1}$, see \eqref{det:eq:corrected-bulk-growing} in the step 3 of the Proof of Lemma~\ref{lem:det-hybrid}, and finally remove the truncation. 

Here in the exact centering case, we decompose this sequence into martingale differences and conditional means.
For $0\le i<m$ define
\begin{equation}\label{ex:eq:mart-pred-def}
 g_i:=\E_i\zeta_{i+1},\qquad
 \eta_{i+1}:=\zeta_{i+1}-g_i,
 \qquad
 \mathcal M_m:=\sum_{i=0}^{m-1}\eta_{i+1},
 \qquad
 \mathcal D_m:=\sum_{i=0}^{m-1}(g_i-\E g_i).
\end{equation}
By \eqref{ex:eq:mart-pred-def}, $\eta_{i+1}$ is $\cF_{i+1}$-measurable and $\E_i\eta_{i+1}=0$; hence $(\eta_{i+1})$ is a martingale-difference sequence.  Since the deterministic correction $1/k_i$ cancels under exact centering, recall the definition of $\zeta_{i+1}$ in \eqref{eq:common-truncated-increment}, the centered sum of the lower-truncated bulk increments has the exact identity
\begin{equation}\label{ex:eq:exact-decomp}
 \sum_{i=0}^{m-1}(\zeta_{i+1}-\E\zeta_{i+1})
 =\mathcal M_m+\mathcal D_m.
\end{equation}

Denote the centered Gaussian terminal block as $\mathcal G_{s_1}^{\mathrm c} = \sum_{i=m}^{n-1}(\log Z_{i+1} - \E \log Z_{i+1})$.
By Gram--Schmidt, \eqref{ex:eq:exact-decomp}, and Lemma~\ref{lem:gaussian-terminal-block}, we have 
\begin{equation}\label{ex:eq:exact-hybrid-decomposition}
 \log\{(\det B_n^{\mathrm{hyb}})^2\}
 -
 \E\log\{(\det B_n^{\mathrm{hyb}})^2\}
 =
 \mathcal M_m
 +
 \mathcal D_m
 +
 \mathcal G_{s_1}^{\mathrm c}
 +
 \mathcal R_n.
\end{equation}
Hence
\[
 W_n^{\mathrm e}(B_n^{\mathrm{hyb}})
 =
 \frac{
 \mathcal M_m
 +
 \mathcal D_m
 +
 \mathcal G_{s_1}^{\mathrm c}
 +
 \mathcal R_n
 }{\sqrt{2\ell}},
\]
the remainder is
\[
 \mathcal R_n
 =
 \sum_{i=0}^{m-1}
 \left[
 \log Z_{i+1}-L_{i+1}
 -
 \E\bigl(\log Z_{i+1}-L_{i+1}\bigr)
 \right].
\]
The following lemma gives a bound for the lower-truncation error which can be used to control the remainder $\mathcal R_n$.

\begin{lemma}[Lower-truncation error]\label{lem:exact-lower-truncation}
Assume the bounded-density hypotheses of
Proposition~\ref{prop:exact-growing-fourth}. Then there exist
constants $c>0$ and $n_0<\infty$ such that, for every $n\ge n_0$,
\begin{equation}\label{ex:eq:lower-truncation-error}
 \E\sum_{i=0}^{m-1}
 \left|\log Z_{i+1}-L_{i+1}\right|
 \le
 \exp\{-c\ell^{3a}\log\ell\}.
\end{equation}
Here $n_0$ depend on the common density bound, while $c$ depend only on $a$.
\end{lemma}

The proof of Lemma~\ref{lem:exact-lower-truncation} is given in Appendix~\ref{sec:proof-lemmas}.
Lemma~\ref{lem:exact-lower-truncation}
gives
\begin{equation}
  \begin{aligned}\label{eq:remainder-Rn-contral}
    \E|\mathcal R_n|
    \le
    2\exp\{-c\ell^{3a}\log\ell\}.
  \end{aligned}
\end{equation}
Moreover, $\mathcal G_{s_1}^{\mathrm c}$ is independent of
$\cF_m$, while $\mathcal M_m$, $\mathcal D_m$, and
$\mathcal R_n$ are $\cF_m$-measurable.  

\medskip
\noindent\textbf{Step 2: identify the variance of the main random part.}

Since $W_n^{\mathrm e}(B_n^{\mathrm{hyb}})=
\frac{
\mathcal M_m
+
\mathcal D_m
+
\mathcal G_{s_1}^{\mathrm c}
+
\mathcal R_n
}{\sqrt{2\ell}}$,
we divide it into a main random part and two small error terms.  The main random part is
\[
 (\mathcal M_m+\mathcal G_{s_1}^{\mathrm c})/\sqrt{2\ell},
\]
which contains the martingale bulk $\mathcal M_m$ and the centered Gaussian terminal block $\mathcal G_{s_1}^{\mathrm c}$. The following Lemma~\ref{lem:exact-mart-bounds} give the moment estimates for the martingale differences $\eta_{i+1}$, which will be used to control $\mathcal M_m$.

\begin{lemma}\label{lem:exact-mart-bounds}
Let $(B_n)_{n\ge2}$, with
$B_n=(b_{ij})_{1\le i,j\le n}$, be a sequence of random matrices
whose entries are independent and real-valued.
Suppose that each $b_{ij}$ has a density $f_{b_{ij}}$ and that,
for fixed constants $R_3,M_1,K_0>0$ independent of $n$,
\[
 \E b_{ij}=0,\qquad
 \E b_{ij}^2=1,\qquad
 \E|b_{ij}|^3\le R_3,\qquad
 \|f_{b_{ij}}\|_\infty\le M_1,
\]
and
\[
 |b_{ij}|\le (K_0n\ell^{3/4})^{1/2}
 \quad\text{almost surely}
\]
for every $n\ge2$ and $1\le i,j\le n$, where $\ell=\log n$.
Put $M:=\max\left\{3,\max_{1\le i,j\le n}\E|b_{ij}|^4\right\}$,
and assume that $M=o(\ell)$ as $n\to\infty$.

Retain the scales and projection notation of
Section~\ref{subsec:common-determinant-setup}.
Let $\eta_{i+1}$ and $h_i$ be defined by
\eqref{ex:eq:mart-pred-def} and
\eqref{eq:common-one-step-error}, respectively,
using the first $m$ rows of $B_n$.
These bulk quantities are unchanged when $B_n$ is replaced by
$B_n^{\mathrm{hyb}}$, with the same fourth-moment bound $M$.
Let
\begin{equation}
  \begin{aligned}\label{eq:def-martingale-for-exact-centering}
    q_i=\E\eta_{i+1}^2,
    \qquad s_{\eta}^2=\sum_{i=0}^{m-1}\E\eta_{i+1}^2,
    \qquad \sigma_{\eta,i+1}^2=\E\eta_{i+1}^2/s_{\eta}^2,
    \qquad v_{\eta,i}=\sum_{j=i+1}^{m-1}\sigma_{\eta,j+1}^2.    
  \end{aligned}
\end{equation}
Then
\begin{align}
 \E|\E_i\eta_{i+1}^2-\E\eta_{i+1}^2| \le Ch_i,\label{ex:eq:mart-var-bound}\\
 \E|\E_i\eta_{i+1}^3|+\E|\eta_{i+1}|^4 \le Ch_i,\label{ex:eq:mart-third-fourth-bound}\\
 \E\eta_{i+1}^2 =2/k_i+O(h_i),\qquad \E\eta_{i+1}^2 \le CM/k_i,\label{ex:eq:mart-q}\\
 s_{\eta}^2 =2\sum_{k=s_1+1}^{n}k^{-1}+O(M)
 =2\ell-2\log s_1+O(M+1).\label{ex:eq:mart-s}
\end{align}
Moreover,
\begin{equation}\label{ex:eq:mart-max}
 \max_{i<m}\frac{\ell^2\E\eta_{i+1}^2}{s_{\eta}^2}=o(1).
\end{equation}
\end{lemma}

The proof of Lemma~\ref{lem:exact-mart-bounds} is given in Appendix~\ref{sec:proof-lemmas}.
Using Lemma~\ref{lem:exact-mart-bounds}, we can identify the variance of the main random part $\mathcal M_m+\mathcal G_{s_1}^{\mathrm c}$.
Set
\[
 s_*^2
 =
 s_{\eta}^2
 +
 \Var(\mathcal G_{s_1}^{\mathrm c}).
\]
Since under assumption, $M=o(\ell)$.  By
\eqref{ex:eq:mart-s} and
\eqref{eq:gaussian-terminal-variance}, we have
\begin{equation}\label{ex:eq:sstar}
 \begin{aligned}
 s_*^2
 &=
 2\ell+O(M+1),
 \qquad
 \frac{s_{\eta}^2}{s_*^2}\ge c>0.
 \end{aligned}
\end{equation}
In particular, $s_*^2\asymp\ell$.

\medskip
\noindent\textbf{Step 3: approximate the main random part at scale $s_*$.}

Since \eqref{ex:eq:sstar} gives $s_*^2 = 2\ell+O(M+1)$, we firstly consider
$\frac{\mathcal M_m+\mathcal G_{s_1}^{\mathrm c}}{s_*}$,
and then replace $s_*$ by the tareget scale $\sqrt{2\ell}$ in the next step.

Let $N_0\sim\Normal$ be independent of all rows.  Since
$\mathcal G_{s_1}^{\mathrm c}$ is independent of $\mathcal M_m$,
conditioning on $\mathcal G_{s_1}^{\mathrm c}$ gives
\begin{equation}
  \begin{aligned}
    &\dk\left(
    \frac{\mathcal M_m+\mathcal G_{s_1}^{\mathrm c}}{s_*},
    \frac{s_{\eta}N_0+\mathcal G_{s_1}^{\mathrm c}}{s_*}
    \right)
    \le
    \dk\left(
    \frac{\mathcal M_m}{s_{\eta}},
    \Normal
    \right). \nonumber
  \end{aligned} 
\end{equation}
Then, we will bound the Kolmogorov distance $\dk\left(\frac{\mathcal M_m}{s_{\eta}},\Normal\right)$ by the characteristic function comparison.  Let $\phi_m$ be the characteristic function of $\mathcal M_m/s_{\eta}$.

For each $0\le i<m$, let $Q_{i+1}\sim \mathcal N\bigl(0,\E\eta_{i+1}^2\bigr)$
be independent Gaussian variables, independent also of all rows.
Since $s_\eta^2 = \sum_{i=0}^{m-1}\E\eta_{i+1}^2$,
we have
$
 \frac1{s_\eta}\sum_{i=0}^{m-1}Q_{i+1}
 \sim \Normal.
$
For $0\le i\le m$, define
\[
 T_i
 =
 \frac1{s_\eta}
 \left(
 \sum_{r=0}^{i-1}\eta_{r+1}
 +
 \sum_{r=i}^{m-1}Q_{r+1}
 \right).
\]
Then $T_0\sim\Normal$ and $T_m=\frac{\mathcal M_m}{s_\eta}$.
Hence, successive Gaussian replacement gives
\begin{equation}
  \begin{aligned}
    \phi_m(t)-e^{-t^2/2}
    =&
    \sum_{i=0}^{m-1}
    \left(
    \E e^{itT_{i+1}}
    -
    \E e^{itT_i}
    \right)\\
    =&\sum_{i=0}^{m-1}e^{-t^2v_{\eta,i}/2}
    \E\left[e^{it\sum_{r<i}\eta_{r+1}/s_{\eta}}
    \left\{\E_i e^{it\eta_{i+1}/s_{\eta}}-e^{-t^2\E\eta_{i+1}^2/(2s_{\eta}^2)}\right\}\right]. \nonumber
  \end{aligned}
\end{equation}
Since $\E_i\eta_{i+1}=0$, Taylor's formula gives, for
$0\le t\le\ell$,
\[
 \begin{aligned}
 \E_i e^{it\eta_{i+1}/s_\eta}
 &=
 1
 -
 \frac{t^2}{2s_\eta^2}\E_i\eta_{i+1}^2
 -
 \frac{it^3}{6s_\eta^3}\E_i\eta_{i+1}^3
 +
 O\left(
 \frac{t^4}{s_\eta^4}\E_i|\eta_{i+1}|^4
 \right),
 \end{aligned}
\]
while
\[
 \begin{aligned}
 \exp\left\{
 -\frac{t^2\E\eta_{i+1}^2}{2s_\eta^2}
 \right\}
 &=
 1
 -
 \frac{t^2\E\eta_{i+1}^2}{2s_\eta^2}
 +
 O\left(
 \frac{t^4(\E\eta_{i+1}^2)^2}{s_\eta^4}
 \right).
 \end{aligned}
\]
Since $(\E\eta_{i+1}^2)^2 \le \E|\eta_{i+1}|^4$,
subtracting the two expansions gives
\begin{align*}
 \left|\frac{\phi_m(t)-e^{-t^2/2}}t\right|
 &\le \frac{Ct}{s_{\eta}^2}\sum_i e^{-t^2v_{\eta,i}/2}\E|\E_i\eta_{i+1}^2-\E\eta_{i+1}^2|\\
 &\quad+\frac{Ct^2}{s_{\eta}^3}\sum_i e^{-t^2v_{\eta,i}/2}\E|\E_i\eta_{i+1}^3|\\
 &\quad+\frac{Ct^3}{s_{\eta}^4}\sum_i e^{-t^2v_{\eta,i}/2}\E|\eta_{i+1}|^4.
\end{align*}
The most difference between the deterministic case (Step~3 in the Proof of Lemma~\ref{lem:det-hybrid}) is that, the linear term vanishes because $\E_i\eta_{i+1}=0$. Then, using Lemma~\ref{lem:esseen} with cutoff $T=\ell$, we have
\begin{equation}
 \begin{aligned}
 \dk\left(\frac{\mathcal M_m}{s_\eta},\Normal\right)
 &\le\frac{C}{\ell}
 +\frac{C}{s_\eta^2}\int_0^\ell
 t\sum_{i=0}^{m-1}e^{-t^2v_{\eta,i}/2}
 \E|\E_i\eta_{i+1}^2-\E\eta_{i+1}^2|\dd t\\
 &\quad+\frac{C}{s_\eta^3}\int_0^\ell
 t^2\sum_{i=0}^{m-1}e^{-t^2v_{\eta,i}/2}
 \E|\E_i\eta_{i+1}^3|\dd t\\
 &\quad+\frac{C}{s_\eta^4}\int_0^\ell
 t^3\sum_{i=0}^{m-1}e^{-t^2v_{\eta,i}/2}
 \E|\eta_{i+1}|^4\dd t. \nonumber
 \end{aligned}
\end{equation}
To estimate the three integrals, we give the following Lemma~\ref{lem:exact-weighted}.

\begin{lemma}\label{lem:exact-weighted}
Under the assumptions of Proposition~\ref{prop:exact-growing-fourth}.
If $k_i\ge s_2$, then $v_{\eta,i}\ge c_0>0$.  Moreover, for $p=0,1,2,3$,
\begin{equation}\label{ex:eq:weighted}
 \int_0^\ell t^p\sum_{i=0}^{m-1}h_i e^{-ct^2v_{\eta,i}}\dd t\le C_pM.
\end{equation}
\end{lemma}

\begin{proof}[Proof of Lemma~\ref{lem:exact-weighted}]
By \eqref{eq:common-h-sum-point}, \eqref{ex:eq:mart-q}, and \eqref{ex:eq:mart-s}, for $k_i\ge s_2$,
\[
 s_{\eta}^2v_{\eta,i}=\sum_{j=i+1}^{m-1}\E\eta_{j+1}^2
 \ge2\sum_{q=s_1+1}^{s_2-1}q^{-1}-C\sum_jh_j
 \ge2\log(s_2/s_1)-CM-C\ge c\ell.
\]
Here $\log(s_2/s_1)=\ell-O(\log\ell)$ and $M=o(\ell)$. Thus $v_{\eta,i}\ge c_0$.  For the indices with $k_i\ge s_2$, the integral of $t^pe^{-ct^2v_{\eta,i}}$ over $[0,\ell]$ is bounded by a constant depending only on $p$, and \eqref{eq:common-h-sum-point} gives a contribution at most $C_pM$.  For the terminal-range indices, $e^{-ct^2v_{\eta,i}}\le1$, so \eqref{eq:common-h-terminal} and $p\le3$ give
\[
 \int_0^\ell t^p\sum_{s_1\le k_i<s_2}h_i e^{-ct^2v_{\eta,i}}\dd t
 \le C\ell^4\,M\ell^{-12}\le CM\ell^{-8}.
\]
Adding the two ranges proves \eqref{ex:eq:weighted}. Thus we complete the proof of Lemma~\ref{lem:exact-weighted}.
\end{proof}

We now return to the proof of Lemma~\ref{lem:exact-hybrid}.
Using Lemmas~\ref{lem:exact-mart-bounds} and~\ref{lem:exact-weighted}, together with $s_{\eta}^2\asymp\ell$, show that the three integrals over $[0,\ell]$ are bounded by
\[
 CM/\ell,\qquad CM/\ell^{3/2},\qquad CM/\ell^2.
\]
Therefore, the total is at most $C(1+M)/\ell$. Thus we have 
\begin{equation}
  \begin{aligned}\label{eq:comparison-1}
    &\dk\left(
    \frac{\mathcal M_m+\mathcal G_{s_1}^{\mathrm c}}{s_*},
    \frac{s_{\eta}N_0+\mathcal G_{s_1}^{\mathrm c}}{s_*}
    \right)
    \le
    \dk\left(
    \frac{\mathcal M_m}{s_{\eta}},
    \Normal
    \right)
    \le
    C\frac{M}{\ell}.
  \end{aligned} 
\end{equation}

Then, we estimate
$
  \dk\left(
  \frac{s_{\eta}N_0+\mathcal G_{s_1}^{\mathrm c}}{s_*},
  \Normal
  \right)
$
Write $Y_{k} = \log Z_{m+k} - \E \log Z_{m+k}$, then we have $ \mathcal G_{s_1}^{\mathrm c}
=\sum_{k=1}^{s_1}Y_k$. We denote $\tau_k^2=\E Y_k^2$, $\widetilde Y_k\sim\mathcal N(0,\tau_k^2)$,
where the $\widetilde Y_k$ are independent of each other, of $N_0$, and of all rows.  Then
\[
 s_{\eta}N_0
 +
 \sum_{k=1}^{s_1}\widetilde Y_k
 \sim
 \mathcal N(0,s_*^2).
\]
For $|u|\le1$, Taylor's expansion gives
$
 \left|
 \E e^{iuY_k}
 -
 e^{-u^2\tau_k^2/2}
 \right|
 \le
 C|u|^3\E|Y_k|^3.
$
Therefore, for $0\le t\le s_*$, since \eqref{ex:eq:sstar}, $s_{\eta}^2/s_*^2\ge c>0$. Combining with
a product telescope and \eqref{eq:gaussian-terminal-third} give
\begin{equation}
  \begin{aligned}\label{eq:two-chara-func-bound}
    \left|
    \E\exp\left\{
    \frac{it}{s_*}
    \left(
    s_{\eta}N_0+\mathcal G_{s_1}^{\mathrm c}
    \right)
    \right\}
    -
    e^{-t^2/2}
    \right|
    &=
    e^{-s_{\eta}^2t^2/(2s_*^2)}
    \left|
    \prod_{k=1}^{s_1}\E e^{itY_k/s_*}
    -
    \prod_{k=1}^{s_1}\E e^{it\widetilde Y_k/s_*}
    \right|\\
    &\le
    C\frac{t^3}{s_*^3}
    e^{-ct^2}
    \sum_{k=1}^{s_1}\E|Y_k|^3
    \le
    C\frac{t^3}{s_*^3}e^{-ct^2}.
  \end{aligned}
\end{equation}
For $t>s_*$, the Gaussian factor $s_{\eta}N_0$ and
\eqref{ex:eq:sstar} give
\[
 \left|
 \E\exp\left\{
 \frac{it}{s_*}
 \left(
 s_{\eta}N_0+\mathcal G_{s_1}^{\mathrm c}
 \right)
 \right\}
 \right|
 \le
 e^{-ct^2}.
\]
After replacing $c$ by $\min\{c,1/2\}$, we have $e^{-t^2/2}\le e^{-ct^2}$.
Thus both characteristic functions in \eqref{eq:two-chara-func-bound} are bounded by $e^{-ct^2}$.
Lemma~\ref{lem:esseen}, with cutoff $T=s_*^3$, now yields
\begin{equation}
  \begin{aligned}\label{eq:comparison-2}
    \dk\left(
    \frac{s_{\eta}N_0+\mathcal G_{s_1}^{\mathrm c}}{s_*},
    \Normal
    \right)
    &\le
    \frac{C}{s_*^3}
    +
    \frac{C}{s_*^3}
    \int_0^{s_*}t^2e^{-ct^2}\,\dd t
    +
    C\int_{s_*}^{s_*^3}
    \frac{e^{-ct^2}}{t}\,\dd t\\
    &\le
    Cs_*^{-3}
    =
    O(\ell^{-3/2}).
  \end{aligned}
\end{equation}
Combining \eqref{eq:comparison-1} and \eqref{eq:comparison-2}, we obtain
\begin{equation}\label{ex:eq:exact-hybrid-natural-scale}
 \dk\left(
 \frac{\mathcal M_m+\mathcal G_{s_1}^{\mathrm c}}{s_*},
 \Normal
 \right)
 \le
 C\frac{M}{\ell},
\end{equation}
where we used $M\ge3$.

\medskip
\noindent\textbf{Step 4: replace $s_*$ by the target scale $\sqrt{2\ell}$.}
By \eqref{ex:eq:sstar},
\[
 \left|
 \frac{s_*}{\sqrt{2\ell}}-1
 \right|
 =
 \frac{
 |s_*^2-2\ell|
 }{
 \sqrt{2\ell}\{s_*+\sqrt{2\ell}\}
 }
 \le
 C\frac{M+1}{\ell}.
\]
This quantity is at most $1/2$ for all sufficiently large $n$.
Lemma~\ref{lem:perturb}(iv) and
\eqref{ex:eq:exact-hybrid-natural-scale} therefore give
\begin{equation}\label{ex:eq:exact-hybrid-main-normal}
 \dk\left(
 \frac{
 \mathcal M_m+\mathcal G_{s_1}^{\mathrm c}
 }{\sqrt{2\ell}},
 \Normal
 \right)
 \le
 C\frac{M}{\ell}.
\end{equation}

\medskip
\noindent\textbf{Step 5: restore the drift and the lower-truncation remainder.}

We have already controlled the Kolmogorov distance of the main random part $\mathcal M_m+\mathcal G_{s_1}^{\mathrm c}$ at the target scale $\sqrt{2\ell}$ in \eqref{ex:eq:exact-hybrid-main-normal}.  It remains to restore the centered drift $\mathcal D_m$ and the lower-truncation remainder $\mathcal R_n$. First, we control the centered drift $\mathcal D_m$. The following Lemma~\ref{lem:exact-drift} gives a concentration bound for the centered drift $\mathcal D_m$.

\begin{lemma}[Concentration of the sum of conditional means]\label{lem:exact-drift}
Under the assumptions of Proposition~\ref{prop:exact-growing-fourth}. Put $\varepsilon_n=\frac{M}{\ell}$.
Then, we have
\begin{equation}\label{ex:eq:drift-prob}
 \Pp\left(|\mathcal D_m|>\varepsilon_n\sqrt{2\ell}\right)
 \le C\varepsilon_n.
\end{equation}
\end{lemma}

The proof of Lemma~\ref{lem:exact-drift} is given in Appendix~\ref{sec:proof-lemmas}.
Apply Lemma~\ref{lem:perturb}(i) to
$\mathcal D_m/\sqrt{2\ell}$ with threshold $ \varepsilon_n=\frac{M}{\ell}$.
Lemma~\ref{lem:exact-drift} and
\eqref{ex:eq:exact-hybrid-main-normal} give
\begin{align*}
 \dk\left(
 \frac{
 \mathcal M_m
 +
 \mathcal D_m
 +
 \mathcal G_{s_1}^{\mathrm c}
 }{\sqrt{2\ell}},
 \Normal
 \right)
 &\le
 \dk\left(
 \frac{
 \mathcal M_m+\mathcal G_{s_1}^{\mathrm c}
 }{\sqrt{2\ell}},
 \Normal
 \right)
 +
 \Pp\left(
 |\mathcal D_m|
 >
 \varepsilon_n\sqrt{2\ell}
 \right)
 +
 C\varepsilon_n\\
 &\le
 C\frac{M}{\ell}.
\end{align*}
We finally restore $\mathcal R_n$.  By Markov's inequality and the
bound on $\E|\mathcal R_n|$ from \eqref{eq:remainder-Rn-contral} in Step~1,
\[
 \Pp\left(
 |\mathcal R_n|
 >
 \ell^{-2}\sqrt{2\ell}
 \right)
 \le
 C\ell^{3/2}
 \exp\{-c\ell^{3a}\log\ell\}.
\]
A second application of Lemma~\ref{lem:perturb}(i), now with threshold
$\ell^{-2}$, together with
\eqref{ex:eq:exact-hybrid-decomposition}, gives
\[
 \begin{aligned}
 \dk\left(
 W_n^{\mathrm e}(B_n^{\mathrm{hyb}}),
 \Normal
 \right)
 &\le
 C\frac{M}{\ell}
 +
 C\ell^{-2}
 +
 C\ell^{3/2}
 \exp\{-c\ell^{3a}\log\ell\}
 \le
 C\frac{M}{\ell}.
 \end{aligned}
\]
This proves \eqref{ex:eq:hybrid-bound}.

For the asymptotic assertion, assume that $M=o(\ell)$.  Then
\[
 \frac{1+M}{\ell}\longrightarrow0,
 \qquad
 \frac{s_*^2}{2\ell}
 =
 1+O\left(\frac{M+1}{\ell}\right)
 \longrightarrow1,
 \qquad
 \varepsilon_n=\frac{M}{\ell}\longrightarrow0.
\]
The terminal Gaussian comparison is $O(\ell^{-3/2})$, and the
lower-truncation remainder is super-polynomially small.  Therefore,
the same argument above gives
\[
 \dk\left(
 W_n^{\mathrm e}(B_n^{\mathrm{hyb}}),
 \Normal
 \right)
 \longrightarrow0.
\]
This proves \eqref{ex:eq:hybrid-asymptotic}. Thus we have completed the proof of Lemma~\ref{lem:exact-hybrid}.
\end{proof}

We now return to the proof of Proposition~\ref{prop:exact-growing-fourth}.
Under the hypothesis $M\le K_1\ell^{1-\gamma}$, Lemma~\ref{lem:exact-hybrid} and \eqref{eq:exact-growing-total-replacement} give
\begin{equation}\label{eq:exact-growing-quantitative-combination}
 \dk(W_n^{\mathrm e}(B_n),\Normal)
 \le o(\ell^{-1})+C\frac{M}{\ell},
\end{equation}
which proves \eqref{eq:exact-growing-bound}.

For the asymptotic assertion, suppose that $M=o(\ell)$. Combining \eqref{ex:eq:hybrid-asymptotic} with \eqref{eq:exact-growing-total-replacement} and the triangle inequality,
we obtain
\begin{align*}
 \dk(W_n^{\mathrm e}(B_n),\Normal)
 &\le
 \sup_x\left|
 \Pp(W_n^{\mathrm e}(B_n)\le x)
 -
 \Pp(W_n^{\mathrm e}(B_n^{\mathrm{hyb}})\le x)
 \right|\\
 &\quad+
 \dk(W_n^{\mathrm e}(B_n^{\mathrm{hyb}}),\Normal)\\
 &=o(\ell^{-1})+o(1)
 =o(1).
\end{align*}
This proves \eqref{eq:exact-growing-asymptotic}. Thus we have completed the proof of Proposition~\ref{prop:exact-growing-fourth}.
\end{proof}

\subsection{Proof of Proposition~\ref{prop:variance-growing-fourth}}
\label{subsec:proof-prop-variance-growing-fourth}

\begin{proof}[Proof of Proposition~\ref{prop:variance-growing-fourth}]
We use the scales and notation defined in
\eqref{eq:common-parameters}--\eqref{eq:common-heights}.
Our goal is to estimate $\dk(W_n^{\mathrm v}(B_n),\Normal)$.
Recall that $B_n^{\mathrm{hyb}}$ is obtained from $B_n$ by
keeping its first $m=n-s_1$ rows and replacing its last $s_1$
rows by independent standard Gaussian rows, independent of $B_n$.
We first decompose the bound into the normal approximation error
for $B_n^{\mathrm{hyb}}$ and the errors arising from these row
replacements.

By Lemma~\ref{lem:variance-transfer},
\[
 0<\Var(\log|\det B_n|)<\infty,
 \qquad
 0<\Var(\log|\det B_n^{\mathrm{hyb}}|)<\infty.
\]
The triangle inequality, together with the elementary bounds
for shifting and rescaling the normal distribution function, gives
\begin{equation}\label{eq:variance-growing-hybrid-comparison}
 \begin{aligned}
 \dk\bigl(W_n^{\mathrm v}(B_n),\Normal\bigr)
 &\le
 \dk\bigl(W_n^{\mathrm v}(B_n^{\mathrm{hyb}}),\Normal\bigr)
 +\dk\bigl(\log|\det B_n|,\log|\det B_n^{\mathrm{hyb}}|\bigr)\\
 +&
 C\frac{|\E\log|\det B_n|-\E\log|\det B_n^{\mathrm{hyb}}||}
 {\sqrt{\Var(\log|\det B_n^{\mathrm{hyb}}|)}}
 +
 C\left|
 \frac{\sqrt{\Var(\log|\det B_n|)}}
 {\sqrt{\Var(\log|\det B_n^{\mathrm{hyb}}|)}}-1
 \right|.
 \end{aligned}
\end{equation}
The following lemma gives the variance bounds and the normal
approximation for the hybrid matrix $B_n^{\mathrm{hyb}}$.

\begin{lemma}\label{lem:variance-hybrid}
Under the assumptions of Proposition~\ref{prop:variance-growing-fourth},
for all sufficiently large $n$,
\begin{equation}\label{eq:variance-hybrid-variance}
 c\ell\le\Var\bigl(\log|\det B_n^{\mathrm{hyb}}|\bigr)
 \le C(\ell+M),
\end{equation}
and
\begin{equation}\label{eq:variance-hybrid-bound}
 \dk\bigl(W_n^{\mathrm v}(B_n^{\mathrm{hyb}}),\Normal\bigr)
 \le C\min\left\{
 \frac{M}{\ell^{3/2}},
 \frac{\log\ell}{\sqrt\ell}
 \right\}.
\end{equation}
The constants $c,C$ depend only on $R_3,M_1,K_0,K_1,\alpha$.
\end{lemma}

\begin{proof}[Proof of Lemma~\ref{lem:variance-hybrid}]
We divide the proof into five steps.

\medskip
\noindent\textbf{Step 1: decompose the centered logarithmic determinant.}

We use \eqref{eq:common-parameters} with
$a=20+4\lceil\alpha\rceil$, together with the notation in
\eqref{ex:eq:mart-pred-def} and
\eqref{eq:def-martingale-for-exact-centering}.
By \eqref{ex:eq:exact-hybrid-decomposition},
\[
 2\bigl(\log|\det B_n^{\mathrm{hyb}}|
 -\E\log|\det B_n^{\mathrm{hyb}}|\bigr)
 =\mathcal M_m+\mathcal G_{s_1}^{\mathrm c}
 +\mathcal D_m+\mathcal R_n.
\]
We first approximate the martingale part, then add the Gaussian
terminal block and the remaining terms.

\medskip
\noindent\textbf{Step 2: approximate the martingale part.}

The following lemma gives the variance bounds and normal
approximation for $\mathcal M_m$ under the present assumptions.

\begin{lemma}\label{lem:variance-martingale}
Under the assumptions of Proposition~\ref{prop:variance-growing-fourth},
let $\mathcal M_m$ and $s_\eta^2$ be defined by
\eqref{ex:eq:mart-pred-def} and
\eqref{eq:def-martingale-for-exact-centering}.
For all sufficiently large $n$,
\begin{equation}\label{eq:variance-martingale-variance}
 c\ell\le s_\eta^2\le C(\ell+M),
\end{equation}
and
\begin{equation}\label{eq:variance-martingale-bound}
 \dk(\mathcal M_m/s_\eta,\Normal)
 \le C\min\left\{
 \frac{M}{\ell^{3/2}},
 \frac{\log\ell}{\sqrt\ell}
 \right\}.
\end{equation}
The constants and the required lower bound on $n$ depend only on
$R_3,M_1,K_0,K_1,\alpha$.
\end{lemma}

The proof is given in Appendix~\ref{sec:proof-lemmas}.

\medskip
\noindent\textbf{Step 3: add the Gaussian terminal block.}

By Lemma~\ref{lem:gaussian-terminal-block},
$\mathcal G_{s_1}^{\mathrm c}$ is independent of $\cF_m$ and
has variance $O(\log\ell)$. Recall that
\[
 s_*^2=s_\eta^2+\Var(\mathcal G_{s_1}^{\mathrm c}).
\]
Lemma~\ref{lem:variance-martingale} gives
\begin{equation}\label{eq:variance-hybrid-main-variance}
 c\ell\le s_*^2\le C(\ell+M),\qquad
 \frac{s_\eta^2}{s_*^2}\ge c>0.
\end{equation}
These bounds and \eqref{eq:gaussian-terminal-third} allow us to
apply the Gaussian replacement argument leading to
\eqref{eq:comparison-2}. For an independent $N_0\sim\Normal$,
\[
 \dk\left(
 \frac{s_\eta N_0+\mathcal G_{s_1}^{\mathrm c}}{s_*},\Normal
 \right)
 \le Cs_*^{-3}\le C\ell^{-3/2}.
\]
Since $M\ge3$, both terms in the minimum in
\eqref{eq:variance-martingale-bound} are at least $\ell^{-3/2}$
for all sufficiently large $n$.
Conditioning on $\mathcal G_{s_1}^{\mathrm c}$ and using
\eqref{eq:variance-martingale-bound}, we therefore obtain
\begin{equation}\label{eq:variance-hybrid-main-bound}
 \dk\left(
 \frac{\mathcal M_m+\mathcal G_{s_1}^{\mathrm c}}{s_*},\Normal
 \right)
 \le C\min\left\{
 \frac{M}{\ell^{3/2}},
 \frac{\log\ell}{\sqrt\ell}
 \right\}.
\end{equation}

\medskip
\noindent\textbf{Step 4: restore the drift and the lower-truncation remainder.}

To restore $\mathcal D_m+\mathcal R_n$ and control its effect on
the variance, we use the following estimate.

\begin{lemma}\label{lem:variance-remainder}
Under the assumptions of Proposition~\ref{prop:variance-growing-fourth},
let $\mathcal D_m$ and $\mathcal R_n$ be the terms in
\eqref{ex:eq:exact-hybrid-decomposition}.
For all sufficiently large $n$,
\begin{equation}\label{eq:variance-remainder-L2}
 \|\mathcal D_m+\mathcal R_n\|_2\le C\ell^{-2}.
\end{equation}
The constant and the required lower bound on $n$ depend only on
$R_3,M_1,K_0,K_1,\alpha$.
\end{lemma}

The proof is given in Appendix~\ref{sec:proof-lemmas}.
Chebyshev's inequality gives
\[
 \Pp\bigl(|\mathcal D_m+\mathcal R_n|>\ell^{-1}\bigr)
 \le C\ell^{-2}.
\]
Since $\ell^{-1}/s_*\le C\ell^{-3/2}$,
Lemma~\ref{lem:perturb}(i) and
\eqref{eq:variance-hybrid-main-bound} imply
\begin{equation}\label{eq:variance-hybrid-full-bound}
 \dk\left(
 \frac{2\{\log|\det B_n^{\mathrm{hyb}}|
 -\E\log|\det B_n^{\mathrm{hyb}}|\}}{s_*},\Normal
 \right)
 \le C\min\left\{
 \frac{M}{\ell^{3/2}},
 \frac{\log\ell}{\sqrt\ell}
 \right\}.
\end{equation}

\medskip
\noindent\textbf{Step 5: use the standard deviation of the full sum.}

All terms in the decomposition in Step~1 are centered, and
$\Var(\mathcal M_m+\mathcal G_{s_1}^{\mathrm c})=s_*^2$.
The triangle inequality in $L^2$ and
Lemma~\ref{lem:variance-remainder} give
\[
 \left|2\sqrt{\Var(\log|\det B_n^{\mathrm{hyb}}|)}-s_*\right|
 \le\|\mathcal D_m+\mathcal R_n\|_2\le C\ell^{-2}.
\]
Dividing by $s_*\ge c\sqrt\ell$, we obtain
\[
 \left|
 \frac{2\sqrt{\Var(\log|\det B_n^{\mathrm{hyb}}|)}}{s_*}-1
 \right|
 \le C\ell^{-5/2}.
\]
Together with \eqref{eq:variance-hybrid-main-variance}, this proves
\eqref{eq:variance-hybrid-variance}.
Finally, Lemma~\ref{lem:perturb}(iv) applied to
\eqref{eq:variance-hybrid-full-bound} proves
\eqref{eq:variance-hybrid-bound}.
\end{proof}

We now return to the proof of
Proposition~\ref{prop:variance-growing-fourth}.
It remains to compare the distributions, expectations, and standard
deviations of the two logarithmic determinants.
Use the intermediate matrices $B_n^{[r]}$, $0\le r\le s_1$,
defined in the proof of Proposition~\ref{prop:exact-growing-fourth}.
Thus $B_n^{[0]}=B_n$ and
$B_n^{[s_1]}=B_n^{\mathrm{hyb}}$.

Fix $1\le r\le s_1$ and move row $m+r$ to the last position in both
$B_n^{[r-1]}$ and $B_n^{[r]}$.
Denote the resulting matrices by $C_n$ and $\overline C_n$,
respectively. This permutation preserves their absolute determinants.
Condition on their common first $n-1$ rows and retain
$\cF_{n-1},\Lambda,w,S_0,S_G$ from
\eqref{eq:det-Cn-set-Cn-bar-decompose}.
Then
\[
 \log|\det C_n|=\log\Lambda+\log|S_0|,
 \qquad
 \log|\det\overline C_n|=\log\Lambda+\log|S_G|,
\]
and the conditional distribution of $S_G$ is standard normal.
The conditional Berry--Esseen bound \eqref{ex:eq:row-BE} and
Lemma~\ref{lem:terminal-diagonal} give
\[
 \E\Delta(w)^2
 \le CR_3^2\E\max_j|w_j|^2
 =CR_3^2\E\max_jp_{jj}(n-1)
 \le C\ell^{-8a}.
\]
The projection estimate applies uniformly to the intermediate
matrices, whose entries are independent and have fourth moments
bounded by $M$.
Since the common factor $\Lambda$ only shifts the conditional
logarithmic distributions, we obtain
\[
 \dk\bigl(\log|\det C_n|,\log|\det\overline C_n|\bigr)
 \le2\E\Delta(w)
 \le C\ell^{-4a}.
\]

We next compare the conditional logarithmic moments.
Lemma~\ref{lem:projection-density} and the conditional second
moments of $S_0,S_G$ give, for every $t\ge0$,
\[
 \Pp\bigl(|\log|S_0||>t\mid\cF_{n-1}\bigr)
 +\Pp\bigl(|\log|S_G||>t\mid\cF_{n-1}\bigr)
 \le Ce^{-t}.
\]
In particular, their conditional absolute logarithmic moments of
orders one and two are uniformly bounded.
Apply \eqref{eq:layer-cake-identity} to the positive and negative
parts of these logarithms, splitting the integrals at $4a\log\ell$.
The conditional distribution functions of $\log|S_0|$ and
$\log|S_G|$ differ by at most $2\Delta(w)$.
For the first moments this yields
\[
 \left|
 \E(\log|S_0|\mid\cF_{n-1})-\E\log|S_G|
 \right|
 \le C\Delta(w)\log\ell+C\ell^{-4a}.
\]
For the second moments, the same calculation has the additional
factor $2t$ in the integrand.
Subtracting the squares of the conditional means therefore gives
\[
 \left|
 \Var(\log|S_0|\mid\cF_{n-1})-\Var(\log|S_G|)
 \right|
 \le C\Delta(w)(\log\ell)^2+C\ell^{-4a}\log\ell.
\]
Consequently,
\begin{equation}\label{eq:variance-growing-one-row-moments}
 \begin{aligned}
 \left\|
 \E(\log|S_0|\mid\cF_{n-1})-\E\log|S_G|
 \right\|_2
 &\le C\ell^{-4a}\log\ell,\\
 \E\left|
 \Var(\log|S_0|\mid\cF_{n-1})-\Var(\log|S_G|)
 \right|
 &\le C\ell^{-4a}(\log\ell)^2.
 \end{aligned}
\end{equation}

Since $\log|\det\overline C_n|$ and $\log|S_G|$ are
square-integrable, the same holds for $\log\Lambda$.
The conditional mean of $\log|S_G|$ is constant.
The law of total variance thus gives
\[
 \begin{aligned}
 \Var(\log|\det C_n|)
 &=\Var\bigl(\log\Lambda+
 \E(\log|S_0|\mid\cF_{n-1})\bigr)
 +\E\Var(\log|S_0|\mid\cF_{n-1}),\\
 \Var(\log|\det\overline C_n|)
 &=\Var(\log\Lambda)+\Var(\log|S_G|).
 \end{aligned}
\]
By Cauchy--Schwarz and
\eqref{eq:variance-growing-one-row-moments},
\[
 \begin{aligned}
 &\left|
 \Var(\log|\det C_n|)-\Var(\log|\det\overline C_n|)
 \right|\\
 &\qquad\le C\biggl\{
 \ell^{-4a}\log\ell\,
 \sqrt{\Var(\log|\det\overline C_n|)}
 +\ell^{-8a}(\log\ell)^2
 +\ell^{-4a}(\log\ell)^2
 \biggr\}.
 \end{aligned}
\]
Moreover,
\[
 \Var(\log|\det\overline C_n|)
 \ge\Var(\log|S_G|)>0,
\]
where the last quantity is an absolute constant.
Dividing the variance difference by the sum of the two standard
deviations therefore yields
\[
 \left|
 \sqrt{\Var(\log|\det C_n|)}
 -\sqrt{\Var(\log|\det\overline C_n|)}
 \right|
 \le C\ell^{-4a}(\log\ell)^2.
\]
The first estimate in
\eqref{eq:variance-growing-one-row-moments} also gives
\[
 \left|
 \E\log|\det C_n|-\E\log|\det\overline C_n|
 \right|
 \le C\ell^{-4a}\log\ell.
\]

Summing these estimates over the $s_1\le\ell^{3a}$ replacements,
we obtain
\begin{equation}\label{eq:variance-growing-terminal-comparison}
 \begin{aligned}
 \dk\bigl(\log|\det B_n|,\log|\det B_n^{\mathrm{hyb}}|\bigr)
 &\le C\ell^{-a},\\
 \left|\E\log|\det B_n|-\E\log|\det B_n^{\mathrm{hyb}}|\right|
 &\le C\ell^{-a}\log\ell,\\
 \left|
 \sqrt{\Var(\log|\det B_n|)}
 -\sqrt{\Var(\log|\det B_n^{\mathrm{hyb}}|)}
 \right|
 &\le C\ell^{-a}(\log\ell)^2.
 \end{aligned}
\end{equation}
Together with \eqref{eq:variance-hybrid-variance}, the last estimate
proves \eqref{eq:variance-growing-variance} for all sufficiently
large $n$.
Finally, substituting
\eqref{eq:variance-hybrid-variance},
\eqref{eq:variance-hybrid-bound}, and
\eqref{eq:variance-growing-terminal-comparison} into
\eqref{eq:variance-growing-hybrid-comparison} gives
\[
 \begin{aligned}
 \dk\bigl(W_n^{\mathrm v}(B_n),\Normal\bigr)
 &\le C\min\left\{
 \frac{M}{\ell^{3/2}},
 \frac{\log\ell}{\sqrt\ell}
 \right\}
 +C\ell^{-a}
 +C\ell^{-a-1/2}(\log\ell)^2\\
 &\le C\min\left\{
 \frac{M}{\ell^{3/2}},
 \frac{\log\ell}{\sqrt\ell}
 \right\},
 \end{aligned}
\]
since $a\ge20$, $M\ge3$, and
$\ell^{-3/2}\le(\log\ell)/\sqrt\ell$ for all sufficiently large $n$.
This proves \eqref{eq:variance-growing-bound} and completes the
proof of Proposition~\ref{prop:variance-growing-fourth}.
\end{proof}

\newpage

\begin{appendix}

\section{Proof of lemmas}\label{sec:proof-lemmas}

\begin{proof}[Proof of Lemma~\ref{prop:exact-transfer}]
Before comparing the two exact means, we should verify that they are well defined. This follows from Lemma~\ref{lem:log-integrability}.

\begin{lemma}[Integrability of logarithmic determinants]\label{lem:log-integrability}
Let $C_n$ have independent rows, and suppose that within each row the coordinates are independent, centered, have variance one, and have densities bounded by $B$.  Then
\[
 \E|\log|\det C_n||<\infty.
\]
The same conclusion holds for every matrix obtained by replacing any collection of rows by independent rows satisfying the same assumptions.
\end{lemma}

\begin{proof}[Proof of Lemma~\ref{lem:log-integrability}]
Write $c_1^T,\ldots,c_n^T$ for the rows and let $P_i$ be the projection onto the orthogonal complement of the first $i$ rows.  Absolute continuity implies that these rows are linearly independent at every intermediate stage almost surely.  Gram--Schmidt gives
\[
 \log|\det C_n|=\sum_{i=0}^{n-1}\log\gamma_{i+1},
 \qquad \gamma_{i+1}^2=c_{i+1}^TP_ic_{i+1}.
\]
Conditionally on the first $i$ rows, write $\E_i$ for the corresponding conditional expectation.  Then $P_i$ is a deterministic projection of rank $k_i=n-i$.  Put
\[
 Z_{i+1}=\frac{c_{i+1}^TP_ic_{i+1}}{k_i}.
\]
Notice that $|\log \gamma_{i+1}| = \log^+\gamma_{i+1} + (-\log\gamma_{i+1})_+$.
For the positive part, $\log^+\gamma_{i+1}\le \gamma_{i+1}^2$ and
$\E_i\gamma_{i+1}^2=\tr P_i=k_i$.  For the negative part, since
$-\log\gamma_{i+1} = -\frac{1}{2}\log k_{i}-\frac{1}{2}\log Z_{i+1} \leq -\frac{1}{2}\log Z_{i+1}$, we have 
\[
 (-\log\gamma_{i+1})_+
 \le \frac12(-\log Z_{i+1})_+,
\]
and the layer-cake identity \eqref{eq:layer-cake-identity} together with Lemma~\ref{lem:projection-density} gives
\[
 \E_i(-\log Z_{i+1})_+
 =\int_0^\infty \Pp_i(Z_{i+1}\le e^{-u})\dd u
 \le u_B+\int_{u_B}^\infty e^{-k_i u/4}\dd u<\infty.
\]
Therefore, $\E|\log \gamma_{i+1}|< \infty$ for $0\leq i\leq n-1$.
There are only $n$ summands.  This proves the assertion and also justifies every conditional logarithmic expectation used below.
\end{proof}

We now return to the proof of Lemma~\ref{prop:exact-transfer}.
Lemma~\ref{lem:log-integrability} shows that
\[
 \log|\det A_n|
 \qquad\text{and}\qquad
 \log|\det\widetilde A_n|
\]
are integrable.  Hence the two exact means are well defined.

We first prove the mean transfer estimate
\eqref{eq:general-mean-transfer}.  To compare the two exact means, we
replace the rows of $A_n$ by the corresponding rows of
$\widetilde A_n$ one at a time.  For $0\le r\le n$, define
\begin{equation}
  \begin{aligned}
    A_{n}^{(r)} =  
    \left( 
    \widetilde a_1,
    \cdots,
    \widetilde a_r,
    a_{r+1},
    \cdots,
    a_{n}
    \right)^{\mathrm T}, \nonumber
  \end{aligned}
\end{equation} 
where $A_{n}^{(0)}=A_n$ and $A_{n}^{(n)}=\widetilde A_n$, $a_{i}$ is the $i$-th row of $A_n$ and $\widetilde a_{i}$ is the $i$-th row of $\widetilde A_n$. Then we have 
\begin{equation}
  \begin{aligned}
    \E\log|\det A_n|-\E\log|\det \widetilde A_n|
    &=\sum_{r=0}^{n-1}\left(\E\log|\det A_{n}^{(r)}|-\E\log|\det A_{n}^{(r+1)}|\right). \nonumber
  \end{aligned}
\end{equation}
Let \(A_{(s,k)}^{(r)}\) be the \((n-1)\times(n-1)\) matrix obtained by $A_{n}^{(r)}$ deleting row $s$ and column \(k\), then $\alpha_{sk}^{(r)} = (-1)^{s+k}\det(A_{(s,k)}^{(r)})$, $1\leq s,k\leq n$, are the cofactor of the $(s,k)$ entry of $A_{n}^{(r)}$. Thus 
\begin{equation}
  \begin{aligned}
    \det A_{n}^{(r)} &= \sum_{k=1}^{n} a_{r+1,k}\alpha_{r+1,k}^{(r)},
    \qquad
    \det A_{n}^{(r+1)} &= \sum_{k=1}^{n} \widetilde a_{r+1,k}\alpha_{r+1,k}^{(r)}. \nonumber
  \end{aligned}
\end{equation}
For $0\leq r \leq n-1$, fix one replacement step and condition on $\mathcal G_{r+1} = \sigma(\text{unchanged rows})$.  Those $n-1$ rows have full row rank almost surely, so their cofactor vector is nonzero.  Write cofactor vector's Euclidean norm as $\Lambda$ and its normalization as $w=(w_1,\ldots,w_n)$, 
\begin{equation}
  \begin{aligned}
    \Lambda &= \sqrt{\sum_{k=1}^{n}|\alpha_{r+1,k}^{(r)}|^2}, 
    \qquad
    w_k &= \frac{\alpha_{r+1,k}^{(r)}}{\Lambda},\qquad 1\leq k\leq n. \nonumber
  \end{aligned}
\end{equation}
where $\sum_jw_j^2=1$.
The two determinants at this step have the form
\begin{equation}\label{eq:exact-transfer-linear-forms}
 \det A_{n}^{(r)}  = \Lambda S,
 \qquad 
 \det A_{n}^{(r+1)} = \Lambda\widetilde S,
 \qquad
 S=\sum_{j=1}^nw_ja_{r+1,j},
 \qquad
 \widetilde S=\sum_{j=1}^nw_j\widetilde a_{r+1,j}.
\end{equation}
Conditionally on the unchanged rows, both $S$ and $\widetilde S$ have variance one
\begin{equation}
  \begin{aligned}\label{eq:conditional-second-moment}
    \E(S^2\mid \mathcal G_{r+1}) &= \sum_{j=1}^{n}w_j^2\E(a_{r+1,j}^2) = 1,
    \quad
    \E(\widetilde S^2\mid \mathcal G_{r+1}) &= \sum_{j=1}^{n}w_j^2\E(\widetilde a_{r+1,j}^2) = 1.
  \end{aligned}
\end{equation}

To control their logarithms, a variance bound is not enough; we also need a uniform bound on the probability of lying close to zero. Thus, applying Lemma~\ref{lem:projection-density} with rank one gives density bounds depending only on $M_0$,
\begin{equation}
  \begin{aligned}\label{eq:conditional-density-bound}
    \|g_S \mid \mathcal G_{r+1}\|_\infty &\le \sqrt2 M_0,
    \qquad
    \|g_{\widetilde S} \mid \mathcal G_{r+1}\|_\infty \le \sqrt2 M_0,
  \end{aligned}
\end{equation}
where $g_S \mid \mathcal G_{r+1}$ and $g_{\widetilde S} \mid \mathcal G_{r+1}$ are the conditional densities of $S$ and $\widetilde S$, respectively. 
The coupling of the changing row is independent of the conditioned rows, and therefore \eqref{eq:row-coupling-probability} gives
\begin{equation}\label{eq:exact-transfer-one-row-coupling}
 \Pp(S\ne\widetilde S\mid \mathcal G_{r+1})\le\rho_n.
\end{equation}
Then, we can use Lemma~\ref{lem:log-comparison} to get a comparison of logarithmic expectations. Using Lemma~\ref{lem:log-comparison}, together with \eqref{eq:conditional-second-moment},\eqref{eq:conditional-density-bound} and \eqref{eq:exact-transfer-one-row-coupling}, yields
\begin{equation}\label{eq:exact-transfer-one-row-mean}
 \left|\E(\log|S|\mid \mathcal G_{r+1}) )
       -\E(\log|\widetilde S|\mid \mathcal G_{r+1})\right|
 \le C_{M_0}\rho_n\log\frac e{\rho_n}.
\end{equation}
The common term $\log\Lambda$ cancels.  Taking expectations in
\eqref{eq:exact-transfer-one-row-mean} and summing the $n$ row
replacements proves \eqref{eq:general-mean-transfer}.

We now prove the distribution transfer estimate
\eqref{eq:general-distribution-transfer}.  We first compare the two
logarithmic determinants with the same center, using the coupling.
And then  using \eqref{eq:general-mean-transfer} to change the center from
$\E\log|\det A_n|$ to $\E\log|\det\widetilde A_n|$.

On the event $\{A_n=\widetilde A_n\}$,
\[
 \log|\det A_n|
 =
 \log|\det\widetilde A_n|.
\]
Thus the laws of
\begin{equation}
  \begin{aligned}
    \frac{\log|\det A_n| - \E \log|\det A_n|}{ \sqrt{\tfrac12\log n} }, \quad \frac{\log|\det\widetilde A_n| - \E \log|\det A_n|}{ \sqrt{\tfrac12\log n} } \nonumber
  \end{aligned}
\end{equation}
differ by at most $\Pp(A_n\ne\widetilde A_n)$ in Kolmogorov distance.  The last step is to quantify the effect of changing a center.  We record the elementary perturbation bounds in Lemma~\ref{lem:perturb} used here and later in the paper.

Since 
\begin{equation}
  \begin{aligned}
    \frac{\log|\det\widetilde A_n| - \E \log|\det A_n|}{ \sqrt{\tfrac12\log n} } + \frac{\E \log|\det A_n| - \E \log|\det\widetilde A_n|}{\sqrt{\tfrac12\log n}}
    = W_{n}^{\mathrm e}(\widetilde A_n), \nonumber 
  \end{aligned}
\end{equation}
we apply Lemma~\ref{lem:perturb} (iii), gives
\begin{equation}
  \begin{aligned}\label{eq:exact-transfer-distribution-step}
    \dk(W_n^{\mathrm e}(A_n),\Normal)
    \le& \dk(W_n^{\mathrm e}(\widetilde A_n),\Normal)
    +\Pp(A_n\ne\widetilde A_n)\\
    &+C\frac{|\E \log|\det A_n| - \E \log|\det\widetilde A_n||}{\sqrt{\log n}}.
  \end{aligned}
\end{equation}
Substitution of \eqref{eq:general-mean-transfer} into \eqref{eq:exact-transfer-distribution-step} proves \eqref{eq:general-distribution-transfer}. Thus, we complete the proof of Lemma~\ref{prop:exact-transfer}.
\end{proof}

\begin{proof}[Proof of Lemma~\ref{lem:variance-transfer}]
Use the intermediate matrices $A_n^{(r)}$, $0\le r\le n$, from
the proof of Lemma~\ref{prop:exact-transfer}.
Throughout this proof, $C$ depends only on $M_0$.

We first verify square-integrability.
For any of these matrices, let $\gamma_{i+1}$ be its
Gram--Schmidt distance at step $i$, and condition on the first
$i$ rows. Then $\E_i\gamma_{i+1}^2=n-i$.
The distance dominates the absolute projection onto any unit
direction in the remaining subspace, so the rank-one density
bound in Lemma~\ref{lem:projection-density} gives
\[
 \Pp_i(\gamma_{i+1}>e^t)\le(n-i)e^{-2t},
 \qquad
 \Pp_i(\gamma_{i+1}<e^{-t})\le Ce^{-t},
 \qquad t\ge0.
\]
Integrating these tails shows that $\log\gamma_{i+1}\in L^2$.
Thus $\log|\det A_n^{(r)}|\in L^2$ for every $r$.
Write
\[
 \sigma_{n,r}^2=\Var(\log|\det A_n^{(r)}|),
 \qquad \sigma_{n,0}=\sigma_n,
 \qquad \sigma_{n,n}=\widetilde\sigma_n.
\]

Fix $0\le r<n$ and use $\mathcal G_{r+1},\Lambda,S,\widetilde S$
from \eqref{eq:exact-transfer-linear-forms}. Then
\[
 \log|\det A_n^{(r)}|=\log\Lambda+\log|S|,
 \qquad
 \log|\det A_n^{(r+1)}|=\log\Lambda+\log|\widetilde S|.
\]
The conditional second-moment and density bounds
\eqref{eq:conditional-second-moment}--\eqref{eq:conditional-density-bound}
imply
\begin{equation}\label{eq:variance-transfer-log-tails}
 \Pp\bigl(|\log|S||>t\mid\mathcal G_{r+1}\bigr)
 +\Pp\bigl(|\log|\widetilde S||>t\mid\mathcal G_{r+1}\bigr)
 \le Ce^{-t},\qquad t\ge0.
\end{equation}
In particular, both conditional logarithmic second moments are
bounded, and $\log\Lambda\in L^2$ by the displayed factorization.
The conditional laws of $S$ and $\widetilde S$ have densities,
so neither logarithm can be almost surely constant.
The law of total variance therefore gives $\sigma_{n,r}>0$
for every $0\le r\le n$.

By \eqref{eq:exact-transfer-one-row-coupling} and
\eqref{eq:variance-transfer-log-tails}, for $p=1,2$,
\[
 \begin{aligned}
 &\E\bigl(|\log|S||^p\1_{\{S\ne\widetilde S\}}
       \mid\mathcal G_{r+1}\bigr)\\
 &\qquad\le\int_0^\infty
 pt^{p-1}\min\{\rho_n,Ce^{-t}\}\dd t
 \le C\rho_n\left(\log\frac e{\rho_n}\right)^p.
 \end{aligned}
\]
The same estimate holds for $\widetilde S$.
Since the logarithms agree on $\{S=\widetilde S\}$, this gives
\begin{equation}\label{eq:variance-transfer-conditional-moments}
 \begin{aligned}
 \left|\E\bigl(\log|S|-\log|\widetilde S|
       \mid\mathcal G_{r+1}\bigr)\right|
 &\le C\rho_n\log\frac e{\rho_n},\\
 \left|\Var\bigl(\log|S|\mid\mathcal G_{r+1}\bigr)
 -\Var\bigl(\log|\widetilde S|\mid\mathcal G_{r+1}\bigr)\right|
 &\le C\rho_n\left(\log\frac e{\rho_n}\right)^2.
 \end{aligned}
\end{equation}
For the second inequality, compare the second moments and then
the squared means, using the bounded conditional first moments
from \eqref{eq:variance-transfer-log-tails}.

Apply the law of total variance to the two logarithmic
determinants. Since
\[
 \Var\bigl(\log\Lambda+
 \E(\log|\widetilde S|\mid\mathcal G_{r+1})\bigr)
 \le\sigma_{n,r+1}^2,
\]
Cauchy--Schwarz and
\eqref{eq:variance-transfer-conditional-moments} yield
\begin{equation}\label{eq:variance-transfer-one-row}
 |\sigma_{n,r}^2-\sigma_{n,r+1}^2|
 \le C\rho_n\log\frac e{\rho_n}\,\sigma_{n,r+1}
 +C\rho_n\left(\log\frac e{\rho_n}\right)^2.
\end{equation}
Here the squared mean-difference term is absorbed into the last
term because $\rho_n\le1$.

Summing \eqref{eq:variance-transfer-one-row} gives
\begin{equation}\label{eq:variance-transfer-sum}
 \sum_{r=0}^{n-1}|\sigma_{n,r}^2-\sigma_{n,r+1}^2|
 \le Cn\rho_n\log\frac e{\rho_n}
       \max_{0\le r\le n}\sigma_{n,r}
 +Cn\rho_n\left(\log\frac e{\rho_n}\right)^2.
\end{equation}
Every intermediate variance is at most
$\widetilde\sigma_n^2$ plus the sum on the left. Hence
\[
 \left(\max_{0\le r\le n}\sigma_{n,r}\right)^2
 \le\widetilde\sigma_n^2
 +Cn\rho_n\log\frac e{\rho_n}
       \max_{0\le r\le n}\sigma_{n,r}
 +Cn\rho_n\left(\log\frac e{\rho_n}\right)^2.
\]
Solving this quadratic inequality yields
\[
 \max_{0\le r\le n}\sigma_{n,r}
 \le\widetilde\sigma_n
 +Cn\rho_n\log\frac e{\rho_n}
 +C\sqrt{n\rho_n}\log\frac e{\rho_n}.
\]
Substituting into \eqref{eq:variance-transfer-sum}, and using
\[
 (n\rho_n)^{3/2}\le\frac12\bigl(n^2\rho_n^2+n\rho_n\bigr),
\]
proves \eqref{eq:general-variance-transfer}. Thus, we complete the proof of Lemma~\ref{lem:variance-transfer}.
\end{proof}

\begin{proof}[Proof of Lemma~\ref{lem:terminal-diagonal}]
We use the notation in \eqref{eq:common-parameters}, with
$\alpha$ and $a$ as in the statement of the lemma.
Fix $i$ in the stated range $n-s_2\le i\le n-1$.  Let $A=A(i)$, let $b_r$ be its $r$th column, and let $A_r$ be obtained by deleting $b_r$.  Put
\begin{equation}
  \begin{aligned}\label{def:T-Tr}
    G_r &=(n^{-1}A_rA_r^T+\lambda I_i)^{-1},\qquad T_r=n^{-1}\tr G_r,    \\
    G &=(n^{-1}AA^T+\lambda I_i)^{-1},\qquad T=n^{-1}\tr G.
  \end{aligned}
\end{equation}

Both $G_r$ and $G$ are positive definite; in particular,
$(G_r)_{uu}>0$.  Our goal is to show that $p_{rr}(i)$ is small
uniformly in $r$.  The following lemma reduces this problem to a
lower bound for the quadratic form $n^{-1}b_r^TG_rb_r$ and records
the rank-one trace comparison needed below.

\begin{lemma}\label{lem:projection-identities}
Let $A$ be an $i\times n$ matrix ($i \leq n-1$), let $b_r$ be its $r$th column, and let $A_r$ be obtained by deleting that column.  Assume that both $A$ and $A_r$ have full row rank.  Let
$P=I_{n}-A^T(AA^T)^{-1}A=(p_{st})_{1\le s,t\le n}$. Then we have
\begin{equation}\label{ex:eq:exact-diagonal-projection}
 p_{rr}=\{1+b_r^T(A_rA_r^T)^{-1}b_r\}^{-1}.
\end{equation}
For $\lambda>0$ and
$G_r=(n^{-1}A_rA_r^T+\lambda I_{i})^{-1}$,
\begin{equation}\label{ex:eq:regularized-diagonal-projection}
 p_{rr}\le\{1+n^{-1}b_r^TG_rb_r\}^{-1}.
\end{equation}
Moreover, if $K$ is any positive-semidefinite matrix, $v$ is
a vector of the appropriate dimension, and $\lambda>0$, then
\begin{equation}\label{eq:rank-one-regularized-trace}
 0\le
 \tr(K+\lambda I)^{-1}
 -
 \tr(K+vv^T+\lambda I)^{-1}
 \le \lambda^{-1}.
\end{equation}
Consequently, adding a rank-one positive-semidefinite matrix,
or deleting one while preserving positive semidefiniteness,
changes the trace of the regularized inverse by at most
$\lambda^{-1}$.
\end{lemma}

\begin{proof}[Proof of Lemma~\ref{lem:projection-identities}]
Put $H=A_rA_r^T$ and $u=b_r$.  Full row rank of $A_r$ makes $H$ positive definite.  Since $AA^T=H+uu^T$, the Sherman--Morrison formula gives
\begin{equation}\label{eq:projection-SM-inverse}
 (H+uu^T)^{-1}=H^{-1}-\frac{H^{-1}uu^TH^{-1}}{1+u^TH^{-1}u}.
\end{equation}
Since $P=I_n-A^T(AA^T)^{-1}A$, its $r$th diagonal entry is
\[
 p_{rr}=1-u^T(H+uu^T)^{-1}u.
\]
Applying \eqref{eq:projection-SM-inverse} to the quadratic form on the
right-hand side, we obtain
\[
\begin{aligned}
 p_{rr}
 &=1-\left\{
 u^TH^{-1}u
 -\frac{(u^TH^{-1}u)^2}{1+u^TH^{-1}u}
 \right\} \\
 &=1-\frac{u^TH^{-1}u}{1+u^TH^{-1}u}
 =\frac{1}{1+u^TH^{-1}u},
\end{aligned}
\]
which proves \eqref{ex:eq:exact-diagonal-projection}.

Regularization is useful because $H^{-1}$ can be unstable when
the smallest singular value of $A_r$ is small. For symmetric matrices $S$ and $T$, the notation $S\succeq T$ means that $S-T$ is positive semidefinite, and $S\preceq T$ means $T\succeq S$.  In this order,
\[
 H^{-1}\succeq (H+n\lambda I_i)^{-1}
 =n^{-1}(n^{-1}H+\lambda I_i)^{-1}=n^{-1}G_r
\]
shows that $u^TH^{-1}u\ge n^{-1}u^TG_ru$.  Inserting this inequality into \eqref{ex:eq:exact-diagonal-projection} proves \eqref{ex:eq:regularized-diagonal-projection}.

Finally, let $R=(K+\lambda I)^{-1}$ for an arbitrary positive-semidefinite matrix $K$.  The Sherman--Morrison formula
gives
\[
 (K+vv^T+\lambda I)^{-1}
 =
 R-\frac{Rvv^TR}{1+v^TRv}.
\]
Taking traces yields
\[
 \tr(K+\lambda I)^{-1}
 -
 \tr(K+vv^T+\lambda I)^{-1}
 =
 \frac{v^TR^2v}{1+v^TRv}.
\]
Since $K\succeq0$, we have
$0\preceq R\preceq\lambda^{-1}I$, and hence
$R^2\preceq\lambda^{-1}R$.  Therefore,
\[
 0
 \le
 \frac{v^TR^2v}{1+v^TRv}
 \le
 \lambda^{-1}\frac{v^TRv}{1+v^TRv}
 \le
 \lambda^{-1},
\]
which proves \eqref{eq:rank-one-regularized-trace}.  The corresponding
bound for deleting a rank-one positive-semidefinite matrix follows by
reversing the addition.
\end{proof}

We now return to the proof of Lemma~\ref{lem:terminal-diagonal}.
By \eqref{ex:eq:regularized-diagonal-projection},
\begin{equation}\label{ex:eq:terminal-projection-bound}
 p_{rr}(i)
 \le
 \{1+n^{-1}b_r^TG_rb_r\}^{-1}.
\end{equation}
Applying \eqref{eq:rank-one-regularized-trace} with
$K=n^{-1}A_rA_r^T$ and $v=n^{-1/2}b_r$, and then dividing by $n$,
gives
\begin{equation}\label{ex:eq:terminal-trace-comparison}
 0\le T_r-T\le(n\lambda)^{-1},
\end{equation}
where $T, T_{r}$ are defined in \eqref{def:T-Tr}.

Thus, it remains to obtain a lower bound for
$n^{-1}b_r^TG_rb_r$, uniformly in $r$.  Conditionally on $A_r$, the
matrix $G_r$ is fixed and the coordinates of $b_r$ are independent,
centered, and have variance one.  Hence
\[
 \E\left(n^{-1}b_r^TG_rb_r\mid A_r\right)
 =
 n^{-1}\tr G_r
 =
 T_r.
\]
Thus, $T_r$ is exactly the conditional mean of
$n^{-1}b_r^TG_rb_r$ given $A_r$.  We first establish a
high-probability lower bound for $T=n^{-1}\tr G$.  Since
\eqref{ex:eq:terminal-trace-comparison} gives $T_r\ge T$ for every
$r$, this yields a simultaneous lower bound for all $T_r$.  We then
use this simultaneous bound to prove a uniform lower bound for
$n^{-1}b_r^TG_rb_r$, $1\le r\le n$.

Set
\[
 y=\frac{i}{n},
 \qquad
 s_y(\lambda)
 =
 \frac{2}{
 1+\lambda-y+
 \sqrt{(1+\lambda-y)^2+4y\lambda}
 }.
\]
Since $k_i=n-i\le s_2$, we have
\[
 1-y=\frac{k_i}{n}\le\ell^{-20a}.
\]
Moreover, $y\ge1/2$ for all sufficiently large $n$.  From the
explicit formula for $s_y(\lambda)$,
\[
 ys_y(\lambda)
 \ge
 c\min\{(1-y)^{-1},\lambda^{-1/2}\}.
\]
Indeed, if $1-y\ge\sqrt\lambda$, the denominator in $s_y(\lambda)$
is $O(1-y)$, whereas if $1-y<\sqrt\lambda$, it is
$O(\sqrt\lambda)$.  Finally,
\[
 (1-y)^{-1}\ge\ell^{20a},
 \qquad
 \lambda^{-1/2}=n^{1/12}\ge\ell^{20a}
\]
for all sufficiently large $n$.  Hence
\[
 ys_y(\lambda)\ge c\ell^{20a}
\]
uniformly over $n-s_2\le i\le n-1$.

We now turn this deterministic estimate into a high-probability lower
bound for $T=n^{-1}\tr G$.  For this purpose, we compare $\E T$ with
$ys_y(\lambda)$ and control the variance of $T$.  The following lemma
provides exactly these two estimates.

\begin{lemma}\label{lem:regularized-inverse}
Let $X$ be a $p\times n$ independent-entry matrix whose entries are centered, have variance one, and fourth moments at most $M$, where $M\ge3$ and $n/2\le p\le n$.  Let $I_p$ denote the $p\times p$ identity matrix, and put
\[
 G=(n^{-1}XX^T+\lambda I_p)^{-1},\quad y=p/n,
\]
and
\[
 s_y(\lambda)=\frac{2}{1+\lambda-y+\sqrt{(1+\lambda-y)^2+4y\lambda}},
\]
where $\lambda = n^{-1/6}$ is defined in \eqref{eq:common-parameters}.
Then
\begin{align}
 |\E (n^{-1}\tr G) - y s_y(\lambda)|&\le C\{n^{-2/3}+Mn^{-1/2}\},\label{ex:eq:regularized-mean}\\
 \Var( n^{-1}\tr G )&\le Cn^{-2/3}.
 \label{ex:eq:regularized-variance}
\end{align}
\end{lemma}

For all sufficiently large $n$, we have $i\ge n-s_2\ge n/2$, so
Lemma~\ref{lem:regularized-inverse} applies with $X=A$ and $p=i$.
Since $M\le C_0\ell^\alpha$ and $\alpha$ is fixed,
\[
 n^{-2/3}+Mn^{-1/2}
 \le n^{-2/3}+C_0\ell^\alpha n^{-1/2}
 =o(1).
\]
Combining this with the preceding bound for $ys_y(\lambda)$ gives
\[
 \E T
 \ge
 ys_y(\lambda)-C\{n^{-2/3}+Mn^{-1/2}\}
 \ge
 c\ell^{20a},
 \qquad
 \Var(T)\le Cn^{-2/3}.
\]

Since $\ell^{10a}=o(\ell^{20a})$, after decreasing $c$ if necessary,
\[
 \E T-\ell^{10a}\ge c\ell^{20a}
\]
for all sufficiently large $n$.  Chebyshev's inequality therefore
gives
\begin{equation}\label{ex:eq:terminal-trace-high}
 \begin{aligned}
 \Pp(T<\ell^{10a})
 &\le
 \Pp\left(
 |T-\E T|\ge \E T-\ell^{10a}
 \right)
 \le
 \frac{\Var(T)}{(\E T-\ell^{10a})^2}
 \le
 Cn^{-2/3}\ell^{-40a}
 =
 o(\ell^{-8a}).
 \end{aligned}
\end{equation}

Define
\begin{equation}\label{def:B-t-event}
 \mathcal B_T
 =
 \bigcup_{r=1}^n\{T_r<\ell^{10a}\}. 
\end{equation}
The lower bound $T_r\ge T$ in
\eqref{ex:eq:terminal-trace-comparison} implies $\mathcal B_T \subseteq \{T<\ell^{10a}\}$.
Hence,
\begin{equation}\label{ex:eq:terminal-leave-one-out-high}
 \Pp(\mathcal B_T)
 \le
 \Pp(T<\ell^{10a})
 =
 o(\ell^{-8a}).
\end{equation}   
Equivalently, outside $\mathcal B_T$,
\begin{equation}\label{eq:low-bound-esti-Tr}
  T_r\ge\ell^{10a} 
\end{equation}
simultaneously for all $r$.

We now return to the quadratic forms.  For each $r$, condition on
$A_r$.  Then $G_r$ is fixed, while the entries of $b_r$ are
independent.  We decompose
\begin{equation}\label{eq:Gr-decompose}
 n^{-1}b_r^TG_rb_r
 =
 D_r+O_r,
\end{equation}
where
\[
 D_r
 =
 n^{-1}\sum_u(G_r)_{uu}b_{ur}^2,
 \qquad
 O_r
 =
 n^{-1}\sum_{u\ne v}(G_r)_{uv}b_{ur}b_{vr}.
\]
It remains to show that, outside an additional event of probability
$o(\ell^{-8a})$,
\begin{equation}\label{ex:eq:terminal-decomposition-good}
 D_r\ge\ell^{-a}T_r,
 \qquad
 |O_r|\le1
\end{equation}
simultaneously for all $r$.

Together with the lower bound
\eqref{eq:low-bound-esti-Tr}, the two estimates in
\eqref{ex:eq:terminal-decomposition-good} will yield a uniform lower
bound for $n^{-1}b_r^TG_rb_r$.  We now prove these two estimates
separately, beginning with the diagonal term $D_r$.
We first prove the diagonal estimate
\[
 D_r
 =
 n^{-1}\sum_u(G_r)_{uu}b_{ur}^2
 \ge
 \ell^{-a}T_r.
\]
Since $T_r = n^{-1}\sum_u(G_r)_{uu}$,
we need a lower bound for the weighted sum of the squares
$b_{ur}^2$.  We truncate and normalize these entries as follows:
\[
 \widehat b_{ur}
 =
 b_{ur}\1_{\{|b_{ur}|\le\ell^a\}},
 \qquad
 \widetilde b_{ur}
 =
 \frac{\widehat b_{ur}-\E\widehat b_{ur}}
 {\sqrt{\Var(\widehat b_{ur})}}.
\]
The fourth-moment bound gives, uniformly in $u,r$,
\begin{equation}\label{ex:eq:internal-trunc}
 |\E\widehat b_{ur}|\le M\ell^{-3a},\quad
 \Var(\widehat b_{ur})\ge1-M\ell^{-2a}-M^2\ell^{-6a}\ge\frac12,
\end{equation}
\begin{equation}\label{ex:eq:internal-8}
 \E|\widetilde b_{ur}|^4\le CM,
 \qquad \E|\widetilde b_{ur}|^8\le CM\ell^{4a}.
\end{equation}
From \eqref{ex:eq:internal-trunc}, uniformly in $u,r$,
\[
 \widetilde b_{ur}^{\,2}
 \le4\widehat b_{ur}^{\,2}+C M^2\ell^{-6a}.
\]

Because $(G_r)_{uu}>0$, $\widehat b_{ur}^{\,2}\le b_{ur}^2$, and
$n^{-1}\sum_u(G_r)_{uu}=T_r$, on the event
$\{D_r<\ell^{-a}T_r\}$ we have
\[
 \begin{aligned}
 n^{-1}\sum_u(G_r)_{uu}\widetilde b_{ur}^{\,2}
 &\le
 4n^{-1}\sum_u(G_r)_{uu}\widehat b_{ur}^{\,2}
 +
 CM^2\ell^{-6a}T_r\\
 &\le
 4D_r+CM^2\ell^{-6a}T_r\\
 &\le
 \{4\ell^{-a}+CM^2\ell^{-6a}\}T_r
 \le
 \frac12T_r
 \end{aligned}
\]
for all sufficiently large $n$, since
$M\le C_0\ell^\alpha$ and $a>3\alpha+12$ imply
\[
 4\ell^{-a}+CM^2\ell^{-6a}
 \le 4\ell^{-a}+CC_0^2\ell^{2\alpha-6a}
 \longrightarrow0.
\]
Define
$
 S_r
 :=
 \sum_u(G_r)_{uu}
 \bigl(\widetilde b_{ur}^{\,2}-1\bigr).
$
Since $\sum_u(G_r)_{uu}=nT_r$, on the event
\[
 \{D_r<\ell^{-a}T_r\}
 \cap
 \{T_r\ge\ell^{10a}\}
\]
we have
\[ 
 \begin{aligned}
 S_r
 &=
 \sum_u(G_r)_{uu}\widetilde b_{ur}^{\,2}
 -
 \sum_u(G_r)_{uu}
 \le
 \frac n2T_r-nT_r
 =
 -\frac n2T_r
 \le
 -\frac n2\ell^{10a}.
 \end{aligned}
\]
Consequently,
\[
 \{D_r<\ell^{-a}T_r\}
 \cap
 \{T_r\ge\ell^{10a}\}
 \subseteq
 \{|S_r|\ge cn\ell^{10a}\}.
\]

To bound the probability of the event on the right, we estimate the
conditional fourth moment of $S_r$.  Conditionally on $A_r$, the
matrix $G_r$ is deterministic, while the variables
$\widetilde b_{ur}$ are independent, centered, and have variance one.
The following Lemma~\ref{lem:quadratic-forms} gives the required fourth-moment
estimate.  Its off-diagonal estimate will also be used later to
control $O_r$.

\begin{lemma}[Quadratic-form estimates]\label{lem:quadratic-forms}
Let $x=(x_1,\ldots,x_N)^T$ have independent centered coordinates with $\E x_j^2=1$ and $\max_j\E|x_j|^4\le M$.
Let $A$ be a deterministic symmetric matrix, then
\begin{equation}
  \begin{aligned}\label{eq:var-part}
    \Var(x^TAx)\le 4M\tr A^2.    
  \end{aligned}
\end{equation}
and
\begin{equation}
  \begin{aligned}\label{eq:off-diag-part}
    \E\left|\sum_{r\ne s}a_{rs}x_rx_s\right|^4
    \le CM^2(\tr A^2)^2.
  \end{aligned}
\end{equation} 
\item If in addition $\max_j\E|x_j|^8\le M_8$ and $A\succeq0$, then
\begin{equation}
  \begin{aligned}\label{eq:diag-part}
    \E\left|\sum_j a_{jj}(x_j^2-1)\right|^4
    \le C\{M_8\tr A^4+M^2(\tr A^2)^2\}.    
  \end{aligned}
\end{equation}
\end{lemma}

\begin{proof}[Proof of Lemma~\ref{lem:quadratic-forms}]
For \eqref{eq:var-part}, write
\[
 x^TAx-\tr A
 =\sum_j a_{jj}(x_j^2-1)+2\sum_{r<s}a_{rs}x_rx_s.
\]
The two sums are uncorrelated.  Independence gives
\[
 \Var\left(\sum_j a_{jj}(x_j^2-1)\right)
 \le M\sum_j a_{jj}^2,
 \qquad
 \Var\left(2\sum_{r<s}a_{rs}x_rx_s\right)
 =4\sum_{r<s}a_{rs}^2,
\]
which proves \eqref{eq:var-part}.

For \eqref{eq:off-diag-part}, replace $A$ by its off-diagonal part.  Let $\delta_1,\ldots,\delta_N$ be independent Bernoulli variables with parameter $1/2$, independent of $x$, and put
\[
 S_\delta=\sum_{r,s}\delta_r(1-\delta_s)a_{rs}x_rx_s.
\]
Because $a_{rr}=0$ and $\E_\delta[\delta_r(1-\delta_s)]=1/4$ for $r\ne s$,
$
 \sum_{r\ne s}a_{rs}x_rx_s=4\E_\delta S_\delta,
$
where $\E_{\delta}$ means taking expectation on $\delta$. Jensen's inequality therefore gives
\begin{equation}\label{eq:qf-random-partition}
 \E\left|\sum_{r\ne s}a_{rs}x_rx_s\right|^4
 \le4^4\E_\delta\E_{x}|S_\delta|^4.
\end{equation}
Fix the selectors and write $I=\{r:\delta_r=1\}$ and $J=I^c$.  Conditional on $(x_r)_{r\in I}$, the variable $S_\delta$ is the independent linear form
$\sum_{s\in J}c_sx_s$, where $c_s=\sum_{r\in I}a_{rs}x_r$.  Direct expansion of its fourth moment gives
\begin{equation}\label{eq:qf-linear-fourth}
 \E\left(|S_\delta|^4\mid(x_r)_{r\in I}\right)
 \le CM\left(\sum_{s\in J}c_s^2\right)^2.
\end{equation}
Let $B_\delta=(a_{rs})_{r\in I,s\in J}$.  Then
$\sum_{s\in J}c_s^2=x_I^TB_\delta B_\delta^Tx_I$.  Recall that a real symmetric matrix $C$ is \emph{positive semidefinite} when $z^TCz\ge0$ for every real vector $z$; we write this as $C\succeq0$.  
Since $B_\delta B_\delta^T\succeq0$, \eqref{eq:var-part}, together with
\[
 \E(x_I^TCx_I)=\tr C
 \qquad\text{and}\qquad
 \tr(C^2)\le(\tr C)^2
\]
for $C\succeq0$, implies
\begin{equation}\label{eq:qf-positive-second}
 \E\left(x_I^TB_\delta B_\delta^Tx_I\right)^2
 \le CM\{\tr(B_\delta B_\delta^T)\}^2
 \le CM(\tr A^2)^2.
\end{equation}
Substitution of \eqref{eq:qf-linear-fourth}--\eqref{eq:qf-positive-second} into
\eqref{eq:qf-random-partition} proves \eqref{eq:off-diag-part}.

For \eqref{eq:diag-part}, set $Y_j=a_{jj}(x_j^2-1)$.  The variables $Y_j$ are independent and centered.  Expanding the fourth power and using independence gives
\[
 \E\left|\sum_jY_j\right|^4
 =\sum_j\E Y_j^4+6\sum_{r<s}\E Y_r^2\E Y_s^2
 \le\sum_j\E Y_j^4+3\left(\sum_j\E Y_j^2\right)^2.
\]
Now
\[
 \sum_j\E|Y_j|^4\le CM_8\sum_j a_{jj}^4
 \le CM_8\tr A^4,
\]
while
\[
 \left(\sum_j\E Y_j^2\right)^2
 \le CM^2\left(\sum_j a_{jj}^2\right)^2
 \le CM^2(\tr A^2)^2.
\]
This proves \eqref{eq:diag-part}.
\end{proof}

We now return to the proof of Lemma~\ref{lem:terminal-diagonal}.
Conditionally on $A_r$, apply \eqref{eq:diag-part} of
Lemma~\ref{lem:quadratic-forms} with
\[
 x=(\widetilde b_{1r},\ldots,\widetilde b_{ir})^T
 \qquad\text{and}\qquad
 A=G_r.
\]
By \eqref{ex:eq:internal-8},
\[
 \E\left(|S_r|^4\mid A_r\right)
 \le
 C\left\{
 M\ell^{4a}\tr G_r^4
 +
 M^2(\tr G_r^2)^2
 \right\}.
\]
Moreover, since
\[
 0\prec G_r\preceq\lambda^{-1}I_i,
\]
we have
\[
 \tr G_r^4\le n\lambda^{-4},
 \qquad
 \tr G_r^2\le n\lambda^{-2}.
\]
Therefore, Markov's inequality and the preceding event inclusion give
\begin{equation}\label{eq:one-prob-need-sum}
 \begin{aligned}
 &\Pp\left(
 D_r<\ell^{-a}T_r,\,
 T_r\ge\ell^{10a}
 \mid A_r
 \right)\\
 &\quad\le
 \Pp\left(
 |S_r|\ge cn\ell^{10a}
 \mid A_r
 \right)\\
 &\quad\le
 Cn^{-4}\ell^{-40a}
 \left\{
 M\ell^{4a}\tr G_r^4
 +
 M^2(\tr G_r^2)^2
 \right\}\\
 &\quad\le
 C\left\{
 Mn^{-3}\lambda^{-4}\ell^{-36a}
 +
 M^2n^{-2}\lambda^{-4}\ell^{-40a}
 \right\}.
 \end{aligned}
\end{equation}
Let
\[
 \mathcal B_D
 =
 \bigcup_{r=1}^n\{D_r<\ell^{-a}T_r\}.
\]

By the definition \eqref{def:B-t-event} and \eqref{eq:low-bound-esti-Tr}, on
$\mathcal B_T^c$ we have
\[
 T_r\ge\ell^{10a}
\]
for every $r$. Taking expectations in
\eqref{eq:one-prob-need-sum} and summing over $r$ therefore gives
\[
 \begin{aligned}
 \Pp(\mathcal B_D\cap\mathcal B_T^c)
 &\le
 \sum_{r=1}^n
 \Pp\left(
 D_r<\ell^{-a}T_r,\,
 T_r\ge\ell^{10a}
 \right)\\
 &\le
 C\{Mn^{-4/3}\ell^{-36a}
 +M^2n^{-1/3}\ell^{-40a}\}\\
 &=
 o(\ell^{-8a}).
 \end{aligned}
\]
Together with \eqref{ex:eq:terminal-leave-one-out-high}, this yields
\begin{equation}\label{ex:eq:terminal-diagonal-failure}
 \Pp(\mathcal B_D)
 \le
 \Pp(\mathcal B_D\cap\mathcal B_T^c)
 +
 \Pp(\mathcal B_T)
 =
 o(\ell^{-8a}).
\end{equation}

It remains to prove the second estimate in
\eqref{ex:eq:terminal-decomposition-good}, namely
\[
 |O_r|\le1
\]
simultaneously for all $r$, outside an event of probability
$o(\ell^{-8a})$.  Recall that
\[
 O_r
 =
 n^{-1}\sum_{u\ne v}(G_r)_{uv}b_{ur}b_{vr}.
\]
Conditionally on $A_r$, the matrix $G_r$ is deterministic, while the
entries of $b_r$ are independent, centered, and have variance one.
Therefore, Lemma~\ref{lem:quadratic-forms}
\eqref{eq:off-diag-part} gives
\[
 \E(|O_r|^4\mid A_r)
 \le
 CM^2n^{-4}(\tr G_r^2)^2.
\]
Since $0\prec G_r\preceq\lambda^{-1}I_i$,
we have
\[
 \tr G_r^2
 \le
 n\lambda^{-2}.
\]
Hence
\[
 \E(|O_r|^4\mid A_r)
 \le
 CM^2n^{-2}\lambda^{-4}.
\]
By Markov's inequality and the union bound over $1\le r\le n$,
\begin{equation}\label{ex:eq:terminal-offdiagonal-failure}
 \begin{aligned}
 \Pp\left(\max_{1\le r\le n}|O_r|>1\right)
 &\le
 \sum_{r=1}^n\E|O_r|^4
 \le
 CM^2n^{-1}\lambda^{-4}
 =
 CM^2n^{-1/3}
 =
 o(\ell^{-8a}).
 \end{aligned}
\end{equation}
Define
\[
 \mathcal E
 :=
 \mathcal B_T
 \cup
 \mathcal B_D
 \cup
 \left\{
 \max_{1\le r\le n}|O_r|>1
 \right\}.
\]
By \eqref{ex:eq:terminal-leave-one-out-high},
\eqref{ex:eq:terminal-diagonal-failure}, and
\eqref{ex:eq:terminal-offdiagonal-failure},
\[
 \Pp(\mathcal E)=o(\ell^{-8a}).
\]
On $\mathcal E^c$, we have, simultaneously for all $r$,
\[
 T_r\ge\ell^{10a},
 \qquad
 D_r\ge\ell^{-a}T_r,
 \qquad
 |O_r|\le1.
\]
Consequently, \eqref{eq:Gr-decompose} gives
\[
 \begin{aligned}
 n^{-1}b_r^TG_rb_r
 &=
 D_r+O_r\\
 &\ge
 D_r-|O_r|\\
 &\ge
 \ell^{-a}T_r-1\\
 &\ge
 \ell^{9a}-1\\
 &\ge
 c\ell^{9a}
 \end{aligned}
\]
for every $r$ and all sufficiently large $n$.  Hence, by
\eqref{ex:eq:terminal-projection-bound},
\[
 \max_{1\le r\le n}p_{rr}(i)
 \le
 C\ell^{-9a}
\]
on $\mathcal E^c$.
Since $P_i$ is an orthogonal projection,
\[
 0\le p_{rr}(i)\le1
\]
for every $r$.  Therefore,
\[
 \begin{aligned}
 \E\max_{1\le r\le n}p_{rr}(i)
 &=
 \E\left[
 \max_{1\le r\le n}p_{rr}(i)
 \mathbf 1_{\mathcal E^c}
 \right]
 +
 \E\left[
 \max_{1\le r\le n}p_{rr}(i)
 \mathbf 1_{\mathcal E}
 \right]\\
 &\le
 C\ell^{-9a}
 +
 \Pp(\mathcal E)\\
 &=
 C\ell^{-9a}
 +
 o(\ell^{-8a})\\
 &\le
 C\ell^{-8a}.
 \end{aligned}
\]
All the estimates above are uniform over
$n-s_2\le i\le n-1$, and hence this proves
\eqref{ex:eq:terminal-diagonal-max}.  The argument uses only
independence, centering, unit variances, and the common fourth-moment
bound.  Since standard Gaussian entries satisfy the same assumptions
with fourth moment $3\le M$, the estimate remains valid when any
subset of the rows is Gaussian.  This completes the proof of
Lemma~\ref{lem:terminal-diagonal}.
\end{proof}

\begin{proof}[Proof of Lemma~\ref{lem:gaussian-terminal-block}]
Condition on the first $n-r$ rows.  By the rank assumption, their
orthogonal complement $W$ has dimension $r$.  The orthogonal
projections of the last $r$ Gaussian rows onto $W$ are independent
standard Gaussian vectors in $W$.

Consider these projected vectors successively.  Before the first one
is added, the available orthogonal space has dimension $r$, so the
first squared Gram--Schmidt distance has law $\chi_r^2$.  After
$q-1$ projected vectors have been added, their span has dimension
$q-1$ almost surely.  Conditional on the preceding projected vectors,
the orthogonal complement of their span in $W$ has dimension
$r-q+1$.  By rotational invariance, the squared norm of the projection
of the $q$th Gaussian vector onto this complement has law
$\chi_{r-q+1}^2$.  This conditional law does not depend on the
preceding vectors.  Induction therefore shows that the successive
squared distances are independent and have laws
\[
 \chi_r^2,\chi_{r-1}^2,\ldots,\chi_1^2.
\]
Since the corresponding normalizing dimensions are
$r,r-1,\ldots,1$, the normalized squared heights have the asserted
laws.  Their conditional joint law is a fixed product law, independent
of the first $n-r$ rows; hence the heights are also independent of the
preceding block.

For $k\ge1$ and $t>-k/2$,
\begin{equation}\label{eq:chi-square-power-moment}
 \E\exp\left\{
 t\log(\chi_k^2/k)
 \right\}
 =
 \left(\frac2k\right)^t
 \frac{\Gamma(k/2+t)}{\Gamma(k/2)}.
\end{equation}
Differentiating its logarithm at $t=0$ gives
\[
 \E\log(\chi_k^2/k)
 =
 \log(2/k)
 +
 \left.
 \frac{\mathrm d}{\mathrm dx}\log\Gamma(x)
 \right|_{x=k/2}.
\]
The differentiated Stirling expansion
\[
 \frac{\mathrm d}{\mathrm dx}\log\Gamma(x)
 =
 \log x-\frac1{2x}+O(x^{-2})
 \qquad (x\to\infty)
\]
therefore we have
\begin{equation}\label{eq:chi-square-log-mean}
 \E\log(\chi_k^2/k)
 =
 -\frac1k+O(k^{-2}).
\end{equation}
Summing
\eqref{eq:chi-square-log-mean} proves
\eqref{eq:gaussian-terminal-mean}.

Set
\[
 Y_k
 =
 \log(\chi_k^2/k)-\E\log(\chi_k^2/k).
\]
For every integer $j\ge2$,
\begin{equation}\label{eq:log-gamma-higher-derivatives}
 \frac{\mathrm d^j}{\mathrm dx^j}\log\Gamma(x)
 =
 (-1)^j(j-1)!
 \sum_{m=0}^{\infty}(x+m)^{-j}.
\end{equation}
Moreover, centering changes only the linear term in the logarithmic
moment-generating function.  Thus, for $j\ge2$, the $j$th cumulant of
$Y_k$ equals the left-hand side of
\eqref{eq:log-gamma-higher-derivatives} evaluated at $x=k/2$.

For $j=2$, integral comparison gives
\[
 \sum_{m=0}^{\infty}(x+m)^{-2}
 =
 \frac1x+O(x^{-2}),
\]
and hence
\begin{equation}\label{eq:chi-square-log-variance}
 \E Y_k^2
 =
 \frac2k+O(k^{-2}).
\end{equation}
Since the variables $Y_1,\ldots,Y_r$ are independent,
summing \eqref{eq:chi-square-log-variance} proves
\eqref{eq:gaussian-terminal-variance}.

For $j=4$, \eqref{eq:log-gamma-higher-derivatives} gives
\[
 |\kappa_4(Y_k)|
 =
 \left|
 \left.
 \frac{\mathrm d^4}{\mathrm dx^4}\log\Gamma(x)
 \right|_{x=k/2}
 \right|
 \le Ck^{-3}.
\]
Using
\[
 \E Y_k^4
 =
 \kappa_4(Y_k)+3\{\E Y_k^2\}^2
\]
together with \eqref{eq:chi-square-log-variance}, we obtain
\[
 \E Y_k^4\le Ck^{-2}.
\]
Therefore, by Cauchy--Schwarz,
\[
 \E|Y_k|^3
 \le
 (\E Y_k^2)^{1/2}(\E Y_k^4)^{1/2}
 \le Ck^{-3/2}.
\]
Since $\sum_{k=1}^{\infty}k^{-3/2}<\infty$, summing this estimate
proves \eqref{eq:gaussian-terminal-third}. Thus we complete the proof of Lemma~\ref{lem:gaussian-terminal-block}.
\end{proof}

\begin{proof}[Proof of Lemma~\ref{lem:common-one-step}]

We use the notation in
\eqref{eq:common-parameters}--\eqref{eq:common-heights}
and \eqref{eq:truncate-notation-definition}, with
$\alpha$ and $a$ as in the statement of the lemma.
Thus \eqref{eq:truncate-notation-definition} gives

\begin{align}
 \theta&=\ell^{-a/2},
 &H&=|\log\theta|=\frac a2\log\ell,
 \label{eq:common-lower-truncation}\\
 L_{i+1}&=\log(Z_{i+1}\vee\theta),
 &\zeta_{i+1}&=L_{i+1}+\frac1{k_i},
 \qquad 0\le i<m.
 \label{eq:common-truncated-increment}
\end{align}
Also we can write
\begin{equation}\label{eq:common-Z-decompose}
 Z_{i+1}-1=U_{i+1}+V_{i+1},
\end{equation}
where
\begin{equation}
  \begin{aligned}\label{eq:common-UV-setup}
    U_{i+1}
    =
    \sum_rq_{rr}(i)(a_{i+1,r}^2-1),
    \qquad
    V_{i+1}
    =
    \sum_{r\ne s}q_{rs}(i)a_{i+1,r}a_{i+1,s}.    
  \end{aligned}
\end{equation}

Conditionally on $\cF_i$, the current row is independent of $Q_i$.
Lemma~\ref{lem:quadratic-forms} gives
\begin{align}
 \E_iU_{i+1}^2
 &\le CMd_i,
 \label{det:eq:UV-U2}\\
 \E_iV_{i+1}^2
 &=
 2\sum_{r\ne s}q_{rs}(i)^2
 =
 \frac2{k_i}-2d_i,
 \label{det:eq:UV-V2}\\
 \E_iV_{i+1}^4
 &\le CM^2k_i^{-2},
 \label{det:eq:UV-V4}\\
 \E_i|U_{i+1}V_{i+1}|
 &\le
 C\sqrt M\sqrt{\frac{d_i}{k_i}},
 \label{det:eq:UV-UV}\\
 \E_i|V_{i+1}|^3
 &\le
 CM^{3/2}k_i^{-3/2}.
 \label{det:eq:UV-V3}
\end{align}

For all sufficiently large $n$, we have
\[
 \theta<\frac12,\qquad H\ge T_0,\qquad
 3\le M\le C_0\ell^\alpha,\qquad
 k_i\ge s_1+1\asymp\ell^{3a}\gg M.
\]
Fix $0\le i<m$ and introduce the good event
\[
 \mathcal A_i
 =
 \{|U_{i+1}|\le1/4,\ |V_{i+1}|\le1/4\}.
\]
On $\mathcal A_i$,
$
 Z_{i+1}
 =
 1+U_{i+1}+V_{i+1}
 \ge\frac12>\theta$,
and hence
\[
 L_{i+1}
 =
 \log(1+U_{i+1}+V_{i+1}).
\]
For $|x|\le1/2$, Taylor's formula gives
\[
 \left|\log(1+x)-x+\frac{x^2}{2}\right|
 \le C|x|^3,
 \qquad
 |\log(1+x)|\le C|x|.
\]
Applying these estimates with
$x=U_{i+1}+V_{i+1}$ and using
$|U_{i+1}|,|V_{i+1}|\le1/4$, we obtain on $\mathcal A_i$
\begin{align}
 \left|
 L_{i+1}-U_{i+1}-V_{i+1}
 +\frac12V_{i+1}^2
 \right|
 &\le
 C\left\{
 U_{i+1}^2
 +|U_{i+1}V_{i+1}|
 +|V_{i+1}|^3
 \right\},
 \label{det:eq:good-log-first}\\
 \left|L_{i+1}^2-V_{i+1}^2\right|
 &\le
 C\left\{
 U_{i+1}^2
 +|U_{i+1}V_{i+1}|
 +|V_{i+1}|^3
 \right\},
 \label{det:eq:good-log-second}\\
 |L_{i+1}|^3
 &\le
 C\left\{U_{i+1}^2+|V_{i+1}|^3\right\},
 \label{det:eq:good-log-third}\\
 |L_{i+1}|^4
 &\le
 C\left\{U_{i+1}^2+V_{i+1}^4\right\}.
 \label{det:eq:good-log-fourth}
\end{align}
For example, the first estimate follows from
\[
 -\frac12(U_{i+1}+V_{i+1})^2+\frac12V_{i+1}^2
 =
 -\frac12U_{i+1}^2-U_{i+1}V_{i+1},
\]
while
\[
 |U_{i+1}+V_{i+1}|^3
 \le
 C\bigl(|U_{i+1}|^3+|V_{i+1}|^3\bigr)
 \le
 C\bigl(U_{i+1}^2+|V_{i+1}|^3\bigr)
\]
on $\mathcal A_i$.  The second estimate follows by writing
\[
 L_{i+1}^2-V_{i+1}^2
 =
 \bigl\{L_{i+1}^2-(U_{i+1}+V_{i+1})^2\bigr\}
 +U_{i+1}^2+2U_{i+1}V_{i+1}.
\]

Using
\eqref{det:eq:UV-U2}--\eqref{det:eq:UV-V3}, the conditional
expectation of the right-hand side of each of
\eqref{det:eq:good-log-first}--\eqref{det:eq:good-log-fourth}
is bounded by $CH_i^*$.

We next control the complement of $\mathcal A_i$.  Markov's
inequality, \eqref{det:eq:UV-U2}, and
\eqref{det:eq:UV-V4} give
\begin{equation}\label{det:eq:bad-event-growing}
 \Pp_i(\mathcal A_i^c)
 \le
 C\left\{Md_i+M^2k_i^{-2}\right\}.
\end{equation}
Let $L_{i+1}^{+}=\max\{L_{i+1},0\}$ and
$L_{i+1}^{-}=\max\{-L_{i+1},0\}$.  Since $\theta<1$ and
$Z_{i+1}\le1+|U_{i+1}|+|V_{i+1}|$,
\[
 L_{i+1}^{+}
 \le
 \log\bigl(1+|U_{i+1}|+|V_{i+1}|\bigr).
\]
On $\mathcal A_i^c$,
$|U_{i+1}|+|V_{i+1}|\ge1/4$, and therefore, for every
$1\le r\le4$,
\[
 \left\{
 \log\bigl(1+|U_{i+1}|+|V_{i+1}|\bigr)
 \right\}^{r}
 \le
 C_r\bigl(|U_{i+1}|+|V_{i+1}|\bigr)^2.
\]
Moreover, on $\mathcal A_i^c$,
\[
 V_{i+1}^2
 \le
 U_{i+1}^2+16V_{i+1}^4.
\]
Indeed, if $|V_{i+1}|\ge1/4$, then
$V_{i+1}^2\le16V_{i+1}^4$; otherwise
$|U_{i+1}|>1/4$ and $V_{i+1}^2\le U_{i+1}^2$.
Consequently,
\begin{equation}\label{det:eq:bad-positive-growing}
 \E_i\left[
 (L_{i+1}^{+})^r\1_{\mathcal A_i^c}
 \right]
 \le
 C\left\{Md_i+M^2k_i^{-2}\right\},
 \qquad 1\le r\le4.
\end{equation}

For the negative part, note that on $\mathcal A_i$,
$L_{i+1}\ge-\log2$.  Since $T_0>1>\log2$, the event
$\{L_{i+1}^{-}>T_0\}$ is contained in $\mathcal A_i^c$.
The layer-cake identity therefore gives, for $1\le r\le4$,
\begin{align*}
 \E_i\left[
 (L_{i+1}^{-})^r\1_{\mathcal A_i^c}
 \right]
 &\le
 C\Pp_i(\mathcal A_i^c)
 +r\int_{T_0}^{H}
 u^{r-1}\Pp_i(L_{i+1}^{-}>u)\dd u.
\end{align*}
For $T_0\le u<H$, one has
$\theta=e^{-H}<e^{-u}$, and hence
\[
 \{L_{i+1}^{-}>u\}
 =
 \{L_{i+1}<-u\}
 =
 \{Z_{i+1}<e^{-u}\}.
\]
Lemma~\ref{lem:common-lower-tail} and
\eqref{det:eq:bad-event-growing} thus imply
\begin{equation}\label{det:eq:bad-negative-growing}
 \E_i\left[
 (L_{i+1}^{-})^r\1_{\mathcal A_i^c}
 \right]
 \le
 C\left\{
 Md_i+M^2k_i^{-2}+\lambda_i^{(r)}
 \right\}.
\end{equation}
Combining
\eqref{det:eq:bad-positive-growing} and
\eqref{det:eq:bad-negative-growing}, we obtain
\begin{equation}\label{det:eq:bad-log-growing}
 \E_i\left[
 |L_{i+1}|^r\1_{\mathcal A_i^c}
 \right]
 \le
 C\left\{
 Md_i+M^2k_i^{-2}+\lambda_i^{(r)}
 \right\}
 \le CH_i^*,
 \qquad 1\le r\le4.
\end{equation}

We also need the contributions of $U_{i+1}$,
$V_{i+1}$, and $V_{i+1}^2$ on $\mathcal A_i^c$.
By Cauchy--Schwarz and \eqref{det:eq:bad-event-growing},
\begin{align}
 \E_i\left[
 |U_{i+1}|\1_{\mathcal A_i^c}
 \right]
 &\le
 \{\E_iU_{i+1}^2\}^{1/2}
 \{\Pp_i(\mathcal A_i^c)\}^{1/2}
 \le
 C\left\{Md_i+M^2k_i^{-2}\right\},
 \label{det:eq:bad-U-growing}\\
 \E_i\left[
 |V_{i+1}|\1_{\mathcal A_i^c}
 \right]
 &\le
 \{\E_iV_{i+1}^2\}^{1/2}
 \{\Pp_i(\mathcal A_i^c)\}^{1/2}
 \le
 C\left\{
 \sqrt M\sqrt{d_i/k_i}
 +M^{3/2}k_i^{-3/2}
 \right\}.
 \label{det:eq:bad-V-growing}
\end{align}
For the last inequality we used
$\E_iV_{i+1}^2\le2/k_i$ and $M\ge1$.
Similarly, Hölder's inequality gives
\begin{align}
 \E_i\left[
 V_{i+1}^2\1_{\mathcal A_i^c}
 \right]
 &\le
 \{\E_iV_{i+1}^4\}^{1/2}
 \{\Pp_i(\mathcal A_i^c)\}^{1/2}
 \notag\\
 &\le
 \frac{CM}{k_i}
 \left\{Md_i+M^2k_i^{-2}\right\}^{1/2}
 \notag\\
 &\le
 C\left\{Md_i+M^2k_i^{-2}\right\},
 \label{det:eq:bad-V2-growing}
\end{align}
where the last step follows from $2xy\le x^2+y^2$.

Since
$\E_iU_{i+1}=\E_iV_{i+1}=0$, we may combine
\eqref{det:eq:good-log-first}--\eqref{det:eq:good-log-fourth}
with
\eqref{det:eq:bad-log-growing}--\eqref{det:eq:bad-V2-growing}.
This yields
\begin{align}
 \left|
 \E_iL_{i+1}
 +\frac12\E_iV_{i+1}^2
 \right|
 &\le CH_i^*,
 \label{det:eq:L-mean-growing}\\
 \left|
 \E_iL_{i+1}^2-\E_iV_{i+1}^2
 \right|
 &\le CH_i^*,
 \label{det:eq:L-second-growing}\\
 \left|\E_i(L_{i+1}^3)\right|
 &\le CH_i^*,
 \label{det:eq:L-third-growing}\\
 \E_i|L_{i+1}|^4
 &\le CH_i^*.
 \label{det:eq:L-fourth-growing}
\end{align}

We now pass from $L_{i+1}$ to
$\zeta_{i+1}=L_{i+1}+k_i^{-1}$.
By \eqref{det:eq:UV-V2},
\[
 \frac12\E_iV_{i+1}^2
 =
 \frac1{k_i}-d_i.
\]
Therefore
\[
 \E_i\zeta_{i+1}
 =
 \left(
 \E_iL_{i+1}
 +\frac12\E_iV_{i+1}^2
 \right)
 +d_i.
\]
Since $M\ge1$, the term $d_i$ is bounded by the first term
$Md_i$ in $H_i^*$, and hence
\begin{equation}\label{det:eq:conditional-mean-growing}
 |\E_i\zeta_{i+1}|
 \le CH_i^*.
\end{equation}

For the second moment, direct expansion gives
\begin{align*}
 \E_i\zeta_{i+1}^2-\frac2{k_i}
 &=
 \left(
 \E_iL_{i+1}^2-\E_iV_{i+1}^2
 \right)
 +
 \frac2{k_i}
 \left(
 \E_iL_{i+1}
 +\frac12\E_iV_{i+1}^2
 \right)
 -2\left(1-\frac1{k_i}\right)d_i
 -\frac1{k_i^2}.
\end{align*}
Each term on the right-hand side is bounded by $CH_i^*$:
the first two by
\eqref{det:eq:L-mean-growing}--\eqref{det:eq:L-second-growing},
the third by $Md_i\le H_i^*$, and the last by
$k_i^{-2}\le M^2k_i^{-2}\le H_i^*$.  Thus
\begin{equation}\label{det:eq:conditional-second-growing}
 \left|
 \E_i\zeta_{i+1}^2-\frac2{k_i}
 \right|
 \le CH_i^*.
\end{equation}

To treat the third moment, first note from
\eqref{det:eq:L-mean-growing} and
\eqref{det:eq:L-second-growing} that
\[
 |\E_iL_{i+1}|
 \le
 \frac1{k_i}+CH_i^*,
 \qquad
 \E_iL_{i+1}^2
 \le
 \frac2{k_i}+CH_i^*.
\]
Using
\[
 \zeta_{i+1}^3
 =
 L_{i+1}^3
 +\frac3{k_i}L_{i+1}^2
 +\frac3{k_i^2}L_{i+1}
 +\frac1{k_i^3},
\]
we obtain
\begin{equation}
  \begin{aligned}\label{det:eq:conditional-third-growing}
  \left|\E_i(\zeta_{i+1}^3)\right|
  &\le
  \left|\E_i(L_{i+1}^3)\right|
  +\frac3{k_i}\E_iL_{i+1}^2
  +\frac3{k_i^2}|\E_iL_{i+1}|
  +\frac1{k_i^3}
  \le
  C\left\{H_i^*+k_i^{-2}\right\}
  \le CH_i^*,    
  \end{aligned}
\end{equation}
because $H_i^*\ge M^2k_i^{-2}\ge k_i^{-2}$.
Likewise,
$
 |\zeta_{i+1}|^4
 \le
 C|L_{i+1}|^4+Ck_i^{-4}$,
and therefore
\begin{equation}\label{det:eq:conditional-fourth-growing}
 \E_i|\zeta_{i+1}|^4
 \le CH_i^*.
\end{equation}
Adding the preceding four estimates \eqref{det:eq:conditional-mean-growing},
\eqref{det:eq:conditional-second-growing}, \eqref{det:eq:conditional-third-growing} and \eqref{det:eq:conditional-fourth-growing} proves
\eqref{eq:common-one-step-conditional}.

Taking expectations in in these four estimates gives proves
\eqref{eq:common-one-step-mean},
\eqref{eq:common-one-step-third},
\eqref{eq:common-one-step-fourth}, and
\eqref{eq:common-one-step-second-mean}.  Moreover,
\begin{align*}
 \E\left|
 \E_i(\zeta_{i+1}^2)-\E\zeta_{i+1}^2
 \right|
 &\le
 \E\left|
 \E_i(\zeta_{i+1}^2)-\frac2{k_i}
 \right|
 +
 \left|
 \E\zeta_{i+1}^2-\frac2{k_i}
 \right|
 \le Ch_i,
\end{align*}
which proves \eqref{eq:common-one-step-second}.

It remains to estimate the sums of $h_i$.  By the definition of
$D_i$ and $R_i$,
\[
 h_i
 =
 MD_i+\sqrt M\,R_i
 +M^{3/2}k_i^{-3/2}
 +M^2k_i^{-2}
 +\sum_{r=1}^4\lambda_i^{(r)}.
\]

Thus, to estimate the sums of $h_i$, it is enough to estimate the sums of $D_i$ and $R_i$.  This is done in Lemma~\ref{lem:diagonal-projection-bounds} below.  The remaining terms are easily summed by using Lemma~\ref{lem:common-lower-tail}.

\begin{lemma}\label{lem:diagonal-projection-bounds}
Let $\alpha\ge0$, $a\ge20$, and $C_0>0$ be fixed, with
$a>3\alpha+12$.
Let $A_n$ be an $n\times n$ random matrix with independent
centered real-valued entries of variance one and fourth moments
at most $M$, where $3\le M\le C_0\ell^\alpha$.
Assume that all square submatrices of $A_n$ are invertible
almost surely.
The estimates below hold for all sufficiently large $n$,
with $C$ depending only on $C_0,\alpha,a$.

Retain the notations of \eqref{eq:common-parameters}--\eqref{eq:common-heights}.
Define
\[
 D_i=\E d_i,
 \qquad R_i=\E\sqrt{d_i/k_i}.
\]
Then, with the parameters in \eqref{eq:common-parameters},
\begin{align}
 k_i\ge n/2:
 &\quad D_i\le k_i^{-1},\quad R_i\le k_i^{-1},\label{ex:eq:projection-bound-early}\\
 s_2\le k_i<n/2:
 &\quad D_i\le \frac Cn+\frac{CM}{k_i n^{5/6}},
 \quad R_i\le C(nk_i)^{-1/2}+\frac{C\sqrt M}{k_i n^{5/12}},\label{ex:eq:projection-bound-middle}\\
 s_1\le k_i<s_2:
 &\quad D_i\le C\ell^{-8a}k_i^{-1},
 \quad R_i\le C\ell^{-4a}k_i^{-1}.
 \label{ex:eq:projection-bound-terminal}
\end{align}
In particular,
\begin{equation}\label{ex:eq:projection-bound-sums}
 \sum_{i=0}^{m-1}D_i+\sum_{i=0}^{m-1}R_i\le C.
\end{equation}
\end{lemma}

\begin{proof}[Proof of Lemma~\ref{lem:diagonal-projection-bounds}]
Since $P_i$ is an orthogonal projection of rank $k_i$, we have
$0\le p_{rr}(i)\le1$, $\sum_{r=1}^n p_{rr}(i)=\tr P_i=k_i$.
Moreover, $q_{rr}(i)=k_i^{-1}p_{rr}(i)$, and hence $d_i=\frac1{k_i^2}\sum_{r=1}^n p_{rr}(i)^2$.

In the early range $k_i\ge n/2$.
Since $p_{rr}(i)^2\le p_{rr}(i)$,
$d_i \le \frac1{k_i^2}\sum_{r=1}^n p_{rr}(i) = \frac1{k_i}$.
It follows that $D_i\le k_i^{-1}$ and, pointwise, $\sqrt{\frac{d_i}{k_i}}\le\frac1{k_i}$. Thus $R_i\le k_i^{-1}$, proving
\eqref{ex:eq:projection-bound-early}.

In the middle range $s_2\le k_i<n/2$.
Let $A$ be the matrix formed by the first $i$ rows, let $b_r$ be its
$r$th column, and let $A_r$ be obtained by deleting $b_r$.  Put
\[
 G_r=(n^{-1}A_rA_r^T+\lambda I_i)^{-1},
 \quad T_r=n^{-1}\tr G_r,
 \qquad
 G=(n^{-1}AA^T+\lambda I_i)^{-1},
 \quad T=n^{-1}\tr G.
\]
Since $AA^T=A_rA_r^T+b_rb_r^T$, Loewner monotonicity gives
$G\preceq G_r$ and therefore $T\le T_r$.

We first obtain a lower-tail estimate for $T_r$.  Write
$y=i/n=1-k_i/n$.  Since $k_i<n/2$, we have $y>1/2$.  Also, from
$s_2=\lfloor n\ell^{-20a}\rfloor$, $\frac{k_i}{n}\ge\frac{s_2}{n}\ge\frac12\ell^{-20a}$
for all sufficiently large $n$.  Since $\sqrt\lambda=n^{-1/12}$, this
implies
\[
 \sqrt\lambda=o\!\left(\frac{k_i}{n}\right).
\]
Using $1+\lambda-y=k_i/n+\lambda$ in the explicit formula for
$s_y(\lambda)$, we therefore have
\[
 \begin{aligned}
  1+\lambda-y+
  \sqrt{(1+\lambda-y)^2+4y\lambda}
  &=
  \frac{k_i}{n}+\lambda+
  \sqrt{\left(\frac{k_i}{n}+\lambda\right)^2+4y\lambda} 
  \le C\frac{k_i}{n}.
 \end{aligned}
\]
Since $y>1/2$, it follows that
\[
 y s_y(\lambda)\ge c\frac{n}{k_i}.
\]
The matrix $A$ has $i>n/2$ rows, so
Lemma~\ref{lem:regularized-inverse} applies directly to $A$.  It gives
\[
 \left|\E T-y s_y(\lambda)\right|
 \le C\{n^{-2/3}+Mn^{-1/2}\}=o(1),
 \qquad
 \Var(T)\le Cn^{-2/3},
\]
where we used $M\le C_0\ell^\alpha$ with fixed $\alpha$.
Since $n/k_i>2$, the preceding lower bound
for $y s_y(\lambda)$ implies that, for some constant $c_1>0$,
\[
 \E T\ge c_1\frac{n}{k_i}.
\]
Choose $c_0=c_1/2$.  By Chebyshev's inequality and $T_r\ge T$,
\[
 \begin{aligned}
 \Pp\left(T_r<c_0\frac{n}{k_i}\right)
 &\le
 \Pp\left(T<c_0\frac{n}{k_i}\right)  
 \le
 \frac{\Var(T)}
 {\left(\E T-c_0n/k_i\right)^2}
 \le C\frac{k_i^2}{n^{8/3}}.
 \end{aligned}
\]

We next control $n^{-1}b_r^TG_rb_r$.  Conditionally on $A_r$, the
matrix $G_r$ is deterministic and $b_r$ has independent centered
coordinates of variance one.  Hence
\[
 \E\left(n^{-1}b_r^TG_rb_r\mid A_r\right)=T_r.
\]
By Lemma~\ref{lem:quadratic-forms} \eqref{eq:var-part},
\[
 \Var\left(n^{-1}b_r^TG_rb_r\mid A_r\right)
 \le CMn^{-2}\tr G_r^2.
\]
Since every eigenvalue of $G_r$ is at most $\lambda^{-1}$, $\tr G_r^2\le\lambda^{-1}\tr G_r =n\lambda^{-1}T_r$,
and therefore
\[
 \Var\left(n^{-1}b_r^TG_rb_r\mid A_r\right)
 \le CM\frac{\lambda^{-1}T_r}{n}.
\]
On the event $T_r\ge c_0n/k_i$, conditional Chebyshev gives
\[
 \begin{aligned}
 &\Pp\left(
 n^{-1}b_r^TG_rb_r<\frac12T_r
 \,\middle|\,A_r
 \right)  
 \le
 \frac{4\Var(n^{-1}b_r^TG_rb_r\mid A_r)}{T_r^2}
 \le
 \frac{CM\lambda^{-1}}{nT_r}
 \le
 CM\frac{k_i}{n^{11/6}},
 \end{aligned}
\]
because $\lambda^{-1}=n^{1/6}$.  Averaging over $A_r$ yields
\[
 \Pp\left(
 n^{-1}b_r^TG_rb_r<\frac12T_r,\,
 T_r\ge c_0\frac{n}{k_i}
 \right)
 \le CM\frac{k_i}{n^{11/6}}.
\]
Outside the two exceptional events just estimated,
\[
 n^{-1}b_r^TG_rb_r
 \ge \frac12T_r
 \ge \frac{c_0}{2}\frac{n}{k_i}.
\]
Thus, by \eqref{ex:eq:regularized-diagonal-projection},
$p_{rr}(i) \le \left(1+n^{-1}b_r^TG_rb_r\right)^{-1} \le C\frac{k_i}{n}$.
Using the trivial bound $p_{rr}(i)^2\le1$ on the exceptional events,
we obtain
\[
 \E p_{rr}(i)^2
 \le
 C\left\{
 \frac{k_i^2}{n^{8/3}}
 +\frac{Mk_i}{n^{11/6}}
 +\left(\frac{k_i}{n}\right)^2
 \right\}.
\]
Consequently,
\[
 \begin{aligned}
 D_i
 &=\frac1{k_i^2}\sum_{r=1}^n\E p_{rr}(i)^2
 \le
 C\left\{
 n^{-5/3}
 +\frac{M}{k_i n^{5/6}}
 +\frac1n
 \right\}
 \le
 \frac Cn+\frac{CM}{k_i n^{5/6}}.
 \end{aligned}
\]
Finally, Jensen's inequality gives
\[
 \begin{aligned}
 R_i
 &=\E\sqrt{\frac{d_i}{k_i}}
 \le\sqrt{\frac{\E d_i}{k_i}}
 =\sqrt{\frac{D_i}{k_i}}
 \le
 C(nk_i)^{-1/2}
 +\frac{C\sqrt M}{k_i n^{5/12}}.
 \end{aligned}
\]
This proves \eqref{ex:eq:projection-bound-middle}.

In the terminal range $s_1\le k_i<s_2$.
Using $\sum_rp_{rr}(i)=k_i$, we have
\[
 d_i
 \le
 \frac{\max_rp_{rr}(i)}{k_i^2}
 \sum_{r=1}^n p_{rr}(i)
 =
 \frac1{k_i}\max_rp_{rr}(i).
\]
Lemma~\ref{lem:terminal-diagonal} therefore gives
$
 D_i
 \le
 \frac1{k_i}\E\max_rp_{rr}(i)
 \le
 C\ell^{-8a}k_i^{-1}.
$
Moreover,
\[
 \sqrt{\frac{d_i}{k_i}}
 \le
 \frac1{k_i}\sqrt{\max_rp_{rr}(i)},
\]
so Jensen's inequality gives
\[
 R_i
 \le
 \frac1{k_i}\sqrt{\E\max_rp_{rr}(i)}
 \le
 C\ell^{-4a}k_i^{-1}.
\]
This proves \eqref{ex:eq:projection-bound-terminal}.

It remains to sum the estimates.  As $i$ ranges from $0$ to $m-1$,
the numbers $k_i$ range through the integers from $n$ down to
$s_1+1$.  The early-range contribution is bounded by
$
 2\sum_{n/2\le k\le n}\frac1k=O(1).
$
For the middle-range $D_i$ terms,
\[
 \sum_{s_2\le k<n/2}
 \left(\frac Cn+\frac{CM}{kn^{5/6}}\right)
 \le
 C+CMn^{-5/6}\log\frac{n}{s_2}.
\]
Since $M\le C_0\ell^\alpha$ and $n/s_2\le C\ell^{20a}$,
\[
 Mn^{-5/6}\log\frac{n}{s_2}
 \le Cn^{-5/6}\ell^\alpha\log\ell=o(1).
\]
Similarly, the middle-range $R_i$ terms satisfy
\[
 \begin{aligned}
 &\sum_{s_2\le k<n/2}
 \left\{C(nk)^{-1/2}
 +\frac{C\sqrt M}{kn^{5/12}}\right\}
 \le
 C+C\sqrt M\,n^{-5/12}\log\frac{n}{s_2}
 \le
 C+Cn^{-5/12}\ell^{\alpha/2}\log\ell
 =O(1).
 \end{aligned}
\]

Finally, the terminal-range contributions are bounded by
\[
 C\ell^{-8a}\sum_{s_1\le k<s_2}\frac1k
 \le C\ell^{1-8a},
 \qquad
 C\ell^{-4a}\sum_{s_1\le k<s_2}\frac1k
 \le C\ell^{1-4a}.
\]
Since $a\ge20$, both are uniformly bounded.
Combining the three ranges
proves
\[
 \sum_{i=0}^{m-1}D_i+\sum_{i=0}^{m-1}R_i\le C,
\]
and completes the proof of Lemma~\ref{lem:diagonal-projection-bounds}.
\end{proof}

We now return to the proof of Lemma~\ref{lem:common-one-step}.
Lemma~\ref{lem:diagonal-projection-bounds} gives
\[
 \sum_{i=0}^{m-1}D_i
 +
 \sum_{i=0}^{m-1}R_i
 \le C.
\]
Since the integers $k_i$ range from $n$ down to $s_1+1$,
\[
 \sum_{i=0}^{m-1}k_i^{-3/2}
 \le Cs_1^{-1/2},
 \qquad
 \sum_{i=0}^{m-1}k_i^{-2}
 \le Cs_1^{-1}.
\]
Together with \eqref{det:eq:lambda-growing-total}, this gives
\begin{align*}
 \sum_{i=0}^{m-1}h_i
 &\le
 C\left\{
 M+\sqrt M
 +M^{3/2}s_1^{-1/2}
 +M^2s_1^{-1}
 \right\}\\
 &\quad+
 C(\log\ell)^4
 \left\{
 M^2\ell^{-a} 
 +M\ell^{-2a}
 +M^2\ell^{-3a}
 \right\}.
\end{align*}
Since $3\le M\le C_0\ell^\alpha$,
$s_1\asymp\ell^{3a}$, and $a>3\alpha+12$,
all terms on the right-hand side are bounded by $CM$.
Thus
$
 \sum_{i=0}^{m-1}h_i\le CM.
$

We next prove the pointwise estimate.  Since $P_i$ is an
orthogonal projection of rank $k_i$,
\[
 0\le p_{rr}(i)\le1,
 \qquad
 \sum_{r=1}^np_{rr}(i)=k_i.
\]
Consequently,
\[
 d_i
 =
 \frac1{k_i^2}\sum_{r=1}^np_{rr}(i)^2
 \le\frac1{k_i},
 \qquad
 \sqrt{\frac{d_i}{k_i}}\le\frac1{k_i}.
\]
It follows that
\begin{equation}
  \begin{aligned}\label{eq:point-bound-term1}
    D_i\le\frac1{k_i},
    \qquad
    R_i\le\frac1{k_i}.
  \end{aligned}
\end{equation}
Moreover, $k_i\ge s_1\gg M$, so
\begin{equation}
  \begin{aligned}\label{eq:point-bound-term2}
    M^{3/2}k_i^{-3/2}
    +M^2k_i^{-2}
    \le \frac{CM}{k_i}.    
  \end{aligned}
\end{equation}
Finally, since $H=(a/2)\log\ell$ and $e^{4H}=\ell^{2a}$,
the definition of $\lambda_i^{(r)}$ gives
\begin{align*}
 \sum_{r=1}^4\lambda_i^{(r)}
 &\le
 C(\log\ell)^4
 \left\{
 M^2\ell^{2a}k_i^{-2}
 +M\ell^{4a}k_i^{-3}
 +M^2k_i^{-2}
 \right\}
 \le
 \frac{CM}{k_i},
\end{align*}
where the last inequality follows from
$k_i\ge s_1\asymp\ell^{3a}$,
$M\le C_0\ell^\alpha$, and $a>3\alpha+12$.
Together with \eqref{eq:point-bound-term1} and \eqref{eq:point-bound-term2}, we have
\[
 h_i\le\frac{CM}{k_i},
 \qquad 0\le i<m.
\]

For the terminal range $s_1\le k_i<s_2$,
Lemma~\ref{lem:diagonal-projection-bounds} gives
\[
 D_i\le C\ell^{-8a}k_i^{-1},
 \qquad
 R_i\le C\ell^{-4a}k_i^{-1}.
\]
Since $\sum_{s_1\le k_i<s_2}\frac1{k_i}\le C\ell$
and $s_1\asymp\ell^{3a}$, we obtain
\begin{align*}
 \sum_{s_1\le k_i<s_2}h_i
 &\le
 CM\ell^{1-8a}
 +C\sqrt M\,\ell^{1-4a}
 +CM^{3/2}\ell^{-3a/2}
 +CM^2\ell^{-3a}\\
 &\quad+
 C(\log\ell)^4
 \left\{
 M^2\ell^{-a}
 +M\ell^{-2a}
 +M^2\ell^{-3a}
 \right\}.
\end{align*}
Since $3\le M\le C_0\ell^\alpha$, $\alpha\ge0$, and $a\ge20$,
the preceding bound yields
\[
 \begin{aligned}
 \sum_{s_1\le k_i<s_2}h_i
 &\le
 CM\ell^{\alpha-a}\bigl\{1+(\log\ell)^4\bigr\}
 \le
 CM\ell^{-(2\alpha+12)}
 \end{aligned}
\]
for all sufficiently large $n$, where the last inequality follows
from $a-3\alpha-12>0$.
This proves \eqref{eq:common-h-terminal} and completes the proof
of Lemma~\ref{lem:common-one-step}.

\end{proof}

\begin{proof}[Proof of Lemma~\ref{lem:common-lower-tail}]
Fix $0\le i<m$.  Since $Q_i=k_i^{-1}P_i$ and $P_i$ is an orthogonal projection, $Q_i\succeq0$. Its diagonal entries are
nonnegative. From
\[
 Z_{i+1}
 =
 \sum_rq_{rr}(i)b_{i+1,r}^2+V_{i+1},
\]
we have
\[
 \{Z_{i+1}<e^{-u}\}
 \subseteq
 \left\{
 \sum_rq_{rr}(i)b_{i+1,r}^2<2e^{-u}
 \right\}
 \cup
 \{|V_{i+1}|>e^{-u}\}.
\]
By Markov's inequality and
\eqref{det:eq:UV-V4},
\begin{equation}\label{det:eq:lower-tail-offdiag}
 \Pp_i(|V_{i+1}|>e^{-u})
 \le
 e^{4u}\E_i|V_{i+1}|^4
 \le
 CM^2e^{4u}k_i^{-2}.
\end{equation}
For the diagonal part, let
\[
 \widehat b_r
 =
 b_{i+1,r}\1_{\{|b_{i+1,r}|\le\ell^a\}},
 \qquad
 \widetilde b_r
 =
 \frac{\widehat b_r-\E\widehat b_r}
 {\sqrt{\Var(\widehat b_r)}}.
\]
As in \eqref{ex:eq:internal-trunc}--\eqref{ex:eq:internal-8},
\begin{equation}\label{det:eq:internal-current-row}
 |\E\widehat b_r|\le M\ell^{-3a},
 \qquad
 \Var(\widehat b_r)\ge\frac12,
 \qquad
 \E|\widetilde b_r|^4\le CM,
 \qquad
 \E|\widetilde b_r|^8\le CM\ell^{4a}.
\end{equation}
The first two estimates in
\eqref{det:eq:internal-current-row} imply
\[
 \widetilde b_r^{\,2}
 \le
 4\widehat b_r^{\,2}+CM^2\ell^{-6a}.
\]
Since $\widehat b_r^{\,2}\le b_{i+1,r}^2$ and
$\sum_rq_{rr}(i)=\tr Q_i=1$, the event
\[
 \sum_rq_{rr}(i)b_{i+1,r}^2<2e^{-u}
\]
implies
\[
 \sum_rq_{rr}(i)\widetilde b_r^{\,2}
 \le
 8e^{-u}+CM^2\ell^{-6a}.
\]
Since $u\le H$, we have
$e^{-u}\ge e^{-H}=\ell^{-a/2}$.
Moreover, $M\le C_0\ell^\alpha$ and $a>3\alpha+12$ give
\[
 M^2\ell^{-6a}
 \le C_0^2\ell^{2\alpha-6a}
 =o(\ell^{-a/2}).
\]
Consequently, for all sufficiently large $n$,
\[
 \sum_rq_{rr}(i)\widetilde b_r^{\,2}
 \le 9e^{-u}.
\]
Choose the absolute constant $T_0$ in
\eqref{det:eq:lambda-growing-origin} sufficiently large that
$9e^{-T_0}\le1/2$.
Then, for $T_0\le u\le H$,
\[
 \sum_rq_{rr}(i)b_{i+1,r}^2<2e^{-u}
 \quad\Longrightarrow\quad
 \sum_rq_{rr}(i)\widetilde b_r^{\,2}\le\frac12.
\]
Since $\sum_rq_{rr}(i)=1$, this further implies
\[
 \left|
 \sum_rq_{rr}(i)(\widetilde b_r^{\,2}-1)
 \right|
 \ge\frac12.
\]

Conditionally on $\cF_i$, the variables $\widetilde b_r$ are
independent, centered, and have variance one, while $Q_i$ is a
deterministic positive-semidefinite matrix.  Since
$Q_i=k_i^{-1}P_i$ and $P_i$ is a rank-$k_i$ orthogonal projection,
\[
 \tr Q_i^2=k_i^{-1},
 \qquad
 \tr Q_i^4=k_i^{-3}.
\]
Lemma~\ref{lem:quadratic-forms} \eqref{eq:diag-part},
\eqref{det:eq:internal-current-row}, and Markov's inequality therefore
give
\begin{align}
 \Pp_i\left(
 \sum_rq_{rr}(i)b_{i+1,r}^2<2e^{-u}
 \right)
 &\le
 \Pp_i\left(
 \left|
 \sum_rq_{rr}(i)(\widetilde b_r^{\,2}-1)
 \right|\ge\frac12
 \right) \notag\\
 &\le
 C\left\{
 M\ell^{4a}\tr Q_i^4
 +
 M^2(\tr Q_i^2)^2
 \right\} \notag\\
 &\le
 C\left\{
 M\ell^{4a}k_i^{-3}
 +
 M^2k_i^{-2}
 \right\}.
 \label{det:eq:lower-tail-diagonal}
\end{align}
Combining
\eqref{det:eq:lower-tail-offdiag} and
\eqref{det:eq:lower-tail-diagonal} proves
\eqref{det:eq:lower-tail-Z}.

As $i$ ranges from $0$ to $m-1$, the integers $k_i$ range from $n$
down to $s_1+1$.  Since $s_1=\lfloor\ell^{3a}\rfloor$,
\[
 \sum_{i=0}^{m-1}k_i^{-2}\le C\ell^{-3a},
 \qquad
 \sum_{i=0}^{m-1}k_i^{-3}\le C\ell^{-6a}.
\]
Furthermore, $H=(a/2)\log\ell$, $e^{4H}=\ell^{2a}$, and, for
$1\le r\le4$,
\[
 \int_{T_0}^Hu^{r-1}e^{4u}\dd u
 \le CH^4e^{4H},
 \qquad
 \int_{T_0}^Hu^{r-1}\dd u
 \le CH^4.
\]
Consequently,
\[
 \begin{aligned}
 \sum_{i=0}^{m-1}\sum_{r=1}^4\lambda_i^{(r)}
 &\le
 CH^4\left\{
 M^2e^{4H}\ell^{-3a}
 +
 M\ell^{4a}\ell^{-6a}
 +
 M^2\ell^{-3a}
 \right\}\\
 &\le
 C(\log\ell)^4
 \left\{
 M^2\ell^{-a}
 +
 M\ell^{-2a}
 +
 M^2\ell^{-3a}
 \right\},
 \end{aligned}
\]
which proves \eqref{det:eq:lambda-growing-total}.

Finally,
\[
 L_{i+1}\ne\log Z_{i+1}
 \quad\Longleftrightarrow\quad
 Z_{i+1}<\theta=e^{-H}.
\]
For all sufficiently large $n$, $H\ge T_0$.  Hence, by the union bound,
the tower property, and \eqref{det:eq:lower-tail-Z} with $u=H$,
\[
 \begin{aligned}
 \Pp\left\{
 L_{i+1}\ne\log Z_{i+1}
 \text{ for some }0\le i<m
 \right\}
 &\le
 \sum_{i=0}^{m-1}
 \E\Pp_i(Z_{i+1}<e^{-H})\\
 &\le
 C\left\{
 M^2e^{4H}\ell^{-3a}
 +
 M\ell^{4a}\ell^{-6a}
 +
 M^2\ell^{-3a}
 \right\}\\
 &\le
 C\left\{
 M^2\ell^{-a}
 +
 M\ell^{-2a}
 +
 M^2\ell^{-3a}
 \right\}.
 \end{aligned}
\]
This proves
\eqref{det:eq:lower-truncation-growing-error}. Thus we complete the proof of Lemma~\ref{lem:common-lower-tail}.
\end{proof}

\begin{proof}[Proof of Lemma~\ref{lem:exact-lower-truncation}]
Fix $0\le i<m$ and condition on $\cF_i$.  Then $P_i$ is a
deterministic orthogonal projection of rank $k_i=n-i$, while the
coordinates of the current row $a_{i+1}$ remain independent and have
densities bounded by the common constant $M_1$.  Since
$Q_i=k_i^{-1}P_i$,
\[
 Z_{i+1}
 =
 a_{i+1}^TQ_i a_{i+1}
 =
 \frac{a_{i+1}^TP_i a_{i+1}}{k_i}.
\]
Let $u_{M_1}$ be the constant in
Lemma~\ref{lem:projection-density}.  Since
\[
 H=\frac a2\log\ell\longrightarrow\infty,
\]
we have $H\ge u_{M_1}$ for all sufficiently large $n$.  Therefore,
for every $u\ge H$,
\begin{equation}\label{ex:eq:exact-lower-tail-conditional}
 \Pp_i(Z_{i+1}\le e^{-u})
 \le e^{-k_i u/4}.
\end{equation}
In particular, letting $u\to\infty$ in
\eqref{ex:eq:exact-lower-tail-conditional} shows that
$\Pp_i(Z_{i+1}=0)=0$.

Recall that $\theta=e^{-H}$ and
$L_{i+1}=\log(Z_{i+1}\vee\theta)$.  Hence
\[
 \left|\log Z_{i+1}-L_{i+1}\right|
 =
 \left(-\log Z_{i+1}-H\right)_+.
\]
The layer-cake identity and Tonelli's theorem therefore give
\begin{align*}
 \E_i\left|\log Z_{i+1}-L_{i+1}\right|
 &=
 \int_0^\infty
 \Pp_i\left(-\log Z_{i+1}-H>t\right)\dd t
 \notag\\
 &=
 \int_H^\infty
 \Pp_i(Z_{i+1}<e^{-u})\dd u
 \notag\\
 &\le
 \int_H^\infty e^{-k_i u/4}\dd u
 =
 \frac4{k_i}e^{-k_iH/4}.
\end{align*}

Taking expectations, summing over $0\le i<m$, and using
$k_i=n-i$, we obtain
\begin{align*}
 \E\sum_{i=0}^{m-1}
 \left|\log Z_{i+1}-L_{i+1}\right|
 &\le
 4\sum_{i=0}^{m-1}
 \frac1{k_i}e^{-k_iH/4}
 =
 4\sum_{k=s_1+1}^{n}
 \frac1k e^{-kH/4}.
\end{align*}
For all sufficiently large $n$, $H\ge4\log2$,
and hence
\begin{align*}
 4\sum_{k=s_1+1}^{n}\frac1k e^{-kH/4}
 &\le
 \frac4{s_1+1}\sum_{k=s_{1}+1}^{\infty}e^{-kH/4}
 =
 \frac4{s_1+1}
 \frac{e^{-(s_1+1)H/4}}{1-e^{-H/4}}
 \le
 \frac8{s_1+1}e^{-(s_1+1)H/4}.
\end{align*}
For all sufficiently large $n$, $s_1+1\ge8$, so the last expression is
at most $e^{-(s_1+1)H/4}$.  Since $s_1+1>\ell^{3a}$, $H=\frac a2\log\ell$,
we have
\[
 \frac{(s_1+1)H}{4}
 >
 \frac a8\ell^{3a}\log\ell.
\]
Consequently,
\[
 \E\sum_{i=0}^{m-1}
 \left|\log Z_{i+1}-L_{i+1}\right|
 \le
 \exp\left\{
 -\frac a8\ell^{3a}\log\ell
 \right\},
\]
which proves \eqref{ex:eq:lower-truncation-error}. Thus we complete the proof of Lemma~\ref{lem:exact-lower-truncation}.
\end{proof}

\begin{proof}[Proof of Lemma~\ref{lem:exact-mart-bounds}]

Recall that
\[
 \eta_{i+1}
 =
 \zeta_{i+1}-g_i,
 \qquad
 g_i=\E_i\zeta_{i+1}.
\]
Lemma~\ref{lem:common-one-step} already gives the required moment
bounds for $\zeta_{i+1}$.  To transfer these bounds to
$\eta_{i+1}$, we also need a pointwise estimate for the conditional
mean $g_i$.

Conditionally on $\mathcal F_i$, the projection $P_i$ is deterministic
of rank $k_i$, while the current row is independent of
$\mathcal F_i$.  Thus $g_i$ is the logarithmic drift associated with
a fixed rank-$k_i$ projection.  The following lemma gives the bound
\[
 |g_i|
 \le
 \frac{CM}{k_i},
\]
which will be used below to control the terms containing powers of
$g_i$.

\begin{lemma}\label{lem:exact-fixed-drift}
Let $X$ have independent centered variance-one coordinates, fourth moments at most $M$, and densities bounded by $M_1$.  Let $P$ be a deterministic rank-$k$ projection, where $k\ge s_1$, and put $Z=X^TPX/k$.  Then
\begin{equation}\label{ex:eq:fixed-drift}
 \left|\E\left[\log(Z\vee\theta)+\frac1k\right]\right|
 \le \frac{CM}{k}.
\end{equation}
\end{lemma}

\begin{proof}[Proof of Lemma~\ref{lem:exact-fixed-drift}]
Set $Y=Z-1$.  Lemma~\ref{lem:quadratic-forms} \eqref{eq:var-part} gives
\[
 \E Y=0,\qquad \E Y^2\le CM/k.
\]
Let $A=\{|Y|\le1/2\}$.  Since $\theta<1/2$, the lower truncation is inactive on $A$, and Taylor's formula yields
\[
 |\E[\log Z;A]|
 \le |\E[Y;A]|+C\E[Y^2;A]
 \le |\E[Y;A^c]|+C\E Y^2
 \le CM/k.
\]
Here $\E Y=0$ and $|Y|\1_{A^c}\le2Y^2$ were used.  On the upper tail $Z>3/2$, $\log Z\le C(Z-1)^2$, so its contribution is $O(M/k)$.  For the lower tail, the layer-cake identity \eqref{eq:layer-cake-identity} gives
\[
 \E[-\log(Z\vee\theta);Z<1/2]
 \le(\log2)\Pp(Z<1/2)
 +\int_{\log2}^{H}\Pp(Z\le e^{-u})\dd u.
\]
Chebyshev controls the fixed interval up to the threshold in Lemma~\ref{lem:projection-density}, and the remaining integral is bounded by $C/k$ using \eqref{ex:eq:projection-smallball}.  Thus $|\E\log(Z\vee\theta)|\le CM/k$.  Adding $1/k$ proves \eqref{ex:eq:fixed-drift}. Thus we complete the proof of Lemma~\ref{lem:exact-fixed-drift}.
\end{proof}

We now return to the proof of Lemma~\ref{lem:exact-mart-bounds}.
Lemma~\ref{lem:exact-fixed-drift}, applied conditionally on the first $i$ rows, gives
\begin{equation}\label{ex:eq:g-point}
 |g_i|\le CM/k_i\quad\text{a.s.}
\end{equation}
Together with Lemma~\ref{lem:common-one-step},
\[
 \E g_i^2\le \|g_i\|_\infty\E|g_i|
 \le C(M/k_i)h_i\le Ch_i.
\]
Since $\eta_{i+1}=\zeta_{i+1}-g_i$,
\[
 \E_i\eta_{i+1}^2 = \E_i\zeta_{i+1}^2-g_i^2,
 \qquad
 \E\eta_{i+1}^2 = \E\zeta_{i+1}^2-\E g_i^2.
\]
This proves \eqref{ex:eq:mart-var-bound} and \eqref{ex:eq:mart-q}.  Also
\[
 \E_i\eta_{i+1}^3=\E_i\zeta_{i+1}^3-3g_i\E_i\zeta_{i+1}^2+2g_i^3.
\]
Taking absolute values and then expectations,
\begin{align*}
 \E|\E_i\eta_{i+1}^3|
 &\le \E|\E_i\zeta_{i+1}^3|
   +3\E\{|g_i|\E_i\zeta_{i+1}^2\}+2\E|g_i|^3\\
 &\le Ch_i+\frac{CM}{k_i}\left(\frac C{k_i}+Ch_i\right)
   +C\left(\frac{M}{k_i}\right)^2\E|g_i|\\
 &\le Ch_i.
\end{align*}
Here $M/k_i\le1$ for all sufficiently large $n$, and the term $M/k_i^2$ is absorbed by the $M^2k_i^{-2}$ component of $h_i$.  This proves the third-moment part of \eqref{ex:eq:mart-third-fourth-bound}.  The fourth-moment part follows from
$|\eta_{i+1}|^4\le8|\zeta_{i+1}|^4+8|g_i|^4$ and
$\E|g_i|^4\le (CM/k_i)^3\E|g_i|\le Ch_i$.
Summing \eqref{ex:eq:mart-q} and using \eqref{eq:common-h-sum-point} proves \eqref{ex:eq:mart-s}.  Finally, $s_{\eta}^2\asymp\ell$, $\E\eta_{i+1}^2 \le CM/k_i$, $k_i\ge s_1$, and $M=o(\ell)$ eventually imply
\[
 \max_i\frac{\ell^2\E\eta_{i+1}^2}{s_{\eta}^2}
 \le C\frac{M\ell}{s_1}=o(1).
\]
This proves \eqref{ex:eq:mart-max} and completes the proof of Lemma~\ref{lem:exact-mart-bounds}.
\end{proof}

\begin{proof}[Proof of Lemma~\ref{lem:exact-drift}]

Recall that $\mathcal D_m = \sum_{i=0}^{m-1}(g_i-\E g_i)$.
We split
\[
 \mathcal D_m
 =
 \mathcal D_{\mathrm{early}}
 +
 \mathcal D_{\mathrm{terminal}},
\]
where
\[
 \mathcal D_{\mathrm{early}}
 =
 \sum_{k_i\ge s_2}(g_i-\E g_i),
 \qquad
 \mathcal D_{\mathrm{terminal}}
 =
 \sum_{s_1\le k_i<s_2}(g_i-\E g_i).
\]
We estimate the two parts separately.  We first control
$\mathcal D_{\mathrm{terminal}}$.
Lemma~\ref{lem:common-one-step} gives
\begin{equation}\label{ex:eq:terminal-drift-L1}
 \begin{aligned}
 \E|\mathcal D_{\mathrm{terminal}}|
 &\le
 2\sum_{s_1\le k_i<s_2}\E|g_i|
 \le
 C\sum_{s_1\le k_i<s_2}h_i
 \le
 CM\ell^{-12}.
 \end{aligned}
\end{equation}

We now turn to $\mathcal D_{\mathrm{early}}$.
To write each conditional mean $g_i$ explicitly as a function of
the projection $P_i$, we introduce the following deterministic
functional.  For the law of row $a_{i+1}$ and a deterministic
rank-$k_i$ projection $P$, define
\begin{equation}\label{ex:eq:Gamma-i}
 \Gamma_i(P)
 :=
 \E_{a_{i+1}}
 \left[
 \log\left(
 \frac{a_{i+1}^TPa_{i+1}}{k_i}\vee\theta
 \right)
 +
 \frac1{k_i}
 \right],
\end{equation}
where $\E_{a_{i+1}}$ denotes expectation with respect to the law of
$a_{i+1}$ only.  By
\eqref{eq:common-normalized-projections},
\eqref{eq:common-heights},
\eqref{eq:common-truncated-increment}, and
\eqref{ex:eq:mart-pred-def}, $g_i=\Gamma_i(P_i)$.
Hence, with $F = \sum_{k_i\ge s_2}\Gamma_i(P_i)$,
we have
\[
 \mathcal D_{\mathrm{early}}
 =
 F-\E F.
\]

Thus, to control $\mathcal D_{\mathrm{early}}$, it is enough to
control the fluctuation of $F$. 
We control the fluctuation of $F$ by comparing it with the same
functional after replacing one row.
Fix $r$, let $a_r'$ be an independent copy of $a_r$, and replace $(a_1,\ldots,a_r,\ldots,a_n)$
by $(a_1,\ldots,a_r',\ldots,a_n)$.
Let $P_i^{(r)}$ be the projection obtained from the first $i$ rows
after this replacement, and define
\[
 F^{(r)}
 =
 \sum_{k_i\ge s_2}\Gamma_i(P_i^{(r)}).
\]
We will first bound $|F-F^{(r)}|$
uniformly in $r$.  The Efron--Stein inequality will then give
\[
 \Var(F)
 \le
 \frac12\sum_r\E(F-F^{(r)})^2.
\]

If $i<r$, the first $i$ rows are unchanged, so $P_i^{(r)}=P_i$.
Therefore,
\[
 F-F^{(r)}
 =
 \sum_{\substack{k_i\ge s_2\\ i\ge r}}
 \left\{
 \Gamma_i(P_i)-\Gamma_i(P_i^{(r)})
 \right\}.
\]

For $i\ge r$, the two collections of the first $i$ rows differ in only
one row.  To estimate each difference
\[
 \Gamma_i(P_i)-\Gamma_i(P_i^{(r)}),
\]
we first compare the ranges of the two projections.  The following
lemma shows that these two rank-$k_i$ ranges have an intersection of
dimension at least $k_i-1$.

\begin{lemma}\label{lem:row-resample-geometry}
Let $1\le i\le n$, and suppose that
$x_1,\ldots,x_i$ and $x_1',\ldots,x_i'$ are two linearly
independent families in $\mathbb R^n$. Suppose that, for some
$t\in\{1,\ldots,i\}$,
\[
 x_j=x_j' \qquad\text{for every }j\ne t.
\]
Let $P$ and $P'$ be the orthogonal projections onto
\[
 \operatorname{Span}\{x_1,\ldots,x_i\}^{\perp}
 \quad\text{and}\quad
 \operatorname{Span}\{x_1',\ldots,x_i'\}^{\perp},
\]
respectively. Then
\[
 \operatorname{rank}P=\operatorname{rank}P'=k:=n-i,
\]
and
\[
 \dim(\operatorname{Range}P\cap\operatorname{Range}P')
 \ge k-1.
\]
\end{lemma}

\begin{proof}[Proof of Lemma~\ref{lem:row-resample-geometry}]
Set
\[
 S=\operatorname{Span}\{x_1,\ldots,x_i\},
 \qquad
 S'=\operatorname{Span}\{x_1',\ldots,x_i'\}.
\]
Since both families are linearly independent,
$\dim S=\dim S'=i$, and hence
\[
 \operatorname{rank}P=\dim S^\perp=n-i=k,
 \qquad
 \operatorname{rank}P'=\dim(S')^\perp=n-i=k.
\]
The unchanged vectors $\{x_j:j\ne t\}$ are linearly
independent and span an $(i-1)$-dimensional subspace contained
in $S\cap S'$. Thus
\[
 \dim(S\cap S')\ge i-1,
\]
and consequently
\[
 \dim(S+S')
 =\dim S+\dim S'-\dim(S\cap S')
 \le i+1.
\]
Finally,
\[
 \operatorname{Range}P\cap\operatorname{Range}P'
 =S^\perp\cap(S')^\perp
 =(S+S')^\perp.
\]
Therefore
\[
 \dim(\operatorname{Range}P\cap\operatorname{Range}P')
 =n-\dim(S+S')
 \ge n-(i+1)=k-1.
\]
\end{proof}

We now return to the proof of Lemma~\ref{lem:exact-drift}.
For every $i\ge r$, Lemma~\ref{lem:row-resample-geometry} gives
\[
 \operatorname{rank}P_i
 =
 \operatorname{rank}P_i^{(r)}
 =
 k_i,
 \qquad
 \dim\left(
 \operatorname{Range}P_i
 \cap
 \operatorname{Range}P_i^{(r)}
 \right)
 \ge
 k_i-1.
\]
Thus the two projections can differ in at most one direction in
each of their ranges, and in particular
\[
 \operatorname{rank}
 \left(P_i-P_i^{(r)}\right)
 \le2.
\]

We now need to convert this geometric fact into a bound for
\[
 \left|
 \Gamma_i(P_i)-\Gamma_i(P_i^{(r)})
 \right|.
\]
Since the term $1/k_i$ in the definition of $\Gamma_i$ does not
depend on the projection, it cancels in this difference.  Hence it is
enough to control the change of
\[
 \E_{a_{i+1}}
 \log\left(
 \frac{a_{i+1}^TPa_{i+1}}{k_i}\vee\theta
 \right)
\]
when $P$ is replaced by another rank-$k_i$ projection whose range
has an intersection of dimension at least $k_i-1$ with the range of
$P$.  The following lemma gives exactly this stability estimate.

\begin{lemma}\label{lem:exact-rank-stability}
Let $X=(X_1,\ldots,X_n)^T$ have independent centered variance-one coordinates satisfying $|X_j|\le B$ and $\E|X_j|^4\le M$.  
And suppose its coordinate densities are bounded by $M_1$.  Let $P,P'$ be rank-$k$ orthogonal projections whose ranges have an intersection of dimension at least $k-1$, where $k\ge s_1$.  Define
\[
 \Gamma(P)=\E\log\left(\frac{X^TPX}{k}\vee\theta\right).
\]
Then
\begin{equation}\label{ex:eq:rank-stability}
 |\Gamma(P)-\Gamma(P')|
 \le C\left\{\frac{M^2}{k^2}+\frac{B^2M}{k^3}
 +\left(H+\log\left(e+\frac{nB^2}{k}\right)+\frac{nB^2}{k}\right)e^{-ck}\right\}.
\end{equation}
The constants depend only on $M_1$.
\end{lemma}

We apply Lemma~\ref{lem:exact-rank-stability} with
\[
 \begin{aligned}
 X&=a_{i+1},\qquad P=P_i,\qquad P'=P_i^{(r)},\\
 k&=k_i,\qquad B=(K_0n\ell^{3/4})^{1/2}.
 \end{aligned}
\]
The geometric assumption of the lemma was verified above.
The moment and density assumptions follow from
Proposition~\ref{prop:exact-growing-fourth}.
Moreover, \eqref{eq:exact-growing-support} gives
\[
 |a_{i+1,j}|\le (K_0n\ell^{3/4})^{1/2}
\]
for the bulk row $a_{i+1}$.

Therefore, uniformly in the realized rows,
\[
 \begin{aligned}
 \left|\Gamma_i(P_i)-\Gamma_i(P_i^{(r)})\right|
 &\le C\Biggl\{
 \frac{M^2}{k_i^2}
 +\frac{K_0n\ell^{3/4}M}{k_i^3}
 +\left(H+\log\left(e+\frac{K_0n^2\ell^{3/4}}{k_i}\right)
 +\frac{K_0n^2\ell^{3/4}}{k_i}\right)e^{-ck_i}
 \Biggr\}.
 \end{aligned}
\]
Summing this estimate over the indices appearing in
$F-F^{(r)}$ gives
\begin{equation}\label{ex:eq:row-influence}
 \begin{aligned}
 |F-F^{(r)}|
 &\le C\sum_{k\ge s_2}\Biggl\{
 \frac{M^2}{k^2}+\frac{K_0n\ell^{3/4}M}{k^3}
 +\left(H+\log\left(e+\frac{K_0n^2\ell^{3/4}}{k}\right)
 +\frac{K_0n^2\ell^{3/4}}{k}\right)e^{-ck}
 \Biggr\}\\
 &\le C\left\{
 \frac{M^2}{s_2}+\frac{K_0n\ell^{3/4}M}{s_2^2}
 +(H+K_0n^2\ell^{3/4})e^{-cs_2}
 \right\}
 =:\Delta_n.
 \end{aligned}
\end{equation}
Here we used $\sum_{k\ge s_2}k^{-2}\le C/s_2$ and
$\sum_{k\ge s_2}k^{-3}\le C/s_2^2$.
Moreover, the factor in parentheses in the exponential term is at most
$C(H+K_0n^2\ell^{3/4})$, while
$\sum_{k\ge s_2}e^{-ck}\le Ce^{-cs_2}$.
Since $s_2=n\ell^{-20a}+O(1)$ and $M\le C\ell$ eventually in both
regimes considered in Proposition~\ref{prop:exact-growing-fourth},
Hence
\[
 \frac{M^2}{s_2}\le\frac{C\ell^{20a+2}}{n},
 \qquad
 \frac{K_0n\ell^{3/4}M}{s_2^2}
 \le\frac{C\ell^{40a+7/4}}{n}.
\]
Here the constants may depend on $K_0$.
The exponential term is smaller than every negative power of $n$.  Thus, with the explicit choice $C_a=40a+2$,
\begin{equation}\label{ex:eq:Delta-poly}
 \Delta_n\le C n^{-1}\ell^{C_a}.
\end{equation}
The functional $F$ depends only on rows that can enter some $P_i$ with $k_i\ge s_2$, hence on at most $n$ independent rows.  The Efron--Stein inequality \eqref{eq:efron-stein-definition} now yields
\begin{equation}\label{ex:eq:drift-var}
 \Var(F)\le\frac12\sum_r\E(F-F^{(r)})^2
 \le \frac n2\Delta_n^2
 \le C n^{-1}\ell^{2C_a}.
\end{equation}

Chebyshev's inequality and \eqref{ex:eq:drift-var} give
\begin{equation}\label{ex:eq:early-drift-tail}
 \Pp\left(|\mathcal D_{\mathrm{early}}|>\tfrac12\varepsilon_n\sqrt{2\ell}\right)
 \le \frac{C\ell^{2C_a}}{n\varepsilon_n^2\ell}.
\end{equation}
Because $\varepsilon_n\ge\ell^{-1}$ and $n$ dominates every fixed power of $\ell$, the right-hand side of \eqref{ex:eq:early-drift-tail} is at most $C\varepsilon_n$ for all sufficiently large $n$.  Markov's inequality and \eqref{ex:eq:terminal-drift-L1} give
\[
 \Pp\left(|\mathcal D_{\mathrm{terminal}}|>\tfrac12\varepsilon_n\sqrt{2\ell}\right)
 \le C\ell^{-23/2}\le C\varepsilon_n.
\]
This proves \eqref{ex:eq:drift-prob}. Thus we complete the proof of Lemma~\ref{lem:exact-drift}.
\end{proof}

\begin{proof}[Proof of Lemma~\ref{lem:regularized-inverse}]
Write $x_k$ for the $k$th row of $X$, let $X^{(k)}$ be obtained by deleting that row, and put
\[
 G^{(k)}=(n^{-1}X^{(k)}X^{(k)T}+\lambda I_{p-1})^{-1}.
\]
The block-inverse formula gives
\begin{equation}\label{ex:eq:regularized-block}
 G_{kk}=D_k^{-1},\qquad
 D_k=\lambda+n^{-1}x_kx_k^T
 -n^{-2}x_kX^{(k)T}G^{(k)}X^{(k)}x_k^T.
\end{equation}
Set $H_k=X^{(k)T}G^{(k)}X^{(k)}$.  From $n^{-1}X^{(k)}X^{(k)T}G^{(k)}=I_{p-1}-\lambda G^{(k)} $
we obtain
\[
 n^{-2}\tr H_k=\frac{p-1}{n}-\frac\lambda n\tr G^{(k)}.
\]
Since $\E(x_{k}x_{k}^{T} \mid X^{(k)}) = n$, $\E(x_{k}H_{k}x_{k}^{T} \mid X^{(k)}) = \tr H_{k}$, we have
\begin{equation}
  \begin{aligned}\label{eq:conditional-E-Dk}
     \E(D_k\mid X^{(k)})
      =1+\lambda - n^{-2}\tr H_{k}
      =1+\lambda-\frac{p}{n}+\frac1n+\frac\lambda n\tr G^{(k)}.    
  \end{aligned}
\end{equation}
Let $A=1+\lambda-\frac{p}{n}+\frac{\lambda}{n}\E\tr G$, and let $\varepsilon_k := D_k-A$, thus
\begin{equation}
  \begin{aligned}
    \E(\varepsilon_k \mid X^{(k)}) = \frac{1}{n} + \frac{\lambda}{n}(\tr G^{(k)} - \E\tr G). \nonumber
  \end{aligned}
\end{equation}
We first prove
\begin{equation}\label{ex:eq:eps-bounds}
 |\E\varepsilon_k|\le Cn^{-1},\qquad
 \E\varepsilon_k^2\le CMn^{-1}.
\end{equation}

To compare the traces after deleting one row, introduce the dual
resolvents
\begin{equation}
  \begin{aligned}\label{eq:dual-resolvent-def}
    \widehat G=(n^{-1}X^TX+\lambda I_n)^{-1},
    \qquad
    \widehat G^{(k)}
    =
    (n^{-1}X^{(k)T}X^{(k)}+\lambda I_n)^{-1}.    
  \end{aligned}
\end{equation}
Since
\[
 n^{-1}X^TX
 =
 n^{-1}X^{(k)T}X^{(k)}
 +
 n^{-1}x_k^Tx_k,
\]
Lemma~\ref{lem:projection-identities}, applied with
\[
 K=n^{-1}X^{(k)T}X^{(k)},
 \qquad
 v=n^{-1/2}x_k^T,
\]
gives
\[
 0\le
 \tr\widehat G^{(k)}-\tr\widehat G
 \le\lambda^{-1}.
\]
Since $BB^T$ and $B^TB$ have the same nonzero eigenvalues,
\begin{equation}
  \begin{aligned}\label{eq:dual-trace-diff}
    \tr G
    =
    \tr\widehat G-\frac{n-p}{\lambda},
    \qquad
    \tr G^{(k)}
    =
    \tr\widehat G^{(k)}-\frac{n-p+1}{\lambda}.    
  \end{aligned}
\end{equation}
Consequently,
\[
 0\le
 \tr G-\tr G^{(k)}
 =
 \lambda^{-1}
 -
 \{\tr\widehat G^{(k)}-\tr\widehat G\}
 \le\lambda^{-1}.
\]
Hence $|\E D_k-A|\le C/n$.  This is the first assertion in \eqref{ex:eq:eps-bounds}.

For the second assertion, since $\E \varepsilon_{k}^{2} = \Var(D_{k}) + (\E D_{k} - A)^{2}$. Using the conditional variance decomposition and  the first assertion in \eqref{ex:eq:eps-bounds}, we have
\begin{equation}
  \begin{aligned}\label{eq:second-assertion-decompose}
    \E \varepsilon_{k}^{2} \le \E (\Var(D_{k}\mid X^{(k)})) + \Var(\E(D_{k}\mid X^{(k)})) + \frac{C}{n^{2}}. 
  \end{aligned}
\end{equation}
Considering, $\E (\Var(D_{k}\mid X^{(k)}))$, by \eqref{ex:eq:regularized-block} and \eqref{eq:conditional-E-Dk},
\begin{equation}
  \begin{aligned}
    D_{k} - \E (D_{k} \mid X^{(k)}) = \frac{1}{n}\sum_{j = 1}^{n}(x_{kj}^{2}-1) - \frac{1}{n^{2}}(x_{k}H_{k}x_{k}^{T} - \tr H_{k}). \nonumber
  \end{aligned}
\end{equation}
Using $(a-b)^{2} \leq 2a^{2} + 2b^{2}$, we have $\Var(D_{k}\mid X^{(k)})\leq 2 \Var( \frac{1}{n}\sum_{j = 1}^{n}(x_{kj}^{2}-1) \mid X^{(k)}) + 2\Var(\frac{1}{n^{2}}(x_{k}H_{k}x_{k}^{T} - \tr H_{k}) \mid X^{(k)})$. Using the assumption of fourth moment bound $M$ and Lemma~\ref{lem:quadratic-forms} \eqref{eq:var-part}, we have
\begin{equation}
  \begin{aligned}
    \Var \left( \frac{1}{n}\sum_{j = 1}^{n}(x_{kj}^{2}-1) \mid X^{(k)} \right) =& \frac{1}{n^{2}}\sum_{j = 1}^{n}(\E X_{kj}^{4} - 1) \leq \frac{M}{n}.\\
    \Var \left(\frac{1}{n^{2}}(x_{k}H_{k}x_{k}^{T} - \tr H_{k}) \mid X^{(k)} \right) \leq & CMn^{-4}\tr H_k^2. \nonumber
  \end{aligned}
\end{equation}
If $\eta_1,\ldots,\eta_{p-1}$ are the eigenvalues of
$n^{-1}X^{(k)}X^{(k)T}$, then the nonzero eigenvalues of $H_k$ are
$n\eta_j/(\eta_j+\lambda)$.  Thus each is at most $n$ and
$\tr H_k^2\le(p-1)n^2\le n^3$.  Consequently
\begin{equation}
  \begin{aligned}\label{eq:second-assertion-first-part}
     \E\{\Var(D_k\mid X^{(k)})\}\le CM/n.    
  \end{aligned}
\end{equation}
Then we consider $\Var(\E(D_{k}\mid X^{(k)}))$. By \eqref{eq:conditional-E-Dk},
$\E(D_{k}\mid X^{(k)}) = 1+\lambda-p/n+n^{-1}+\lambda n^{-1}\tr G^{(k)}$. Thus 
\begin{equation}
  \begin{aligned}
    \Var(\E(D_{k}\mid X^{(k)})) = \lambda^{2}\Var(n^{-1}\tr G^{(k)}). \nonumber
  \end{aligned}
\end{equation}
Reveal the rows of $X^{(k)}$ one by one.  Similar to \eqref{eq:dual-resolvent-def} and \eqref{eq:dual-trace-diff}, replacing one row changes
$n^{-1}X^{(k)}X^{(k)T}$ in rank at most two; the trace of its regularized inverse changes by at most $2\lambda^{-1}$.  We now use the Efron--Stein inequality.  If $F=F(X_1,\ldots,X_N)$ is a square-integrable function of independent inputs and $F^{(r)}$ is obtained by replacing $X_r$ by an independent copy, then
\begin{equation}\label{eq:efron-stein-definition}
 \Var(F)\le\frac12\sum_{r=1}^N\E(F-F^{(r)})^2.
\end{equation}
Applying \eqref{eq:efron-stein-definition} to the row inputs gives
\[
 \Var(n^{-1}\tr G^{(k)})
 \le C(p-1)(n\lambda)^{-2}\le \frac{C}{n\lambda^2}.
\]
After multiplication by $\lambda^2$, $\Var(\E(D_{k}\mid X^{(k)}))$ contributes $C/n$. Combining with \eqref{eq:second-assertion-decompose} and \eqref{eq:second-assertion-first-part} proves the second assertion in \eqref{ex:eq:eps-bounds}.

Since $A$ is deterministic and does not depend on $k$, the exact identity
$\frac1{A+z}=\frac1A-\frac z{A^2}+\frac{z^2}{A^2(A+z)}$ with $z=\varepsilon_k$ gives
\begin{equation}
  \begin{aligned}
    \frac{1}{D_{k}} = \frac{1}{A} - \frac{\varepsilon_{k}}{A^{2}} + \frac{\varepsilon_{k}^{2}}{A^{2}D_{k}}. \nonumber
  \end{aligned}
\end{equation}
Averaging and taking expectation, and by \eqref{ex:eq:regularized-block}, $G_{kk} = D_{k}^{-1}$, we have
\begin{equation}\label{ex:eq:regularized-fixed-point}
 \frac{1}{p}\E\tr G = \frac{1}{p}\sum_{k = 1}^{p}\E G_{kk} = \frac{1}{p}\sum_{k = 1}^{p}\E \frac{1}{D_{k}} =\frac1{A}+\Delta,
 \quad
 |\Delta|\le C\{n^{-1}\lambda^{-2}+Mn^{-1}\lambda^{-3}\},
\end{equation}
where $\Delta = -\frac{1}{pA^{2}}\sum_{k = 1}^{p}\E \varepsilon_{k} + \frac{1}{pA^{2}}\sum_{k = 1}^{p}\E\frac{\varepsilon_{k}^{2}}{D_{k}}$.
Indeed, before \eqref{eq:second-assertion-first-part}, we know that all eigenvalues of $H_{k}$ are less or equal to $n$. Thus $H_{k}\preceq nI_{n}$, and therefore $n^{-1}I_{n} - n^{-2}H_{k} \succeq 0$. By the definition of $D_{k}$ in \eqref{ex:eq:regularized-block} and the definiton of $H_{k}$, we have $D_{k} = \lambda + x_{k}( n^{-1}I_{n} - n^{-2}H_{k} )x_{k}^{T} \geq \lambda$. Moreover, $A = 1+\lambda -pn^{-1}+ \lambda n^{-1}\E \tr G \geq \lambda$.
Using $A,D_k\ge\lambda$ and \eqref{ex:eq:eps-bounds}, the linear term is bounded by
$A^{-2}|\E\varepsilon_k|\le Cn^{-1}\lambda^{-2}$ and the quadratic remainder by
$A^{-2}\E\varepsilon_k^2/D_k\le CMn^{-1}\lambda^{-3}$.

The positive solution of $s=(1+\lambda-y+y\lambda s)^{-1}$ is $s_y(\lambda)$.  Since
$f(t)=t-(1+\lambda-y+y\lambda t)^{-1}$ satisfies $f'(t)\ge1$, $f(\frac{1}{p}\E\tr G) = \Delta$ and $f(s_{y}(\lambda)) = 0$.
The mean-value theorem and \eqref{ex:eq:regularized-fixed-point} imply
$|p^{-1}\E\tr G - s_y(\lambda)|\le|\Delta|$.  Multiplication by $y$ and the choice $\lambda=n^{-1/6}$ prove \eqref{ex:eq:regularized-mean}.

For the variance, reveal the $n$ columns.  Replacing one column changes $n^{-1}XX^T$ by the difference of two rank-one positive semidefinite matrices. Using Lemma~\ref{lem:projection-identities}, it follows that addition or deletion of one such matrix changes the normalized trace $\frac{1}{n}\tr G$ by at most $1/(n\lambda)$, and a replacement by at most $2/(n\lambda)$.  The Efron--Stein inequality \eqref{eq:efron-stein-definition} therefore gives
\[
 \Var\left(\frac{1}{n}\tr G \right)\le Cn(n\lambda)^{-2}=\frac{C}{n\lambda^2}=Cn^{-2/3}.
\]
Thus we have proved \eqref{ex:eq:regularized-variance} and we complete the proof of Lemma~\ref{lem:regularized-inverse}.
\end{proof}

\begin{proof}[Proof of Lemma~\ref{lem:exact-rank-stability}]
If $P=P'$, there is nothing to prove.  Otherwise their common range has dimension $k-1$.  Let $R$ be the projection onto that common range.  There are unit vectors $u,v\in\ker R$ such that
\[
 P=R+uu^T,\qquad P'=R+vv^T.
\]
Put
\[
 W=X^TRX,\qquad U=u^TX,\qquad V=v^TX,
 \qquad \mu=\E W=k-1.
\]
Let $u_0$ be the threshold in Lemma~\ref{lem:projection-density} and set
$c_0=e^{-u_0}/4$.  Since $R$ has rank $k-1$, for all sufficiently large $k$,
\[
 \{W<c_0k\}\subseteq
 \left\{\frac{W}{k-1}\le e^{-u_0}\right\}.
\]
Consequently, after decreasing $c>0$ if necessary,
\begin{equation}\label{ex:eq:W-small}
 \Pp(W<c_0k)\le e^{-ck}.
\end{equation}
For large $n$, $c_0k>k\theta$.  On $\mathcal G=\{W\ge c_0k\}$ the lower truncation is inactive, and
\begin{equation}\label{ex:eq:log-rank-exp}
 \log(W+U^2)-\log(W+V^2)
 =\frac{U^2-V^2}{W}+\rho,
 \qquad |\rho|\le\frac{U^4+V^4}{2W^2}.
\end{equation}
The fourth moment of a unit linear form is at most $CM$, hence
\begin{equation}\label{ex:eq:rank-rem}
 \E(|\rho|;\mathcal G)\le CM/k^2.
\end{equation}
Since $X^TPX\le\|X\|^2\le nB^2$, on $\mathcal G^c$ both logarithms after lower truncation have absolute value at most
$H+\log(e+nB^2/k)$.  Hence \eqref{ex:eq:W-small} controls the bad-event contribution by
$C\{H+\log(e+nB^2/k)\}e^{-ck}$.

Choose a twice continuously differentiable function $f:[0,\infty)\to\R$ such that
$f(w)=w^{-1}$ for $w\ge c_0k$, $f(w)=0$ for $w\le c_0k/2$, and
\[
 \|f\|_\infty\le C/k,
 \qquad \|f'\|_\infty\le C/k^2,
 \qquad \|f''\|_\infty\le C/k^3.
\]
The two expressions differ only on $\{W<c_0k\}$.  Since
$|U|,|V|\le B\sqrt n$ and $\|f\|_\infty\le C/k$,
\begin{align}
 &\left|\E[(U^2-V^2)W^{-1};\mathcal G]
       -\E[(U^2-V^2)f(W)]\right|\notag\\
 &\qquad\le \frac Ck\E[(U^2+V^2)\1_{\{W<c_0k\}}]
 \le C\frac{nB^2}{k}e^{-ck}.
 \label{ex:eq:f-transition}
\end{align}
Taylor expansion of $f$ at $\mu$ gives
\begin{align}
 \E[(U^2-V^2)f(W)]
 &=f'(\mu)\E[(U^2-V^2)(W-\mu)]+\mathcal R,\label{ex:eq:f-expand}\\
 |\mathcal R|
 &\le\frac{C}{k^3}\E[(U^2+V^2)(W-\mu)^2].\label{ex:eq:f-rem}
\end{align}
The constant term vanishes because $\E U^2=\E V^2=1$.  For deterministic symmetric matrices $C,D$, expanding the four indices and retaining only pairings and all-equal indices gives the exact identity
\begin{equation}\label{ex:eq:qf-covariance}
 \Cov(X^TCX,X^TDX)
 =2\tr(CD)+\sum_j(\E X_j^4-3)c_{jj}d_{jj}.
\end{equation}
Apply this with $C=uu^T$ and $D=R$.  Since $Ru=0$,
\[
 \E[(U^2-1)(W-\mu)]
 =2u^TRu+\sum_j(\E X_j^4-3)u_j^2r_{jj}
 =\sum_j(\E X_j^4-3)u_j^2r_{jj}.
\]
Its absolute value is at most $CM$; the same holds with $v$.
Thus the first term in \eqref{ex:eq:f-expand} is $O(M/k^2)$.

It remains to control the Taylor remainder in
\eqref{ex:eq:f-rem}.  Since
\[
 U=u^TX,\qquad
 V=v^TX,\qquad
 W-\mu=X^TRX-(k-1),
\]
we need to estimate
\[
 \begin{aligned}
 \E[(U^2+V^2)(W-\mu)^2]
 &=
 \E\left[
 (u^TX)^2\{X^TRX-(k-1)\}^2
 \right]\\
 &\quad+
 \E\left[
 (v^TX)^2\{X^TRX-(k-1)\}^2
 \right].
 \end{aligned}
\]
Thus we need a bound for the mixed moment
\[
 \E\left[
 (w^TX)^2\{X^TRX-r\}^2
 \right]
\]
when $w$ is a unit vector and $R$ is a rank-$r$ projection.
The following Lemma~\ref{lem:exact-mixed-six} gives exactly this estimate.

\begin{lemma}\label{lem:exact-mixed-six}
Let $X=(X_1,\ldots,X_n)^T$ have independent centered variance-one coordinates satisfying $|X_j|\le B$ and $\E|X_j|^4\le M$.  Let $u$ be a unit vector and let $R$ be an orthogonal projection of rank $r$.  Then
\begin{equation}\label{ex:eq:mixed-six}
 \E\bigl[(u^TX)^2\{X^TRX-r\}^2\bigr]
 \le C\{M^2r+B^2M\}.
\end{equation}
\end{lemma}

\begin{proof}[Proof of Lemma~\ref{lem:exact-mixed-six}]
Put $A=uu^T$.  For a symmetric matrix $C$, write
\[
 D_C=\sum_jc_{jj}(X_j^2-1),\qquad
 O_C=\sum_{j\ne k}c_{jk}X_jX_k.
\]
Then $X^TAX=1+D_A+O_A$ and $X^TRX-r=D_R+O_R$.  The elementary moment bounds
\begin{align*}
 \E D_A^2&\le CM,& \E D_R^2&\le CMr,\\
 \E O_A^2&\le C,& \E O_A^4&\le CM^2,
 &\E O_R^4&\le CM^2r^2
\end{align*}
follow from independence and Lemma~\ref{lem:quadratic-forms} \eqref{eq:off-diag-part}.  Also
$\E(D_R+O_R)^2\le CMr$.

Expand
\begin{align*}
 (1+D_A+O_A)(D_R+O_R)^2
 &=D_R^2+2D_RO_R+O_R^2+D_AD_R^2+2D_AD_RO_R+D_AO_R^2\\
 &\quad+O_AD_R^2+2O_AD_RO_R+O_AO_R^2.
\end{align*}
The expectation of $2D_RO_R$ is zero, and the expectation of the first line without $D_A$ is
$\E(D_R+O_R)^2=O(Mr)$.  We now bound the six terms containing $D_A$ or $O_A$.
For $\E D_AD_R^2$, independence and the centering of $X_j^2-1$ force all three diagonal indices to coincide; hence
\[
 |\E D_AD_R^2|
 \le C\max_j\E(|X_j|^2+1)^3\sum_ja_{jj}r_{jj}^2
 \le CB^2M.
\]
For $\E D_AD_RO_R$, a nonzero term must match the two off-diagonal indices with the two diagonal indices.  Since $|\E X_j^3|\le M^{1/2}$,
\[
 |\E D_AD_RO_R|
 \le CM\sum_{p\ne q}a_{pp}r_{qq}|r_{pq}|
 \le CM\sqrt r.
\]
Here $R^2=R$ implies $\sum_qr_{pq}^2=r_{pp}$, while
$\sum_qr_{qq}^2\le\sum_qr_{qq}=r$.  Thus
\[
 \sum_qr_{qq}|r_{pq}|\le\sqrt{r r_{pp}},
 \qquad
 \sum_pa_{pp}\sqrt{r_{pp}}
 =\sum_pu_p^2\sqrt{r_{pp}}\le1.
\]
Cauchy--Schwarz then gives
\[
 |\E D_AO_R^2|\le(\E D_A^2)^{1/2}(\E O_R^4)^{1/2}
 \le CM^{3/2}r.
\]
Likewise, direct index matching gives
\begin{align*}
 \E O_AD_R^2
 &=2\sum_{p\ne q}u_pu_qr_{pp}r_{qq}
   \E X_p^3\E X_q^3,\\
 |\E O_AD_R^2|
 &\le CM\left(\sum_p|u_p|r_{pp}\right)^2
 \le CMr,
\end{align*}
where
$\sum_p|u_p|r_{pp}\le(\sum_pu_p^2r_{pp})^{1/2}(\sum_pr_{pp})^{1/2}\le\sqrt r$.
Furthermore,
\[
 |\E O_AD_RO_R|
 \le(\E D_R^2)^{1/2}(\E O_A^2O_R^2)^{1/2}
 \le C M^{3/2}r,
\]
because $\E O_A^2O_R^2\le(\E O_A^4\E O_R^4)^{1/2}\le CM^2r$.  Finally,
$|\E O_AO_R^2|\le(\E O_A^2)^{1/2}(\E O_R^4)^{1/2}\le CMr$.
The displayed estimates cover, respectively,
$D_AD_R^2$, $D_AD_RO_R$, $D_AO_R^2$, $O_AD_R^2$, $O_AD_RO_R$, and $O_AO_R^2$.
Since $M\ge1$ and $r\ge1$, all of them except the all-equal sixth-order diagonal contribution are bounded by $CM^2r$; that exceptional contribution is bounded by $CB^2M$.  This proves \eqref{ex:eq:mixed-six}.
\end{proof}

We now return to the proof of Lemma~\ref{lem:exact-rank-stability}.
Using Lemma~\ref{lem:exact-mixed-six}, applied to $u$ and $v$, bounds \eqref{ex:eq:f-rem} by
\begin{equation}\label{ex:eq:rank-taylor-remainder}
 C\left(\frac{M^2}{k^2}+\frac{B^2M}{k^3}\right).
\end{equation}
The good-event logarithmic remainder in \eqref{ex:eq:rank-rem} is $CM/k^2$, the linear term is $CM/k^2$, and the Taylor remainder is bounded by \eqref{ex:eq:rank-taylor-remainder}.  Because $M\ge 1$, both $M/k^2$ terms are absorbed by $CM^2/k^2$.  The transition error \eqref{ex:eq:f-transition} and the two bad-event logarithms contribute at most
\[
 C\left\{H+\log\left(e+\frac{nB^2}{k}\right)+\frac{nB^2}{k}\right\}e^{-ck}.
\]
Together with \eqref{ex:eq:log-rank-exp}, these bounds give exactly \eqref{ex:eq:rank-stability}. Thus we complete the proof of Lemma~\ref{lem:exact-rank-stability}.
\end{proof}

\begin{proof}[Proof of Lemma~\ref{lem:variance-martingale}]
We use \eqref{eq:common-parameters} with
$a=20+4\lceil\alpha\rceil$, so $a>3\alpha+12$.
All bulk quantities depend only on the first $m$ rows, which satisfy
\eqref{eq:variance-growing-support}.
Constants below may depend on $R_3,M_1,K_0,K_1,\alpha$.

By Lemma~\ref{lem:exact-fixed-drift}, $|g_i|\le CM/k_i$ and
$\max_{i<m}M/k_i\le C\ell^{\alpha-3a}=o(1)$.
Together with Lemma~\ref{lem:common-one-step}, this gives
\[
 \E|g_i|^p
 \le(CM/k_i)^{p-1}\E|g_i|
 \le Ch_i,
 \qquad p=2,3,4.
\]
Expanding the powers of $\eta_{i+1}=\zeta_{i+1}-g_i$ therefore yields
\begin{equation}\label{eq:variance-hybrid-moments}
 \left|q_i-\frac2{k_i}\right|
 +\E\left|\E_i\eta_{i+1}^2-q_i\right|
 +\E|\E_i\eta_{i+1}^3|+\E|\eta_{i+1}|^4
 \le Ch_i.
\end{equation}
In the third-moment expansion, the additional term $M/k_i^2$
is absorbed by the $M^2/k_i^2$ term in $h_i$.
Together with \eqref{eq:common-h-sum-point}, this yields
$q_i\le CM/k_i$ and
\[
 s_\eta^2
 =2\sum_{k=s_1+1}^{n}\frac1k+O(M)
 \le C(\ell+M).
\]

For the lower bound, \eqref{eq:common-h-terminal} gives
\[
 \begin{aligned}
 s_\eta^2
 &\ge\sum_{s_1<k_i<s_2}q_i
 \ge2\log(s_2/s_1)-C-CM\ell^{-(2\alpha+12)}
 \ge c\ell,
 \end{aligned}
\]
since $\log(s_2/s_1)=\ell-O(\log\ell)$ and
\[
 M\ell^{-(2\alpha+12)}
 \le K_1\ell^{-\alpha-12}=o(1).
\]
This proves \eqref{eq:variance-martingale-variance}.

For $k_i\ge s_2$, the future sum contains all indices $j$ with
$s_1<k_j<s_2$. The same lower bound therefore gives
\begin{equation}\label{eq:variance-hybrid-future-variance}
 s_\eta^2v_{\eta,i}
 =\sum_{j=i+1}^{m-1}q_j
 \ge c\ell.
\end{equation}

Apply Lemma~\ref{lem:esseen} to $\mathcal M_m/s_\eta$ with cutoff
$s_\eta\ell$ and substitute $u=s_\eta t/\sqrt\ell$. Since
$s_\eta\ge c\sqrt\ell$,
\begin{equation}\label{eq:variance-hybrid-smoothing}
 \dk(\mathcal M_m/s_\eta,\Normal)
 \le C\ell^{-3/2}+C\int_0^{\ell^{3/2}}
 \frac{\left|\varphi_{\mathcal M_m/\sqrt\ell}(t)
 -e^{-s_\eta^2t^2/(2\ell)}\right|}{t}\dd t.
\end{equation}
The successive Gaussian replacements in the proof of
Lemma~\ref{lem:exact-hybrid}, now at scale $\sqrt\ell$, give
\begin{equation}\label{eq:variance-hybrid-characteristic}
 \begin{aligned}
 &\frac{\left|\varphi_{\mathcal M_m/\sqrt\ell}(t)
 -e^{-s_\eta^2t^2/(2\ell)}\right|}{t}\\
 &\quad\le\frac{Ct}{\ell}\sum_{i<m}
 e^{-s_\eta^2t^2v_{\eta,i}/(2\ell)}
 \E\left|\E_i\eta_{i+1}^2-q_i\right|\\
 &\qquad+\frac{Ct^2}{\ell^{3/2}}\sum_{i<m}
 e^{-s_\eta^2t^2v_{\eta,i}/(2\ell)}\E|\E_i\eta_{i+1}^3|\\
 &\qquad+\frac{Ct^3}{\ell^2}\sum_{i<m}
 e^{-s_\eta^2t^2v_{\eta,i}/(2\ell)}\E|\eta_{i+1}|^4,
 \qquad t>0.
 \end{aligned}
\end{equation}
The linear term vanishes since $\E_i\eta_{i+1}=0$.
The third-order Taylor remainder for $e^{ix}$ is at most $|x|^4/24$,
while $|e^{-y}-1+y|\le y^2/2$ for $y\ge0$ and
$q_i^2\le\E|\eta_{i+1}|^4$. Thus the estimate holds throughout
$0<t\le\ell^{3/2}$.

The last two sums are controlled by \eqref{eq:variance-hybrid-moments}.
For the first, we need a sharper bound when $k_i\ge s_2$.
Conditionally on the first $i$ rows, $\E_i\eta_{i+1}^2$ is the
variance of the truncated logarithm associated with $P_i$.
Lemma~\ref{lem:row-resample-geometry} describes the change in $P_i$
when one of these rows is replaced. The following lemma bounds
the resulting change in variance.

\begin{lemma}\label{lem:variance-rank-stability}
Let $\alpha\ge0$ and $M_1,K_0,K_1>0$ be fixed.
Use the notation in \eqref{eq:common-parameters} and
\eqref{eq:truncate-notation-definition}, with
$a=20+4\lceil\alpha\rceil$.
Let $X=(X_1,\ldots,X_n)^T$ have independent centered real-valued
coordinates of variance one satisfying, for some $M\ge3$,
\[
 \E|X_j|^4\le M\le K_1\ell^\alpha,\qquad
 |X_j|\le\bigl(K_0n\ell^{(\alpha+3)/2}\bigr)^{1/2},\qquad
 \|f_{X_j}\|_\infty\le M_1.
\]
Let $P,P'$ be rank-$k$ orthogonal projections whose ranges have
an intersection of dimension at least $k-1$, where $k\ge s_1$.
For all sufficiently large $n$,
\begin{equation}\label{eq:variance-rank-stability}
 \begin{aligned}
 &\left|
 \Var\left(\log\left(\frac{X^TPX}{k}\vee\theta\right)\right)
 -\Var\left(\log\left(\frac{X^TP'X}{k}\vee\theta\right)\right)
 \right|\\
 &\qquad\le C\left(
 \frac{M^2}{k^2}
 +\frac{K_0n\ell^{(\alpha+3)/2}M}{k^3}
 \right)+n^{-100}.
 \end{aligned}
\end{equation}
The constant and the required lower bound on $n$ depend only on
$M_1,K_0,K_1,\alpha$.
\end{lemma}

\begin{proof}[Proof of Lemma~\ref{lem:variance-rank-stability}]
The case $P=P'$ is immediate. Otherwise write
\[
 P=R+uu^T,\qquad P'=R+vv^T,
\]
where $R$ has rank $k-1$ and $u,v\in\ker R$ are unit vectors.
Put $W=X^TRX$, $U=u^TX$, $V=v^TX$, and $\mu=\E W=k-1$.
Lemma~\ref{lem:projection-density} gives constants $c_0,c>0$,
depending only on $M_1$, such that
\[
 \Pp(W<c_0k)\le e^{-ck}.
\]
For large $n$, $\theta<c_0$. With $F(w)=\{\log(w/k)\}^2$,
\[
 \sup_{w\ge c_0k}|F^{(j)}(w)|\le Ck^{-j},
 \qquad j=1,2,3.
\]
Thus, on $\{W\ge c_0k\}$, the lower truncation is inactive and
\[
 \left|F(W+U^2)-F(W+V^2)-F'(W)(U^2-V^2)\right|
 \le Ck^{-2}(U^4+V^4).
\]
Its expectation is at most $CM/k^2$, since
$\E U^4+\E V^4\le CM$.

Choose a smooth function $f$ agreeing with $F'$ on $[c_0k,\infty)$,
vanishing on $[0,c_0k/2]$, and satisfying
\[
 \|f^{(j)}\|_\infty\le Ck^{-j-1},
 \qquad j=0,1,2.
\]
The support assumption bounds both truncated logarithms in absolute
value by $C(\ell+H)$ and gives
$U^2+V^2\le2K_0n^2\ell^{(\alpha+3)/2}$.
Hence the discarded terms on $\{W<c_0k\}$, including the extension
to $\E[(U^2-V^2)f(W)]$, are bounded by
\[
 C\left\{
 (\ell+H)^2+\frac{n^2\ell^{(\alpha+3)/2}}{k}
 \right\}e^{-ck}
 \le n^{-102},
\]
because $k\ge s_1\asymp\ell^{3a}$.
The absolute difference of the raw second moments is therefore
at most
\[
 \left|\E[(U^2-V^2)f(W)]\right|
 +\frac{CM}{k^2}+n^{-102}.
\]
Taylor expansion at $\mu$, together with $\E U^2=\E V^2=1$, gives
\[
 \begin{aligned}
 \left|\E[(U^2-V^2)f(W)]\right|
 &\le\frac{C}{k^2}
 \left|\E[(U^2-V^2)(W-\mu)]\right|
 +\frac{C}{k^3}\E[(U^2+V^2)(W-\mu)^2]\\
 &\le\frac{CM}{k^2}
 +\frac{C}{k^3}
 \left(M^2k+K_0n\ell^{(\alpha+3)/2}M\right).
 \end{aligned}
\]
Here the covariance identity \eqref{ex:eq:qf-covariance} and
$Ru=Rv=0$ bound the first expectation by $CM$; the second is
bounded by Lemma~\ref{lem:exact-mixed-six}.

Lemma~\ref{lem:exact-rank-stability}, applied with
$B^2=K_0n\ell^{(\alpha+3)/2}$, gives the same bound for the
difference of the means, with its exponentially small term
absorbed into $n^{-102}$.
Moreover, Lemma~\ref{lem:exact-fixed-drift} gives
\[
 \begin{aligned}
 &\left|\E\log\left(\frac{X^TPX}{k}\vee\theta\right)\right|
 +\left|\E\log\left(\frac{X^TP'X}{k}\vee\theta\right)\right|
 \le\frac{CM}{k}
 \le C\ell^{\alpha-3a}=o(1).
 \end{aligned}
\]
The difference of the squared means thus satisfies the same bound.
Subtracting it from the difference of the raw second moments,
and absorbing the negligible terms into $n^{-100}$, proves
\eqref{eq:variance-rank-stability}.
\end{proof}

We return to the proof of Lemma~\ref{lem:variance-martingale}.
Apply Lemma~\ref{lem:variance-rank-stability} to each row replacement.
The Efron--Stein inequality \eqref{eq:efron-stein-definition} and
Cauchy--Schwarz yield
\[
 \E\left|\E_i\eta_{i+1}^2-q_i\right|
 \le\sqrt{\Var(\E_i\eta_{i+1}^2)}
 \le C\sqrt n\left(
 \frac{M^2}{k_i^2}
 +\frac{K_0n\ell^{(\alpha+3)/2}M}{k_i^3}
 +n^{-100}\right).
\]
Summing over $k_i\ge s_2$, we obtain
\begin{equation}\label{eq:variance-hybrid-conditional-variance}
 \begin{aligned}
 \sum_{k_i\ge s_2}\E\left|\E_i\eta_{i+1}^2-q_i\right|
 &\le C\sqrt n\left(
 \frac{M^2}{s_2}
 +\frac{K_0n\ell^{(\alpha+3)/2}M}{s_2^2}
 \right)+Cn^{-98}\\
 &\le Cn^{-1/2}\ell^{40a+2\alpha+2},
 \end{aligned}
\end{equation}
where we used $M\le K_1\ell^\alpha$ and
$s_2\asymp n\ell^{-20a}$.

For these indices, \eqref{eq:variance-hybrid-future-variance}
bounds the exponential in \eqref{eq:variance-hybrid-characteristic}
by $e^{-ct^2}$.
Using \eqref{eq:variance-hybrid-conditional-variance},
\eqref{eq:variance-hybrid-moments}, and
$\sum_{i<m}h_i\le CM$, the three integral contributions are at most
\[
 Cn^{-1/2}\ell^{40a+2\alpha+1},
 \qquad CM\ell^{-3/2},
 \qquad CM\ell^{-2}.
\]
For $s_1<k_i<s_2$, bound the exponential by one and use
\eqref{eq:common-h-terminal}.
Integrating $t$, $t^2$, and $t^3$ over $[0,\ell^{3/2}]$
gives the respective bounds
\[
 CM\ell^{-(2\alpha+10)},
 \qquad CM\ell^{-(2\alpha+9)},
 \qquad CM\ell^{-(2\alpha+8)}.
\]
Substituting into \eqref{eq:variance-hybrid-smoothing}, and using
$M\ge3$, gives
\[
 \dk(\mathcal M_m/s_\eta,\Normal)
 \le C\ell^{-3/2}
 +Cn^{-1/2}\ell^{40a+2\alpha+1}
 +CM\ell^{-3/2}
 \le CM\ell^{-3/2}.
\]

We now prove the bound $C(\log\ell)/\sqrt\ell$.
We first estimate the third absolute moments.
Recall that $\eta_{i+1}=L_{i+1}-\E_iL_{i+1}$.
Lemma~\ref{lem:exact-fixed-drift} gives, for all sufficiently large $n$,
\[
 |\E_iL_{i+1}|\le CM/k_i\le1,
 \qquad
 \eta_{i+1}\ge-H-1.
\]
Since $\E_iZ_{i+1}=1$, Lemma~\ref{lem:quadratic-forms} gives
$\E_i(Z_{i+1}-1)^2\le CM/k_i$.
Thus, for $u\ge2$,
\[
 \Pp_i(\eta_{i+1}>u)
 \le \Pp_i(Z_{i+1}>e^{u-1})
 \le \frac{CM}{k_i}e^{-2u}.
\]
Integrating this upper-tail bound above $2H+4$ and taking
expectations yields
\begin{equation}\label{eq:variance-hybrid-absolute-third}
 \begin{aligned}
 \E|\eta_{i+1}|^3
 &\le (2H+4)q_i
 +\frac{CM}{k_i}(H+1)^3e^{-4H}\\
 &\le C(\log\ell)q_i
 +\frac{CM}{k_i}(\log\ell)^3\ell^{-2a}.
 \end{aligned}
\end{equation}

Using the second-order Taylor expansion in the same Gaussian
replacements gives
\begin{equation}\label{eq:variance-hybrid-characteristic-second}
 \begin{aligned}
 \frac{\left|\varphi_{\mathcal M_m/\sqrt\ell}(t)
 -e^{-s_\eta^2t^2/(2\ell)}\right|}{t}
 &\le\frac{Ct}{\ell}\sum_{i<m}
 e^{-s_\eta^2t^2v_{\eta,i}/(2\ell)}
 \E\left|\E_i\eta_{i+1}^2-q_i\right|\\
 &\quad+\frac{Ct^2}{\ell^{3/2}}\sum_{i<m}
 e^{-s_\eta^2t^2v_{\eta,i}/(2\ell)}
 \E|\eta_{i+1}|^3,
 \qquad t>0.
 \end{aligned}
\end{equation}
Here the Gaussian remainder is absorbed using
$|e^{-y}-1+y|\le Cy^{3/2}$ for $y\ge0$ and
$q_i^{3/2}\le\E|\eta_{i+1}|^3$.

The integral of the first term in
\eqref{eq:variance-hybrid-characteristic-second}
has already been bounded by
\[
 Cn^{-1/2}\ell^{40a+2\alpha+1}
 +CM\ell^{-(2\alpha+10)}.
\]
For the second term, first consider $k_i\ge s_2$.
Since $q_i\le CM/k_i$ and
\eqref{eq:variance-hybrid-future-variance} holds,
\[
 \frac{q_i}{s_\eta^2v_{\eta,i}}
 \le \frac{CM}{s_2\ell}
 \le Cn^{-1}\ell^{20a+\alpha-1}
 =o(1).
\]
Consequently,
\[
 \begin{aligned}
 \sum_{k_i\ge s_2}
 \frac{q_i}{(s_\eta^2v_{\eta,i})^{3/2}}
 &\le C\sum_{k_i\ge s_2}
 \int_{s_\eta^2v_{\eta,i}}^{s_\eta^2v_{\eta,i}+q_i}
 x^{-3/2}\,\dd x
 \le C\int_{c\ell}^{\infty}x^{-3/2}\,\dd x
 \le C\ell^{-1/2}.
 \end{aligned}
\]
The first inequality uses $q_i\le s_\eta^2v_{\eta,i}$;
the integration intervals have disjoint interiors because
$s_\eta^2v_{\eta,i}=\sum_{j=i+1}^{m-1}q_j$.

Integrating the second term in
\eqref{eq:variance-hybrid-characteristic-second}
over these indices and using
\eqref{eq:variance-hybrid-absolute-third}, we obtain the bound
\[
 \begin{aligned}
 C\sum_{k_i\ge s_2}
 \frac{\E|\eta_{i+1}|^3}{(s_\eta^2v_{\eta,i})^{3/2}}
 &\le C\frac{\log\ell}{\sqrt\ell}
 +CM(\log\ell)^3\ell^{-2a-3/2}
   \sum_{k_i\ge s_2}\frac1{k_i}\\
 &\le C\frac{\log\ell}{\sqrt\ell}
 +CM(\log\ell)^3\ell^{-2a-1/2}.
 \end{aligned}
\]

For the terminal indices, \eqref{eq:variance-hybrid-moments}
and \eqref{eq:common-h-terminal} give
\[
 \sum_{s_1<k_i<s_2}q_i\le C\ell,
 \qquad
 \sum_{s_1<k_i<s_2}\E|\eta_{i+1}|^4
 \le CM\ell^{-(2\alpha+12)}.
\]
Cauchy--Schwarz therefore yields
\[
 \begin{aligned}
 \sum_{s_1<k_i<s_2}\E|\eta_{i+1}|^3
 &\le
 \left(\sum_{s_1<k_i<s_2}q_i\right)^{1/2}
 \left(\sum_{s_1<k_i<s_2}\E|\eta_{i+1}|^4\right)^{1/2}\\
 &\le C\sqrt M\,\ell^{-\alpha-11/2}.
 \end{aligned}
\]
Bounding the exponential by one and integrating
$t^2/\ell^{3/2}$ over $[0,\ell^{3/2}]$, their contribution is at most
\[
 C\sqrt M\,\ell^{-\alpha-5/2}
 \le C\ell^{-(\alpha+5)/2}.
\]

Combining these estimates with
\eqref{eq:variance-hybrid-smoothing}, and using
$M\le K_1\ell^\alpha$ and $a=20+4\lceil\alpha\rceil$, gives
\[
 \dk(\mathcal M_m/s_\eta,\Normal)
 \le C\frac{\log\ell}{\sqrt\ell}.
\]
Together with the first bound, this proves
\eqref{eq:variance-martingale-bound}.
\end{proof}

\begin{proof}[Proof of Lemma~\ref{lem:variance-remainder}]
We use the same choice $a=20+4\lceil\alpha\rceil$.
As in the proof of Lemma~\ref{lem:exact-drift}, split
\[
 \begin{aligned}
 \mathcal D_m&=\mathcal D_{\mathrm{early}}
              +\mathcal D_{\mathrm{terminal}},
 \quad
 \mathcal D_{\mathrm{early}}
 =\sum_{k_i\ge s_2}(g_i-\E g_i),
 \quad
 \mathcal D_{\mathrm{terminal}}
 =\sum_{s_1<k_i<s_2}(g_i-\E g_i).
 \end{aligned}
\]
For the early part, put $F=\sum_{k_i\ge s_2}g_i$ and let
$F^{(r)}$ be the same functional after replacing row $r$ by
an independent copy.
Then $\mathcal D_{\mathrm{early}}=F-\E F$.
By Lemma~\ref{lem:row-resample-geometry}, the projections before
and after this replacement satisfy the geometric hypothesis of
Lemma~\ref{lem:exact-rank-stability}.
Using the support bound \eqref{eq:variance-growing-support}
in \eqref{ex:eq:rank-stability}, and summing as in
\eqref{ex:eq:row-influence}, gives
\[
 \begin{aligned}
 |F-F^{(r)}|
 &\le C\left\{
 \frac{M^2}{s_2}
 +\frac{K_0n\ell^{(\alpha+3)/2}M}{s_2^2}
 +(H+K_0n^2\ell^{(\alpha+3)/2})e^{-cs_2}
 \right\}\\
 &\le Cn^{-1}\ell^{40a+2\alpha+2},
 \end{aligned}
\]
uniformly in $r$, since $M\le K_1\ell^\alpha$ and
$s_2\asymp n\ell^{-20a}$.
The Efron--Stein inequality \eqref{eq:efron-stein-definition}
therefore yields
\[
 \|\mathcal D_{\mathrm{early}}\|_2^2
 =\Var(F)
 \le\frac12\sum_r\E(F-F^{(r)})^2
 \le Cn^{-1}\ell^{80a+4\alpha+4}.
\]

For the terminal part, \eqref{eq:common-one-step-mean} and
\eqref{eq:common-h-terminal} give
\[
 \E|\mathcal D_{\mathrm{terminal}}|
 \le2\sum_{s_1<k_i<s_2}\E|g_i|
 \le CM\ell^{-(2\alpha+12)}.
\]
Also, Lemma~\ref{lem:exact-fixed-drift}, applied conditionally,
gives $|g_i|\le CM/k_i$, as in \eqref{ex:eq:g-point}. Hence
\[
 |\mathcal D_{\mathrm{terminal}}|
 \le CM\sum_{s_1<k_i<s_2}k_i^{-1}
 \le CM\ell.
\]
Consequently,
\[
 \|\mathcal D_{\mathrm{terminal}}\|_2^2
 \le CM\ell\,\E|\mathcal D_{\mathrm{terminal}}|
 \le CM^2\ell^{-(2\alpha+11)},
\]
and therefore
\[
 \|\mathcal D_{\mathrm{terminal}}\|_2
 \le CM\ell^{-\alpha-11/2}.
\]

For the lower-truncation remainder, the conditional small-ball
bound \eqref{ex:eq:exact-lower-tail-conditional} uses only the
density assumption and remains valid here. Since
$|\log Z_{i+1}-L_{i+1}|=(-\log Z_{i+1}-H)_+$, it gives
\[
 \begin{aligned}
 \E_i|\log Z_{i+1}-L_{i+1}|^2
 &=2\int_0^\infty t\,
   \Pp_i\bigl(Z_{i+1}<e^{-(H+t)}\bigr)\,\dd t
 \le\frac{32}{k_i^2}e^{-k_iH/4}.
 \end{aligned}
\]
By the centered definition of $\mathcal R_n$ following
\eqref{ex:eq:exact-hybrid-decomposition} and Minkowski's inequality,
\[
 \|\mathcal R_n\|_2
 \le\sum_{i<m}\|\log Z_{i+1}-L_{i+1}\|_2
 \le C\sum_{k=s_1+1}^{n}\frac{e^{-kH/8}}{k}
 \le e^{-cs_1H}
\]
for all sufficiently large $n$.
Combining the three bounds, we obtain
\[
 \|\mathcal D_m+\mathcal R_n\|_2
 \le Cn^{-1/2}\ell^{40a+2\alpha+2}
      +CM\ell^{-\alpha-11/2}
      +e^{-cs_1H}
 \le C\ell^{-2},
\]
where the last inequality uses $M\le K_1\ell^\alpha$ and the
scales in \eqref{eq:common-parameters} and
\eqref{eq:truncate-notation-definition}.
This proves \eqref{eq:variance-remainder-L2}.
\end{proof}

\section{Proof of examples}\label{sec:proof-examples}
We first verify the properties of the rare-spike random variables used
in Examples~\ref{ex:det-threshold} and~\ref{ex:exact-threshold}.  The
same argument applies to both examples; only the function $q$ changes.

\begin{lemma}\label{lem:spike-law}
Let
\[
 q(x)=\sqrt{\log(e+x)}
 \qquad\text{or}\qquad
 q(x)=\log(e+x),
\]
and let $Y$ be a fixed symmetric random variable with variance one and
a $C^\infty$ density supported on a bounded interval.  For $m\ge1$,
set
\begin{equation}\label{eq:spike-dimensions}
 B_m=e^m,
 \qquad
 N_m=
 \left\lfloor\frac{B_m^4}{q(B_m)}\right\rfloor,
 \qquad
 p_m=N_m^{-1},
 \qquad
 \tau_m=p_mB_m^2.
\end{equation}
Let $\eta_m$ be Bernoulli with parameter $p_m$, let $\varepsilon_m$
take the values $1$ and $-1$ with equal probability, and assume that
$Y,\eta_m,\varepsilon_m$ are independent.  Define
\begin{equation}\label{eq:spike-variable}
 \xi_m
 =
 \frac{Y+\eta_m\varepsilon_mB_m}{\sqrt{1+\tau_m}}.
\end{equation}
Then $(N_m)_{m\ge1}$ is strictly increasing,
\[
 N_m\longrightarrow\infty,
 \qquad
 \tau_m\longrightarrow0,
\]
and the variables $\xi_m$ are centered, have variance one, and have
$C^\infty$ densities bounded by a common constant.

Put
\[
 \kappa_m:=\E\xi_m^4-3,
 \qquad
 m_{8,m}:=\E|\xi_m|^8.
\]
Then
\begin{align}
 \kappa_m
 &=(1+o(1))q(B_m),
 \label{eq:spike-fourth-correction}\\
 m_{8,m}
 &\le CN_mq(B_m)^2,
 \label{eq:spike-eighth}\\
 \sup_{m\ge1}
 \E\frac{|\xi_m|^4}{q(|\xi_m|)}
 &<\infty.
 \label{eq:spike-critical-L1}
\end{align}
However, the family
\begin{equation}\label{eq:spike-critical-not-ui}
 \left\{
 \frac{|\xi_m|^4}{q(|\xi_m|)}:
 m\ge1
 \right\}
\end{equation}
is not uniformly integrable.

If $\ell_m=\log N_m$, then
\begin{align}
 q(x)=\sqrt{\log(e+x)}
 \quad&\Longrightarrow\quad
 \frac{p_mB_m^4}{\sqrt{\ell_m}}\longrightarrow\frac12,
 \quad
 \frac{\kappa_m}{\sqrt{\ell_m}}\longrightarrow\frac12,
 \quad
 m_{8,m}\le CN_m\ell_m,
 \label{eq:spike-det-scales}\\
 q(x)=\log(e+x)
 \quad&\Longrightarrow\quad
 \frac{p_mB_m^4}{\ell_m}\longrightarrow\frac14,
 \quad
 \frac{\kappa_m}{\ell_m}\longrightarrow\frac14,
 \quad
 m_{8,m}\le CN_m\ell_m^2.
 \label{eq:spike-exact-scales}
\end{align}
\end{lemma}

\begin{proof}[Proof of Lemma~\ref{lem:spike-law}]
Set
\[
 x_m:=\frac{B_m^4}{q(B_m)}.
\]
For either choice of $q$, one has $x_m\to\infty$ and
\[
 N_m=(1+o(1))x_m
 =(1+o(1))\frac{B_m^4}{q(B_m)}.
\]
Moreover, $\log(e+B_{m+1})\le2\log(e+B_m)$ for $m\ge1$, and hence
$q(B_{m+1})\le2q(B_m)$.  It follows that
\[
 x_{m+1}
 \ge
 \frac{e^4}{2}x_m.
\]
Since $x_1>1$, the sequence $(N_m)_{m\ge1}$ is strictly increasing.
Also,
\begin{equation}\label{eq:spike-tau-asymptotic}
 \tau_m
 =
 (1+o(1))\frac{q(B_m)}{B_m^2}
 \longrightarrow0,
 \qquad
 p_mB_m^4
 =
 (1+o(1))q(B_m).
\end{equation}

Let $h$ be the density of $Y$.  The density of $\xi_m$ is a mixture
of translated and rescaled copies of $h$.  Its supremum norm is at
most
\[
 \sqrt{1+\tau_m}\,\|h\|_\infty.
\]
Since $\sup_m\tau_m<\infty$, the densities of the variables $\xi_m$
are $C^\infty$ and bounded by a common constant.

Symmetry gives $\E Y=0$, and therefore $\E\xi_m=0$.  Moreover,
\[
 \E(Y+\eta_m\varepsilon_mB_m)^2
 =
 \E Y^2+p_mB_m^2
 =
 1+\tau_m,
\]
so $\E\xi_m^2=1$.

All mixed terms containing an odd power of $\varepsilon_m$ have
expectation zero.  Thus
\[
 \E(Y+\eta_m\varepsilon_mB_m)^4
 =
 \E Y^4+6\tau_m+p_mB_m^4.
\]
Consequently,
\[
 \kappa_m
 =
 \frac{
 p_mB_m^4+(\E Y^4-3)-3\tau_m^2
 }
 {(1+\tau_m)^2}.
\]
Since $q(B_m)\to\infty$, relation
\eqref{eq:spike-tau-asymptotic} proves
\eqref{eq:spike-fourth-correction}.

The bounded support of $Y$ and the inequality
$|u+v|^8\le2^7(|u|^8+|v|^8)$ give
\[
 m_{8,m}
 \le
 C\{1+p_mB_m^8\}.
\]
Furthermore,
\[
 p_mB_m^8
 \asymp
 B_m^4q(B_m)
 \asymp
 N_mq(B_m)^2,
\]
which proves \eqref{eq:spike-eighth}.

We next consider the weighted fourth moments.  Put
\[
 U_m:=\frac{|\xi_m|^4}{q(|\xi_m|)}.
\]
On $\{\eta_m=0\}$, the variables $\xi_m$ remain in a fixed bounded
interval, and hence
\[
 \sup_m\E[U_m;\eta_m=0]<\infty.
\]
On $\{\eta_m=1\}$, the bounded support of $Y$ and
$\tau_m\to0$ imply that, for all sufficiently large $m$,
\begin{equation}\label{eq:spike-size-comparison}
 cB_m\le|\xi_m|\le CB_m.
\end{equation}
For either choice of $q$, this gives
\begin{equation}\label{eq:spike-q-comparison}
 cq(B_m)
 \le
 q(|\xi_m|)
 \le
 Cq(B_m)
 \qquad\text{on }\{\eta_m=1\}.
\end{equation}
Together with \eqref{eq:spike-tau-asymptotic}, these estimates yield
\[
 \E[U_m;\eta_m=1]
 \le
 C\frac{p_mB_m^4}{q(B_m)}
 \le C.
\]
This proves \eqref{eq:spike-critical-L1}.

The reverse estimates in
\eqref{eq:spike-size-comparison}--\eqref{eq:spike-q-comparison} give
\[
 U_m
 \ge
 c\frac{B_m^4}{q(B_m)}
 \ge
 cN_m
 \qquad\text{on }\{\eta_m=1\}
\]
for all sufficiently large $m$.  Given $R>0$, choose $m$ so large
that $cN_m>R$.  Since $\Pp(\eta_m=1)=p_m=N_m^{-1}$,
\[
 \E\bigl[U_m\1_{\{U_m>R\}}\bigr]
 \ge
 \E[U_m;\eta_m=1]
 \ge
 cN_mp_m
 =
 c.
\]
Therefore the family in \eqref{eq:spike-critical-not-ui} is not
uniformly integrable.

Finally,
\begin{equation}\label{eq:spike-log-dimension}
 \ell_m
 =
 \log N_m
 =
 4\log B_m-\log q(B_m)+o(1).
\end{equation}
If $q(x)=\sqrt{\log(e+x)}$, then
$q(B_m)\sim\sqrt{\log B_m}$ and $\ell_m\sim4\log B_m$.
Together with
\eqref{eq:spike-tau-asymptotic},
\eqref{eq:spike-fourth-correction}, and
\eqref{eq:spike-eighth}, this proves
\eqref{eq:spike-det-scales}.  If $q(x)=\log(e+x)$, then
$q(B_m)\sim\log B_m$ and again $\ell_m\sim4\log B_m$.
The same three estimates prove \eqref{eq:spike-exact-scales}.
\end{proof}

For $q(x)=\sqrt{\log(e+x)}$, consider the array defined in
Example~\ref{ex:det-threshold}; for $q(x)=\log(e+x)$, consider the
array defined in Example~\ref{ex:exact-threshold}.  In either case,
at a dimension $n=N_m$, the entries in the first
$r_n:=\lfloor n/2\rfloor$ rows are independent copies of $\xi_m$,
and the remaining entries are independent standard Gaussian random
variables.  At every other dimension, all entries are independent
standard Gaussian random variables.

Lemma~\ref{lem:spike-law} shows that every entry is centered and has
variance one, and that all entries have $C^\infty$ densities bounded
by a common constant.  In particular, the first $i$ rows are linearly
independent almost surely for every $0\le i<n$.

Since $q\ge1$, a standard Gaussian random variable $g$ satisfies
\[
 \E\frac{|g|^4}{q(|g|)}<\infty.
\]
Together with \eqref{eq:spike-critical-L1}, this gives
\begin{equation}\label{eq:spike-array-critical-L1}
 \sup_{n\ge2}\max_{1\le i,j\le n}
 \E\frac{|a_{ij}|^4}{q(|a_{ij}|)}
 <\infty.
\end{equation}

At every dimension $n=N_m$, the entry $a_{11}^{(N_m)}$ has the same
distribution as $\xi_m$.  Hence
\[
 \left\{
 \frac{|a_{11}^{(N_m)}|^4}{q(|a_{11}^{(N_m)}|)}:
 m\ge1
 \right\}
\]
is not uniformly integrable by
\eqref{eq:spike-critical-not-ui}.  Therefore the full family
\begin{equation}\label{eq:spike-array-not-ui}
 \left\{
 \frac{|a_{ij}|^4}{q(|a_{ij}|)}:
 n\ge2,\ 1\le i,j\le n
 \right\}
\end{equation}
is not uniformly integrable.

We next record the conditional moment identities used in both
examples.  Fix one of the two examples, fix $m\ge1$, and put
$n=N_m$.  Let $a_1^T,\ldots,a_n^T$ denote the rows of $A_n$.  We
retain the notation of
\eqref{eq:common-parameters}--\eqref{eq:common-heights} and set
\begin{equation}
\begin{aligned}\label{eq:spike-height-definitions}
 Z_{i+1}
 &:=
 a_{i+1}^TQ_i a_{i+1},
 \quad 0\le i<n,\\
 X_{i+1}
 &:=
 Z_{i+1}-1,
 \quad
 d_i
 :=
 \sum_{j=1}^nq_{jj}(i)^2,
 \quad 0\le i<r_n,\\
 D_n
 &:=
 \sum_{i=0}^{r_n-1}d_i.
\end{aligned}
\end{equation}
For the rank-$k_{i}$ projection $P_i=(p_{rs}(i))_{1\le r,s\le n}$, Cauchy--Schwarz and $0\le p_{jj}(i)\le1$ give
\begin{equation}\label{eq:spike-energy-bounds}
 \frac1n\le d_i\le\frac1{k_i},
 \qquad
 \frac{r_n}{n}
 \le
 D_n
 \le
 \sum_{k=n-r_n+1}^{n}\frac1k
 =
 \log2+o(1).
\end{equation}

For $0\le i<r_n$, the entries of $a_{i+1}$ are independent copies
of $\xi_m$ and are independent of $\mathcal F_i$.  By
\eqref{eq:common-projection-traces}, $\E_iZ_{i+1}=\tr Q_i=1$.
Applying the covariance identity
\eqref{ex:eq:qf-covariance} with $C=D=Q_i$ and using $\E\xi_m^4=3+\kappa_m$
gives
\begin{align}
 \E_iX_{i+1}
 &=0,
 \label{eq:spike-conditional-mean}\\
 \E_iX_{i+1}^2
 &=
 2\tr Q_i^2
 +
 \kappa_m\sum_{j=1}^nq_{jj}(i)^2
 =
 \frac2{k_i}+\kappa_md_i.
 \label{eq:spike-conditional-variance}
\end{align}
Splitting $X_{i+1}$ into its diagonal and off-diagonal parts and
applying Lemma~\ref{lem:quadratic-forms} gives
\begin{equation}\label{eq:spike-conditional-fourth-general}
 \E_i|X_{i+1}|^4
 \le
 C\left\{
 \frac{m_{8,m}}{k_i^3}
 +
 \frac{(\E\xi_m^4)^2}{k_i^2}
 \right\}.
\end{equation}
Since $k_i\ge n/2$ for $0\le i<r_n$,
\eqref{eq:spike-det-scales}--\eqref{eq:spike-exact-scales} imply
\begin{align}
 q(x)=\sqrt{\log(e+x)}
 \quad&\Longrightarrow\quad
 \E_i|X_{i+1}|^4
 \le
 \frac{C\ell}{n^2},
 \label{eq:spike-conditional-fourth-det}\\
 q(x)=\log(e+x)
 \quad&\Longrightarrow\quad
 \E_i|X_{i+1}|^4
 \le
 \frac{C\ell^2}{n^2}.
 \label{eq:spike-conditional-fourth-exact}
\end{align}

\subsection{Proof of Example~\ref{ex:det-threshold}}
\label{subsec:proof-det-example}

\begin{proof}[Proof of Example~\ref{ex:det-threshold}]

Apply Lemma~\ref{lem:spike-law} with
$q(x)=\sqrt{\log(e+x)}$.
Equations~\eqref{eq:spike-array-critical-L1} and
\eqref{eq:spike-array-not-ui} give
\eqref{eq:det-example-critical-bound} and the failure of uniform
integrability in \eqref{eq:det-clt-assumption}, respectively.
It remains to prove \eqref{eq:det-example-failure}.

For the rest of the proof, fix $m$ and put $n=N_m$.  Write
\[
 r_n=\lfloor n/2\rfloor,
 \qquad
 s_n=n-r_n,
 \qquad
 \ell=\log n.
\]
We use the notation in \eqref{eq:spike-height-definitions}. 

\medskip
\noindent
\emph{Step 1: expansion of the heavy block.}
Define
\begin{equation}\label{eq:det-example-heavy-log}
 S_{H,n}
 :=
 \sum_{i=0}^{r_n-1}\log Z_{i+1}
 =
 \sum_{i=0}^{r_n-1}\log(1+X_{i+1}).
\end{equation}
By \eqref{eq:spike-conditional-mean}, each $X_{i+1}$ is
$\mathcal F_{i+1}$-measurable and satisfies $\E(X_{i+1}\mid\mathcal F_i)=0$.
Thus $(X_{i+1})_{0\le i<r_n}$ is a martingale-difference array
with respect to the filtration $(\mathcal F_i)_{i\ge0}$.
By \eqref{eq:spike-conditional-mean}, the variables
$X_{i+1}$ are martingale differences with respect to
$(\mathcal F_{i+1})_{i\ge0}$.  Their orthogonality,
\eqref{eq:spike-conditional-variance}, and
\eqref{eq:spike-energy-bounds} give
\begin{align}
 \E\left(\sum_{i<r_n}X_{i+1}\right)^2
 &=
 \E\sum_{i<r_n}\E_iX_{i+1}^2
=
 2\sum_{k=s_n+1}^{n}\frac1k
 +
 \kappa_m\E D_n
 =
 O(\sqrt\ell).
 \label{eq:det-example-linear-bound}
\end{align}
Indeed, $D_n\le\log2+o(1)$ and $\kappa_m=O(\sqrt\ell)$
by \eqref{eq:spike-fourth-correction} and
\eqref{eq:spike-det-scales}.  Dividing
\eqref{eq:det-example-linear-bound} by $\ell$ and applying
Chebyshev's inequality yields
\begin{equation}\label{eq:det-example-linear-negligible}
 \frac1{\sqrt\ell}\sum_{i<r_n}X_{i+1}
 \xrightarrow{\Pp}0.
\end{equation}

Next, the variables $X_{i+1}^2-\E_iX_{i+1}^2$, $0\le i<r_n$,
are also martingale differences.  Hence
\begin{align}
 \E\left|
 \sum_{i<r_n}
 \bigl(X_{i+1}^2-\E_iX_{i+1}^2\bigr)
 \right|^2
 &=
 \sum_{i<r_n}
 \E\left|
 X_{i+1}^2-\E_iX_{i+1}^2
 \right|^2
 \le
 \sum_{i<r_n}\E|X_{i+1}|^4
 \le
 \frac{C\ell}{n}
 \longrightarrow0,
 \label{eq:det-example-quadratic-concentration}
\end{align}
where the last inequality follows from
\eqref{eq:spike-conditional-fourth-det}.  Combining
\eqref{eq:det-example-quadratic-concentration} with
\eqref{eq:spike-conditional-variance}, we obtain
\begin{align}
 \sum_{i<r_n}X_{i+1}^2
 &=
 \sum_{i<r_n}\E_iX_{i+1}^2
 +
 o_{\Pp}(1)
 =
 2\sum_{k=s_n+1}^{n}\frac1k
 +
 \kappa_mD_n
 +
 o_{\Pp}(1)
 =
 O_{\Pp}(\sqrt\ell).
 \label{eq:det-example-square-sum}
\end{align}

The same fourth-moment estimate shows that the largest increment is
small.  For every fixed $\eta>0$,
\begin{align*}
 \Pp\left(
 \max_{i<r_n}|X_{i+1}|>\eta
 \right)
 &\le
 \sum_{i<r_n}
 \Pp(|X_{i+1}|>\eta)
 \le
 \eta^{-4}
 \sum_{i<r_n}\E|X_{i+1}|^4
 \le
 \frac{C\ell}{n\eta^4}
 \longrightarrow0.
\end{align*}
Therefore
\begin{equation}\label{eq:det-example-max-small}
 \max_{i<r_n}|X_{i+1}|=o_{\Pp}(1).
\end{equation}

Define the Taylor remainder
\begin{equation}\label{eq:det-example-heavy-remainder}
 R_{H,n}
 :=
 S_{H,n}
 -
 \sum_{i<r_n}X_{i+1}
 +
 \frac12\sum_{i<r_n}X_{i+1}^2.
\end{equation}
On the event $\{\max_{i<r_n}|X_{i+1}|\le\frac12\}$,
Taylor's formula gives
\[
 \left|
 \log(1+x)-x+\frac{x^2}{2}
 \right|
 \le
 C|x|^3,
 \qquad
 |x|\le\frac12.
\]
Consequently,
\begin{align*}
 |R_{H,n}|
 &\le
 C\sum_{i<r_n}|X_{i+1}|^3
 \le
 C\max_{i<r_n}|X_{i+1}|
 \sum_{i<r_n}X_{i+1}^2.
\end{align*}
Equations \eqref{eq:det-example-square-sum} and
\eqref{eq:det-example-max-small}, together with the fact that the
preceding event has probability tending to one, imply
\begin{equation}\label{eq:det-example-heavy-remainder-bound}
 R_{H,n}=o_{\Pp}(\sqrt\ell).
\end{equation}

Substituting
\eqref{eq:det-example-linear-negligible},
\eqref{eq:det-example-square-sum}, and
\eqref{eq:det-example-heavy-remainder-bound} into
\eqref{eq:det-example-heavy-remainder}, we conclude that
\begin{equation}\label{eq:det-example-heavy-expansion}
 S_{H,n}
 =
 -\sum_{k=s_n+1}^{n}\frac1k
 -
 \frac{\kappa_m}{2}D_n
 +
 o_{\Pp}(\sqrt\ell).
\end{equation}

\medskip
\noindent
\emph{Step 2: the Gaussian block.}
Let
\begin{equation}\label{eq:det-example-gaussian-block}
 S_{G,n}
 :=
 \sum_{i=r_n}^{n-1}\log Z_{i+1}
\end{equation}
be the logarithmic sum of the actual normalized squared
Gram--Schmidt heights in the Gaussian block.

Conditional on $\mathcal F_{r_n}$, the remaining rows are independent
standard Gaussian vectors.  Rotational invariance and the
Gram--Schmidt construction show that the corresponding normalized
squared heights are independent and have respective distributions
\[
 \frac{\chi_{s_n}^2}{s_n},
 \frac{\chi_{s_n-1}^2}{s_n-1},
 \ldots,
 \chi_1^2.
\]
Moreover, this conditional joint distribution does not depend on
$\mathcal F_{r_n}$.  Thus $S_{G,n}$ is independent of the first block
and
\begin{equation}\label{eq:det-example-gaussian-law}
 S_{G,n}
 \stackrel{d}{=}
 \sum_{k=1}^{s_n}\log(\chi_k^2/k),
\end{equation}
where the chi-square variables on the right-hand side are
independent.

Define
\[
 G_n
 :=
 \frac{S_{G,n}-\E S_{G,n}}{\sqrt{2\ell}}.
\]
Lemma~\ref{lem:gaussian-terminal-block} gives
\begin{align}
 \E S_{G,n}
 &=
 -\sum_{k=1}^{s_n}\frac1k+O(1),
 \label{eq:det-example-gaussian-mean}\\
 \Var(S_{G,n})
 &=
 2\sum_{k=1}^{s_n}\frac1k+O(1)
 =
 2\ell+O(1),
 \label{eq:det-example-gaussian-variance}\\
 \sum_{k=1}^{s_n}
 \E\left|
 \log(\chi_k^2/k)-\E\log(\chi_k^2/k)
 \right|^3
 &\le C.
 \label{eq:det-example-gaussian-third}
\end{align}

Let $\sigma_{G,n}^2:=\Var(S_{G,n})$.
The classical Berry--Esseen inequality and
\eqref{eq:det-example-gaussian-third} imply
\[
 \dk\left(
 \frac{S_{G,n}-\E S_{G,n}}{\sigma_{G,n}},
 \Normal
 \right)
 \le
 \frac{C}{\sigma_{G,n}^3}
 \longrightarrow0.
\]
Furthermore, $\frac{\sigma_{G,n}}{\sqrt{2\ell}} \rightarrow1$
by \eqref{eq:det-example-gaussian-variance}.  Slutsky's theorem
therefore gives
\begin{equation}\label{eq:det-example-gaussian-asymptotics}
 G_n\Rightarrow\Normal.
\end{equation}

\medskip
\noindent
\emph{Step 3: recombining the two blocks.}
Because the entries of $A_n$ have densities, the matrix $A_n$ and
all intermediate collections of rows have full rank almost surely.
For every row $i+1$, the squared Gram--Schmidt height equals
$a_{i+1}^TP_i a_{i+1} = k_iZ_{i+1}$.
The product formula for the determinant therefore gives
\begin{equation}\label{eq:det-example-gram-schmidt}
 \log\det(A_nA_n^T)
 =
 \sum_{i=0}^{n-1}\log(k_iZ_{i+1})
 =
 \log n!+S_{H,n}+S_{G,n}.
\end{equation}
Since $2\log|\det A_n| = \log\det(A_nA_n^T)$,
the definition of $W_n^{\mathrm d}$ yields
\begin{equation}\label{eq:det-example-W-representation}
 W_n^{\mathrm d}(A_n)
 =
 \frac{
 \log n+S_{H,n}+S_{G,n}
 }{\sqrt{2\ell}}.
\end{equation}
Let $H_n:=\sum_{k=1}^{n}\frac1k$. By \eqref{eq:det-example-heavy-expansion} and
\eqref{eq:det-example-gaussian-mean},
\[
 -\sum_{k=s_n+1}^{n}\frac1k
 +
 \E S_{G,n}
 =
 -H_n+O(1).
\]
Since $\log n-H_n=O(1)$,
and $S_{G,n} = \E S_{G,n}+\sqrt{2\ell}\,G_n$,
substitution into \eqref{eq:det-example-W-representation} gives
\begin{equation}\label{eq:det-example-final-shift}
 W_n^{\mathrm d}(A_n)
 =
 G_n
 -
 \frac{\kappa_mD_n}{2\sqrt{2\ell}}
 +
 o_{\Pp}(1).
\end{equation}

\medskip
\noindent
\emph{Step 4: the shift does not vanish.}
By \eqref{eq:spike-fourth-correction} and
\eqref{eq:spike-det-scales}, $\frac{\kappa_m}{\sqrt\ell} \rightarrow \frac12$.
Moreover, \eqref{eq:spike-energy-bounds} gives the almost-sure lower
bound $D_n\ge\frac{r_n}{n}$.
Consequently, almost surely,
\[
 \frac{\kappa_mD_n}{2\sqrt{2\ell}}
 \ge
 \frac{\kappa_m}{2\sqrt{2\ell}}
 \frac{r_n}{n},
\]
and the deterministic expression on the right converges to
$\frac1{2\sqrt2} \cdot\frac12 \cdot\frac12 = \frac1{8\sqrt2}$.
It follows that, with $a_0:=\frac1{16\sqrt2}$,
we have
\begin{equation}\label{eq:det-example-shift-lower}
 \frac{\kappa_mD_n}{2\sqrt{2\ell}}
 \ge a_0
 \qquad\text{almost surely}
\end{equation}
for every sufficiently large $m$.

Write \eqref{eq:det-example-final-shift} as
\[
 W_n^{\mathrm d}(A_n)
 =
 G_n
 -
 \frac{\kappa_mD_n}{2\sqrt{2\ell}}
 +
 R_n,
 \qquad
 R_n=o_{\Pp}(1).
\]
Let $\Phi$ denote the standard normal distribution function.  Fix
$0<\epsilon<a_0$.  On the event
\[
 \{G_n\le a_0-\epsilon\}
 \cap
 \{|R_n|\le\epsilon\},
\]
equation \eqref{eq:det-example-shift-lower} gives
\[
 W_n^{\mathrm d}(A_n)
 \le
 (a_0-\epsilon)-a_0+\epsilon
 =
 0.
\]
Therefore, for every sufficiently large $m$,
\begin{align*}
 \Pp\bigl(W_n^{\mathrm d}(A_n)\le0\bigr)
 &\ge
 \Pp\bigl(
 G_n\le a_0-\epsilon,\ |R_n|\le\epsilon
 \bigr)\\
 &\ge
 \Pp(G_n\le a_0-\epsilon)
 -
 \Pp(|R_n|>\epsilon).
\end{align*}
Using \eqref{eq:det-example-gaussian-asymptotics} and
$R_n=o_{\Pp}(1)$, we obtain
\[
 \liminf_{m\to\infty}
 \Pp\bigl(
 W_{N_m}^{\mathrm d}(A_{N_m})\le0
 \bigr)
 \ge
 \Phi(a_0-\epsilon).
\]
Letting $\epsilon\downarrow0$ gives
\[
 \liminf_{m\to\infty}
 \Pp\bigl(
 W_{N_m}^{\mathrm d}(A_{N_m})\le0
 \bigr)
 \ge
 \Phi(a_0)
 >
 \frac12.
\]
Since the standard normal distribution function equals $1/2$ at
zero,
\begin{align*}
 \liminf_{m\to\infty}
 \dk\bigl(
 W_{N_m}^{\mathrm d}(A_{N_m}),
 \Normal
 \bigr)
 &\ge
 \liminf_{m\to\infty}
 \left|
 \Pp\bigl(
 W_{N_m}^{\mathrm d}(A_{N_m})\le0
 \bigr)
 -
 \frac12
 \right|\\
 &\ge
 \Phi(a_0)-\frac12
 >
 0.
\end{align*}
This proves \eqref{eq:det-example-failure}.
\end{proof}

\subsection{Proof of Example~\ref{ex:exact-threshold}}
\label{subsec:proof-exact-example}

\begin{proof}[Proof of Example~\ref{ex:exact-threshold}]

Apply Lemma~\ref{lem:spike-law} with $q(x)=\log(e+x)$.
This gives the stated density bound, while
\eqref{eq:spike-array-critical-L1} and
\eqref{eq:spike-array-not-ui} give
\eqref{eq:exact-example-critical-bound} and the failure of uniform
integrability in \eqref{eq:exact-clt-assumption}, respectively.
It remains to establish \eqref{eq:exact-example-limit}.
  
For each $m$, put $n=N_m$ and write
\[
 r_n=\lfloor n/2\rfloor,
 \qquad
 s_n=n-r_n,
 \qquad
 \ell=\log n.
\]
We retain the notation of \eqref{eq:common-parameters}-\eqref{eq:common-heights} and use the notation in \eqref{eq:spike-height-definitions}.

\medskip
\noindent
\emph{Step 1: concentration of the projection energy.}
For $1\le t<r_n$, replace the $t$ th heavy row by an independent copy,
and let $D_n^{(t)}$ denote the resulting value of $D_n$ defined in \eqref{eq:spike-height-definitions}. Outside a
null event, both the original and resampled collections of rows are
linearly independent at every intermediate stage.  Terms with $i<t$
do not change.  For $t\le i<r_n$, Lemma~\ref{lem:row-resample-geometry}
shows that the ranges of the old and new rank-$k_i$ projections
$P_i$ and $P_i^{(t)}$ have an intersection of dimension at least
$k_i-1$.

If $P_i=P_i^{(t)}$, then $d_i=d_i^{(t)}$.  Otherwise, the intersection
has dimension $k_i-1$, and the two projections may be written as
\[
 P_i=R+uu^T,
 \qquad
 P_i^{(t)}=R+vv^T,
\]
where $R$ is the orthogonal projection onto the common intersection
and $u,v$ are unit vectors orthogonal to its range.  Writing
$P_i=(p_{jk}(i))$ and $P_i^{(t)}=(p_{jk}^{(t)}(i))$, we have
\begin{align*}
 \left|
 \sum_{j=1}^n p_{jj}(i)^2
 -
 \sum_{j=1}^n p_{jj}^{(t)}(i)^2
 \right|
 &\le
 \sum_{j=1}^n
 |p_{jj}(i)-p_{jj}^{(t)}(i)|
 \{p_{jj}(i)+p_{jj}^{(t)}(i)\}\\
 &\le
 2\sum_{j=1}^n|u_j^2-v_j^2|\\
 &\le4.
\end{align*}
Consequently,
\[
 |d_i-d_i^{(t)}|
 \le
 \frac4{k_i^2}.
\]
Since $k_i\ge s_n+1\ge n/2$ throughout the heavy block,
\begin{equation}\label{eq:exact-example-row-influence}
 |D_n-D_n^{(t)}|
 \le
 4\sum_{i=t}^{r_n-1}\frac1{k_i^2}
 \le
 4\sum_{k\ge n/2}\frac1{k^2}
 \le
 \frac Cn.
\end{equation}
The Efron--Stein inequality \eqref{eq:efron-stein-definition}, applied
to the independent heavy rows on which $D_n$ depends, therefore gives
\begin{equation}\label{eq:exact-example-D-variance}
 \Var(D_n)
 \le
 \frac12\sum_{t<r_n}
 \E(D_n-D_n^{(t)})^2
 \le
 \frac Cn.
\end{equation}

Taking $n=N_m$ in \eqref{eq:spike-energy-bounds}, we obtain
\[
 \frac{r_{N_m}}{N_m}
 \le
 \E D_{N_m}
 \le
 \sum_{j=s_{N_m}+1}^{N_m}\frac1j
 =
 \log2+o(1).
\]
Hence the sequence $(\E D_{N_m})_{m\ge1}$ is bounded.  There therefore
exist indices $m_k\uparrow\infty$ and a number
\[
 d_0\in\left[\frac12,\log2\right]
\]
such that
\begin{equation}\label{eq:exact-example-D-limit}
 \E D_{N_{m_k}}\longrightarrow d_0.
\end{equation}
Moreover, \eqref{eq:exact-example-D-variance} gives $\Var(D_{N_{m_k}}) \le \frac{C}{N_{m_k}} \longrightarrow0$.
Consequently,
\begin{equation}\label{eq:exact-example-D-probability}
 D_{N_{m_k}}\xrightarrow{\Pp}d_0.
\end{equation}
Define
\begin{equation}\label{eq:exact-example-sigma-definition}
 \sigma^2
 :=
 1+\frac{d_0}{8}.
\end{equation}
Since $d_0\in[1/2,\log2]$, we have $\sigma^2 \in \left[ \frac{17}{16}, 1+\frac{\log2}{8} \right]$.

\medskip
\noindent
\emph{Step 2: a martingale central limit theorem for the heavy block.}
For each $k$, consider the Gram--Schmidt variables associated with
the matrix $A_{N_{m_k}}$, and define
\begin{equation}\label{eq:exact-example-martingale}
 \begin{aligned}
 M_{H,N_{m_k}}
 &:=
 \sum_{i=0}^{r_{N_{m_k}}-1}X_{i+1},
 \quad
 V_{N_{m_k}}
 :=
 \sum_{i=0}^{r_{N_{m_k}}-1}\E_iX_{i+1}^2.
 \end{aligned}
\end{equation}
Here $X_{i+1}$ and $\E_i$ refer to the variables and conditional
expectations for $A_{N_{m_k}}$.
By \eqref{eq:spike-conditional-variance},
\begin{equation}\label{eq:exact-example-V-formula}
 V_{N_{m_k}}
 =
 2\sum_{j=s_{N_{m_k}}+1}^{N_{m_k}}\frac1j
 +
 \kappa_{m_k}D_{N_{m_k}}.
\end{equation}
For the present choice $q(x)=\log(e+x)$,
\eqref{eq:spike-exact-scales} gives $\frac{\kappa_{m_k}}{\log N_{m_k}} \longrightarrow \frac14$.
Combining this with \eqref{eq:exact-example-D-probability} and the fact that $\sum_{j=s_{N_{m_k}}+1}^{N_{m_k}}\frac1j=O(1)$,
we obtain
\begin{equation}\label{eq:exact-example-V-limit}
 \frac{V_{N_{m_k}}}{\log N_{m_k}}
 \xrightarrow{\Pp}
 \frac{d_0}{4}.
\end{equation}
For $\varepsilon>0$, define
\[
 L_{N_{m_k}}(\varepsilon)
 :=
 \frac{1}{\log N_{m_k}}
 \sum_{i=0}^{r_{N_{m_k}}-1}
 \E_i\left[
 X_{i+1}^2
 \1_{\{
 |X_{i+1}|>\varepsilon\sqrt{\log N_{m_k}}
 \}}
 \right].
\]
Using $x^2\1_{\{ |x|>\varepsilon\sqrt{\log N_{m_k}} \}} \le \frac{x^4}{\varepsilon^2\log N_{m_k}}$
and \eqref{eq:spike-conditional-fourth-exact}, we obtain
\begin{align}
 \E L_{N_{m_k}}(\varepsilon)
 &\le
 \frac{1}{
 \varepsilon^2(\log N_{m_k})^2
 }
 \sum_{i=0}^{r_{N_{m_k}}-1}\E|X_{i+1}|^4
 \le
 \frac{C}{N_{m_k}}
 \longrightarrow0.
 \label{eq:exact-example-lindeberg}
\end{align}
Thus $L_{N_{m_k}}(\varepsilon)\xrightarrow{\Pp}0$,
by Markov's inequality.  The martingale Lindeberg--Feller theorem,
together with \eqref{eq:exact-example-V-limit}, therefore gives
\begin{equation}\label{eq:exact-example-martingale-clt}
 \frac{M_{H,N_{m_k}}}
 {\sqrt{\log N_{m_k}}}
 \Rightarrow
 \mathcal N\left(0,\frac{d_0}{4}\right).
\end{equation}

\medskip
\noindent
\emph{Step 3: replacing the heavy logarithmic sum by its linear part.}
For each $k$, define
\begin{equation}\label{eq:exact-example-heavy-log}
 S_{H,N_{m_k}}
 :=
 \sum_{i=0}^{r_{N_{m_k}}-1}\log Z_{i+1}
 =
 \sum_{i=0}^{r_{N_{m_k}}-1}\log(1+X_{i+1}),
\end{equation}
and
\begin{equation}\label{eq:exact-example-Taylor-remainder}
 T_{H,N_{m_k}}
 :=
 S_{H,N_{m_k}}
 -
 M_{H,N_{m_k}}
 +
 \frac12
 \sum_{i=0}^{r_{N_{m_k}}-1}X_{i+1}^2.
\end{equation}
We first prove
\begin{equation}\label{eq:exact-example-Taylor-L1}
 \E|T_{H,N_{m_k}}|
 \longrightarrow0.
\end{equation}

For $0\le i<r_{N_{m_k}}$, we have $N_{m_k}-i \ge s_{N_{m_k}}+1 \ge \frac{N_{m_k}}2$.
Hence \eqref{eq:spike-conditional-variance}, \eqref{eq:spike-energy-bounds}, and $|\kappa_{m_k}|=O(\log N_{m_k})$
give
\[
 \E_iX_{i+1}^2
 \le
 \frac{2}{N_{m_k}-i}
 +
 \frac{|\kappa_{m_k}|}{N_{m_k}-i}
 \le
 \frac{C\log N_{m_k}}{N_{m_k}}.
\]
Together with \eqref{eq:spike-conditional-fourth-exact} and
Cauchy--Schwarz, this yields
\begin{align}
 \sum_{i=0}^{r_{N_{m_k}}-1}\E|X_{i+1}|^3
 &\le
 \sum_{i=0}^{r_{N_{m_k}}-1}
 (\E X_{i+1}^2)^{1/2}
 (\E|X_{i+1}|^4)^{1/2}
 \notag\\
 &\le
 CN_{m_k}
 \left(
 \frac{\log N_{m_k}}{N_{m_k}}
 \right)^{1/2}
 \left(
 \frac{(\log N_{m_k})^2}{N_{m_k}^2}
 \right)^{1/2}
 \notag\\
 &=
 O\left(
 \frac{(\log N_{m_k})^{3/2}}
 {\sqrt{N_{m_k}}}
 \right)
 \longrightarrow0.
 \label{eq:exact-example-third-sum}
\end{align}
For $x>-1$, put $r(x):=\log(1+x)-x+\frac{x^2}{2}$.
On $|x|\le1/2$, $|r(x)|\le C|x|^3$, while on $x>1/2$, $|r(x)|\le Cx^4$.
Finally, on $-1<x<-1/2$,
\[
 |r(x)|\le C\{1-\log(1+x)\}.
\]
It remains to control the logarithmic singularity in the last
region.  Let $u_0$ be the constant in
Lemma~\ref{lem:projection-density}, corresponding to the common
density bound of the entries.  Since
$Z_{i+1}=1+X_{i+1}$ and $Z_{i+1}<1/2$ implies
$|X_{i+1}|>1/2$, the layer-cake identity
\eqref{eq:layer-cake-identity}, conditional Markov's inequality, and
\eqref{ex:eq:projection-smallball} give
\begin{align}
 &\E_i[-\log Z_{i+1};\,Z_{i+1}<1/2]
 \notag\\
 &\quad=
 (\log2)\Pp_i(Z_{i+1}<1/2)
 +
 \int_{\log2}^{\infty}
 \Pp_i(Z_{i+1}<e^{-u})\,\dd u
 \notag\\
 &\quad\le
 C\E_i|X_{i+1}|^4
 +
 \int_{u_0}^{\infty}
 e^{-(N_{m_k}-i)u/4}\,\dd u
 \notag\\
 &\quad\le
 C\E_i|X_{i+1}|^4
 +
 Ce^{-c(N_{m_k}-i)}.
 \label{eq:exact-example-negative-tail}
\end{align}
For $u\in[\log 2,u_0]$, we have $e^{-u}\le1/2$, and hence
\[
 \{Z_{i+1}<e^{-u}\}
 \subseteq
 \{|X_{i+1}|>1/2\},
\]
because $Z_{i+1}=1+X_{i+1}$.  Therefore,
\[
 \begin{aligned}
 \int_{\log 2}^{u_0}
 \Pp_i(Z_{i+1}<e^{-u})\,\dd u
 &\le
 (u_0-\log 2)\Pp_i(|X_{i+1}|>1/2)\\
 &\le
 C\E_i|X_{i+1}|^4,
 \end{aligned}
\]
where the last inequality follows from Markov's inequality.
Using the three bounds for $r(x)$,
\eqref{eq:exact-example-third-sum},
\eqref{eq:exact-example-negative-tail}, and
\eqref{eq:spike-conditional-fourth-exact}, we obtain
\[
 \begin{aligned}
 \E|T_{H,N_{m_k}}|
 &\le
 C\sum_{i=0}^{r_{N_{m_k}}-1}\E|X_{i+1}|^3
 +
 C\sum_{i=0}^{r_{N_{m_k}}-1}\E|X_{i+1}|^4
 +
 C\sum_{i=0}^{r_{N_{m_k}}-1}
 e^{-c(N_{m_k}-i)}\\
 &\le
 o(1)
 +
 C\frac{(\log N_{m_k})^2}{N_{m_k}}
 +
 CN_{m_k}e^{-cN_{m_k}/2}
 =
 o(1).
 \end{aligned}
\]
This proves \eqref{eq:exact-example-Taylor-L1}.

We next show that the centered quadratic term is negligible on the
scale $\sqrt{\log N_{m_k}}$.  From
\eqref{eq:spike-conditional-variance},
\begin{align}
 &\sum_{i=0}^{r_{N_{m_k}}-1}
 \{X_{i+1}^2-\E X_{i+1}^2\}
 \notag
 =
 \sum_{i=0}^{r_{N_{m_k}}-1}
 \{X_{i+1}^2-\E_iX_{i+1}^2\}
 +
 \kappa_{m_k}
 \{D_{N_{m_k}}-\E D_{N_{m_k}}\}.
 \label{eq:exact-example-quadratic-decomposition}
\end{align}
The first term on the right is a martingale sum and satisfies
\[
 \begin{aligned}
 &\E\left|
 \sum_{i=0}^{r_{N_{m_k}}-1}
 \{X_{i+1}^2-\E_iX_{i+1}^2\}
 \right|^2
 \le
 \sum_{i=0}^{r_{N_{m_k}}-1}\E|X_{i+1}|^4
 \le
 C\frac{(\log N_{m_k})^2}{N_{m_k}}.
 \end{aligned}
\]
For the second term,
\eqref{eq:spike-exact-scales} and
\eqref{eq:exact-example-D-variance} give
\[
 \begin{aligned}
 &\E\left[
 \frac{
 \kappa_{m_k}
 \{D_{N_{m_k}}-\E D_{N_{m_k}}\}
 }
 {\sqrt{\log N_{m_k}}}
 \right]^2
 \le
 C\frac{\log N_{m_k}}{N_{m_k}}
 \longrightarrow0.
 \end{aligned}
\]
Consequently,
\begin{equation}\label{eq:exact-example-quadratic-negligible}
 \frac{1}{\sqrt{\log N_{m_k}}}
 \sum_{i=0}^{r_{N_{m_k}}-1}
 \{X_{i+1}^2-\E X_{i+1}^2\}
 \xrightarrow{\Pp}0.
\end{equation}
From \eqref{eq:exact-example-Taylor-remainder},
\[
 \begin{aligned}
 S_{H,N_{m_k}}-\E S_{H,N_{m_k}}
 &=
 M_{H,N_{m_k}}
 -
 \frac12
 \sum_{i=0}^{r_{N_{m_k}}-1}
 \{X_{i+1}^2-\E X_{i+1}^2\}
 +
 T_{H,N_{m_k}}-\E T_{H,N_{m_k}}.
 \end{aligned}
\]
Equation \eqref{eq:exact-example-Taylor-L1} implies $T_{H,N_{m_k}}-\E T_{H,N_{m_k}} = o_{\Pp}(1)$.
Therefore,
\eqref{eq:exact-example-martingale-clt} and
\eqref{eq:exact-example-quadratic-negligible} yield
\begin{equation}\label{eq:exact-example-heavy-centered}
 \frac{
 S_{H,N_{m_k}}-\E S_{H,N_{m_k}}
 }
 {\sqrt{\log N_{m_k}}}
 \Rightarrow
 \mathcal N\left(0,\frac{d_0}{4}\right).
\end{equation}

\medskip
\noindent
\emph{Step 4: the Gaussian block.}
For each $k$, let
\begin{equation}\label{eq:exact-example-gaussian-block}
 S_{G,N_{m_k}}
 :=
 \sum_{i=r_{N_{m_k}}}^{N_{m_k}-1}\log Z_{i+1}.
\end{equation}
This is the logarithmic sum of the normalized squared
Gram--Schmidt heights in the Gaussian block.  By
Lemma~\ref{lem:gaussian-terminal-block},
$S_{G,N_{m_k}}$ is independent of
$\mathcal F_{r_{N_{m_k}}}$ and
\[
 S_{G,N_{m_k}}
 \stackrel d=
 \sum_{j=1}^{s_{N_{m_k}}}
 \log(\chi_j^2/j),
\]
where the chi-square variables on the right are independent.  The
same lemma gives
\[
 \begin{aligned}
 \Var(S_{G,N_{m_k}})
 &=
 2\sum_{j=1}^{s_{N_{m_k}}}\frac1j+O(1)
 =
 2\log N_{m_k}+O(1),
 \end{aligned}
\]
and the sum of the third absolute centered moments is bounded
uniformly in $k$.  The classical Berry--Esseen inequality, followed
by Slutsky's theorem, therefore yields
\begin{equation}\label{eq:exact-example-Gaussian-clt}
 \frac{
 S_{G,N_{m_k}}-\E S_{G,N_{m_k}}
 }
 {\sqrt{2\log N_{m_k}}}
 \Rightarrow
 \Normal.
\end{equation}

\medskip
\noindent
\emph{Step 5: recombining the two independent blocks.}
For every $k$, the Gram--Schmidt product formula gives
\[
 2\log|\det A_{N_{m_k}}|
 =
 \log(N_{m_k}!)
 +
 S_{H,N_{m_k}}
 +
 S_{G,N_{m_k}}.
\]
After subtracting expectations, the deterministic term
$\log(N_{m_k}!)$ cancels.  Hence
\begin{equation}\label{eq:exact-example-sum-blocks}
 \begin{aligned}
 W_{N_{m_k}}^{\mathrm e}(A_{N_{m_k}})
 &=
 \frac{
 S_{H,N_{m_k}}-\E S_{H,N_{m_k}}
 }
 {\sqrt{2\log N_{m_k}}}
 +
 \frac{
 S_{G,N_{m_k}}-\E S_{G,N_{m_k}}
 }
 {\sqrt{2\log N_{m_k}}}.
 \end{aligned}
\end{equation}
By \eqref{eq:exact-example-heavy-centered}, the first term on the
right-hand side of \eqref{eq:exact-example-sum-blocks} converges to $\mathcal N\left(0,\frac{d_0}{8}\right)$.
By \eqref{eq:exact-example-Gaussian-clt}, the second term converges to
$\Normal$.  Moreover,
$S_{H,N_{m_k}}$ is
$\mathcal F_{r_{N_{m_k}}}$-measurable, whereas
$S_{G,N_{m_k}}$ is independent of
$\mathcal F_{r_{N_{m_k}}}$.  Thus the two terms are independent for
every $k$.  It follows that
\[
 W_{N_{m_k}}^{\mathrm e}(A_{N_{m_k}})
 \Rightarrow
 \mathcal N\left(0,1+\frac{d_0}{8}\right)
 =
 \mathcal N(0,\sigma^2),
\]
where $\sigma^2$ is defined in
\eqref{eq:exact-example-sigma-definition}.  This proves
\eqref{eq:exact-example-limit} and completes the proof.
\end{proof}

\subsection{Proof of Example~\ref{ex:variance-rate-sharpness}}
\label{subsec:proof-variance-rate-sharpness}

To prove the lower bound, we separate a Gaussian logarithmic
determinant from the contribution of the added terms.
The following lemma controls the error in $L^2$.

\begin{lemma}\label{lem:variance-spike-reduction}
Fix $\alpha>0$, and let $p_n,b_n$ be given by
\eqref{eq:variance-example-spike-parameters}.
Realize the entries of $A_n$ as
\[
 a_{ij}=\frac{g_{ij}+b_n\eta_{ij}\varepsilon_{ij}}
 {\sqrt{1+p_nb_n^2}},
\]
using independent copies of the variables in
\eqref{eq:variance-example-spike-variable}, and put
\[
 G_n=(g_{ij})_{1\le i,j\le n},
 \qquad K_n=\sum_{i,j=1}^n\eta_{ij}.
\]
On an enlarged probability space, there exist random variables
$Y_n,R_n$ such that $Y_n$ is independent of $K_n$,
$Y_n\overset{d}=\log|\det G_n|$, and
\begin{equation}\label{eq:variance-spike-representation}
 \mathcal L_n
 =-\frac n2\log(1+p_nb_n^2)
 +Y_n+\frac{K_n}{2}\log\left(1+\frac{b_n^2}{n}\right)+R_n.
\end{equation}
Moreover, for all sufficiently large $n$,
\begin{equation}\label{eq:variance-spike-remainder}
 \|R_n\|_2\le C_\alpha n^{-1/4}(\log n)^3.
\end{equation}
\end{lemma}

\begin{proof}
Write $\ell=\log n$ and $H_n=(b_n\eta_{ij}\varepsilon_{ij})$.
Both choices of parameters satisfy
\[
 n^2p_n\le\ell,
 \qquad b_n^2\le n\ell^{\alpha+1}
\]
for all sufficiently large $n$.
Since $K_n$ is binomial with parameters $n^2,p_n$,
\begin{equation}\label{eq:variance-spike-count-moments}
 \E K_n^j\le C_j\ell^j,\qquad 1\le j\le4.
\end{equation}

\medskip
\noindent\textbf{Step 1: construct the Gaussian term.}

Condition on $H_n$, and move its $r$ nonzero rows to the first
$r$ positions. Let $Q$ consist of the corresponding rows of
$G_n+H_n$ after this permutation.
Let $\widehat G_n$ be an independent standard Gaussian matrix,
and let $\widehat Q$ consist of its first $r$ rows.
Empty Gram determinants are interpreted as one.
Define
\[
 Y_n=\log|\det(G_n+H_n)|
 -\frac12\log\det(QQ^T)
 +\frac12\log\det(\widehat Q\widehat Q^T).
\]
By the Gaussian Gram--Schmidt decomposition used in
Lemma~\ref{lem:gaussian-terminal-block}, conditionally on $H_n$,
the difference of the first two terms has the law of
$\frac12\sum_{j=1}^{n-r}\log\chi_j^2$.
The last term is independent of this difference and has the law
of $\frac12\sum_{j=n-r+1}^{n}\log\chi_j^2$,
with independent chi-square variables in both sums.
Thus the conditional law of $Y_n$ is the same Gaussian
logarithmic determinant law for every $H_n$.
In particular, $Y_n$ is independent of $H_n$ and hence of $K_n$.

Set
\begin{equation}\label{eq:variance-spike-error-definition}
 R_n=\frac12\log\det(QQ^T)
 -\frac12\log\det(\widehat Q\widehat Q^T)
 -\frac{K_n}{2}\log\left(1+\frac{b_n^2}{n}\right).
\end{equation}
Since $A_n=(G_n+H_n)/\sqrt{1+p_nb_n^2}$,
this proves \eqref{eq:variance-spike-representation}.

\medskip
\noindent\textbf{Step 2: bound the fourth moment of the error.}

For any fixed $H_n$, each Gram--Schmidt distance $D_j$ of $Q$,
conditionally on the preceding rows, satisfies
\[
 \Pp(D_j<e^{-t})\le Ce^{-t},
 \qquad
 \Pp(D_j>e^t)\le(n+nb_n^2)e^{-2t},
 \qquad t\ge0.
\]
For the first bound, project onto any unit vector in the
remaining subspace and use the bounded density of a translated
standard normal variable. The second follows from
$\E D_j^2\le n+nb_n^2$.
Integrating these conditional tails gives
\[
 \bigl(\E(|\log D_j|^4\mid H_n)\bigr)^{1/4}
 \le C\log(e+n+nb_n^2)\le C_\alpha\ell.
\]
The same bound holds for the distances of $\widehat Q$.
Since $r\le K_n$, Minkowski's inequality applied to
\eqref{eq:variance-spike-error-definition} yields
\[
 \bigl(\E(|R_n|^4\mid H_n)\bigr)^{1/4}
 \le C_\alpha K_n\ell.
\]
Together with \eqref{eq:variance-spike-count-moments}, this gives
\begin{equation}\label{eq:variance-spike-error-fourth}
 \|R_n\|_4\le C_\alpha\ell^2.
\end{equation}

\medskip
\noindent\textbf{Step 3: estimate the error in $L^2$.}

The probability that two nonzero entries of $H_n$ lie in the
same row or column is at most
\[
 Cn^3p_n^2\le\frac{C\ell^2}{n}.
\]
Condition on a configuration without such a pair, and write
$k=K_n$. Then $r=k$, and the nonzero rows of $H_n$ are
orthogonal with squared norm $b_n^2$.
Expanding Gaussian second moments gives
\[
 \begin{aligned}
 \E\left(\left\|\frac{QQ^T}{n+b_n^2}-I_k\right\|_F^2
          \mathrel{\Big|}H_n\right)
 &=\frac{k(k+1)(n+2b_n^2)}{(n+b_n^2)^2},\\
 \E\left(\left\|\frac{\widehat Q\widehat Q^T}{n}-I_k\right\|_F^2
          \mathrel{\Big|}H_n\right)
 &=\frac{k(k+1)}{n}.
 \end{aligned}
\]
Both are at most $Ck^2/n$, where $\|\cdot\|_F$ denotes
the Frobenius norm.
Let $\mathcal E_n$ be the event that no two nonzero entries of
$H_n$ share a row or column and that both normalized Gram
matrices above differ from the identity by at most $1/2$
in operator norm.
Markov's inequality and \eqref{eq:variance-spike-count-moments}
imply
\[
 \Pp(\mathcal E_n^c)\le\frac{C\ell^2}{n}.
\]

On $\mathcal E_n$, the deterministic terms in
\eqref{eq:variance-spike-error-definition} cancel, so
\[
 R_n=\frac12\log\det\left(\frac{QQ^T}{n+b_n^2}\right)
 -\frac12\log\det\left(\frac{\widehat Q\widehat Q^T}{n}\right).
\]
Using $|\log(1+x)|\le2|x|$ for $|x|\le1/2$ on the eigenvalues,
we obtain
\[
 |R_n|\le C\sqrt{k}\left(
 \left\|\frac{QQ^T}{n+b_n^2}-I_k\right\|_F
 +\left\|\frac{\widehat Q\widehat Q^T}{n}-I_k\right\|_F
 \right).
\]
The preceding second-moment bounds therefore give
\[
 \E(R_n^2\1_{\mathcal E_n})
 \le\frac Cn\E K_n^3\le\frac{C\ell^3}{n}.
\]
For the complement, \eqref{eq:variance-spike-error-fourth}
and H\"older's inequality yield
\[
 \|R_n\1_{\mathcal E_n^c}\|_2
 \le\|R_n\|_4\Pp(\mathcal E_n^c)^{1/4}
 \le C_\alpha n^{-1/4}\ell^{5/2}.
\]
Combining the two estimates proves
\eqref{eq:variance-spike-remainder}.
\end{proof}

\begin{proof}[Proof of Example~\ref{ex:variance-rate-sharpness}]
The case $\alpha=0$ follows from
\eqref{eq:Gaussian-exact-var-BE-bound}.
Fix $\alpha>0$ and write $\ell=\log n$.
Constants below may depend on $\alpha$.

\medskip
\noindent\textbf{Step 1: verify the entry assumptions.}

The definition of $\xi_n$ gives $\E\xi_n=0$ and
$\E\xi_n^2=1$.
Its density is a mixture of three rescaled Gaussian densities,
so it is smooth and
\[
 \|f_{\xi_n}\|_\infty
 \le\frac{\sqrt{1+p_nb_n^2}}{\sqrt{2\pi}}\le C.
\]
Here we used $p_nb_n^2\to0$.
For the shifted Gaussian terms, split according to
$|G|\le b_n/2$ and its complement.
On the first event the absolute value of each shifted,
normalized variable is comparable to $b_n$; on the complement
use the Gaussian tail bound.
Since $p_nb_n^4=\ell^\alpha$ and
$\log(e+b_n)\sim\ell/2$, this gives
\[
 \E\frac{|\xi_n|^4}{\{\log(e+|\xi_n|)\}^\alpha}
 \le C\left(1+
 \frac{p_nb_n^4}{\{\log(e+b_n)\}^\alpha}\right)
 \le C.
\]
The remaining finitely many dimensions have Gaussian entries.
Thus all assumptions in Theorem~\ref{thm:variance-be} hold.

\medskip
\noindent\textbf{Step 2: compute the variance and third cumulant.}

Use $Y_n,K_n,R_n$ from Lemma~\ref{lem:variance-spike-reduction},
and put
\begin{equation}\label{eq:variance-example-main-term}
 \begin{aligned}
 J_n&=\frac12\log\left(1+\frac{b_n^2}{n}\right),\qquad
 X_n=Y_n-\E Y_n+J_n(K_n-n^2p_n),\\
 s_n^2&=\Var(X_n).
 \end{aligned}
\end{equation}
The Gaussian Gram--Schmidt decomposition and
\eqref{eq:chi-square-power-moment},
\eqref{eq:log-gamma-higher-derivatives} give
\[
 \kappa_j(Y_n)
 =(-1)^j(j-1)!
 \sum_{k=1}^n\sum_{m=0}^{\infty}(k+2m)^{-j},
 \qquad j\ge2,
\]
where $\kappa_j$ denotes the $j$th cumulant.
Consequently,
\[
 \Var(Y_n)=\frac12\ell+O(1),\qquad
 |\kappa_j(Y_n)|\le C(j-1)!,\quad j\ge3.
\]
Indeed, grouping terms with the same value of $k+2m$ bounds
the double sum by $\sum_{q\ge1}q^{1-j}$ for $j\ge3$.
Since $Y_n$ and $K_n$ are independent and
$K_n\sim\operatorname{Bin}(n^2,p_n)$,
\begin{equation}\label{eq:variance-example-main-cumulants}
 \begin{aligned}
 s_n^2
 &=\Var(Y_n)+n^2p_n(1-p_n)J_n^2,\\
 \kappa_3(X_n)
 &=\kappa_3(Y_n)+n^2p_n(1-p_n)(1-2p_n)J_n^3.
 \end{aligned}
\end{equation}

For $0<\alpha\le1$,
\[
 n^2p_n=\ell^\alpha,\qquad J_n=\frac12\log2.
\]
For $\alpha>1$,
\[
 n^2p_n=\frac{\ell}{(\log\ell)^2},\qquad
 J_n\sim\frac{\alpha-1}{4}\log\ell.
\]
In both cases, $s_n^2\asymp\ell$ and
$\kappa_3(X_n)\sim n^2p_nJ_n^3>0$ for all sufficiently large $n$.
In particular,
\begin{equation}\label{eq:variance-example-third-scale}
 \frac{\kappa_3(X_n)}{s_n^3}
 \asymp
 \begin{cases}
 \ell^{\alpha-3/2},&0<\alpha\le1,\\[1mm]
 (\log\ell)/\sqrt\ell,&\alpha>1.
 \end{cases}
\end{equation}
These quantities tend to zero. Also, $J_n/s_n\to0$ and
$n^2p_nJ_n^3\to\infty$, so
\begin{equation}\label{eq:variance-example-fourth-negligible}
 \frac{1+n^2p_nJ_n^4}{s_n\kappa_3(X_n)}
 \le C\left(
 \frac{1}{n^2p_nJ_n^3s_n}+\frac{J_n}{s_n}
 \right)\longrightarrow0.
\end{equation}

\medskip
\noindent\textbf{Step 3: detect the third-order term.}

The cumulant bound above gives a remainder $O(u^4)$ in the
third-order expansion of
$\log\E e^{iu(Y_n-\E Y_n)}$, uniformly for $|u|\le1/2$.
For the centered binomial term, expand
\[
 n^2\log\{1+p_n(e^{iuJ_n}-1)\}-iuJ_nn^2p_n.
\]
For $p_n\le1/4$ and $|uJ_n|\le1/2$, the remainder is
$O(n^2p_n|uJ_n|^4)$.
Independence and \eqref{eq:variance-example-main-cumulants}
therefore yield, uniformly for $0\le t\le1$,
\begin{equation}\label{eq:variance-example-characteristic-expansion}
 \log\E e^{itX_n/s_n}
 =-\frac{t^2}{2}
 -\frac{i\kappa_3(X_n)t^3}{6s_n^3}
 +O\left(\frac{1+n^2p_nJ_n^4}{s_n^4}t^4\right).
\end{equation}

Choose the fixed function
\[
 f(x)=\int_0^1t(1-t)^2\sin(tx)\,\dd t.
\]
It is bounded and odd, and
\[
 f'(x)=\int_0^1t^2(1-t)^2\cos(tx)\,\dd t.
\]
Two integrations by parts give
$|f'(x)|\le C/(1+x^2)$.
Thus $f$ is Lipschitz and $f'\in L^1(\R)$.
For $N\sim\Normal$, we have $\E f(N)=0$.
Integration by parts against distribution functions gives,
for any real-valued random variable $W$,
\begin{equation}\label{eq:variance-example-test-distance}
 |\E f(W)|\le\|f'\|_{L^1(\R)}\dk(W,\Normal).
\end{equation}

Exponentiating \eqref{eq:variance-example-characteristic-expansion},
taking imaginary parts, and integrating against $t(1-t)^2$,
we obtain
\[
 \begin{aligned}
 \E f(X_n/s_n)
 &=-\frac{\kappa_3(X_n)}{6s_n^3}
 \int_0^1t^4(1-t)^2e^{-t^2/2}\,\dd t\\
 &\quad+O\left(
 \frac{1+n^2p_nJ_n^4}{s_n^4}
 +\frac{\kappa_3(X_n)^2}{s_n^6}
 \right).
 \end{aligned}
\]
By \eqref{eq:variance-example-third-scale} and
\eqref{eq:variance-example-fourth-negligible}, the error is
$o(\kappa_3(X_n)/s_n^3)$.
Since the integral is strictly positive,
\begin{equation}\label{eq:variance-example-test-lower}
 |\E f(X_n/s_n)|\ge c\frac{\kappa_3(X_n)}{s_n^3}
\end{equation}
for all sufficiently large $n$.

\medskip
\noindent\textbf{Step 4: restore the true mean and variance.}

By \eqref{eq:variance-spike-representation},
\[
 \mathcal L_n-\E\mathcal L_n
 =X_n+R_n-\E R_n.
\]
The triangle inequality in $L^2$ gives
\[
 |\sigma_n-s_n|
 \le\|R_n-\E R_n\|_2\le\|R_n\|_2.
\]
Since $s_n\asymp\sqrt\ell$, Lemma~\ref{lem:variance-spike-reduction}
implies $\sigma_n/s_n\to1$ and
\[
 \begin{aligned}
 \|W_n^{\mathrm v}(A_n)-X_n/s_n\|_2
 &\le\frac{\|R_n-\E R_n\|_2+|\sigma_n-s_n|}{\sigma_n}\\
 &\le Cn^{-1/4}\ell^{5/2}
 =o\left(\frac{\kappa_3(X_n)}{s_n^3}\right).
 \end{aligned}
\]
The Lipschitz property of $f$ and
\eqref{eq:variance-example-test-lower} therefore give
\[
 |\E f(W_n^{\mathrm v}(A_n))|
 \ge c\frac{\kappa_3(X_n)}{s_n^3}.
\]
Applying \eqref{eq:variance-example-test-distance} and then
\eqref{eq:variance-example-third-scale} proves
\eqref{eq:variance-rate-lower-bound}.
This completes the proof.
\end{proof}

\end{appendix}

\begin{acks}[Acknowledgments]
Liu S.H. was partially supported by the Fundamental Research Funds for the Central Universities DUT25RC(3)133.
Shao Q.M. was partially supported by National Nature Science Foundation of China NSFC 12031005 and Shenzhen Outstanding Talents Training Fund, China.
\end{acks}



\bibliographystyle{imsart-nameyear} 
\bibliography{refs.bib}       

\begin{thebibliography}{18}

\bibitem[\protect\citeauthoryear{Bao, Pan and Zhou}{2015}]{BaoPanZhou2015}
\begin{barticle}[author]
\bauthor{\bsnm{Bao},~\bfnm{Zhigang}\binits{Z.}}, \bauthor{\bsnm{Pan},~\bfnm{Guangming}\binits{G.}} \AND \bauthor{\bsnm{Zhou},~\bfnm{Wang}\binits{W.}}
(\byear{2015}).
\btitle{The logarithmic law of random determinant}.
\bjournal{Bernoulli}
\bvolume{21}
\bpages{1600--1628}.
\bdoi{10.3150/14-BEJ615}
\end{barticle}
\endbibitem

\bibitem[\protect\citeauthoryear{Dembo}{1989}]{Dembo1989}
\begin{barticle}[author]
\bauthor{\bsnm{Dembo},~\bfnm{Amir}\binits{A.}}
(\byear{1989}).
\btitle{On random determinants}.
\bjournal{Quarterly of Applied Mathematics}
\bvolume{47}
\bpages{185--195}.
\end{barticle}
\endbibitem

\bibitem[\protect\citeauthoryear{Dumitriu and Edelman}{2002}]{DumitriuEdelman2002}
\begin{barticle}[author]
\bauthor{\bsnm{Dumitriu},~\bfnm{Ioana}\binits{I.}} \AND \bauthor{\bsnm{Edelman},~\bfnm{Alan}\binits{A.}}
(\byear{2002}).
\btitle{Matrix models for beta ensembles}.
\bjournal{Journal of Mathematical Physics}
\bvolume{43}
\bpages{5830--5847}.
\end{barticle}
\endbibitem

\bibitem[\protect\citeauthoryear{Forsythe and Tukey}{1952}]{ForsytheTukey1952}
\begin{barticle}[author]
\bauthor{\bsnm{Forsythe},~\bfnm{George~E.}\binits{G.~E.}} \AND \bauthor{\bsnm{Tukey},~\bfnm{John~W.}\binits{J.~W.}}
(\byear{1952}).
\btitle{The extent of $n$ random unit vectors}.
\bjournal{Bulletin of the American Mathematical Society}
\bvolume{58}
\bpages{502}.
\end{barticle}
\endbibitem

\bibitem[\protect\citeauthoryear{Girko}{1979}]{Girko1979}
\begin{barticle}[author]
\bauthor{\bsnm{Girko},~\bfnm{V.~L.}\binits{V.~L.}}
(\byear{1979}).
\btitle{A central limit theorem for random determinants}.
\bjournal{Theory of Probability and Its Applications}
\bvolume{24}
\bpages{729--740}.
\end{barticle}
\endbibitem

\bibitem[\protect\citeauthoryear{Girko}{1990}]{GirkoBook1990}
\begin{bbook}[author]
\bauthor{\bsnm{Girko},~\bfnm{V.~L.}\binits{V.~L.}}
(\byear{1990}).
\btitle{Theory of Random Determinants}.
\bseries{Mathematics and Its Applications (Soviet Series)}
\bvolume{45}.
\bpublisher{Kluwer Academic Publishers}, \baddress{Dordrecht}.
\bnote{Translated from the Russian}.
\end{bbook}
\endbibitem

\bibitem[\protect\citeauthoryear{Girko}{1997}]{Girko1997}
\begin{barticle}[author]
\bauthor{\bsnm{Girko},~\bfnm{V.~L.}\binits{V.~L.}}
(\byear{1997}).
\btitle{A refinement of the central limit theorem for random determinants}.
\bjournal{Theory of Probability and Its Applications}
\bvolume{42}
\bpages{121--129}.
\end{barticle}
\endbibitem

\bibitem[\protect\citeauthoryear{Goodman}{1963}]{Goodman1963}
\begin{barticle}[author]
\bauthor{\bsnm{Goodman},~\bfnm{N.~R.}\binits{N.~R.}}
(\byear{1963}).
\btitle{The distribution of the determinant of a complex {Wishart} distributed matrix}.
\bjournal{Annals of Mathematical Statistics}
\bvolume{34}
\bpages{178--180}.
\end{barticle}
\endbibitem

\bibitem[\protect\citeauthoryear{Klenke}{2014}]{Klenke2014}
\begin{bbook}[author]
\bauthor{\bsnm{Klenke},~\bfnm{Achim}\binits{A.}}
(\byear{2014}).
\btitle{Probability Theory: A Comprehensive Course},
\bedition{2} ed.
\bseries{Universitext}.
\bpublisher{Springer}.
\end{bbook}
\endbibitem

\bibitem[\protect\citeauthoryear{Li et~al.}{2026}]{li2026logarithmic}
\begin{barticle}[author]
\bauthor{\bsnm{Li},~\bfnm{Yanpeng}\binits{Y.}}, \bauthor{\bsnm{Liu},~\bfnm{Zhi}\binits{Z.}}, \bauthor{\bsnm{Xie},~\bfnm{Jiahui}\binits{J.}} \AND \bauthor{\bsnm{Zhou},~\bfnm{Wang}\binits{W.}}
(\byear{2026}).
\btitle{The logarithmic law of sample correlation matrices}.
\bjournal{arXiv preprint arXiv:2603.19800}.
\end{barticle}
\endbibitem

\bibitem[\protect\citeauthoryear{Livshyts, Paouris and Pivovarov}{2016}]{livshyts2016sharp}
\begin{barticle}[author]
\bauthor{\bsnm{Livshyts},~\bfnm{Galyna}\binits{G.}}, \bauthor{\bsnm{Paouris},~\bfnm{Grigoris}\binits{G.}} \AND \bauthor{\bsnm{Pivovarov},~\bfnm{Peter}\binits{P.}}
(\byear{2016}).
\btitle{On sharp bounds for marginal densities of product measures}.
\bjournal{Israel Journal of Mathematics}
\bvolume{216}
\bpages{877--889}.
\end{barticle}
\endbibitem

\bibitem[\protect\citeauthoryear{Nguyen and Vu}{2014}]{NguyenVu2014}
\begin{barticle}[author]
\bauthor{\bsnm{Nguyen},~\bfnm{Hoi~H.}\binits{H.~H.}} \AND \bauthor{\bsnm{Vu},~\bfnm{Van~H.}\binits{V.~H.}}
(\byear{2014}).
\btitle{Random matrices: Law of the determinant}.
\bjournal{Annals of Probability}
\bvolume{42}
\bpages{146--167}.
\bdoi{10.1214/12-AOP791}
\end{barticle}
\endbibitem

\bibitem[\protect\citeauthoryear{Nyquist, Rice and Riordan}{1954}]{NyquistRiceRiordan1954}
\begin{barticle}[author]
\bauthor{\bsnm{Nyquist},~\bfnm{Harry}\binits{H.}}, \bauthor{\bsnm{Rice},~\bfnm{S.~O.}\binits{S.~O.}} \AND \bauthor{\bsnm{Riordan},~\bfnm{John}\binits{J.}}
(\byear{1954}).
\btitle{The distribution of random determinants}.
\bjournal{Quarterly of Applied Mathematics}
\bvolume{12}
\bpages{97--104}.
\end{barticle}
\endbibitem

\bibitem[\protect\citeauthoryear{Pr{\'e}kopa}{1967}]{Prekopa1967}
\begin{barticle}[author]
\bauthor{\bsnm{Pr{\'e}kopa},~\bfnm{Andr{\'a}s}\binits{A.}}
(\byear{1967}).
\btitle{On random determinants. {I}}.
\bjournal{Studia Scientiarum Mathematicarum Hungarica}
\bvolume{2}
\bpages{125--132}.
\end{barticle}
\endbibitem

\bibitem[\protect\citeauthoryear{Rouault}{2007}]{Rouault2007}
\begin{barticle}[author]
\bauthor{\bsnm{Rouault},~\bfnm{Alain}\binits{A.}}
(\byear{2007}).
\btitle{Asymptotic behavior of random determinants in the {Laguerre}, {Gram} and {Jacobi} ensembles}.
\bjournal{ALEA. Latin American Journal of Probability and Mathematical Statistics}
\bvolume{3}
\bpages{181--230}.
\end{barticle}
\endbibitem

\bibitem[\protect\citeauthoryear{Szekeres and Tur{\'a}n}{1937}]{SzekeresTuran1937}
\begin{barticle}[author]
\bauthor{\bsnm{Szekeres},~\bfnm{G.}\binits{G.}} \AND \bauthor{\bsnm{Tur{\'a}n},~\bfnm{P.}\binits{P.}}
(\byear{1937}).
\btitle{On an extremal problem in the theory of determinants}.
\bjournal{Math. Naturwiss. Anz. Ungar. Akad. Wiss.}
\bvolume{56}
\bpages{796--806}.
\bnote{In Hungarian}.
\end{barticle}
\endbibitem

\bibitem[\protect\citeauthoryear{Tao and Vu}{2006}]{TaoVu2006}
\begin{barticle}[author]
\bauthor{\bsnm{Tao},~\bfnm{Terence}\binits{T.}} \AND \bauthor{\bsnm{Vu},~\bfnm{Van~H.}\binits{V.~H.}}
(\byear{2006}).
\btitle{On random {$\pm 1$} matrices: Singularity and determinant}.
\bjournal{Random Structures \& Algorithms}
\bvolume{28}
\bpages{1--23}.
\end{barticle}
\endbibitem

\bibitem[\protect\citeauthoryear{Tao and Vu}{2012}]{tao2012central}
\begin{barticle}[author]
\bauthor{\bsnm{Tao},~\bfnm{Terence}\binits{T.}} \AND \bauthor{\bsnm{Vu},~\bfnm{Van}\binits{V.}}
(\byear{2012}).
\btitle{A central limit theorem for the determinant of a Wigner matrix}.
\bjournal{Advances in Mathematics}
\bvolume{231}
\bpages{74--101}.
\end{barticle}
\endbibitem

\end{thebibliography}

\end{document}